\documentclass[11pt]{article}
\usepackage{fullpage}
\usepackage{preamble}
\usepackage{import}
\usepackage{amsmath}
\usepackage{bbm}
\usepackage[utf8]{inputenc}
\usepackage{mathtools, bm}
\usepackage{amssymb, bm}
\usepackage{xcolor}
\usepackage{emptypage}
\usepackage{pdfpages}
\date{}

\usepackage[
backend=biber,
style=alphabetic,
sorting=nyt,
maxnames=4,
maxalphanames=4,
backref=true
]{biblatex}
\title{The fine spectral expansion of the Rankin--Selberg period}
\author{Paul Boisseau}
\date{\today}

\begin{document}
\maketitle
\begin{abstract}
    We state and prove the spectral expansion of the theta series attached to the Rankin--Selberg spherical variety $(\GL_{n+1} \times \GL_n)/\GL_n$. This is a key result towards the fine spectral expansion of the Jacquet--Rallis trace formula. Our expansion is written in terms of regularized Rankin--Selberg periods for non-tempered automorphic representations, which we show compute special values of $L$-functions. The proof relies on shifts of contours of integration à la Langlands. We also establish two technical but crucial results on bounds and singularities for discrete Eisenstein series of $\GL_n$ in the positive Weyl chamber.
\end{abstract}

\setcounter{tocdepth}{1}
\tableofcontents

\section{Introduction}
\label{sec:intro}
\subsection{Motivations}

\subsubsection{Spectral decomposition of the Rankin--Selberg spherical variety}

Let $n \geq 1$ be an integer and let $F$ be a number field with ring of adeles $\bA$. For any algebraic group $\LAG$ over $F$, set $[\LAG]:=\LAG(F) \backslash \LAG(\bA)$. Set $G=\GL_n \times \GL_{n+1}$ and $H=\GL_n$, both considered as algebraic groups over $F$. Embed $H$ as a subgroup of $G$ using the map $h \mapsto \left(h,\begin{pmatrix}
    h & \\
    & 1
\end{pmatrix}\right)$, and consider the affine $G$-spherical \emph{Rankin--Selberg} variety $X=G/H$. For any Schwartz function $\Phi$ on $X(\bA)$ we may form the theta series
\begin{equation*}
    \Theta^X_\Phi(g)=\sum_{x \in X(F)} \Phi(g^{-1}.x), \quad g \in G(\bA).
\end{equation*}
The series $\Theta^X_\Phi$ form a $G(\bA)$-invariant subspace of the space of smooth functions on $[G]$. Our goal is to describe this representation in terms of the automorphic spectrum of $G$. 

We may reformulate the problem as follows. Because $X(\bA)=G(\bA) \backslash H(\bA)$, any Schwartz function on $X(\bA)$ can be obtained from a Schwartz function $f$ on $G(\bA)$ via the integral
\begin{equation*}
    \Phi(x)=\int_{H(\bA)} f(gh) dh,
\end{equation*}
where $dh$ is a Haar-measure on $H(\bA)$ and $g$ is any representative of $x$ in $G(\bA)$. We can consider the automorphic kernel
\begin{equation}
    \label{eq:kernel_function}
    K_f(x,y)=\sum_{\gamma \in G(F)} f(x^{-1} \gamma y), \quad x, y \in G(\bA),
\end{equation}
and form the Rankin--Selberg automorphic kernel
\begin{equation*}
    J^H(g,f)=\int_{[H]} K_f(g,h) dh, \quad g \in G(\bA),
\end{equation*}
where $[H]$ is equipped with the quotient measure. Because $X(F)=G(F) \backslash H(F)$, we have $\Theta^X_\Phi(g)=J^H(g,f)$ and we now need to write the spectral decomposition of the distribution $f \mapsto J^H(g,f)$.

\subsubsection{The Jacquet--Rallis relative trace formula}

The motivation to study the distribution $J^H(g,f)$ stems from the Jacquet--Rallis relative trace formula introduced in \cite{JR}. More precisely, let $E$ be a quadratic extension of $F$, and write $G_E$ and $H_E$ for the restriction of scalars $\Res_{E/F}(\GL_{n} \times \GL_{n+1})$ and $\Res_{E/F} GL_{n}$ respectively. We have the distribution $J^{H_E}(g,f)$. Inside $G_E$ lies the subgroup $G'=\GL_n \times \GL_{n+1}$. The integral along $[G']$ is the \emph{Flicker--Rallis} studied in \cite{Fli}. We can write any $g' \in G'(\bA)$ as $g'=(g'_n,g'_{n+1})$. Let $\eta$ be the quadratic character associated to $E/F$ by class field theory. For $f$ a Schwartz function on $G_E(\bA)$, the Jacquet--Rallis relative trace formula for general linear groups introduced in \cite{Zyd20} is the regularized version of the integral
\begin{equation}
    \label{eq:JR_RTF}
    \int_{[G']} J^{H_E}(g',f) \eta(\det(g'_n))^{n+1}\eta(\det(g'_{n+1}))^{n}dg'.
\end{equation}
It plays a significant role in recent works on the global Gan--Gross--Prasad and Ichino--Ikeda conjectures for unitary groups, in particular \cite{Zhang2}, \cite{BPLZZ}, \cite{BPCZ} and \cite{BPC}. To any cuspidal datum $\chi$ of $G_E$ we can attach a distribution $J^{H_E}_\chi(g,f)$ obtained by integrating the $K_\chi$ part of the kernel. A key result obtained in \cite{BPCZ} is the spectral decomposition of $J^{H_E}_\chi(g,f)$ for certain \emph{regular} cuspidal data $\chi$. Combined with an analog result for the Flicker--Rallis period, the authors of \cite{BPCZ} are able to derive the regular part of the spectral expansion of the Jacquet--Rallis trace formula. This turns out to be sufficient to prove the conjecture of \cite{GGP} for cuspidal generic representations of unitary groups. However, to tackle the non-tempered version of the Gan--Gross--Prasad conjecture introduced in \cite{GGP2} it is necessary to understand all residual and Eisenstein contributions to the trace formula.

In this regard, our paper provides the fine spectral expansion of the distribution $J^{H_E}_\chi(g,f)$ with no restriction on cuspidal data. This is an important step towards the fine spectral expansion of the trace formula itself. In \cite{Ch}, Chaudouard proved the corresponding result for the Flicker--Rallis distributions $J^{G'}(g,f)$ obtained by considering the symmetric space $G_E(\bA) / G'(\bA)$. It now remains to combine the two expansions to derive that of \eqref{eq:JR_RTF}, which we will do in a future work.

\subsection{The fine spectral expansion of the Rankin--Selberg period}

\subsubsection{Preliminary notations}

Before stating the main result of this paper, we fix some notations. For every place $v$ of $F$, $F_v$ is the localization of $F$ at $v$. We say that a parabolic subgroup $P$ of $G$ is semi-standard if it contains the torus of diagonal matrices, and standard if it contains the Borel subgroup of upper triangular matrices. In this last case, it admits a standard Levi decomposition $P=M_P N_P$, where $N_P$ is the unipotent radical of $P$. Let $\Pi_\disc(M_P)$ be the set of discrete irreducible automorphic representations of $M_P$, with trivial central character on the central subgroup $A_P^\infty$ (see \S\ref{subsubsec:automorphic_quotients}). If $\pi \in \Pi_\disc(M_P)$, we can write its decomposition into local components $\pi=\otimes'_v \pi_v$. We also have the space $\cA_{P,\pi}(G)$ of automorphic forms on $A_P^\infty M_P(F) N_P(\bA) \backslash G(\bA)$ induced from $\pi$ (see \S\ref{subsubsec:discrete_automorphic_rep}). Let $\fa_{P,\cc}^*$ be the complex vector space of unramified characters of $M_P(\bA)$, and let $i \fa_{P}^*$ be its real subspace of unitary characters. For $\lambda \in \fa_{P,\cc}^*$, we have a map $\varphi \mapsto \varphi_\lambda$ that identifies $\cA_{P,\pi}(G)$ with the induction $\cA_{P,\pi,\lambda}(G)$ of $\pi \otimes \lambda$. A Schwartz function $f \in \cS(G(\bA))$ acts on $\cA_{P,\pi,\lambda}(G)$ and thus on $\cA_{P,\pi}(G)$ by transporting the structure, and we denote this action by $I_P(f,\lambda)$. We finally write $E(\varphi,\lambda)$ for the Eisenstein series induced from $\varphi_\lambda$ (see \S\ref{subsubsec:Eisenstein}).

\subsubsection{Relevant inducing data}
\label{subsubsec:relevant_intro}

Let $k \geq 1$. By \cite{MW89}, any discrete automorphic representation $\pi \in \Pi_\disc(\GL_k)$ is obtained by taking residues of Eisenstein series of automorphic forms in an induction $\cA_{P,\sigma^{\boxtimes d}}(\GL_k)$, where $P$ is a standard parabolic subgroup with $M_P=\GL_r^d$, $\sigma$ is a cuspidal automorphic representation of $\GL_r$ and $r$ and $d$ are integers with $rd=k$. We then write $\pi=\Speh(\sigma,d)$. By \cite{Langlands}, inductions of discrete automorphic representations exhaust the spectral decomposition of $L^2([\GL_k])$. We will refer to these representations as being of \emph{Arthur type}.

By analogy with the local notion of derivatives introduced in \cite{BZ}, we define the automorphic derivative of $\pi=\Speh(\sigma,d)$ to be
\begin{equation*}
    \pi^-=\Speh(\sigma,d-1) \in \Pi_\disc(\GL_{r(d-1)}).
\end{equation*}
Note that if $\pi$ is cuspidal, i.e. if $d=1$, then $\pi^-$ is the trivial representation of the trivial group. We also write $\pi^\vee$ for the dual of $\pi$, and we easily check that $\pi^{-,\vee}=\pi^{\vee,-}$.

Our spectral expansion of $J^H$ will be indexed by a set $\Pi_H$ of \emph{relevant inducing data}. This is the set of triples $(I,P,\pi)$ satisfying the following desiderata.
\begin{itemize}
    \item $I$ is a tuple of non-negative integers $(n_+,n_1,n_2,n_-)$ such that $n_2^-:=n-n_+-n_1-n_-$ and $n_1^-:=n+1-n_+-n_2-n_-$ are non-negative. We then let $P_I$ be the standard parabolic subgroup of $G$ with standard Levi subgroup 
     \begin{equation}
        \label{eq:P_decompo}
        M_I:=\left( \GL_{n_+} \times \GL_{n_1} \times \GL_{n_2^-} \times \GL_{n_-} \right) \times \left( \GL_{n_+} \times \GL_{n_1^-} \times \GL_{n_2} \times \GL_{n_-} \right).
    \end{equation} 
    \item $P$ is a standard parabolic subgroup of $G$ included in $P_I$.
    \item $\pi$ is a discrete automorphic representation of $M_P$ which, with respect to \eqref{eq:P_decompo}, decomposes as
    \begin{equation}
        \label{eq:pi_decompo_intro}
        \left( \boxtimes_{i=1}^{m_+} \pi_{+,i} \boxtimes_{i=1}^{m_1} \pi_{1,i} \boxtimes_{i=1}^{m_2} \pi_{2,i}^{-,\vee} \boxtimes_{i=1}^{m_-} \pi_{-,i} \right) \boxtimes  \left( \boxtimes_{i=1}^{m_+} \pi_{+,i}^\vee \boxtimes_{i=1}^{m_1} \pi_{1,i}^{-,\vee} \boxtimes_{i=1}^{m_2} \pi_{2,i} \boxtimes_{i=1}^{m_-} \pi_{-,i}^\vee \right),
    \end{equation}
    where $m_+, m_1, m_2$ and $m_-$ are non-negative integers, and all the $\pi_{\hdots}$ are discrete automorphic representations of some general linear groups. 
\end{itemize} 

Let $(I,P,\pi) \in \Pi_H$. We associate to this triple three additional pieces of data. 
\begin{itemize}
    \item Let $\fa_{\pi,\cc}^* \subset \fa_{P,\cc}^*$ be the subspace of unramified characters $\lambda$ of $M_P(\bA)$ such that $(I,P,\pi \otimes \lambda)$ belongs to $\Pi_H$ if we lift the requirement that the restriction of the central character of $\pi \otimes \lambda$ to $A_P^\infty$ is trivial. We refer to \eqref{eq:a_pi_defi} for an explicit description of $\fa_{\pi,\cc}^*$ in coordinates. We denote by $i \fa_{\pi}^*$ its subset of unitary characters.
    \item With respect to the basis coming from the lattice of algebraic characters in $\fa_{P_I,\cc}^*$ (see \S\ref{subsubsec:reductive_gp}), let $\underline{\rho}_\pi$ be the element in $\fa_{P_I,\cc}^* \subset \fa_{P,\cc}^*$ with coordinates
    \begin{equation*}
        \underline{\rho}_\pi=((1/4,0,0,-1/4),(1/4,0,0,-1/4)).
    \end{equation*}
    Note that $\underline{\rho}_\pi$ does not belong to $\fa_{\pi,\cc}^*$ unless $n_+=n_-=0$, in which case it is zero.
    \item Let $W(\pi)$ be the subset of the Weyl group of $G$ defined in \S\ref{sec:functional_equation}. It has the property that, for any $w \in W(\pi)$, if we write $w.P$ for the unique standard parabolic subgroup of $G$ with standard Levi subgroup $w M_P w^{-1}$, and $w.\pi$ for the discrete automorphic representation of $w M_P w^{-1}$ obtained by conjugation, we have $(I,w.P,w.\pi) \in \Pi_H$.
\end{itemize}
Note that $i \fa_\pi^*$, $\underline{\rho}_\pi$ and $W(\pi)$ really depend on $(I,P,\pi)$, but we only highlight the relation to $\pi$ to ease notations.

\begin{ex}
    \label{ex:cuspi}
    Take $I=(0,n,n+1,0)$. Then $P_I$ is $G$. It follows that triples $(I,P,\pi) \in \Pi_H$ are in bijection with couples $(P,\pi)$ where $P$ is a standard parabolic subgroup of $G$ and $\pi$ is a cuspidal automorphic representation of $M_P$. Indeed, with the notation of \eqref{eq:pi_decompo_intro}, $\pi_{2,i}^{-,\vee}$ and $\pi_{1,j}^{-,\vee}$ for $1 \leq i \leq m_2$ and $1 \leq j \leq m_1$ are representations of the trivial group. Moreover, $\fa_{\pi,\cc}^*$ is $\fa_{P,\cc}^*$, $\underline{\rho}_\pi=0$ and $W(\pi)$ is the subset $W(P)$ of Weyl elements permutating the blocks of $M_P$.
\end{ex}

\begin{ex}
    \label{ex:resi}
    Take $I=(0,0,n+1,0)$. Then $P_I$ is again $G$. However, if $(I,P,\pi) \in \Pi_H$, then $\pi$ is of the form $\left( \boxtimes_{i=1}^{m_2} \pi_{2,i}^{-,\vee} \right) \boxtimes \left( \boxtimes_{i=1}^{m_2} \pi_{2,i} \right)$. This is only possible if $m_2=1$ and $\pi_{2,1}=\Speh(\chi,n+1)$ for $\chi$ an automorphic character of $\GL_1$, so that $\pi$ is the character $\chi^\vee \circ \det \boxtimes \chi \circ \det$ of $G$. By a remark of \cite[Section~9]{GGP2}, all residual representations of $G$ which admit a $H(\bA)$-invariant linear form are of this shape. In this case, $P=G$, $\fa_{\pi,\cc}^*$ is the diagonal subspace $(\lambda,-\lambda)$ in $\fa_{G,\cc}^*$, $\underline{\rho}_\pi=0$ and $W(\pi)$ is trivial.
\end{ex}

\begin{ex}
        \label{ex:scalar}
    We now give a non-Arthur example coming from $\Pi_H$. Take $I=(n,0,1,0)$ so that $P_I$ has standard Levi $M_I=\GL_n \times (\GL_n \times \GL_1)$. Then for $(I,P,\pi)$ to belong to $\Pi_H$, $P$ has to be a standard parabolic subgroup of $G$ with $M_P$ of the form $M_+ \times (M_+ \times \GL_1)$, and $\pi$ to be a discrete automorphic representation of $M_P$ which decomposes as $\pi_+ \boxtimes (\pi_+^\vee \boxtimes \chi)$ accordingly, where $\pi_+ \in \Pi_\disc(M_+)$ and $\chi$ is an automorphic character of $\GL_1$. Relatively to $P \subset P_I$, $\fa_{\pi,\cc}^*$ is the subspace $(\lambda,(-\lambda,\mu))$ where $\lambda \in \fa_{M_+,\cc}^*$ and $\mu \in \fa_{\GL_1,\cc}^*$, $\underline{\rho}_\pi=(1/4,(1/4,0))$. Finally, if we set $P_+=P \cap H$, then $W(\pi)=W(P_+) \subset H \subset G$. 
\end{ex}

\subsubsection{Regularized periods}

We now give some motivation for the definition of $\Pi_H$. For $\varphi$ an automorphic form on $G$, consider the (a priori non-convergent) \emph{Rankin--Selberg} period
\begin{equation*}
    \cP_H(\varphi)=\int_{[H]} \varphi(h) dh.
\end{equation*}
If $\varphi$ belongs to a cuspidal automorphic representation $\pi$ of $G$ (so that the integral converges), a celebrated theorem of \cite{JPSS83} states that $\cP_H$ vanishes if an only if the central value of the Rankin--Selberg $L$-function $L(1/2,\pi)$ does. In \cite{GGP2}, Gan, Gross and Prasad conjectured that the restriction of (a suitable regularization of) $\cP_H$ to an arbitrary automorphic representation $\pi$ of $G$ of Arthur type should vanish unless the Arthur parameter of $\pi$ was \emph{relevant} and a special value of a certain quotient of $L$-functions was non-zero. If we specify the definition of $\Pi_H$ to the case $n_+=n_-=0$, then these relevant representations singled out by \cite{GGP2} are exactly the inductions $\cA_{P,\pi,\lambda}(G)$ with $\lambda \in i \fa_{\pi}^*$. They are those expected to appear in the spectral expansion of $J^H$. However, it turns out that they are not enough to fully describe it and that additional representations are needed. They are those of the form $\cA_{P,\pi,\lambda}(G)$ for $(I,P,\pi) \in \Pi_H$ and $\lambda \in i\fa_{\pi}^*-\underline{\rho}_\pi$. These representations are not of Arthur type as soon as $\underline{\rho}_\pi \neq 0$.

Our first result is the definition of a regularization of $\cP_H$ on these inductions $\cA_{P,\pi,\lambda}(G)$. By \cite{MW89}, there exist $P_\pi$ a standard parabolic subgroup of $G$, $\sigma_\pi$ a cuspidal automorphic representation of $M_{P_\pi}$ and $\nu_\pi \in \fa_{P_\pi}^*$ such that $\cA_{P,\pi}(G)$ is spanned by residues of Eisenstein series induced from $\cA_{P_\pi,\sigma_\pi}(G)$ at $-\nu_\pi$. We denote this map by $E^{P,*}(\cdot,0)$. For $\phi \in \cA_{P_\pi,\sigma_\pi}(G)$, we can form the Eisenstein series $E(\phi,\lambda)$ and further take its Rankin--Selberg Zeta integral $Z(E(\phi,\lambda))$. It is a meromorphic function in $\lambda \in \fa_{P_\pi,\cc}^*$ by \cite{IY}. If we identify $\fa_{\pi,\cc}^*-\underline{\rho}_\pi-\nu_\pi$ as an affine subspace of $\fa_{P_\pi,\cc}^*$, we see that it is contained in a finite union of singularities of $Z(E(\phi,\lambda))$ which are all affine hyperplanes. By taking iterated residues, we obtain a meromorphic function $\Res \; Z(E(\phi,\lambda))$ on $\fa_{\pi,\cc}^*-\underline{\rho}_\pi-\nu_\pi$. A priori, this construction depends on the order of the residues taken. In the following theorem, by "for $\lambda \in \fa_{\pi,\cc}^*-\underline{\rho}_\pi$" in general position we mean that it lies outside of a countable union of affine hyperplanes.

\begin{theorem}
    \label{thm:GGP_global_intro}
    Let $(I,P,\pi) \in \Pi_H$. The following assertions hold.
    \begin{itemize}
        \item For $\lambda \in \fa_{\pi,\cc}^*-\underline{\rho}_\pi$ in general position, the residue $\mathrm{Res} \; Z(E(\phi,\lambda-\nu_\pi))$ is independent of the order and factors through $E^{P,*}(\cdot,\lambda-\nu_\pi) : \cA_{P_\pi,\sigma_\pi,\lambda-\nu_\pi}(G)\twoheadrightarrow \cA_{P,\pi,\lambda}(G)$. We denote by $\cP_\pi(\cdot,\lambda)$ the resulting $H(\bA)$-invariant linear form on $\cA_{P,\pi,\lambda}(G)$.
        \item For $\varphi \in \cA_{P,\pi}(G)$, the map $\lambda \in \fa_{\pi,\cc}^*-\underline{\rho}_\pi \mapsto \cP_\pi(\varphi,\lambda)$ is meromorphic, and for $\lambda$ in general position the map $\varphi \in \cA_{P,\pi,\lambda}(G) \mapsto \cP_{\pi}(\varphi,\lambda)$ is continuous.
        \item For $\varphi=E^{P,*}(\phi,\lambda-\nu_\pi) \in \cA_{P,\pi,\lambda}(G)$ with $\phi=\otimes \phi_v$ and $\lambda$ in general position, there exists a finite set of places $\tS$ such that
        \begin{equation*}
            \cP_{\pi}(\varphi,\lambda)=\cL(\pi,\lambda)  \prod_{v \in \tS} Z_{\sigma_\pi,v}^\natural(\phi_v,\lambda-\nu_\pi),
        \end{equation*}
        where $\cL(\pi,\lambda)$ is the quotient of Rankin--Selberg $L$-functions described in \S\ref{subsubsec:proof_main_theorem}, and the linear forms $Z_{\sigma_\pi,v}^\natural(\phi_v,\lambda-\nu_\pi)$ are residues of local Zeta integrals built on the inductions $I_{P_\pi}^G (\sigma_{\pi,v} \otimes (\lambda-\nu_\pi))$ (see \S\ref{subsec:RS_and_Zeta}).
    \end{itemize}
\end{theorem}

\begin{ex}
    \label{ex:periods}
    If $I=(0,n,n+1,0)$ and $(I,P,\pi) \in \Pi_H$ as in Example~\ref{ex:cuspi} (so that $\pi$ is cuspidal), then $\cP_\pi(\varphi,\lambda)=Z(E(\varphi,\lambda))$. If now $I=(0,0,n+1,0)$ as in Example~\ref{ex:resi}, so that $P=G$ and $\pi$ is a character of $G$, then $\cP_{\pi}(\varphi,\lambda)$ is simply evaluation at $1$ (up to constant). Note that for any $(\lambda,-\lambda) \in \fa_{\pi,\cc}^*$ the character $\pi_{(\lambda,-\lambda)}$ remains $H(\bA)$-invariant. Finally, take $I=(n,0,1,0)$ as in Example~\ref{ex:scalar}. We use the notation from there. The representation $\cA_{P_+,\pi_+}(\GL_n)$ is equipped with the Petersson inner product $\langle \cdot,\cdot \rangle_n$. If $\varphi=\varphi_n \otimes \varphi_{n+1} \in \cA_{P,\pi}(G)$, then $\varphi_n \in \cA_{P_+,\pi_+}(\GL_n)$ and $g_n \in \GL_n(\bA) \mapsto \varphi_{n+1}(g_{n})$ belongs to $\cA_{P_+,\pi_+^\vee}(\GL_n) \otimes \Val{\det}^{1/2}$. It follows that for any $\lambda \in \fa_{\pi,\cc}^*$, the linear form 
    \begin{equation*}
        \varphi=\varphi_n \otimes \varphi_{n+1} \in \cA_{P,\pi,\lambda-\underline{\rho}_\pi}(G) \mapsto \langle \varphi_n \otimes \overline{\varphi}_{n+1} \rangle_n
    \end{equation*}
    is non-zero and $H(\bA)$-invariant. This is our map $\cP_{\pi,\lambda}$ (up to constant).
\end{ex}

In \cite{BoiZ}, we built these periods in the Arthur case, i.e. when $n_+=n_-=0$, and proved the corresponding version of Theorem~\ref{thm:GGP_global_intro}. The procedure we use in the current paper for the case of general $(I,P,\pi)$ is the same, and the main idea is to realize $Z(E(\phi,\lambda))$ as a regularization of a truncated period using \cite{IY} and \cite{Zydor}. The approach of \cite{BoiZ} generalizes quite easily to prove the first and third points of Theorem~\ref{thm:GGP_global_intro}. In \cite[Theorem~1.2]{BoiZ}, we additionally studied the vanishing of the period $\cP_\pi$. We showed that the local linear forms $Z_{\sigma_\pi,v}^\natural$ are always non-zero for $\lambda \in i \fa_{\pi}^*$, and therefore that $\cP_\pi$ vanishes if an only if $\cL(\pi,\lambda)$ does. This proved the non-tempered Gan--Gross-Prasad conjecture from \cite{GGP2}. In this text, we will not deal directly with the local factors $Z^\natural_{\sigma_\pi,v}$, and in particular we will not settle the question of their non-vanishing. We leave this question to a future work. Finally, we emphasize that an alternative description of $\cP_\pi$ using parabolic descent is given in Proposition~\ref{prop:parabolic_descent}.

We give two conceptual explanations for the appearance of $\Pi_H$ in the spectral decomposition of $J^H$. First, in \cite{BZSV} Ben-Zvi, Sakellaridis and Venkatesh have attached to the spherical variety $X=G/H$ a hyperspherical hamiltonian variety $M$ (in this case, the contangent bundle $T^* X$) and a $\hat{G}(\cc)=G(\cc)$-dual variety $M^\vee$ which here is $T^* \mathrm{St}_n \otimes \mathrm{St}_{n+1}$ the cotangent bundle of the tensor product of the standard representations. According to \cite[Conjecture~14.3.5]{BZSV}, if $\Pi$ is an automorphic representation of $G$ with Arthur parameter $\Psi$, the (regularized) period $\cP_H$ should be zero on $\Pi$ as soon as the set of fixed points for the action of the hypothetical Langlands dual group of $F$ via $\Psi$ on a Slodowy slice $M^\vee_\mathrm{slice}$ of $M^\vee$ (which depends on $\Psi$) is empty. Using the reformulation of \cite[Section~4]{GGP2}, one can check that if $(I,P,\pi) \in \Pi_H$ and $\lambda \in \fa_{\pi,\cc}^*-\underline{\rho}_\pi$, then the action of the Arthur parameter of the induced representation $\cA_{P,\pi,\lambda}(G)$ on $M^\vee_\mathrm{slice}$ indeed admits fixed points. The second reason comes from local theory. It has recently been checked in \cite{Pat} that, over $p$-adic fields, the unitary $H$-distinguished representations of $G$ are exactly the local counterparts of the $\cA_{P,\pi,\lambda}(G)$, that is inductions of Speh representations that satisfy the same combinatorial condition as \S\ref{subsubsec:relevant_intro}. This is consistent with \cite{BZSV}, and provides further evidence that the linear forms $Z_{\sigma_\pi,v}^\natural$ should be non-zero.

\subsubsection{The fine spectral expansion}

We can finally describe the spectral expansion of $J^H$. Let $(I,P,\pi) \in \Pi_H$. We write $\cB_{P,\pi}$ for the Hilbert basis of $\cA_{P,\pi}(G)$ for the Petersson innner product defined in \S\ref{subsec:bases}. For a Schwartz function $f \in \cS(G(\bA))$, $g \in G(\bA)$ and $\lambda \in \fa_{\pi,\cc}^*$ in general position, we can consider the relative character 
\begin{equation}
    \label{eq:relative_character_intro}
     J_{(I,P,\pi)}^H(g,f,\lambda)=\sum_{\varphi \in \cB_{P,\pi}} E(g,I_P(f,\lambda+\underline{\rho}_\pi)\varphi,\lambda+\underline{\rho}_\pi) \overline{\cP_{\pi}(\varphi,-\overline{\lambda}-\underline{\rho}_\pi)}.
\end{equation}
The regularized period $\cP_{\pi}(\varphi,-\lambda-\underline{\rho}_\pi)$ may have poles in the region $i \fa_\pi^*$, but they are compensated by zeros of the Eisenstein series (which are regular in this region). Their product therefore defines an holomorphic function on $i \fa_\pi^*$ which is moreover of rapid decay. The sum in \eqref{eq:relative_character_intro} is absolutely convergent and the map $f \mapsto J_{(I,P,\pi)}^H(g,f,\lambda)$ is continuous. All these properties are proved in Lemma~\ref{lem:bound_relative_character}. Finally, we have functional equations (Corollary~\ref{cor:independence_choice_couple})
\begin{equation*}
    J^H_{(I,w.P,w.\pi)}(g,f,w \lambda)=J^H_{(I,P,\pi)}(g,f,\lambda), \quad w \in W(\pi).
\end{equation*}
We now equip $i \fa_\pi^*$ with the Haar measure described in \S\ref{subsubsec:i_a_pi_measure}, which depends on that of $H(\bA)$. We write the main result of this paper.

\begin{theorem}
    \label{thm:spectral_expansion_intro}
    Let $f \in \cS(G(\bA))$ be a Schwartz function, let $g \in G(\bA)$. Then we have
    \begin{equation}
        \label{eq:J_expansion}
        J^H(g,f)=\sum_{(I,P,\pi) \in \Pi_H} \frac{1}{\Val{W(\pi)}} \int_{i \fa_\pi^*} J_{(I,P,\pi)}^H(g,f,\lambda) d \lambda,
    \end{equation}
    where this double integral is absolutely convergent.
\end{theorem}

The maps $\lambda \mapsto  J_{(I,P,\pi)}^H(g,f,\lambda)$ are meromorphic and therefore the contour in the integral \eqref{eq:J_expansion} may be shifted. In particular, one can replace $\underline{\rho}_\pi$ by any element in $\fa_{P_I}^*$ of the form $((t,0,0,-s),(1/2-t,0,0,-1/2+s))$, $0<t,s<1/2$. Moving the contour further away might result in poles of Eisenstein series.

\subsection{About Theorem~\ref{thm:spectral_expansion_intro}}

Although the statement of Theorem~\ref{thm:spectral_expansion_intro} seems analogous at first glance to fine spectral expansions for other relative trace formulae, in particular those of \cite{Lap06} and \cite{Ch}, we emphasize that our case presents an additional difficulty. Namely, the strategy in \cite{Lap06} and \cite{Ch} is to start from the spectral expansion of the kernel function $K_f$ from \cite{Art78}, which follows from the spectral expansion of $L^2([G])$ of \cite{Langlands}, and then to make it commute with the considered automorphic period. This involves serious analytic obstacles which are overcomed by using truncation in the spirit of \cite{JLR}. A key feature of these proofs is that the contours of integration never leave a neighborhood of the unitary axis $i \fa_P^*$. In our case, this approach proves to be impracticable as to make the different truncated integrals converge one has to shift parts of the spectral expansion of $K_f$ off the unitary axis. This basically amounts to reversing the proof of \cite{Langlands} and, as one can expect, produces very intricate residual contributions. This specificity of the Rankin--Selberg period is reflected in the expansion \eqref{eq:J_expansion} of Theorem~\ref{thm:spectral_expansion_intro} by the appearance of the $\underline{\rho}_\pi$ shifts. In contrast, the formulae in \cite{Lap06} and \cite{Ch} only involve unitary terms. 

Our strategy to prove Theorem~\ref{thm:spectral_expansion_intro} does not start from the decomposition of $L^2([G])$. Instead, we reproduce the argument of \cite{Langlands} in the case of our Rankin--Selberg period, which means that we proceed by shifting contours of from integrals along unramified characters with very positive real parts back to the unitary axis. However, we emphasize that we do use \cite{Langlands} later in the proof as a black box to simplify some computations of residual contributions. The miracle in our proof is that, after applying this trick and choosing the contours carefully, all the residues we gain along the way actually contribute to the expansion of Theorem~\ref{eq:J_expansion}. This is in sharp contrast with \cite{Langlands} where intricate compensations occur (see e.g. \cite[Section~4]{Labesse}). 

We now present the main steps of the proof of Theorem~\ref{thm:spectral_expansion_intro} and the required technical inputs. We will also explicitly write our argument for the simple $\GL_1 \times \GL_2$ example in Section~\ref{sec:example}.

\subsubsection{Step 1: coarse unfolding of the Rankin--Selberg integral}
\label{subsubsec:coarse_unfolding}

The first step does not involve shift of contours and is the subject of \S\ref{subsec:unfolding}. For fixed $f \in \cS(G(\bA))$ and $g \in G(\bA)$, set $F(g')=K_f(g,g')$ for $g' \in [G]$. Then $F$ is a Schwartz function on $[G]$ (see \S\ref{subsubsec:spaces_functions}). For every integer $0 \leq r \leq n$, let $P_r$ be the standard parabolic subgroup of $G$ with standard Levi factor $M_r:=(\GL_r \times \GL_{n-r}) \times (\GL_r \times \GL_{n+1-r})$. Let $F_{P_r}$ be the constant term of $F$ along $P_r$. Let $K_H$ be the standard maximal compact subgroup of $H(\bA)$ and write $R$ for its action by right translations on $F$. By applying a classical Rankin--Selberg unfolding argument, we arrive at an expression of the form 
\begin{equation}
    \label{eq:RS_unfolding}
    \int_{[H]} F(h) dh=\sum_{r=0}^{n} \int_{K_H}\left( \langle \cdot,\cdot \rangle_r \otimes Z_{n-r}(\cdot,0) \right)(R(k)F_{P_r}) dk.
\end{equation}
Here we mean that we regard $F_{P_r}$ as a function on $[M_r]$ to which we apply the inner-product obtained by integrating on the diagonal $[\GL_r]$ and the Rankin--Selberg Zeta integral relative to $\GL_{n-r} \subset \GL_{n-r} \times \GL_{n-r+1}$, evaluated at zero. We refer to Proposition~\ref{prop:zeta_unfold} for a precise statement.

Let $0 \leq r \leq n$. We now compute the spectral expansion of the linear form $ \langle \cdot,\cdot \rangle_r \otimes Z_{n-r}(\cdot,0) $ which we view as defined on a certain space of functions on the Levi $[M_r]$. For $\langle \cdot,\cdot \rangle_r$, we can directly use the spectral expansion of \cite{Langlands}. For the Zeta integral, it is not too difficult to write that of $Z_{n-r}(\cdot,s)$ for $\Re(s)$ large enough. This boils down to the fact that this map is continuous for functions of fixed moderate growth (see Lemma~\ref{lem:zeta_convergence}). Moreover, $Z_{n-r}$ will kill all non-generic (hence here non-cuspidal) contributions. Thanks to the adjunction between constant terms and Eisenstein series and the description of the periods $\cP_\pi$ by parabolic descent (Proposition~\ref{prop:parabolic_descent}), we can induce our expansion back to $G$. The final result in the language of Theorem~\ref{thm:spectral_expansion_intro}, is that for any $\Re(s)$ large enough
\begin{equation}
    \label{eq:to_shift}
    \int_{[H]} F(h) dh=\sum_{r=0}^n \sum_{(I_r,P,\pi) \in \Pi_H} \frac{1}{\Val{W(\pi)}}\int_{i \fa_\pi^* +\underline{\rho}_\pi - s \underline{z}_r} \sum_{\varphi \in \cB_{P,\pi}}E(g,I_P(f,\lambda)\varphi,\lambda) \overline{\cP_{\pi}(\varphi,-\overline{\lambda})}  d\lambda,
\end{equation}
where $I_r=(r,n-r,n+1-r,0)$ and $\underline{z}_r \in \fa_{P_r}^*$ is the element with coordinates $((0,1),(0,1))$. Note that this imposes that the representations $\pi_{1,i}$ and $\pi_{2,j}$ in \eqref{eq:pi_decompo_intro} are cuspidal. The appearance of $\underline{\rho}_\pi$ boils down to a computation of modular characters. 

\subsubsection{Step 2: residues of Eisenstein series}
\label{subsubsec:step2}

To end the proof of Theorem~\ref{thm:spectral_expansion_intro}, we want to shift the contour of integration in \eqref{eq:to_shift} to the regions $i \fa_\pi^* +\underline{\rho}_\pi$. This has to be done in two main steps, the first being to go to $i \fa_\pi^* +\underline{\rho}_\pi - 1/2 \underline{z}_r$. This is done in \S\ref{subsec:additional}.

To begin with, we need to ensure that our integrand is meromorphic and of rapid decay in vertical strips in our region of integration. This requires a majorization of the Eisenstein series $E(g,I_P(f,\lambda)\varphi,\lambda)$ for $\Re(\lambda)$ in a neighborhood of the positive Weyl chamber. This intermediate result is the content of Theorem~\ref{thm:bound_Eisenstein}. It is an extension of \cite[Theorem~3.9.2.1]{Ch} which derived such a bound for $\lambda$ in a neighborhood of the imaginary axis $i \fa_P^*$, following the strategy of \cite{Lap2}. The proof relies on deep results including bounds towards the Ramanujan conjecture from \cite{LRS} and zero-free regions for automorphic $L$ functions from \cite{Bru06} and \cite{Lap2} which allow us to control Eisenstein series slightly to the left of the imaginary axis. 

We now have to understand the singularities of meromorphic functions
\begin{equation*}
    \lambda \mapsto E(g,I_P(f,\lambda)\varphi,\lambda) \overline{\cP_{\pi}(\varphi,-\overline{\lambda})},
\end{equation*}
When moving the contour to $i \fa_\pi^* +\underline{\rho}_\pi - 1/2 \underline{z}_r$, we can show that all the poles we encounter come from the Eisenstein series $E(g,I_P(f,\lambda)\varphi,\lambda)$. We determine their possible singularities for $\Re(\lambda)$ in a neighborhood of the positive Weyl chamber in Theorem~\ref{thm:analytic_Eisenstein}. We emphasize that the situation for discrete Eisenstein series is far more complicated than for those induced from cuspidal automorphic forms as intricate compensations of poles can occur in their constant terms. We refer to \cite{Hegde} where this phenomenon was studied in details for unramified forms on split reductive groups. For $\GL_n$, it turns out that Theorem~\ref{thm:analytic_Eisenstein} was already contained (up to some mild reformulation) in \cite{MW89}. The answer is that, in the language of \cite{BZ}, singularities arise when segments of $\Speh(\sigma,d)$ and $\Speh(\sigma',d')$ are linked (see \S\ref{subsubsec:BZ_segments}). This useful fact seems to be ignored by later references dealing with the subject (e.g. \cite{HM} or \cite{GS}).

In any case, we now know which singularities we cross during our shift of contours. The next step is to describe the representations spanned by the corresponding residues of $E(g,I_P(f,\lambda)\varphi,\lambda)$. In general, this is a hard question (see e.g. \cite{HM} and \cite{GS}), but for the singularities we consider it is straightforward. More precisely, they arise from links between segments associated to $\Speh(\sigma,d-1)$ and $\sigma$, for some cuspidal representation $\sigma$. The resulting residues is a twist of $\Speh(\sigma,d)$ by \cite{MW89} (see Lemma~\ref{lem:compute_residues}).

The outcome of the second step is an expansion of the form 
\begin{equation}
    \label{eq:to_shift_2}
    \int_{[H]} F(h) dh=\sum_{r=0}^n \sum_{(I,P,\pi) \in \Pi_{H,r}} \frac{1}{\Val{W(\pi)}} \int_{i \fa_\pi^* +\underline{\rho}_\pi - 1/2 \underline{z}_r} \sum_{\varphi \in \cB_{P,\pi}}E(g,I_P(f,\lambda)\varphi,\lambda) \overline{\cP_{\pi}(\varphi,-\overline{\lambda})}  d\lambda,
\end{equation}
where $\Pi_{H,r}$ is a certain subset of $\Pi_H$. We refer to Proposition~\ref{prop:I_pi_exp} for a precise statement.

\subsubsection{Step 3: residues of regularized periods}
\label{subsubsec:step3}
We finally move the contour of integration in \eqref{eq:to_shift_2} to $i \fa_\pi^* +\underline{\rho}_\pi$. This is the content of \S\ref{subsec:this_is_the_end}. The singularities that we now cross are those of $\overline{\cP_{\pi}(\varphi,-\overline{\lambda})}$. By Theorem~\ref{thm:GGP_global_intro}, these regularized periods are built by taking residues of Rankin--Selberg Zeta integrals $Z(E(\phi,\lambda))$ of Eisenstein series induced from cuspidal automorphic representations. Using this description, we see that residues of $\cP_\pi(\varphi,\lambda)$ are also iterated residues of $Z(E(\phi,\lambda))$, and therefore can be written as some $\cP_{\pi'}$ for another triple $(I',P',\pi')$. Therefore, the bookeeping needed to keep track of the residual contributions in this third step is less involved than in the second one. These additional terms are always twisted by a non-zero $\underline{\rho}_\pi$. Once this last wave of shifts of contours is finished, Theorem~\ref{thm:spectral_expansion_intro} is proved.

\subsection{Outline of the paper}

The paper is organized as follows. Section \ref{sec:example} presents our argument in the simple example $G=\GL_1 \times \GL_2$. In Section~\ref{sec:preliminaries}, we fix some notations and prove some preliminary results pertaining to automorphic forms and spaces of functions on automorphic quotients. In Section~\ref{chap:poles}, we recall the main results on the classification of discrete automorphic forms on $\GL_n$ from \cite{MW89}. In particular, we study the poles of global intertwining operators $M(w,\lambda)$ and discrete Eisenstein series $E(\varphi,\lambda)$ for $\Re(\lambda)$ in a neighborhood of the positive Weyl chamber (Proposition~\ref{prop:M_regular} and Theorem~\ref{thm:analytic_Eisenstein}). We finally extend the results of \cite{Ch} to bound Eisenstein series in this region (Theorem~\ref{thm:bound_Eisenstein}). We then proceed in Section~\ref{chap:IYZ_periods} to recall the framework of \cite{BoiZ} on regularized Rankin--Selberg periods, following \cite{IY} and \cite{Zydor}. In particular, we describe them using parabolic descent in Proposition~\ref{prop:parabolic_descent}, and use this to find their singularities in Proposition~\ref{prop:reg_P} and bound them in Proposition~\ref{prop:individual_bound_P}. We finally compute their residues in Proposition~\ref{prop:order_residues}. In Section~\ref{chap:RS_non_tempered}, we extend the construction of the regularized linear forms $\cP_\pi$ of \cite{BoiZ} to the general case of triples $(I,P,\pi) \in \Pi_H$. We prove Theorem~\ref{thm:GGP_global_intro} in \S\ref{subsubsec:proof_main_theorem}. We also present in \S\ref{subsec:increasing} a different set of \emph{increasing} inducing data $\Pi_H^\uparrow$ which appears naturally when computing the spectral expansion. Finally, in Section~\ref{sec:pseudo_spectral} we prove Theorem~\ref{thm:spectral_expansion_intro}. The argument is divided in the three steps of \S\ref{subsec:unfolding}, \S\ref{subsec:additional} and \S\ref{subsec:this_is_the_end} presented above. In \S\ref{subsec:end_proof} we group together all the contributions to end the proof of Theorem~\ref{thm:spectral_expansion_intro}.

\subsection{Acknowledgement}

The author thanks Rapha\"el Beuzart-Plessis for helpful discussions and comments. He is also grateful to Wee Teck Gan, Erez Lapid and Tasho Kaletha for suggestions on an earlier version of this text.

This work was partly funded by the European Union ERC Consolidator Grant, RELANTRA, project number 101044930. Views and opinions expressed are however those of the author only and do not necessarily reflect those of the European Union or the European Research Council. Neither the European Union nor the granting authority can be held responsible for them. The author was also supported by the Max Planck Institute for Mathematics in Bonn, that he thanks for its hospitality
and financial support.

\section{The \texorpdfstring{$\GL_1 \times \GL_2$}{GL1 x GL2}-case}
\label{sec:example}

We now explain our shifts of contours of integration for the simple example $G=\GL_1 \times \GL_2$ and $H=\GL_1$. In that case, the result is certainly not new (see e.g. \cite{Jac86} for a closely related computation). We refer to the main text for some technical analytic bounds that we use. We also ignore the question of normalizing the measures, for which we refer to \S\ref{subsubsec:i_a_pi_measure}.

We will use the notation from Section~\ref{sec:intro}. We take $f \in \cS(G(\bA))$, $g \in G(\bA)$ and set $F(g')=K_f(g,g')$ for $g' \in [G]$. The unfolding expansion for the Rankin--Selberg integral from \eqref{eq:RS_unfolding} reads
\begin{equation}
    \label{eq:coarse_zeta_unfold}
    \int_{[H]} F(h) dh=Z_2(F,0)+\int_{[\GL_1]} F_{B}(h)dh,
\end{equation}
where $B$ is the Borel subgroup of upper triangular matrices in $G$. Here $Z_2(F,0)$ is the Zeta integral from \cite{JPSS83}. It is the integral of a Whittaker coefficient of $F$ along $H(\bA)$ (see \eqref{eq:Z_r}). Let $B_2$ be the Borel subgroup of $\GL_2$, so that $B=\GL_1 \times B_2$. We now write the expansion of each term separately.

\subsection{The main contribution}

We start with $Z_2(F,0)$. For $s \in \cc$, note that $F_{-s}:=F(h) \Val{h}^{-s}$ is still of rapid decay, and that $Z_2(F,0)=Z_2(F_{-s},s)$. We now apply the spectral expansion of \cite{Langlands} to $F_{-s}$. It reads
\begin{align*}
    F_{-s}&=\sum_{\chi_1} \sum_{\pi} \sum_{\varphi \in \chi_1 \boxtimes \pi} \int_{i \fa_G^*} \langle F,\varphi_{\lambda-\overline{s}} \rangle_G \varphi_{\lambda} d \lambda + \frac{1}{2}\sum_{\chi_1} \sum_{\chi_2} \sum_{\varphi \in \chi_1 \boxtimes \cA_{B_2,\chi_2}(\GL_2)} \int_{i \fa_{B}^*} \langle F,E(\varphi,\lambda-\overline{s}) \rangle_G E(\varphi,\lambda) d \lambda \\
    &+ \sum_{\chi_1,\chi_1'}  \int_{i \fa_{G}^*} \langle F,(\chi_1 \boxtimes \chi_1')_{\lambda-\overline{s}} \rangle_G (\chi_1 \boxtimes \chi_1')_{\lambda} d \lambda.
\end{align*}
The notations we use are as follows. $\chi_1$ and $\chi_1'$ range along automorphic characters of $\GL_1$ (and we write again $\chi_1'$ for $\chi_1' \circ \det$), and $\chi_2$ along automorphic characters of $(\GL_1)^2$. These characters are trivial on $A_{\GL_1}^\infty$ and $A_{(\GL_1)^2}^\infty$ respectively. $\pi$ ranges along cuspidal representations of $\GL_2$ (with trivial central character on $A_{\GL_2}^\infty$). $\varphi$ ranges along orthonormal bases for the Petersson inner products. $s$ is the element in $\fa_{G,\cc}^*$ corresponding to the character $(g_1,g_2)\mapsto \Val{g_1}^{s/2}\Val{\det g_2}^{s/2}$. Finally $\langle \cdot,\cdot \rangle_G$ is the inner product given by integrating along $[G]$. By Corollary~\ref{cor:Langlands_spectral_extended}, this expression is absolutely convergent in some space of function with large fixed growth $\cT_N([G])$ independent of $s$, and by \cite[Lemma~7.1.1.1]{BPCZ} $Z_2(\cdot,s)$ defines a continuous linear form on this space for $\Re(s)$ large enough. As characters of $\GL_2$ are not generic, we arrive at
\begin{equation}
    \label{eq:zeta_expansion_intro}
    Z_2(F,0)=\sum_{\chi_1, \pi, \varphi} \int_{i \fa_G^*+s} \langle F,\varphi_{-\overline{\lambda}} \rangle_G Z_2(\varphi_\lambda) d \lambda + \frac{1}{2}\sum_{\chi_1, \chi_2, \varphi} \int_{i \fa_{B}^*+s} \langle F,E(\varphi,-\overline{\lambda}) \rangle_G Z_2(E(\varphi,\lambda)) d\lambda,
\end{equation}
where we write $Z_2(\varphi_\lambda)$ and $Z_2(E(\varphi,\lambda))$ for $Z_2(\varphi_\lambda,0)$ and $Z_2(E(\varphi,\lambda),0)$, both Zeta integral being absolutely convergent for our $\lambda$. This ends the manipulations of step 1 from \S\ref{subsubsec:coarse_unfolding} for this contribution. We now shift the contour in each integral.

\subsubsection{Cuspidal contributions}

We start with the first term in \eqref{eq:zeta_expansion_intro}. We fix $\chi_1,\pi$ and $\varphi$. There is no Eisenstein series here, so that we may skip step 2 an go directly to step 3 from \S\ref{subsubsec:step3}. The only poles come from $Z_2(\varphi_\lambda)$. If we write $\lambda=(\lambda_1,\lambda_2) \in \fa_{G,\cc}^*$, we know that for factorizable $\varphi$ (which we can arrange all our elements in the orthonormal bases to be), we have a finite set of places $\tS$ of $F$ such that 
\begin{equation*}
    Z_2(\varphi_\lambda)=L(1/2+\lambda_1+\lambda_2,\chi_1 \times \pi) \times \prod_{v \in \tS} Z^\natural_{v}(\varphi_{v,\lambda}).
\end{equation*}
The global $L$-functions and normalized local Zeta integrals are regular by \cite{JPSS83} as we only look at the central direction. In that case, we can shift the contour to $i \fa_G^*$ with no issue. In the language of \S\ref{subsubsec:relevant_intro}, this corresponds to the case $(I,G,\chi \boxtimes \pi) \in \Pi_H$ with $I=(0,1,2,0)$ from Example~\ref{ex:cuspi}.

\subsubsection{Continuous contributions}

We now deal with the second contribution in \eqref{eq:zeta_expansion_intro}, and fix $\chi_1$, $\chi_2$ and $\varphi$. We further write $\chi_2=\chi_2^1 \boxtimes \chi_2^2$ and $\lambda=(\lambda_1,(\lambda_2^1,\lambda_2^2))$. The factorization is
\begin{equation}
    \label{eq:Z_2_intro}
     Z_2(E(\varphi,\lambda))=\frac{L(1/2+\lambda_1+\lambda_2^1,\chi_1 \times \chi_2^1)L(1/2+\lambda_1+\lambda_2^2,\chi_1 \times \chi_2^2)}{L(1+\lambda_2^1-\lambda_2^2,\chi_2^1 \times \chi_2^{2,\vee})} \prod_{v \in \tS} \times Z^\natural_{v}(\varphi_{v,\lambda}).
\end{equation}

If $\chi_1 \neq \chi_2^{1,\vee}$ and $\chi_1 \neq \chi_2^{2,\vee}$, we shift the contour in the central direction of $s$. The Eisenstein series and the Zeta integral remain regular and neither step 2 nor 3 are needed. We obtain an integral on $i \fa_B^*$. This corresponds to the contribution $(I,B,\chi \boxtimes (\chi_1^1 \boxtimes \chi_2^2)) \in \Pi_H$ with $I=(0,1,2,0)$ from Example~\ref{ex:cuspi}.

We now assume that $\chi_1=\chi_2^{1,\vee}$ but $\chi_1 \neq \chi_2^{2,\vee}$. As before, the Eisenstein series remains regular. However, we now have a simple pole at $\lambda_1+\lambda_2^1=1/2$ coming from \eqref{eq:Z_2_intro}. Set $I=(0,0,1,1)$ and let $\chi$ be the character $\chi_1 \boxtimes (\chi_2^2 \boxtimes \chi_2^1)$ of $\GL_1 \times (\GL_1)^2$. Then $(I,P,\chi) \in \Pi_H$. Let $w$ be the non-trivial element in the Weyl group of $\GL_2$. Then if we apply $w$, the hyperplane $\lambda_1+\lambda_2^{1}=1/2$ becomes $\fa_{\chi,\cc}^*-\underline{\rho}_\chi$. For $\phi \in \cA_{B,\chi}(G)$, by Proposition~\ref{prop:parabolic_descent} (see also Example~\ref{ex:periods}) we have the simple description
\begin{equation}
    \label{eq:cP_pi_intro}
    \cP_{\chi}(\phi,\mu)=\int_{K_H} \phi_\mu(wk) dk=\phi(w), \quad \mu \in \fa_{\chi,\cc}^*-\underline{\rho}_\chi.
\end{equation}
Note that for $\mu$ in this subspace, this indeeds define a $H(\bA)$-invariant linear form on $\cA_{B,\chi,\mu}(G)$. We claim that we have the equality 
\begin{equation}
    \label{eq:residue_intro}
    \underset{\lambda_1+\lambda_2^1=1/2}{\Res} \; Z_2(E(\varphi,\lambda))=\cP_\chi(M(w,\lambda)\varphi,w\lambda).
\end{equation}
This result is proved in Proposition~\ref{prop:order_residues}.
\begin{rem}
    If we specify \eqref{eq:residue_intro} to the special case where all the characters and $\varphi$ are assumed to be unramified and if $F=\qq$ (so that there is no local factor in \eqref{eq:Z_2_intro} by \cite{CS} and \cite{Sta}), we are simply saying that the residue of the quotient of $L$-functions in \eqref{eq:Z_2_intro} is equal, up to a volume term, to the global factor of $M(w,\lambda)$.
\end{rem}
By the functional equation of Eisenstein series and a change of variable, we arrive at 
\begin{equation}
    \label{eq:intermediate_integral_intro}
    \sum_{\varphi} \int_{i \fa_{B}^*+s} \langle F,E(\varphi,-\overline{\lambda}) \rangle_G Z_2(E(\varphi,\lambda)) d\lambda=\sum_{\varphi \in \cA_{B,\chi}(G)} \int_{i \fa_{\chi}^*-\underline{\rho}_\chi+e} \langle F,E(\varphi,-\overline{\lambda}) \rangle_G \cP_{\chi}(\varphi,\lambda) d\lambda,
\end{equation}
where $\fa_\chi^*$ is the subspace $\{(\lambda_1,(\lambda_2^1,-\lambda_1))\}$, $\underline{\rho}_\chi=(-1/4,(0,-1/4))$ and $e=(0,(1/4,0))$. We now want to shift in the $\lambda_2^1$ variable to $i \fa_{\chi}^*-\underline{\rho}_\chi$ (i.e. to $\Re(\lambda_2^1)=0$). By the descrition of \eqref{eq:cP_pi_intro}, $\cP_\chi$ is clearly regular, and our contour will keep $-\overline{\lambda}$ in the positive Weyl chamber without crossing the pole of the Eisenstein series (which would occur for $\Re(\lambda_2^1)=3/4$). Therefore, we can indeed do this shift of contour. We therefore get two contributions attached to $(I,B,\chi)$ with $I=(0,1,2,0)$ and $(0,0,1,1)$. The case $\chi_1=\chi_2^{2,\vee}$ but $\chi_1  \neq \chi_2^{1,\vee}$ is the same using the functional equation.

We now move to the most difficult case where $\chi_1=\chi_2^{1,\vee}=\chi_2^{2,\vee}$. We a priori have a pole of order two in the numerator of \eqref{eq:cP_pi_intro} when the hyperplanes $\lambda_1+\lambda_{2}^1=1/2$ and $\lambda_1+\lambda_2^2=1/2$ cross. If we assume a weak version of the generalized Riemann hypothesis, we get that $E(\varphi,-\overline{\lambda})$ is regular and that $L(1+\lambda_2^1-\lambda_2^2,\chi_2^1 \times \chi_2^{2,\vee})$ has no zero for $\Re(\lambda_2^1) - \Re(\lambda_2^2) \geq -\varepsilon$ for a small $\varepsilon$. We could take advantage of this as follows. First, we move the contour in \eqref{eq:intermediate_integral_intro} to the region $i\fa_{B}^*+(1/4,(1/4+\varepsilon,1/4+\varepsilon))$. We then shift the contour in the $\lambda_2^1$ variable to the region $i\fa_{B}^*+(1/4,(1/4-\varepsilon,1/4+\varepsilon))$, thus getting the residue along $\lambda_1+\lambda_2^1=1/2$. But we are now exactly in the situation of \eqref{eq:intermediate_integral_intro} and can proceed from there, the key point is that the hyperplane $\lambda_1+\lambda_2^2=1/2$ is no longer singular for our residue. Finally, we shift the main contribution in the $\lambda_2^2$ variable to go to $i\fa_{B}^*+(1/4,(1/4-\varepsilon,1/4-\varepsilon))$ catching the additional residue along $\lambda_1+\lambda_2^1=1/2$. We may now conclude as in the $\chi_1 \neq \chi_2^{2,\vee}$.

If we don't assume some variant of the generalized Riemann hypothesis, we can still use the zero free regions from \cite{Bru06}. Because we know that the integrand in \eqref{eq:intermediate_integral_intro} is of rapid decay by Theorem~\ref{thm:analytic_Eisenstein} and Proposition~\ref{prop:individual_bound_P_pi}, we may cut its tail to focus on the region $\Val{\Im(\lambda_2^1)}, \Val{\Im(\lambda_2^2)} \leq T$ for $T$ large. We can now assume that the product of rectangles in the variables $\lambda_2^1$ and $\lambda_2^2$ centered in $(1/4,(1/4,1/4))$ with real length $2\varepsilon$ and imaginary height $T$ is contained in a zero free region, and therefore do the same manipulations as before. This is the method we use in the core of the text (see Lemma~\ref{lem:residue_n} and Lemma~\ref{lem:residue_n+1}). It is inspired by \cite{Lap06}.

\subsubsection{Final result}

Putting everything together, we arrive at 
\begin{equation}
    \label{eq:Z_2_final}
    Z_2(F,0)=\sum_{\substack{(I,P,\pi) \in \Pi_H \\ I=(0,1,2,0) \\ \text{or } I=(0,1,0,1)}} \int_{i \fa_\pi^*} \frac{1}{\Val{W(\pi)}}J_{(I,P,\pi)}^H(g,f,\lambda) d\lambda,
\end{equation}
where $\Val{W(\pi)}$ is $2$ if $I=(0,1,2,0)$, and $1$ otherwise.

\subsection{The constant term contribution}

We now deal with the second term in \eqref{eq:coarse_zeta_unfold} which is the integral of $F_{B}$ along $[\GL_1]$. Let $T=\GL_1 \times (\GL_1)^2$ be the maximal torus of diagonal matrices. The function $F_{B}$ is not of rapid decay on $[T]$, but if we fix any $N>0$ then there exists $s \in \cc$ with $\Re(s)$ large enough so that 
\begin{equation*}
    (t_1,(t_2^1,t_2^2)) \in [T] \mapsto F_{B}(t_1,(t_2^1,t_2^2)) \Val{\det g_2^2}^s
\end{equation*}
belongs to $\cT_{-N}([T])$, i.e. decreases at least as fast as $\norm{\cdot}^{-N}_T$ (see Lemma~\ref{lem:constant_growth}). Using the same trick as before, we can write the spectral expansion of $F_{B}$ as a function on $[T]$ using \cite{Langlands}, and further upgrade it to a spectral expansion depending on $F$ by the adjunction between constant terms and Eisenstein series. By repackaging things as in \eqref{eq:zeta_expansion_intro} and taking into account the modular characters, we arrive at
\begin{equation}
    \label{eq:constant_term_period}
    \int_{[\GL_1]} F_{B}(h)dh=\sum_{\substack{(I,P,\chi) \in \Pi_H \\ I=(1,0,1,0)}}  \int_{ i \fa_{\chi}^*-\underline{\rho}_\chi+s e_2^1} \sum_{\varphi \in \cA_{P,\chi}(G)} \langle F,E(\varphi,-\overline{\lambda}) \rangle_G \cP_{\chi}(\varphi,\lambda) d\lambda,
\end{equation}
where $\underline{\rho}_\chi=(1/4,(1/4,0))$, $e_2^1=(0,(0,1))$, the condition $I=(1,0,1,0)$ implies that $\chi$ is a character of $[T]$ of the form $\chi_1 \boxtimes (\chi_1^\vee \boxtimes \chi_2)$, $\fa_{\pi}^*=(\lambda_1,(-\lambda_1,\lambda_2))$, and finally
\begin{equation}
    \label{eq:regu_period_intro}
    \cP_{\chi}(\varphi,\lambda)=\int_{K_H} \varphi_\lambda(k) dk=\varphi(1), \quad \lambda \in \fa_{\chi,\cc}^*-\underline{\rho}_\chi.
\end{equation}
We now shift in the variable $\lambda_{2}^1$ to bring the region of integration from $i \fa_{\pi}^*-\underline{\rho}_\pi+s e_2^1$ to $i \fa_{\pi}^*-\underline{\rho}_\pi$. Given \eqref{eq:regu_period_intro}, the only possible singularity comes from $E(\varphi,-\overline{\lambda})$ and occurs at $\lambda_{1}+\lambda_{2}=1$ if $\chi_2=\chi_1^\vee$. If we write $\eta=\chi_1 \boxtimes (\chi_1^\vee \circ \det)$, then the residue spans the character $\eta_{-\overline{\lambda}}$ (where we project $-\overline{\lambda}$ to $\fa_{G,\cc}^*$). By Lemma~\ref{lem:compute_residues}, we have the adjunction for $\lambda$ in the singular hyperplane
\begin{equation}
    \label{eq:adjunction}
    \sum_{\varphi \in \cA_{P,\chi}(G)} \langle F,E^*(\varphi,-\overline{\lambda}) \rangle_G \cP_{\chi}(\varphi,\lambda)=\langle F,\eta_{-\overline{\lambda}} \rangle_G \cP_{\eta}(1,\lambda),
\end{equation}
where $E^*$ is the residue of the Eisenstein series and $\cP_{\eta}(1,\lambda)$ is simply constant equal to $1$. In particular, this expression is now holomorphic. This case corresponds to $((0,0,2,0),G,\eta) \in \Pi_H$ from Example~\ref{ex:resi}, so that $\fa_{\eta}=(\lambda_1,(-\lambda_1,-\lambda_1))$ and $\underline{\rho}_{\eta}=0$. We can shift the contour in the integral of the residue to $i \fa_{\eta}^*-\underline{\rho}_{\eta}$.

If we add up the main contribution of \eqref{eq:constant_term_period} with \eqref{eq:adjunction}, we obtain the formula
\begin{equation}
    \label{eq:constant_final}
    \int_{[\GL_1]} F_{B}(h)dh=\sum_{\substack{(I,B,\pi) \in \Pi_H \\ I=(1,0,1,0) \\ \text{or } I=(0,0,0,2)}} \int_{i \fa_\pi^*}\frac{1}{\Val{W(\pi)}}J_{(I,B,\pi)}^H(g,f,\lambda) d\lambda,
\end{equation}
where $\Val{W(\pi)}$ is always $1$ here. Theorem~\ref{thm:spectral_expansion_intro} now follows from putting \eqref{eq:Z_2_final} and \eqref{eq:constant_final} together.

\section{Preliminaries on automorphic forms}
\label{sec:preliminaries}
\subsection{General notation}

Let $F$ be a field of characteristic zero. All algebraic groups are defined over $F$.

\subsubsection{Reductive groups, parabolic subgroups, characters}
\label{subsubsec:reductive_gp}
Let $G$ be a connected reductive group. Let $Z_G$ be the center of $G$. Let $N_G$ be the unipotent radical of $G$ and let $X^*(G)$ be the group of $F$-algebraic characters of $G$. Set $\fa^*_G=X^*(G) \otimes_{\zz} \rr$ and $\fa_G=\Hom_\zz(X^*(G),\rr)$. Let
\begin{equation}
\label{eq:canonical_pairing}
    \langle \cdot,\cdot \rangle : \fa_G^* \times \fa_G \to \rr
\end{equation}
be the canonical pairing.

Let $P_0$ be a minimal parabolic subgroup of $G$. Let $M_0$ be a Levi factor of $P_0$. We say that a parabolic subgroup of $G$ is standard (resp. semi-standard) if it contains $P_0$ (resp. if it contains $M_0$). If $P$ is a semi-standard parabolic subgroup of $G$, we will denote by $N_P$ its unipotent radical and by $M_P$ its unique Levi factor containing $M_0$, which is said to be semi-standard. We have a decomposition $P=M_P N_P$. We denote by $\cP(M_P)$ the set of semi-standard parabolic subgroups of $G$ with semi-standard Levi $M_P$.

Let $A_G$ be the maximal central $F$-split torus of $G$. If $P$ is a semi-standard parabolic subgroup of $G$, set $A_P=A_{M_P}$. We set $\fa_0^*=\fa_{P_0}^*$, $\fa_0=\fa_{P_0}$ and $A_0=A_{P_0}$.

Let $P \subset Q$ be semi-standard parabolic subgroups of $G$. The restriction maps $X^*(Q) \to X^*(P)$ and $X^*(A_P) \to X^*(A_Q)$ induce dual decompositions $\fa_P=\fa_P^Q \oplus \fa_Q$ and $\fa_P^*=\fa_P^{Q,*} \oplus \fa_Q^*$. In particular, we have projections $\fa_0 \to \fa_P^Q$ and $\fa_0^* \to \fa_P^{Q,*}$ denoted by $X \mapsto X_P^Q$ which only depend on the Levi factors $M_P$ and $M_Q$. If $Q=G$, we omit the exponent $G$ in the previous notation.

Set $\fa_{P,\cc}^Q=\fa_{P}^Q \otimes_\rr \cc$ and $\fa_{P,\cc}^{Q,*}=\fa_{P}^{Q,*} \otimes_\rr \cc$. We still denote by $\langle \cdot, \cdot \rangle$ the pairing obtained by extension of scalars. We have decompositions $ \fa_{P,\cc}^{Q}=\fa_{P}^{Q} \oplus i \fa_{P}^{Q}, \quad \fa_{P,\cc}^{Q,*}=\fa_{P}^{Q,*} \oplus i \fa_{P}^{Q,*}$, where $i^2=-1$. We denote by $\Re$ and $\Im$ the real and imaginary parts associated to these decompositions, and by $\overline{\lambda}$ the complex conjugate of any $\lambda \in \fa_{P,\cc}^{Q,*}$.

\subsubsection{Roots, coroots, weights} \label{subsubsec:roots} Let $P$ be a standard parabolic subgroup of $G$. Let $\Delta_0^P \subset \fa_0^{P,*}$ (resp. $\Sigma_0^P \subset \fa_0^{P,*}$) be the set of simple roots (resp. of roots) of $A_0$ in $M_P \cap P_0$. If $P=G$, we write $\Delta_0$ and $\Sigma_0$. Let $\Delta_P$ (resp. $\Sigma_P$) be the image of $\Delta_0 \setminus \Delta_0^P$ (resp. $\Sigma_0 \setminus \Sigma_0^P$) by the projection $\fa_0^* \to \fa_P^*$. More generally, for $P \subset Q$ let $\Delta_P^Q$ (resp. $\Sigma_P^Q$) be the projection of $\Delta_0^Q \setminus \Delta_0^P$ in $\fa_P^{Q,*}$ (resp. $\Sigma_0^Q \setminus \Sigma_0^P$). Let $\Delta_P^{Q,\vee} \subset \fa_P^Q$ be the set of simple coroots. If $\alpha \in \Delta_{P}^Q$, we denote by $\alpha^\vee$ the associated coroot. By duality, let $\hat{\Delta}_P^Q$ be the set of simple weights. Set 
\begin{equation*}
    \fa_{P}^{Q,*,+}=\left\{ \lambda \in \fa_P^* \; | \; \langle \lambda, \alpha^\vee \rangle > 0, \; \forall \alpha \in \Delta_{P}^Q \right\}.
\end{equation*}
If $Q=G$, we drop the exponent. We denote by $\overline{\fa_{P}^{Q,*,+}}$ the closure of these open subsets in $\fa_P^Q$ and $\fa_P^{Q,*}$ respectively. If $\lambda \in \fa_P^*\setminus \{0\}$, we write $\lambda > 0$ if $\lambda$ is a nonnegative linear combination of the simple roots $\Delta_P$.

We say that a functional $\Lambda$ on $\fa_{P,\cc}^*$ is an affine linear form if it is of the form $\Lambda(\lambda)=\langle \lambda, \gamma^\vee \rangle - a$ for $\gamma^\vee \in \fa_P$ and $a \in \cc$. We call its set of zeros an affine hyperplane. If $\gamma^\vee$ is a coroot, then it is an affine root hyperplane. By "$\lambda \in \fa_{P,\cc}^*$ in general position", we mean that $\lambda$ lies outside of a countable union of affine hyperplanes.

\subsubsection{Weyl group} Let $W$ be the Weyl group of $(G,A_0)$, which is by definition the quotient of the normalizer $N_{G(F)}(A_0(F))$ by the centralizer $Z_{G(F)}(A_0(F))$. It acts on $\fa_0$ and by duality on $\fa_0^*$. If $w \in W$, we write again $w$ for a representative in $G(F)$.

Let $P=M_P N_P$ and $Q=M_Q N_Q$ be two standard parabolic subgroups of $G$. Let ${}_Q W_P$ be the set of $w \in W$ such that $M_P \cap w^{-1} P_0 w=M_P \cap P_0$ and $M_Q \cap w P_0 w^{-1} = M_Q \cap P_0$.

Let $w \in {}_Q W_P$. Set $P_w=(M_P \cap w^{-1} Q w)N_P$. By \cite[Lemme~V.4.6.]{Renard}, $P_w$ is a standard parabolic subgroup of $G$ included in $P$, with standard Levi factor $M_P \cap w^{-1} M_Q w$. In the same way, $Q_w=(M_Q \cap w P w^{-1})N_Q$ is standard parabolic subgroup of $G$ included in $Q$, with standard Levi factor $M_Q \cap w M_P w^{-1}$. Note that $w \Sigma_{P_w}^P \subset \Sigma_{Q_w}$ and $w^{-1} \Sigma_{Q_w}^Q \subset \Sigma_{P_w}$. Set
\begin{align*}
    W(P;Q)&=\{ w \in {}_Q W_P \; | \;  P_w=P \}=\{w \in {}_Q W_P \; | \; M_P \subset w^{-1} M_Q w \}, \\
    W(P,Q)&=\{w \in {}_Q W_P \; | \; M_P = w^{-1} M_Q w \}.
\end{align*}
Note that $w \in {}_Q W_P$ implies $w \in W(P_w,Q_w)$. Set
\begin{equation*}
    W(P)=\bigcup_{Q} W(P,Q).
\end{equation*}
Write $w_P$ for the longest element in $W(P)$.

If $R$ is another standard parabolic subgroup of $G$, we write ${}_Q W^R_P$ (resp. $W^R(P;Q)$ and $W^R(P,Q)$) for ${}_{Q \cap M_R} W_{P\cap M_R}$ (resp. $W(P \cap M_R; Q \cap M_R)$ and $W(P \cap M_R, Q \cap M_R)$) relatively to the reductive group $M_R$.

\subsection{Automorphic quotients and Haar measures}

We now assume that $F$ is a number field. Let $G$ be a connected reductive group over $F$.

\subsubsection{Automorphic quotients} 
\label{subsubsec:automorphic_quotients}

Let $\bA$ be the adele ring of $F$, let $\bA_f$ be its ring of finite adeles. Set $F_\infty=F \otimes_\qq \rr$. Let $V_F$ be the set of places of $F$ and let $V_{F,\infty} \subset V_F$ be the subset of Archimedean places. For $v \in V_F$, let $F_v$ be the completion of $F$ at $v$. If $v$ is non-Archimedean, let $q_v$ be the cardinality of the residual field of $F_v$ and $\oo_{v}$ be its ring of integers. Let $\Val{\cdot}$ be the absolute value $\bA^\times \to \rr_+^\times$ given by taking the product of the normalized absolute values $\Val{\cdot}_v$ on each $F_v$. 

Let $P=M_P N_P$ be a semi-standard parabolic subgroup of $G$. Set 
\begin{equation*}
    [G]_{P}=M_P(F) N_P(\bA) \backslash G(\bA).
\end{equation*}
Let $A_{P,\qq}$ be the maximal $\qq$-split subtorus of the Weil restriction $\Res_{F/\qq}A_P$, and let $A_P^\infty$ be the neutral component of $A_{P,\qq}(\rr)$. Set 
\begin{equation*}
    [G]_{P,0}=A_P^\infty M_P(F) N_P(\bA) \backslash G(\bA).
\end{equation*}
If $P=G$, we simply write $[G]$ and $[G]_0$ for $[G]_G$ and $[G]_{G,0}$ respectively.

Let $P$ be a semi-standard parabolic subgroup of $G$. There is a canonical morphism $ H_P : P(\bA) \to \fa_P$ such that $\langle \chi, H_P(g) \rangle=\log \Val{\chi(g)}$ for any $g \in P(\bA)$ and $\chi \in X^*(P)$. The kernel of $H_P$ is denoted by $P(\bA)^1$. We extend it to $ H_P : G(\bA) \to \fa_P$ which satisfies: for any $g \in G(\bA)$ we have $H_P(g)=H_P(p)$ whenever $g \in pK$ with $p \in P(\bA)$. If $P=P_0$, we write $H_0=H_{P_0}$.

We set
\begin{equation*}
    [G]_P^1=M_P(F) N_P(\bA) \backslash P(\bA)^1 K.
\end{equation*}
If $P=G$, we simply write $[G]^1$.

Let $K=\prod_{v \in V_F} K_v \subset G(\bA)$ be a "good" maximal compact subgroup in good position relative to $M_0$. We write $K=K_\infty K^\infty$ where $K_\infty=\prod_{v \in V_{F,\infty}} K_v$ and $K^\infty=\prod_{v \in V_F \setminus V_{F,\infty}} K_v$. By a \emph{level} $J$ of $G$, we mean an open-compact subgroup $J$ of $G(\bA_f)$.

\subsubsection{Modular characters}
If $P \subset Q$ are semi-standard parabolic subgroups of $G$, let $\rho_P^Q$ be the unique element in $\fa_P^{Q,*}$ such that for every $m \in M_P(\bA)$ we have
\begin{equation*}
    \Val{\det(\mathrm{Ad}_P^Q(m))}=\exp(\langle 2 \rho_P^Q,H_P(m) \rangle),
\end{equation*}
where $\mathrm{Ad}_P^Q$ is the adjoint action of $M_P$ on the Lie algebra of $M_Q \cap N_P$. For every $g \in G(\bA)$, we then set 
\begin{equation*}
    \delta_P^Q(g):=\exp(\langle 2 \rho_P^Q,H_P(g) \rangle).
\end{equation*}
In particular, when restricted to $P(\bA)\cap M_Q(\bA)$ it coincides with the restriction of the modular character of the latter. If $Q=G$, we omit the superscript.

\subsubsection{Haar measures} \label{subsubsec:first_measure} 

We take a Haar measure $dg$ on $G(\bA)$, with factorization $dg=\prod_v d g_v$ where for all place $v$, $dg_v$ is a Haar measure on $G(F_v)$. This implicitly implies that for almost all place $v$ the volume of $K_v$ is $1$. 

Let $P$ be a semi-standard parabolic subgroup of $G$. We equip $\fa_P$ with the Haar measure that gives covolume $1$ to the lattice $\Hom(X^*(P),\zz)$. We equip $A_P^\infty$ with the Haar measure compatible with the isomorphism $A_P^\infty \simeq \fa_P$ induced by $H_P$. If $P \subset Q$, we equip $\fa_P^Q=\fa_P/\fa_Q$ with the quotient measure. 

For each $v \in V_F$, we give $K_v$ the invariant probability measure. This yields a product measure on $K$. If $N$ is an unipotent group, we give $N(\bA)$ the Haar measure whose quotient by the counting measure on $N(F)$ gives $[N]$ volume $1$. We equip $M_P(\bA)$ with the unique Haar measure such that
\begin{equation}
\label{eq:Levi_measure}
    \int_{G(\bA)} f(g) dg= \int_{N_P(\bA)} \int_{M_P(\bA)} \int_K f(nmk) \exp(-\langle 2 \rho_P,H_P(m) \rangle) dkdmdn
\end{equation}
for every continuous and compactly supported function $f$ on $G(\bA)$. We equip $M_P(\bA)^1$ with the Haar measure compatible with the isomorphism $M_P(\bA)^1 \times A_P^\infty \to M_P(\bA)$.

We give $[G]_P$ the quotient of our measure on $G(\bA)$ by the product of the counting measure on $M_P(F)$ with our measure on $N_P(\bA)$. Moreover, note that the action of $a \in A_P^\infty$ by left translation on $[G]_P$ multiplies the measure by $\delta_P^{-1}(a)$. By taking the quotient of the measure on $[G]_P$ by that of $A_P^\infty$, we obtain a "semi-invariant" measure on $[G]_{P,0}$.

\subsection{Functions on automorphic quotients}

We keep the assumption that $F$ is a number field and that $G$ is connected reductive over $F$.

Let $X$ be a set, let $f$ and $g$ be two positive functions on $X$. We write
\begin{equation*}
    f(x) \ll g(x), \quad x \in X,
\end{equation*}
if there exists $C>0$ such that $f(x) \leq C g(x)$ for all $x \in X$.

\subsubsection{Smooth functions} Let $\fg_\infty$ be the Lie algebra of $G(F_\infty)$, let $\cU(\fg_\infty)$ be the enveloping algebra of its complexification and let $\cZ(\fg_\infty)$ be the center of $\cU(\fg_\infty)$. If we only care about $G(F_v)$ for a single Archimedean place $v$ of $F$, we will write $\cU(\fg_{v,\infty})$ instead.

By a level $J$ we mean a normal open compact subgroup of $K^\infty$. If $V$ is a representation of $G(\bA)$, we denote by $V^J$ its subspace of vectors fixed by $J$.

Let $V$ be a Fréchet space. We say that a function $\varphi : G(\bA) \to V$ is smooth if it is right-invariant by some level $J$ and if for every $g_f \in G(\bA_f)$, the function $g_\infty \in G(F_\infty) \mapsto \varphi(g_f g_\infty)$ is smooth in the usual sense (i.e. belongs to $C^\infty(G(F_\infty))$). We write $R$ (resp. $L$) for the actions by right-translation (resp. left-translation) of $G(\bA)$ and $\cU(\fg_\infty)$ on such smooth functions.

\subsubsection{Heights} \label{subsec:heights} We take an embedding $\iota : G \hookrightarrow \GL_n$ for some integer $n>0$. We define a height $\norm{\cdot}$ on $G(\bA)$ by 
\begin{equation*}
    \norm{g}=\prod_v \max_{1 \leq i,j \leq n} (\Val{\iota(g)_{i,j}}_v,\Val{\iota(g^{-1})_{i,j}}_v), \quad g \in G(\bA).
\end{equation*} 
If we choose another embedding $\iota'$ yielding $\norm{\cdot}'$, then there exists $r>0$ such that $\norm{g}^{1/{r}} \ll \norm{g}' \ll \norm{g}^{r}$ for $g \in G(\bA)$. By \cite[Equation~(2.4.1.1)]{BPCZ}, we have the formula
\begin{equation}
    \label{eq:mult_norm}
    \norm{gh} \ll \norm{g}\norm{h}, \quad g,h \in G(\bA).
\end{equation}
If $P$ is a semi-standard parabolic subgroup of $G$, for any $g \in G(\bA)$ we define
\begin{equation*}
    \norm{g}_P= \inf_{\delta \in M_P(F) N_P(\bA)} \norm{\delta g}.
\end{equation*}
The heights satisfy the following properties.

\begin{lem}
    \label{lem:height_properties}
    The following assertions hold.
    \begin{itemize}
        \item There exists $N>0$ such that $g \mapsto \norm{g}_G^{-N}$ is absolutely integrable on $[G]$.
        \item For every $M>0$ there exists a compact set $\cU \subset [G]$ such that $g \notin \cU$ implies $\norm{g}_G \geq M$.
    \end{itemize}
\end{lem}

\begin{proof}
    The first assertion is \cite[Proposition~A.1.1.(vi)]{Beu}. For the second, by \cite[Section~2.4.3]{BPCZ} we are reduced to proving the fact on a Siegel domain of $G$ (see \cite[Section~2.2.13]{BPCZ}), and hence for $g \in A_0^\infty$. But there we can express $\norm{g}_{G}$ using characters by \cite[Section~2.4.3]{BPCZ}, and hence easily conclude.
\end{proof}

\subsubsection{Schwartz functions} For every compact subset $C$ of $G(\bA_f)$ and every level $J$, let $\cS(G(\bA),C,J)$ be the space of smooth functions $f : G(\bA) \to \cc$ such that
\begin{itemize}
    \item $f$ is biinvariant by $J$ and is supported on $G(F_\infty) \times C$;
    \item for every integer $r \geq 1$ and $X,Y \in \cU(\fg_\infty)$ we have
    \begin{equation*}
        \norm{f}_{r,X,Y}:= \sup_{g \in G(\bA)} \norm{g}^r \Val{(R(X)L(Y)f)(g)} < \infty.
    \end{equation*}
\end{itemize}
We equip $\cS(G(\bA),C,J)$ with the family of semi-norms $\norm{ \cdot }_{r,X,Y}$, and let $\cS(G(\bA))$ be the locally convex topological direct limit of the spaces $\cS(G(\bA),C,J)$ over the pairs $(C,J)$. It is the space of Schwartz functions on $G(\bA)$, and it is an algebra for the convolution product $*$. For any level $J$, we denote by $\cS(G(\bA))^J$ its subalgebra of $J$-biinvariant functions. 

\subsubsection{Petersson innner--product} We fix a semi-standard parabolic subgroup $P$ of $G$ for the reminder of this section. We have the Hilbert space $L^2([G]_P)$ of square-integrable functions on $[G]_P$. We will also consider $L^2([G]_{P,0})$ the space of functions on $[G]_P$ that transform by $\delta_P^{1/2}$ under left-translation by $A_P^\infty$ and such that the Petersson-norm
\begin{equation*}
    \norm{\varphi}_{P,\Pet}^2=\langle \varphi,\varphi \rangle_{P,\Pet}:=\int_{[G]_{P,0}} \Val{\varphi(g)}^2 dg,
\end{equation*}
is finite. If $J$ is a level $G$, we write $L^2([G]_{P,0})^{\infty,J}$ for the space of $J$-invariant functions $\varphi$ in $L^2([G]_{P,0})$ such that the orbit map $g \mapsto g. \varphi$ is smooth. The space $L^2([G]_{P,0})^{\infty,J}$ is given the topology induced by the family of semi-norms $\norm{X.\varphi}_{P,\Pet}$ for $X \in \cU(\fg_\infty)$. Then $L^2([G]_{P,0})^\infty=\bigcup_J L^2([G]_{P,0})^{\infty,J}$ is given the locally convex direct limit topology.

\subsubsection{Spaces of functions}
\label{subsubsec:spaces_functions}
For all $N \in \rr$, $X \in \cU(\fg_\infty)$ and any smooth function $\varphi : [G]_P \to \cc$ we define
\begin{equation*}
    \norm{\varphi}_{N,X}=\sup_{x \in [G]_P} \norm{x}_P^N \Val{(R(X)\varphi)(x)}.
\end{equation*}
If $X=1$, we simply write $\norm{\varphi}_{N}$.

For every $N \in \rr$, let $\cT_N([G]_P)$ be the space of smooth functions $\varphi : [G]_P \to \cc$ such that for every $X \in \cU(\fg_\infty)$ we have $\norm{\varphi}_{-N,X} < \infty$. For every level $J$, we equip $\cT_N([G]_P)^J$ with the topology of Fréchet space induced by the family of semi-norms $(\norm{\cdot}_{-N,X})_X$. Set
\begin{equation*}
    \cT([G]_P)=\bigcup_{N>0} \cT_N([G]_P).
\end{equation*}
This is the space of \emph{functions of uniform moderate growth on $[G]_P$}. It is equipped with a natural topology of locally Fréchet space.

Let $\cS([G]_P)$ be the space of smooth functions $\varphi : [G]_P \to \cc$ such that for every $N \geq 0$ and $X \in \cU(\fg_\infty)$ we have $\norm{\varphi}_{N,X} < \infty$. For every level $J$, we equip $\cS([G]_P)^J$ with the Fréchet topology induced by the family of semi-norms $(\norm{\cdot}_{N,X})_{N,X}$. The space $\cS([G]_P)$ is the \emph{Schwartz space} of $[G]_P$.

By \cite[Section~2.5.10]{BPCZ}, $\cS([G]_P)$ is dense in $\cT([G]_P)$. It is in general not dense in $\cT_N([G]_P)$, but we have the following weaker result.

\begin{lem}
\label{lem:approx_infini}
    Let $N\in \rr$. Then the closure of $\cS([G])$ in $\cT_{N+1}([G])$ contains $\cT_{N}([G])$.
\end{lem}

\begin{proof}
    By the Dixmier--Malliavin theorem of \cite{DM}, it is enough to show that the statement holds for the topology induced by the sole norm $\norm{\cdot}_{-N-1}$. If $[G]$ is compact, this is automatic. In general, we have the following fact: for every $M>0$ there exists a compact set $\cU \subset [G]$ such that $g \notin \cU$ implies $\norm{g}_G \geq M$. By Lemma~\ref{lem:height_properties} it is enough to approximate elements in $\cT_{N}([G])$ by functions in $\cS([G])$ on compact sets. But it follows from an easy adaptation of \cite[Theorem~8.4]{MZ} that this can be done using Poincar\'e series of Schwartz functions in $\cS(G(\bA))$. 
\end{proof}

We will also make use of some non-smooth variants of the above spaces. For every $N \in \rr$, let $\cT_N^0([G]_P)$ be the space of complex Radon measures $\varphi$ on $[G]_P$ such that
\begin{equation*}
    \norm{\varphi}_{1,N}:=\int_{[G]_P} \norm{g}^{-N} \Val{\varphi(g)}  < \infty.
\end{equation*}
We equip $\cT^0_N([G]_P)$ with the topology associated to the norm $\norm{\cdot}_{1,N}$ so that it is Banach, and let $\cT^0([G]_P)$ be the locally convex direct limit of the spaces $\cT^0_N([G]_P)$. 

Let $\cS^{00}([G]_P)$ be the space of continuous measurable complex-valued functions on $[G]_P$ such that for every $N>0$ we have $\norm{\varphi}_N < \infty$. We have a pairing
\begin{equation}
\label{eq:pairing_P}
    \langle \varphi, \psi \rangle_P=\int_{[G]_P} \varphi(g) \overline{\psi}(g) , \quad \varphi \in \cS^{00}([G]_P), \quad \psi \in \cT^0([G]_P).
\end{equation}
It identifies the topological dual of $\cS^{00}([G]_P)$ with $\cT^0([G]_P)$ (see \cite[Section~2.5.9]{BPCZ}).

\subsection{Automorphic representations}

We keep the assumption that $F$ is a number field and that $G$ is reductive over $F$. 

\subsubsection{Automorphic forms}
Let $P$ be a semi-standard parabolic subgroup of $G$. We define the space of automorphic forms $\cA_P(G)$ to be the subspace of $\cZ(\fg_\infty)$-finite functions in $\cT([G]_P)$. 

For any ideal $\cJ \subset \cZ(\fg_\infty)$ of finite codimension, we denote by $\cA_{P,\cJ}(G)$ the subspace of $\varphi \in \cA_P(G)$ such that $R(z)\varphi=0$ for every $z \in \cJ$. By \cite[Section~2.7.1]{BPCZ}, there exists $N \geq 1$ such that $\cA_{P,\cJ}(G) \subset \cT_{N}([G]_P)$. We give $\cA_{P,\cJ}(G)$ the induced topology. It is independent from $N$ by \cite[Lemma~2.5.4.1]{BPCZ} and by the open mapping theorem. Then $\cA_P(G)=\bigcup_{\cJ} \cA_{P,\cJ}(G)$ is given the locally convex direct limit topology.

Let $\cA_P^0(G)$ be the subspace of $\varphi \in \cA_P(G)$ such that
\begin{equation*}
    \varphi(ag)=\exp(\langle \rho_P,H_P(a) \rangle) \varphi(g)
\end{equation*}
for every $a \in A_P^\infty$ and $g \in [G]_P$. If $P=G$ we simply write $\cA(G)$ and $\cA^0(G)$.

Let $\cA_{P,\disc}(G) \subset \cA_P^0(G)$ be the subspace of $\varphi$ such that the Petersson norm $\norm{\varphi}_{P,\Pet}$ is finite. The spaces $\cA_{P}^0(G)$ and $\cA_{P,\disc}(G)$ are given the subspace topology.

\begin{rem}
\label{rem:K_finite}
    In contrast with most references, we follow \cite{BPCZ} and do not ask that our automorphic forms are $K_\infty$-finite. By \cite{Lap}, the main results on the analytic extensions of Eisenstein series and intertwining operators in the $K_\infty$-finite case propagate to the smooth case. 
\end{rem}

\subsubsection{Discrete automorphic representations}
\label{subsubsec:discrete_automorphic_rep}
We define a discrete automorphic representation of $G(\bA)$ to be a topologically irreducible subrepresentation of $\cA_{\disc}(G)$. Let $\Pi_{\disc}(G)$ be the set of such representations. For $\pi \in \Pi_{\disc}(G)$, let $\cA_\pi(G)$ be the $\pi$-isotypic component of $\cA_{\disc}(G)$. Note that $\pi$ always has trivial central character on $A_G^\infty$.

For $\pi \in \Pi_{\disc}(M_P)$, let $\cA_{P,\pi}(G)$ be the subspace of $\varphi \in \cA_{P,\disc}(G)$ such that for all $g \in G(\bA)$ the map $ m \in [M_P] \mapsto \delta_P(m)^{-1/2} \varphi(mg)$ belongs to $\cA_\pi(M_P)$. It is a closed subspace of $\cA_{P,\cJ}(G)$ for some ideal of finite codimension $\cJ$ and we give it the induced topology. For any $\lambda \in \fa_{P,\cc}^*$, set $\pi_\lambda=\pi \otimes \exp( \langle \lambda, H_{M_P}(\cdot) \rangle )$ and for $\varphi \in \cA_{P,\pi}(G)$ define
\begin{equation*}
    \varphi_\lambda(g)=\exp( \langle \lambda, H_{P}(g) \rangle ) \varphi(g).
\end{equation*}
The map $\varphi \mapsto \varphi_\lambda$ identifies $\cA_{P,\pi}(G)$ with a subspace of $\cA_P(G)$ denoted by $\cA_{P,\pi,\lambda}(G)$. We denote by $I_P(\lambda)$ the actions of $G(\bA)$ and $\cS(G(\bA))$ we obtain on $\cA_{P,\pi,\lambda}(G)$ by transporting those on $\cA_P(G)$.

Let $\pi \in \Pi_{\disc}(M_P)$. By \cite{Flath}, it decomposes as $\pi=\otimes'_v \pi_v$. For every place $v$, we write $I_P^G \pi_v$ for the smooth parabolic induction of $\pi_v$ for $G(F_v)$.

\subsubsection{Topologies on spaces of automorphic forms}

Let $\pi \in \Pi_\disc(M_P)$ for $M_P$ some standard Levi of $G$. Because $\pi$ is discrete, we have another choice of topology on $\cA_{P,\pi}(G)$ by realizing it as a subspace of $L^2([G]_{P,0})^\infty$. The following lemma explains how to compare these two topologies. We only state it for $G=\GL_n$, although it should hold for any reductive groups. More precisely, the potential issue lies within the first assertion where information on the exponents of discrete automorphic forms is used.

\begin{lem}
\label{lem:sobolev_and_co}
    Assume that $G=\GL_n$ and let $P$ be a standard parabolic subgroup of $G$. Let $J$ be a level. Then for any $N>0$ sufficiently large there exist $X_1, \hdots, X_r \in \cU(\fg_\infty)$ such that for all $\pi \in \Pi_\disc(M_P)$ and all $\varphi \in \cA_{P,\pi}(G)^J$ we have 
    \begin{equation*}
        \norm{\varphi}_{P,\Pet} \leq \sum_{i=1}^r \norm{\varphi}_{-N,X_i}.
    \end{equation*}
    In particular, $\cA_{P,\pi}(G)^J$ is included in the space of smooth vectors $L^2([G]_{P,0})^{J,\infty}$, and if we endow it with the induced topology we have a closed embedding $\cA_{P,\pi}(G)^J \subset \cT_{N}([G]_P)^J$ (where $N$ can be chosen independently of $\pi$). 

    Conversely, without assuming that $G=\GL_n$, for any $N>0$ large enough and any $X \in \cU(\fg_\infty)$, there exist $Y_1, \hdots, Y_r \in \cU(\fg_\infty)$ such that for any $\varphi \in L^2([G]_{P,0})^{J,\infty}$ we have 
    \begin{equation*}
        \norm{\varphi}_{-N,X} \leq \sum_{i=1}^r \norm{R(Y_i) \varphi}_{P,\Pet}.
    \end{equation*}
    In fact, up to constant we can take $Y_i=\Delta^i$ where $\Delta$ is the Laplace--Beltrami operator defined in \eqref{eq:Laplace} below.
\end{lem}

\begin{proof}
    The first assertion is \cite[Lemma~3.1.2.1]{Ch}, the second is the Sobolev inequality (\cite[\S3.4, Key Lemma]{Be2}, see also \cite[Lemma~3.8.1.1]{Ch}). The last part is a consequence of \cite[Proposition~3.5]{BK}.
\end{proof}

\subsubsection{Constant terms and cuspidal representations}

For $Q$ a standard parabolic subgroup and $\varphi \in \cA_P(G)$, we have a constant term $\varphi_Q$ defined by
\begin{equation*}
    \varphi_Q(g)=\int_{[N_Q]} \varphi(ng) dn, \quad g \in [G]_Q.
\end{equation*}
Let $\cA_{P,\cusp}(G) \subset \cA^0_P(G)$ be the subspace of $\varphi$ such that $\varphi_Q=0$ for all $Q \subsetneq P$. Let $\Pi_{\mathrm{cusp}}(G)$ be the set of topologically irreducible subrepresentations of $\cA_{\mathrm{cusp}}(G)$, where we equip this space with the subspace topology from $\cA(G)$. It is a subset of $\Pi_{\disc}(G)$. 

\subsubsection{Intertwining operators} Let $P$ and $Q$ be standard parabolic subgroups of $G$. Let $\pi \in \Pi_\disc(M_P)$. Let $w \in W(P,Q)$ and $\lambda \in \fa_{P,\cc}^*$ such that $\langle \Re(\lambda), \alpha^\vee \rangle$ is large enough for any $\alpha \in \Delta_P$ such that $w \alpha <0$. For $\varphi \in \cA_{P,\pi}(G)$, consider the absolutely convergent integral
\begin{equation*}
    (M(w,\lambda) \varphi)_{w \lambda}(g)=\int_{(N_Q \cap w N_P w^{-1})(\bA) \backslash N_Q(\bA)} \varphi_\lambda(w^{-1} ng)dn, \quad g \in [G]_Q.
\end{equation*}
By \cite{Langlands} and \cite{BL}, it admits a meromorphic continuation to $\fa_{P,\cc}^*$ if $\varphi$ is $K_\infty$-finite. By \cite{Lap}, this holds for any $\varphi \in \cA_{P,\pi}(G)$ and defines a continuous intertwining operator for any regular $\lambda$
\begin{equation}
\label{eq:global_operator}
    M(w,\lambda) : \cA_{P,\pi}(G) \to \cA_{Q,w\pi}(G).
\end{equation}
By \cite[Theorem~2.3]{BL}, the singularities of $M(w,\lambda)$ are located along affine root hyperplanes.

Let $Q' \subset Q$ and $P' \subset P$ such that $w \in W(P',Q')$. Then we have for $\varphi \in \cA_{P,\pi}(G)$
\begin{equation}
\label{eq:constant_inter}
    \left(M(w,\lambda)\varphi \right)_{Q'}=M(w,\lambda)\varphi_{P'}.
\end{equation}
Moreover, if $R$ is another standard parabolic and if $w_1 \in W(P,Q)$ and $w_2 \in W(Q,R)$, by \cite[Theorem~2.3.5]{BL} we have the functional equation
\begin{equation}
\label{eq:M_equation}
    M(w_2,w_1\lambda)M(w_1,\lambda)\varphi=M(w_2 w_1,\lambda)\varphi.
\end{equation}

\subsubsection{Eisenstein series}
\label{subsubsec:Eisenstein}
Let $P \subset Q$ be standard parabolic subgroups of $G$. For any $\varphi \in \cA_{P,\disc}(G)$ and $\lambda \in \fa_{P,\cc}^*$ we define
\begin{equation}
\label{eq:partial_Eisenstein}
    E^Q(g,\varphi,\lambda)=\sum_{\gamma \in P(F) \backslash Q(F)} \varphi_\lambda(\gamma g)=\sum_{\gamma \in M_Q \cap P(F) \backslash M_Q(F)} \varphi_\lambda(\gamma g), \quad g \in G(\bA).
\end{equation}
This sum is absolutely convergent for $\Re(\lambda)$ in a suitable cone. If $\varphi$ is $K_\infty$-finite, it admits once again a meromorphic continuation to $\fa_{P,\cc}^*$ by \cite{Langlands} and \cite{BL}, and this holds for any $\varphi \in \cA_{P,\disc}(G)$ by \cite{Lap}. If $Q=G$, we simply write $E(g,\varphi,\lambda)$. By \cite[Theorem~2.3]{BL}, the singularities of $E^Q(\varphi,\lambda)$ are located along affine root hyperplanes.

For regular $\lambda$, let $E_Q^G(\varphi,\lambda)$ be the constant term of $E(\varphi,\lambda)$ along $Q$. By \cite[Lemma~6.10]{BL}, we have
\begin{equation}
\label{eq:constant_term}
    E_Q^G(\varphi,\lambda)=\sum_{w \in {}_Q W_P} E^Q(M(w,\lambda)\varphi_{P_w},w \lambda).
\end{equation}

We have the following easy relation between intertwining operators and Eisenstein series.

\begin{lem}
\label{lem:ME=EM}
    Let $\varphi \in \cA_{P,\disc}(G)$. Let $Q, Q'$ be two standard parabolic subgroups of $G$ such that $P \subset Q$. Let $w \in W(Q,Q')$. Then for regular $\lambda \in \fa_{P,\cc}^*$ we have
    \begin{equation}
    \label{eq:ME=EM}
        M(w,\lambda)E^Q(\varphi,\lambda)=E^{Q'}(M(w,\lambda) \varphi,w\lambda).
    \end{equation}
\end{lem}

\begin{proof}
   This holds in the region of absolute convergence, and for $\lambda$ in general position by analytic continuation.
\end{proof}

\subsubsection{Cuspidal components and residual automorphic forms} 
Let $\varphi \in \cA_{P}^0(G)$. As cuspidal automorphic forms are of rapid decay (\cite[Section~I.2.18.]{MW95}), for every $\varphi_0 \in \cA_{P,\mathrm{cusp}}(G)$ the pairing $\langle \varphi, \varphi_0 \rangle_{P,\Pet}$ makes sense. By \cite[Section~I.2.18]{MW95}, there exists a unique $\varphi^{\mathrm{cusp}} \in \cA_{P,\mathrm{cusp}}(G)$ such that for all $\varphi_0 \in \cA_{P,\mathrm{cusp}}(G)$ we have $    \langle \varphi, \varphi_0 \rangle_{P,\Pet}=\langle \varphi^{\mathrm{cusp}}, \varphi_0 \rangle_{P,\Pet}$. This definition is then generalized to any $\varphi \in \cA_P(G)$ that is finite under the action by left translation of $A_P^\infty$ (see \cite[Section~I.3.4]{MW95}). The tuple $(\varphi_Q^{\mathrm{cusp}})_{Q \subset P}$ (where $\varphi_Q^{\mathrm{cusp}}$ is the cuspidal component of the constant term $\varphi_Q$) is called the family of cuspidal components of $\varphi$. 

If $\varphi^\cusp=0$, then we say that $\varphi$ is \emph{residual}. For any regular $\lambda$ we have $E(\varphi,\lambda) \in \cA(G)$. If this Eisenstein series is proper (that is if $P\neq G$), then $E(\varphi,\lambda)$ is residual

\subsubsection{Pairs, triples and bases}
\label{subsec:bases}
Let $P$ be a standard parabolic subgroup of $G$. For any level $J$, let $\Pi_{\disc}(M_P)^J$ (resp. $\Pi_{\mathrm{cusp}}(M_P)^J$) be the subset of $\pi \in \Pi_{\disc}(M_P)$ (resp. $\pi \in \Pi_{\mathrm{cusp}}(M_P)$) such that $\cA_{P,\pi}(G)^J \neq \{0\}$. Let $e_J$ the measure supported on $J$ of volume $1$. Then the projection $R(e_J)$ sends $\cA_{P,\pi}(G)$ to $\cA_{P,\pi}(G)^J$.

Let $\widehat{K}_\infty$ be the set of isomorphism classes of irreducible unitary representations of $K_\infty$. For any $\tau \in \widehat{K}_\infty$, let $\cA_{P,\pi}(G)^\tau$ be $\tau$-isotypic component of $\cA_{P,\pi}(G)$. For any level $J$, set
\begin{equation*}
    \cA_{P,\pi}(G)^{\tau,J}=\cA_{P,\pi}(G)^{\tau} \cap \cA_{P,\pi}(G)^{J}.
\end{equation*}

Following \cite[Section~3.2.3]{Ch}, we call $J$-pair any $(P,\pi)$ where $P$ is a standard parabolic subgroup of $G$ and $\pi \in \Pi_{\disc}(M_P)^J$. If $\tau \in \widehat{K}_\infty$, we call $\tau$-pair any $(P,\pi)$ with $\cA_{P,\pi}(G)^\tau \neq \{0\}$. We call a $J$-triple any $(P,\pi,\tau)$ where $(P,\pi)$ is a $J$-pair and $\tau \in \widehat{K}_\infty$ with the additional requirement $\cA_{P,\pi}(G)^{\tau,J} \neq \{0\}$. Note that this subspace is always of finite dimension. Let $\cB_{P,\pi}(\tau,J)$ be an orthonormal basis of $\cA_{P,\pi}(G)^{\tau,J}$ with respect to $\langle \cdot,\cdot \rangle_{P,\Pet}$. We then define $\cB_{P,\pi}(J)$ to be the union over $\tau \in \widehat{K}_\infty$ of the $\cB_{P,\pi}(\tau,J)$. 
If $\tau \in \widehat{K}_\infty$, let $e_\tau$ be the measure supported on $K_\infty$ given by $e_\tau(k)=\deg(\tau)\trace(\tau(k))dk$, where $dk$ is the probability Haar measure on $K_\infty$. Then the projection $e_\tau$ sends $\cA_{P,\pi}(G)$ to $\cA_{P,\pi}(G)^\tau$. For $f \in \cS(G(\bA))$ and $\tau \in \widehat{K}_\infty$, set
\begin{equation}
    \label{eq:f_tau}
    f_{\tau}:=f * e_{\tau}.
\end{equation}

\subsubsection{Numerical invariants} 
\label{subsubsec:numerical_invariants}
We borrow some notation from \cite[Section~3.2.2]{Ch}. Set
    \begin{equation}
    \label{eq:Laplace}
        \Delta=\mathrm{Id}-\Omega_G+2\Omega_K,
    \end{equation}
    where $\Omega_G$ and $\Omega_{K_\infty}$ are the Casimir operator of $G$ and $K_\infty$ respectively associated to the standard Killing form on $\fg_\infty$ corresponding to the trace. 
    
    For any $\tau \in \widehat{K}_\infty$, let $\lambda_\tau$ be the Casimir eigenvalue of $\tau$. Let $P$ be a standard parabolic subgroup of $G$. Let $\pi_\infty$ be an irreducible unitary representation of $M_P(F_\infty)$ and let $\lambda_{\pi_\infty}$ be the Casimir eigenvalue of $\pi_\infty$. Set
\begin{equation*}
    \Lambda_{\pi_\infty}^{M_P}=\sqrt{\lambda_{\pi_\infty}^2+\lambda_\tau^2}, 
\end{equation*}
where $\tau$ is a minimal $K_\infty \cap M_P(F_\infty)$-type of $\pi_\infty$, i.e. whose infinitesimal character has minimal norm, and 
\begin{equation*}
    \Lambda_{\pi_\infty}^G=\min_\tau \sqrt{\lambda_{\pi_\infty}^2+\lambda_\tau^2}, 
\end{equation*}
where the minimum is taken over minimal $K_\infty$-types of $\Ind_{P(F_\infty)}^{G(F_\infty)} \pi_\infty$. 

If $\pi \in \Pi_{\mathrm{disc}}(M_P)$ with Archimedean component $\pi_\infty$, set
\begin{equation*}
    \lambda_\pi=\lambda_{\pi_\infty}, \quad \Lambda_{\pi}^{M_P} =\Lambda_{\pi_\infty}^{M_P}, \quad \Lambda_{\pi}=\Lambda_{\pi_\infty}^G.
\end{equation*}
A key property on the Casimir eigenvalues is that, if $(P,\pi,\tau)$ is a $J$-triple, then by \cite[Lemma~6.1]{Mu2} we have $\lambda_\tau \geq \lambda_\pi$.
    
\subsubsection{$\cR$-regions}
\label{subsec:R_region}
Let $P$ be a standard parabolic subgroup of $G$. Let $\pi \in \Pi_{\disc}(M_P)$. We define for $k>0$, $c>0$ and $C>0$
\begin{equation}
\label{eq:R_defi}
    \cR_{\pi,k,c}=\left\{ \lambda \in \fa_{P,\cc}^* \; \middle| \; \forall \alpha \in \Sigma_P, \; \langle \Re(\lambda),\alpha^\vee \rangle > -c(1+\Lambda_\pi^{M_P}+\Val{\langle \Im(\lambda),\alpha^\vee \rangle})^{-k} \right\},
\end{equation}
and 
\begin{equation}
\label{eq:R_defi_bounded}
    \cR_{\pi,k,c}^C=\left\{ \lambda \in \fa_{P,\cc}^* \; \middle| \; \forall \alpha \in \Sigma_P, \; \langle \Re(\lambda),\alpha^\vee \rangle > -c(1+\Lambda_\pi^{M_P}+\Val{\langle \Im(\lambda),\alpha^\vee \rangle})^{-k}, \quad \norm{\Re(\lambda)}<C  \right\}.
\end{equation}
We also define 
\begin{equation}
\label{eq:S_defi}
    \cS_{\pi,k,c}=\left\{ \lambda \in \fa_{P,\cc}^* \; \middle| \;  \norm{\Re(\lambda)} < c(1+\Lambda_\pi^{M_P}+\norm{\Im(\lambda)})^{-k} \right\},
\end{equation}
These definitions are inspired by \cite[Section~3]{Lap2}. The two differences are that our sets are subsets of $\fa_{P,\cc}^*$ rather than $\fa_{P,\cc}^{G,*}$, and that in the two $\cR$ sets we allow $\Re(\lambda)$ to grow large in the positive direction.

Let $Q$ be another standard parabolic subgroup of $G$. Let $w \in {}_Q W_P$. Set
\begin{equation}
    \label{eq:R(w)_defi}
    \cR_{\pi,k,c}(w)=\bigcap_{\substack{\alpha \in \Sigma_{P_w} \\ w \alpha<0 }}  \left\{ \lambda \in \fa_{P_w}^* \; \middle| \; \langle \Re(\lambda),\alpha^\vee \rangle > -c(1+\Lambda_{\pi}^{M_{P}}+\Val{\langle \Im(\lambda),\alpha^\vee \rangle})^{-k} \right\},
\end{equation}
and 
\begin{equation}
    \label{eq:R(w)_defiC}
    \cR_{\pi,k,c}^C(w)=\bigcap_{\substack{\alpha \in \Sigma_{P_w} \\ w \alpha<0 }}  \left\{ \lambda \in \fa_{P_w}^* \; \middle| \; \langle \Re(\lambda),\alpha^\vee \rangle > -c(1+\Lambda_{\pi}^{M_{P}}+\Val{\langle \Im(\lambda),\alpha^\vee \rangle})^{-k}, \quad \norm{\Re(\lambda)}<C \right\}.
\end{equation}
These regions contain $\cR_{\pi,k,c}$ and $\cR_{\pi,k,c}^C$ respectively.

\subsection[Spectral decompositions]{Pseudo-Eisenstein series and spectral decompositions}

We keep the notation from the previous section. We now present some generalities on pseudo-Eisenstein series and state Langlands spectral decomposition theorem for the scalar product.

\subsubsection{Pseudo-Eisenstein series} \label{subsubsec:pseudo_eisenstein}

Let $\cP\cW(\fa_{P,\cc}^*)$ be the Paley--Wiener space of functions on $\fa_{P,\cc}^*$ obtained as Fourier transforms of compactly supported smooth functions on $\fa_P$. If $\cV$ is a finite-dimensional subspace of $K_\infty$-finite functions in $\cA_{P,\cusp}(G)$, we define $\cP\cW_{P,\cV}$ to be the space of $\cV$-valued entire functions on $\fa_{P,\cc}^*$ of Paley--Wiener type. We write $\cP \cW_P$ for the direct sum of all the $\cP \cW_{P,\cV}$. For $\Phi \in \cP\cW_{P}$ and any $\kappa \in \fa_P^*$, consider 
\begin{equation}
\label{eq:F_Phi_defi}
    F_\Phi(g)=\int_{\substack{\lambda \in \fa_{P,\cc}^* \\ \Re(\lambda)=\kappa}} \Phi(\lambda)(g)\exp(\langle \lambda,H_P(g) \rangle) d\lambda, \quad g \in G(\bA).
\end{equation}
It is independent of the choice of $\kappa$. We define the pseudo-Eisenstein series associated to $\Phi$ by
\begin{equation*}
    E(g,F_\Phi)=\sum_{\gamma \in P(F) \backslash G(F)} F_\Phi(\gamma g), \quad g \in [G].
\end{equation*}
where this sum is actually over a finite set which depends on $g$ by \cite[Lemma~5.1]{Art78}. This pseudo-Eisenstein series is rapidly decreasing. Moreover, by \cite[Section~II.1.11]{MW95} we have
\begin{equation}
\label{eq:pseudo_Eisenstein_unfold}
    E(g,F_\Phi)=\int_{\Re(\lambda)=\kappa} E(g,\Phi(\lambda),\lambda) d\lambda, \quad g \in [G],
\end{equation}
for any $\kappa$ in the region of the absolute convergence of Eisenstein series.

\begin{lem}
    \label{lem:pseudo_dense}
        The vector space 
        \begin{equation}
        \label{eq:pseudo_Eisenstein_space}
            \bigoplus_{P_0\subset P} \bigoplus_{\pi \in \Pi_\cusp(M_P)} E \left( \cP \cW_{P,\pi} \right)
        \end{equation}
        is dense in $\cS([G])$.
    \end{lem}
    
    \begin{proof}
        The space in \eqref{eq:pseudo_Eisenstein_space} is stable by the action of the left $K_\infty$-finite functions in $\cS(G(\bA))$. The latter form a dense subspace $\cS(G(\bA))$ (see \cite[931]{Art78}). By the Dixmier--Malliavin theorem of \cite{DM} (see also \cite[Section~2.5.3]{BPCZ}), it is therefore enough to show that \eqref{eq:pseudo_Eisenstein_space} is dense in $\cS^{00}([G])$. Let $l \in \cS^{00}([G])^*$ the topological dual of this space and assume that it is zero when restricted to \eqref{eq:pseudo_Eisenstein_space}. Using the pairing \eqref{eq:pairing_P}, we can identify it with an element $\varphi$ in some $\cT_N^0([G])$. Let $\delta_n$ be a Dirac sequence in $\cS(G(\bA))$, so that $R(\delta_n) l$ converges weakly to $l$. Then all the $R(\delta_n)l$ correspond to smooth elements $\varphi_n \in \cT_N^0([G]) \cap \cT([G])$. Moreover, by the same procedure as in \cite[Lemma~4.4.3]{GH}, we can choose the $\delta_n$ to be $K_\infty$-finite. It follows that all the $R(\delta_n)l$ are zero on \eqref{eq:pseudo_Eisenstein_space}. By the adjunction between constant terms and Eisenstein series from \cite[Equation~(2.5.13.12)]{BPCZ}, this implies that the $\varphi_n$ are orthogonal to all the spaces $\cP \cW_{P,\pi}$. By \cite[Theorem~II.1.12]{MW95} they must be zero, so that $l$ is as well. This concludes the proof. Note that \cite[Theorem~II.1.12]{MW95} is written for the space $L^2([G])$ but also holds for $\cT([G])$ (see \cite[Proposition~I.3.4]{MW95}).
    \end{proof}

\subsubsection{Langlands' spectral decomposition theorem}
\label{subsubsec:Langlands_spectral}

We now state the version of Langlands' spectral theorem from \cite{Langlands} for pseudo--Eisenstein series.

\begin{theorem}
\label{thm:Langlands_spectral}
   Let $J$ be a level. For every $\Phi \in \cP\cW^J$ and $\Phi' \in \cP\cW^J$ we have
    \begin{equation}
        \langle E(F_{\Phi}),E(F_{\Phi'}) \rangle_G \\
        =\sum_{P_0 \subset P} \frac{1}{\Val{\cP(M_P)}} \sum_{\pi \in \Pi_\disc(M_P)} \int_{i \fa_P^*}  \sum_{\varphi \in \cB_{P,\pi}(J)} \langle E(F_{\Phi'}),E(\varphi,\lambda) \rangle_G \langle E(\varphi,\lambda),E(F_{\Phi'}) \rangle_G d\lambda. \label{eq:Langlands_spectral}
    \end{equation}
\end{theorem}

Note that the sums in \eqref{eq:Langlands_spectral} are finite. In Theorem~\ref{thm:bound_Eisenstein}, we will show that the Eisenstein series involved in \eqref{eq:Langlands_spectral} form an integrable family in some space $\cT_N([G])$. This will let us extend Theorem~\ref{thm:Langlands_spectral} to functions of rapid enough decay in Proposition~\ref{prop:extended_Langlands}.

\section{Discrete Eisenstein series on \texorpdfstring{$\GL_n$}{GLn}}
\label{chap:poles}

In this chapter, the group $G$ is $\GL_n$ for some $n \geq 1$. We will use the following conventions. We choose $P_0$ to be the standard Borel subgroup of upper triangular matrices, and $M_0=T_0$ to be the diagonal maximal torus. The group $K$ is the standard maximal compact subgroup of $\GL_n(\bA)$. 

If $P$ is a standard parabolic subgroup of $\GL_n$, its standard Levi factor is of the form $M_P=\GL_{n_1} \times \hdots \times \GL_{n_m}$ for some integers $n_1, \hdots, n_m$. With this notation, we associate to $P$ the tuple $\underline{n}(P):=(n_1,\hdots,n_m)$. This completely characterizes $P$ among the standard parabolic subgroups of $\GL_n$. We will often write $M_P=\prod \GL_{n_i}$, where we implicitly assume that the product is taken in the order $i=1, \hdots, m$. We identify $\fa_P^*$ with $\rr^m$ by sending the canonical basis $(e_i^*)$ of $X^*(P)$ to the canonical basis of $\rr^m$. We will write $\lambda=(\lambda_1,\hdots,\lambda_m)$ with respect to this basis.

If $w \in W$, we take the representative of $w$ in $\GL_n(F)$ prescribed by \cite[Section~2]{KS}. If $P=M_P N_P$ is a standard parabolic subgroup of $\GL_n$, we have an embedding of $W(M_P)$ the Weyl group of $M_P$ inside $W$. Write $M_P=\GL_{n_1} \times \hdots \times \GL_{n_m}$. We have an identification (of sets) $W(P) \simeq \fS_m$ such that, if $\sigma \in \fS_m$, we have
\begin{equation*}
    \sigma M_P \sigma^{-1}=M_{n_{\sigma^{-1}(1)}} \times \hdots \times \hdots M_{n_{\sigma^{-1}(m)}}.
\end{equation*}
We will often identify a $w \in W(P)$ with an element in $\fS_m$. We will write $w.P$ for the standard parabolic subgroup of $\GL_n$ with standard Levi factor $wM_Pw^{-1}$. We say that $w \in W$ acts by permutation on the blocks of $M_P$ (or simply acts by blocks on $P$) if it belongs to $W(P)$. 

\subsection{Discrete automorphic forms for \texorpdfstring{$\GL_n$}{GLn}}
\label{subsec:residual}

\subsubsection{The classification of~\cite{MW89}}
\label{subsubsec:disc_gln}
Let $\pi \in \Pi_{\mathrm{disc}}(\GL_n)$. There exist integers $r, d \geq 1$ with $n=rd$ and $\sigma \in \Pi_{\mathrm{cusp}}(\GL_r)$ such that any $\varphi \in \cA_\pi(\GL_n)$ is obtained as the residue of an Eisenstein series built from a $\phi \in \cA_{P_{\pi},\sigma^{\boxtimes d}}(\GL_n)$ where $P_{\pi} \subset \GL_n$ is the standard parabolic subgroup of Levi factor $\GL_{r}^{d}$. More precisely, define
\begin{equation}
\label{eq:nu_pi_defi}
    \nu_{\pi}=-\rho_{P_\pi}/r, \quad \sigma_\pi=\sigma^{\boxtimes d} \in \Pi_{\cusp}(M_{P_\pi}),
\end{equation}
and set for $\lambda \in \fa_{P_\pi,\cc}^*$
\begin{equation}
\label{eq:L_pi_res_defi}
    L_{\pi,\mathrm{res}}(\lambda)=\prod_{\alpha \in \Delta_{P_\pi}} \left( \langle \lambda, \alpha^\vee \rangle -1 \right).
\end{equation}
Note that $ L_{\pi,\mathrm{res}}(-\nu_{\pi})=0$. We introduce a minus sign in \eqref{eq:nu_pi_defi} to follow the convention of \cite{Ch}. For every $g \in \GL_n(\bA)$, denote by $E^*(g,\phi,\cdot)$ the map $\lambda \mapsto  L_{\pi,\mathrm{res}}(\lambda) E(g,\phi,\lambda)$. It is holomorphic in a neighborhood of $-\nu_\pi$. By \cite{MW89} we have
\begin{equation}
\label{eq:residual}
    \varphi(g)=E^*(g,\phi,-\nu_{\pi}).
\end{equation}
As $\pi$ is the unique irreducible quotient of $\cA_{P_\pi,\sigma_\pi,-\nu_\pi}(\GL_n)$, it deserves to be called a Speh representation and we write $ \pi=\Speh(\sigma,d)$.

Let $\pi, \pi' \in \Pi_{\mathrm{disc}}(\GL_n)$. Define $r'$, $d'$, $\sigma'$, $P_{\pi'}$ as in \S~\ref{subsubsec:disc_gln} for $\pi'$. By \cite{MW95}, the completed Rankin--Selberg $L$ function $L(s,\pi \times \pi')$ exists and satisfies
\begin{equation}
\label{eq:discrete_L}
    L(s,\pi \times \pi')=\prod_{i=1}^d \prod_{j=1}^{d'} L\left(s+\frac{d-2i+1}{2}+\frac{d'-2j+1}{2},\sigma \times \sigma'\right).
\end{equation}

\subsubsection{The induced case} 
\label{subsubsec:residual_blocks}
We now deal with representations induced from the discrete spectrum of a Levi subgroup, i.e. with representations of Arthur type of $\GL_n$. Let $P=M_P N_P$ be a standard parabolic subgroup. Let $\pi \in \Pi_{\mathrm{disc}}(M_P)$. Write $M_P=\GL_{n_1} \times \hdots \times \GL_{n_m}$ and $\pi = \pi_1 \boxtimes \hdots \boxtimes \pi_m$ accordingly. By \S\ref{subsubsec:disc_gln}, there exist integers $r_i, d_i \geq 1$ with $n_i=r_i d_i$ and some representations $\sigma_i \in \Pi_{\mathrm{cusp}}(\GL_{r_i})$ such that any $\varphi_i \in \cA_{\pi_i}(\GL_{n_i})$ is obtained as the residue of an Eisenstein series built from a $\phi_i \in \cA_{P_{\pi_i},\sigma_i^{\boxtimes d_i}}(\GL_{n_i})$ where $P_{\pi_i} \subset \GL_{n_i}$ is the standard parabolic subgroup of Levi factor $\GL_{r_i}^{d_i}$. Set $P_{\pi}=(P_{\pi_1} \times \hdots \times P_{\pi_m})N_P$ and $\sigma_\pi=\sigma_1^{\boxtimes d_1} \boxtimes \hdots \boxtimes \sigma_m^{\boxtimes d_m}$ which is a cuspidal automorphic representation of $M_{P_\pi}$. 

Note that any $\varphi \in \cA_{P,\pi}(\GL_n)$ is residual unless $\pi$ is cuspidal (i.e. $P=P_\pi$). Indeed, this follows from the fact that Eisenstein series are orthogonal to cusp forms (in the sense of \S\ref{subsubsec:Eisenstein}) and that we may compute the residue under the Petersson inner-product by \cite[Theorem~2.2]{Lap} (see also \cite[Lemma~9.4.2.1]{BoiPhD} for a closely related argument).

Let $Q$ and $R$ be a parabolic of $\GL_n$ such that $P_{\pi} \subset Q \subset P$. Write $Q \cap M_P=Q_1 \times \hdots \times Q_m$. We have a decomposition $\fa_Q^{P,*}=\oplus_{i=1}^m \fa_{Q_i}^{\GL_{n_i},*}$. Set
\begin{equation}
\label{eq:nu_exp_defi}
    \nu_{Q,\pi}=\left(-\rho_{Q_i}^{\GL_{n_i}}/r_i \right)_{1 \leq i \leq r}
\end{equation}
written accordingly. If the context is clear, we omit the subscript $\pi$. If $Q=P_\pi$, we will write $\nu_\pi$. 

By construction, if $P_\pi \subset Q \subset P$, for any $\varphi \in \cA_{P,\pi}(\GL_n)$ we have $\varphi_{Q,-\nu_Q} \in \cA_Q^0(\GL_n)$ (see Lemma~\ref{lem:contant_term}). Moreover, note that $\varphi_Q=0$ unless $P_\pi \subset Q$. If $w \in {}_Q W_P$ such that $P_\pi \subset P_w$, then we set $\varphi_w:=\varphi_{P_w}$ and $\nu_w:=\nu_{P_w}$.

Set
\begin{equation}
\label{eq:residual_L_pi}
     L_{\pi,\mathrm{res}}(\lambda)=\prod_{\alpha \in \Delta_{P_\pi}^P} \left( \langle \lambda, \alpha^\vee \rangle -1 \right).
\end{equation}
For every $g \in \GL_n(\bA)$ and $\phi \in \cA_{P_\pi,\sigma_\pi}(\GL_n)$ denote by $E^{P,*}(g,\phi,\cdot)$ the partial residues of Eisenstein series $\lambda \mapsto  L_{\pi,\mathrm{res}}(\lambda)E^P(g,\phi,\lambda)$. It is holomorphic in a neighborhood of $-\nu_\pi$. By exactness of induction, for every $\varphi \in \cA_{P,\pi}(\GL_n)$ there exists $\phi \in \cA_{P_\pi,\sigma_\pi}(\GL_n)$ such that $\varphi=E^{P,*}(\phi,-\nu_\pi)$. We write $\pi=\boxtimes_{i=1}^m \Speh(\sigma_i,d_i)$.

Finally, let $w^*_\pi$ be the longest element in $W^P(P_\pi,P_\pi)$, i.e. the unique element $w$ in this set such that $w(P_\pi \cap M_P)w^{-1}$ is opposed to $P_\pi \cap M_P$. Note that it acts by identity on $\fa_{P}^*$ and that $w_\pi^* \sigma_\pi=\sigma_\pi$ and $w_\pi^* \nu_{\pi}=-\nu_{\pi}$.

\subsection{Normalization of intertwining operators}
\label{sec:norm_operator}

In this section, let $P$ be a standard parabolic subgroup of $G$ and let $\pi \in \Pi_{\disc}(M_P)$. Let $Q$ be a standard parabolic subgroup of $G$. Let $w \in W(P,Q)$ and take a $\lambda \in \fa_P^*$ in general position. Denote by $M_\pi(w,\lambda)$ the restriction of $M(w,\lambda)$ (defined in \eqref{eq:global_operator}) to the subspace $\cA_{P,\pi}(G) \subset \cA_P(G)$. Following \cite{Art82}, we normalize $M_\pi(w,\lambda)$ as
\begin{equation}
    \label{eq:normalization_non_twisted} M_\pi(w,\lambda)=n_\pi(w,\lambda)N_\pi(w,\lambda).
\end{equation}
Here $n_\pi(w,\lambda)$ is a meromorphic function in $\lambda$ referred to as "the scalar factor", and $N_\pi(w,\lambda)$ is the so-called "normalized operator". If the context is clear, we will remove the subscript $\pi$. We describe these objects below. The goal of this section is to recall the main properties of $M_\pi(w,\lambda)$. We then explain how this operator appears in the constant term of discrete automorphic forms.

\subsubsection{The scalar factor} Write $M_P=\GL_{n_1} \times \hdots \times \GL_{n_m}$ and $\pi=\pi_1 \boxtimes \hdots \boxtimes \pi_m$. Let $\beta \in \Sigma_{P}$ be the positive root of $P$ associated to the two blocks $\GL_{n_i}$ and $\GL_{n_j}$, with $1 \leq i < j \leq m$. Set
\begin{equation}
    \label{eq:n_pi_formula}
    n_\pi(\beta,s)=\frac{L(1-s,\pi_i^\vee \times \pi_j)}{L(1+s,\pi_i \times \pi_j^\vee)}, \quad \text{and} \quad  n_\pi(w,\lambda)=\prod_{\substack{\beta \in \Sigma_{P} \\ w \beta <0}} n_\pi(\beta,\langle \lambda, \beta^\vee \rangle).
\end{equation}
By \eqref{eq:discrete_L}, the scalar factor can be expressed in terms of cuspidal Rankin--Selberg $L$-functions. We write $L_\infty$ (resp. $L^\infty$) for the Archimedean parts of the $L$-functions (resp. the finite part).

\begin{theorem}
    \label{thm:L}
    Let $\sigma_1$ and $\sigma_2$ be cuspidal automorphic representations of $\GL_{r_1}$ and $\GL_{r_2}$ respectively. Set $\delta_{\sigma_1,\sigma_2}=1$ if $\sigma_1 \simeq \sigma_2$ and $0$ otherwise.
    \begin{enumerate}
        \item Poles: the function
        \begin{equation*}
            (s(1-s))^{\delta_{\sigma_1,\sigma_2}}L(s,\sigma_1 \times \sigma_2^\vee)
        \end{equation*}
        is entire of order one.
        \item Functional equation: we have
        \begin{equation*}
            L(s,\sigma_1 \times \sigma_2^\vee)=\epsilon(s,\sigma_1 \times \sigma_2^\vee) L(1-s,\sigma_1^\vee \times \sigma_2),
        \end{equation*}
        with $\epsilon(s,\sigma_1 \times \sigma_2^\vee)=\epsilon_0 q_{\sigma_1 \times \sigma_2^\vee}^{\frac{1}{2}-s}$ where $\epsilon_0$ is a complex number of modulus $1$ and $q_{\sigma_1 \times \sigma_2^\vee} \in \nn$ is the arithmetic conductor of $\sigma_1 \times \sigma_2^\vee$.
         \item Archimedean factors: we have
         \begin{equation*}
             L_\infty(s,\sigma_1 \times \sigma_2^\vee)=\prod_{j=1}^m \Gamma_{\rr}(s-\alpha_j),
         \end{equation*}
         where $m=r_1 r_2 [F:\qq]$, $\Gamma_{\rr}(s)=\pi^{-s/2} \Gamma(s)$ (the usual $\Gamma$ function) and $\alpha_1, \hdots, \alpha_m$ are complex numbers such that for every $j$
         \begin{equation}
         \label{eq:archi_ramanujan}
             1-\Re(\alpha_j)>\frac{1}{r_1^2+1}+\frac{1}{r_2^2+1}.
         \end{equation}
         \item Zero free region: There exists $k$ such that for every level $J$ there exist $c_J>0$ such that for every $J$-pair $(P,\pi)$ and every $i,j$, the meromorphic $s \mapsto L^\infty(s,\sigma_i \times \sigma_j^\vee)$ doesn't have any zero in the region $\Re(s) \geq 1-c_J\left(1+\Lambda_\pi^{M_P}+\Val{s}\right)^{-k}$, where $\sigma_\pi=\sigma_1^{\boxtimes d_1} \boxtimes \hdots \boxtimes \sigma_m^{\boxtimes d_m}$ is the cuspidal representation of $M_{P_\pi}$ defined in \S\ref{subsubsec:residual_blocks}.
        \end{enumerate}
\end{theorem}

\begin{proof}
    1. and 2. are proved in \cite{JPSS83}. The bound \eqref{eq:archi_ramanujan} in 3. is \cite[Proposition~3.3]{MS04} which is based on \cite{LRS}. 4. is \cite[Proposition~3.5]{Lap2} which quotes \cite{Bru06}. More precisely, \cite[Proposition~3.5]{Lap2} writes the zero-free region in terms of the analytic conductor of $\sigma_\pi$. But as noted in \cite[Section~3.4.7]{Ch}, the latter is bounded in terms of $\Lambda_\pi^{M_P}$ as the level $J$ is fixed. Therefore, we conclude that this zero-free region holds.
\end{proof}

\subsubsection{Local normalized intertwining operators} 
\label{subsubsec:local_normalized_operator}

Let $\phi \in \cA_{P,\pi}(G)$. Assume that $\phi=\otimes'_v \phi_v$ is factorizable, so that for all place $v$ we have $\phi_v \in I_P^G \pi_v$. Let $S \subset V_F$ be a finite set of places such that $\phi_v$ is unramified if $v \notin \tS$. By \cite[Theorem~2.1]{Art89}, we have a factorization
\begin{equation}
\label{eq:unram_facto_local_intertwin}
    N_\pi(w,\lambda)\phi=\prod_{v \in \tS} N_{\pi_v}(w,\lambda)\phi_v
\end{equation}
where the $N_{\pi_v}(w,\lambda)$ are meromorphic local intertwining operators $I_{P}^{G} \pi_{v,\lambda} \to I_{Q}^{G} w\pi_{v,\lambda}$. The product notation of \eqref{eq:unram_facto_local_intertwin} means that $N_\pi(w,\lambda)\phi$ is factorizable and that for $v \notin S$ the local component $N_{\pi_v}(w,\lambda) \phi_v$ is the unique unramified vector in $I_Q^G w \pi_\lambda$ such that $\phi_v(1)=N_{\pi_v}(w,\lambda) \phi_v(1)$. 

Let $v \in V_F$. The normalized operator $N_{\pi_v}$ is defined in term of a local $L$-function, so that it makes sense for any smooth irreducible unitary representation $\pi_v$ of $M_P(F_v)$ (which is not necessarily a local constituent of some $\pi \in \Pi_{\mathrm{disc}}(M_P)$). We recall the classification of such representations of $\GL_N(F_v)$ for $N \geq 1$. For any discrete series $\delta$ of some $\GL_r(F_v)$ and any $d \geq 1$, let $\Speh(\delta,d)$ be the unique irreducible quotient of the parabolic induction $\delta_{\frac{d-1}{2}} \times \hdots \times \delta_{\frac{1-d}{2}}$. By \cite[Theorem~A(ii)]{Ta} in the non-Archimedean case, and \cite{Vo} in the Archimedean case, for any smooth irreducible unitary representation $\tau$ of $\GL_N(F_v)$ there exist discrete series $\delta_{1} \hdots \delta_{k}$ of some $\GL_{N_i}(F_v)$, integers $d_{1}, \hdots, d_{k}$ and real numbers $-1/2 < \nu_{1}, \hdots ,\nu_{k} < 1/2$ such that 
\begin{equation*}
    \tau \simeq \Speh(\delta_{1},d_{1})_{\nu_{1}} \times \hdots \times \Speh(\delta_k,d_k)_{\nu_k}.
\end{equation*}
Set
\begin{equation*}
    e(\tau)=2 \inf \left\{ \frac{1}{2}-\Val{\nu_i} \; \big| \; 1 \leq i \leq k \right\}.
\end{equation*}
If now $\pi_v=\pi_{v,1} \boxtimes \hdots \boxtimes \pi_{v,m}$ is a smooth irreducible unitary representation of $M_P(F_v)$, set $e(\pi_v)=\min e(\pi_{v,i})$.

\begin{theorem}
    \label{thm:N}
    Let $\pi_v$ be a smooth irreducible and unitary representation of $M_P(F_v)$.
    \begin{enumerate}
        \item For each $\phi_v \in I_{P}^{G} \pi_{v}$ the vector $N_{\pi_v}(w,\lambda)\phi_v$ is a rational function in $\lambda$ if $v$ is Archimedean, and in $q_v^{-\lambda}$ if $v$ is non-Archimedean. More precisely, it is a rational function in the variables $\langle \lambda, \alpha \rangle$ (resp. $q_v^{-\langle \lambda, \alpha \rangle}$) for the $\alpha \in \Delta_P$ such that $w \alpha <0$ in the Archimedean case (resp. the non-Archimedean case).
        \item If $w_1 \in W(P,Q)$ and $w_2 \in W(Q,R)$ we have
        \begin{equation*}
            N_{w_1 \pi_v}(w_2,w_1\lambda)N_{\pi_v}(w_1,\lambda)=N_{\pi_v}(w_2w_1,\lambda).
        \end{equation*}
         \item The operator $N_{\pi_v}(w,\lambda)$ is holomorphic in the region
         \begin{equation*}
           \bigcap_{\substack{\alpha \in \Sigma_{P} \\ w \alpha <0}} \left\{ \lambda \in \fa_{P,\cc}^* \; | \; \langle \Re(\lambda),\alpha^\vee \rangle > -e(\pi_v) \right\}.
        \end{equation*}
        \item If $\pi_v$ is the local component of some $\pi \in \Pi_{\mathrm{disc}}(M_P)$, then
        \begin{equation*}
            e(\pi_v)>\frac{2}{1+n^2}.
        \end{equation*}
    \end{enumerate}
\end{theorem}

\begin{proof}
    1., 2. are contained in \cite[Theorem~2.1]{Art89}. By decomposing $N_{\pi_v}(w,\lambda)$ as a product of rank one intertwining operators, 3. is \cite[Proposition~I.10]{MW89}. 4. is proved for local components of cuspidal representations in \cite[Proposition~3.3]{MS04}, and for residual representations in \cite[Proposition~3.5]{MS04}.
\end{proof}

\subsubsection{Constant terms of discrete automorphic forms}

\label{subsubsection:constant_terms_discrete}

We now fix a standard parabolic subgroup $P$ of $G$ and $\pi \in \Pi_\disc(M_P)$. We have $P_\pi$ and $\sigma_\pi \in \Pi_{\cusp}(M_{P_\pi})$ as in \S\ref{subsubsec:residual_blocks}. For any $\phi \in \cA_{P_\pi,\sigma_\pi}(G)$ consider the regularized operator 
\begin{equation}
\label{eq:reg_operator}
    M^*(w_\pi^*,\lambda) \phi:= L_{\pi,\mathrm{res}}(\lambda)M(w_\pi^*,\lambda) \phi.
\end{equation}
It follows from the factorization of the global intertwining operators \eqref{eq:normalization_non_twisted}, from the localization of the poles of the global $L$-factors in Theorem~\ref{thm:L}, and from the regularity of the local intertwining operators in Theorem~\ref{thm:N} that there exists a constant $c_\pi$ such that for any $\lambda \in \fa_{P,\cc}^*$ we have
\begin{equation}
\label{eq:regular_is_local}
    M^*(w_\pi^*,-\nu_\pi+\lambda)\phi=c_\pi N_{\sigma_\pi}(w_\pi^*,-\nu_\pi+\lambda)\phi.
\end{equation}
By Theorem~\ref{thm:N}, we see that $M^*(w_\pi^*,-\nu_\pi+\lambda)\phi$ is regular on $\fa_{P,\cc}^*$. Morever, by \cite[Corollary~3.3]{BoiZ} $M^*(w_\pi,-\nu_\pi)$ also realizes the quotient map $\cA_{P_\pi,\sigma_\pi,-\nu_\pi}(G) \to \cA_{P,\pi}(G)$.

We now use the notation of \S\ref{subsubsec:residual_blocks}. Let $P_\pi \subset Q \subset P$. For each $1 \leq i \leq m$ there exist integers $d_{i,1}, \hdots, d_{i,m_i}$ such that $\sum_{j=1}^{m_i} d_{i,j}=d_i$ and $M_Q=\prod_{i=1}^m \prod_{j=1}^{m_i} \GL_{r_i d_{i,j}}$. For each pair $(i,j)$, set $n_{i,j}=r_i d_{i,j}$ and $P_{i,j}=P_{\pi_i} \cap \GL_{n_{i,j}}$. Let $\pi_{i,j}$ be the discrete representation of $\GL_{n_{i,j}}$ spanned by the residues of Eisenstein series built from $\phi_{i,j} \in \cA_{P_{i,j},\sigma_i^{\boxtimes d_{i,j}}}(\GL_{n_{i,j}})$. Set $\pi_Q=\boxtimes_i \boxtimes_j \pi_{i,j}$ which is a discrete representation of $M_Q$. The following is \cite[Lemma~3.2]{BoiZ}.

\begin{lem}
\label{lem:contant_term}
    For every $\varphi \in \cA_{P,\pi}(\GL_n)$ we have $\varphi_{Q,-\nu_Q} \in \cA_{Q,\pi_Q}(\GL_n)$. In particular, $\varphi_{Q,-\nu_Q}$ is residual unless $Q=P_\pi$. Moreover, in that case, assume that $\varphi=E^{P,*}(\phi,-\nu_\pi)$ for $\phi \in \cA_{P_\pi,\sigma_\pi}(\GL_n)$. Then we have
    \begin{equation}
    \label{eq:constant_P_pi}
        \varphi_{P_\pi}=M^*(w_\pi^*,-\nu_{\pi})\phi.
    \end{equation}
\end{lem}

We note that our choices of measures yield an adjunction between residues of Eisenstein series and constant terms.

\begin{prop}
\label{prop:adjunt}
    For every $\phi \in \cA_{P_\pi,\sigma_\pi,-\nu_\pi}(G)$ and $\varphi \in \cA_{P,\pi}(G)$ we have
    \begin{equation*}
        \langle E^{P,*}(\phi,-\nu_\pi),\varphi \rangle_{P,\Pet}=\langle \phi,\varphi_{P_\pi} \rangle_{P_\pi,\Pet}.
    \end{equation*}
\end{prop}

\begin{proof}
    By the Iwasawa decomposition, it is enough to prove that the proposition holds if $P=G$. Write $M_{P_\pi}=\GL_r^{d}$ and $\sigma_\pi=\sigma^{\boxtimes d}$ where $\sigma \in \Pi_{\cusp}(\GL_r)$. Choose $\phi' \in \cA_{P_\pi,\sigma_\pi,-\nu_\pi}$ such that $E^{*}(\phi')=\varphi$. Let $\Lambda^{T,\mathrm{Art}}$ be Arthur's truncation operator from \cite{Art80}. Here $T \in \fa_{0}^{G}$ is a truncation parameter, and we say that it is sufficiently positive if we have $\langle \alpha ,T \rangle >M$ for all $\alpha \in \Delta_0$. This also naturally defines a notion of $\lim \limits_{T \to \infty}$. By the Maa\ss--Selberg relation of \cite[Theorem~3.1.3.1]{Ch}, for $\lambda, \lambda' \in \fa_{P_\pi,\cc}^{G,*}$ in general position and $T \in \fa_{0}^G$ sufficiently positive we have
    \begin{align}
    \label{eq:truncated_product}
        &\langle \Lambda^{T,\mathrm{Art}} E(\phi,\lambda),E(\phi',\overline{\lambda'}) \rangle_{G,\Pet} \nonumber \\
        = &\sum_{w,w' \in W(P_\pi,P_\pi)} \langle M(w,\lambda) \phi,M(w',\overline{\lambda'}) \phi' \rangle_{P_\pi,\Pet} \frac{\exp(\langle w\lambda + w' \lambda',T_{P_\pi} \rangle )}{\theta_{P_\pi}^{\mathrm{Art}}( w\lambda + w' \lambda' )},
    \end{align}
    where for any standard parabolic subgroup $Q$ of $\GL_n$ we write $\vol(\fa_Q^G/ \zz(\Delta_Q^\vee))$ be the covolume of the lattice generated by $\Delta_Q^\vee$ in $\fa_Q^G$ and we set
    \begin{equation*}
        \theta_{Q}^{\mathrm{Art}}(\lambda)=\vol(\fa_{Q}^G/ \zz(\Delta_{Q}^\vee))^{-1} \prod_{\alpha \in \Delta_{Q}} \langle \lambda, \alpha^\vee \rangle.
    \end{equation*}

    We compute residues of this expression in two ways. Set
    \begin{equation*}
        f_{\phi,\phi'}^T(\lambda,\lambda'):= L_{\pi,\mathrm{res}}(\lambda)  L_{\pi,\mathrm{res}}(\lambda')\langle \Lambda^{T,\mathrm{Art}} E(\phi,\lambda),E(\phi',\overline{\lambda'}) \rangle_{G,\Pet}.
    \end{equation*}
    First, by the continuity of the truncation operator from \cite[Theorem~3.9]{Zydor} and of Eisenstein series from \cite[Theorem~2.2]{Lap}, we have
    \begin{equation*}
        f_{\phi,\phi'}^T(-\nu_\pi,-\nu_\pi)=\langle \Lambda^{T,\mathrm{Art}} E^*(\phi,-\nu_\pi),E^*(\phi',-\nu_\pi) \rangle_{G,\Pet}=\langle \Lambda^{T,\mathrm{Art}} E^*(\phi,-\nu_\pi),\varphi \rangle_{G,\Pet}.
    \end{equation*}
    By another application of the Maa\ss--Selberg relation, this is
    \begin{equation*}
        f_{\phi,\phi'}^T(-\nu_\pi,-\nu_\pi)=\sum_{Q \supset P_\pi} \langle E^*(\phi,-\nu_\pi)_Q,\varphi_Q \rangle_{Q,\mathrm{Art}} \frac{\exp(\langle 2\nu_Q,T_Q \rangle )}{\theta_{Q}^{\mathrm{Art}}( 2\nu_Q )},
    \end{equation*}
    which is well defined as for all $P_\pi \subset Q$ and $\alpha \in \Delta_Q$ we have $\langle 2 \nu_Q, \alpha^\vee \rangle <0$. In particular, by taking the limit as $T \to \infty$ we obtain
    \begin{equation}
    \label{eq:Pet_limit_1}
        \lim_{T \to \infty } f_{\phi,\phi'}^T(-\nu_\pi,-\nu_\pi)=\langle E^*(\phi,-\nu_\pi),\varphi \rangle_{G,\Pet}.
    \end{equation}

    We now start from the RHS of \eqref{eq:truncated_product}. For regular $\lambda$ we have
    \begin{equation*}
        f_{\phi,\phi'}^T(\lambda,-\nu_\pi)= L_{\pi,\mathrm{res}}(\lambda)\sum_{w \in W(P_\pi,P_\pi)} \langle M(w,\lambda) \phi,M^*(w_\pi^*,-\nu_\pi) \phi' \rangle_{P_\pi,\Pet} \frac{\exp(\langle w\lambda +\nu_\pi,T_{P_\pi} \rangle )}{\theta_{P_\pi}^{\mathrm{Art}}( w\lambda + \nu_\pi )}.
    \end{equation*}
    Let $w \in W(P_\pi,P_\pi)$. In a neighborhood of $-\nu_\pi$, the poles of $\theta_{P_\pi}^{\mathrm{Art}}( w\lambda + \nu_\pi )^{-1}$ are simple and along the affine hyperplanes $\langle \lambda,w^{-1} \alpha^\vee \rangle =\langle -\nu_\pi,\alpha^\vee \rangle=1$ for $\alpha \in \Delta_{P_\pi}$. In particular, they are distinct from the poles of $M(w,\lambda)$ at $-\nu_\pi$ which are of the form $\langle \lambda, \beta^\vee \rangle =1$ with $\beta^\vee \in \Delta_{P_\pi}$ and $w \beta <0$. We may therefore compute the residues under the sum to obtain
    \begin{equation*}
        f_{\phi,\phi'}^T(-\nu_\pi,-\nu_\pi)=\sum_{w \in W(P_\pi,P_\pi)} \langle M^*(w,-\nu_\pi) \phi,M^*(w_\pi^*,-\nu_\pi) \phi' \rangle_{P_\pi,\Pet} \frac{\exp(\langle -w\nu_\pi +\nu_\pi,T_{P_\pi} \rangle )}{\theta_{P_\pi}^{\mathrm{Art},*}( -w \nu_\pi + \nu_\pi )},
    \end{equation*}
    where $M^*(w,-\nu_\pi)$ and $\theta_{P_\pi}^{\mathrm{Art},*}( -w \nu_\pi + \nu_\pi )$ are the appropriate regularizations. But we see that $\lim_{T \to \infty} \langle -w\nu_\pi +\nu_\pi,T_{P_\pi} \rangle = -\infty$ unless $w=1$, so that by Lemma~\ref{lem:contant_term}
    \begin{equation}
    \label{eq:Pet_limit_2}
        \lim_{T \to \infty } f_{\phi,\phi'}^T(-\nu_\pi,-\nu_\pi)=\frac{\langle \phi,\varphi_{P_\pi} \rangle_{P_\pi,\Pet} }{\theta_{P_\pi}^{\mathrm{Art},*}(0)}=\langle \phi,\varphi_{P_\pi} \rangle_{P_\pi,\Pet} \times \vol(\fa_{P_\pi}^G/ \zz(\Delta_{P_\pi}^\vee)).
    \end{equation}
    By going back to our choices of measures in \S\ref{subsubsec:first_measure} and of coordinates, we see that $\vol(\fa_{P_\pi}^G/ \zz(\Delta_{P_\pi}^\vee))$ is equal to $1$. The proposition now follows from \eqref{eq:Pet_limit_1} and \eqref{eq:Pet_limit_2}.
\end{proof}

\subsection{Regularity of intertwining operators}
\label{subsec:reg_intertwining}
Let $P$ be a standard parabolic subgroup of $G$, let $\pi \in \Pi_\disc(M_P)$. Let $Q$ be another standard parabolic subgroup of $G$ and take $w \in {}_Q W_P$ such that $P_\pi \subset P_w$. Write $\pi_w$ for the representation $\pi_{P_w} \in \Pi_{\disc}(M_{P_w})$ defined in Lemma~\ref{lem:contant_term}. Let $\varphi \in \cA_{P,\pi}(G)$. As before, set $\nu_w=\nu_{P_w}$ and $\varphi_w=\varphi_{P_w}$. By Lemma~\ref{lem:contant_term}, we have the normalization

\begin{equation}
\label{eq:normalization_twisted}
       M(w,\lambda)\varphi_{w}=n_{\pi_w}(w,\lambda+\nu_{w})N_{\pi_w}(w,\lambda+\nu_{w})\varphi_{w,-\nu_{w}}
\end{equation}
The goal of this section is to determine the poles of $M(w,\lambda)$ in a neighborhood of $\fa_{P}^{*,+}$ of the form $\cR_{\pi,k,c_J}$ (see \eqref{eq:R_defi}).

\subsubsection{Regularity of scalar factors}
\label{subsubsec:reg_scalar}

\begin{lem}
\label{lem:n_regular}
    There exist $k>0$ such that for every level $J$ there exists $c_J>0$ such that for every $J$-pair $(P,\pi)$ there exists a product of affine root linear forms $L_{\pi,w} \in \cc[\fa_{P,\cc}^*]$ such that the product
    \begin{equation*}
        \lambda \mapsto L_{\pi,w}(\lambda)n_{\pi_w}(w,\lambda+\nu_{P_w}).
    \end{equation*}
    is regular in the region $\cR_{\pi,k,c_J}$.
\end{lem}

\begin{rem}
\label{rem:f_pi_w}
    As the proof shows, the set $\{ L_{\pi,w} \}$, where $w$ and $\pi$ range as in \S\ref{subsec:reg_intertwining}, is finite. Moreover, if $\pi$ is cuspidal we may take
    \begin{equation*}
        L_{\pi,w}(\lambda)=\prod_{\substack{i<j \\ w(i)>w(j) \\ \pi_i \simeq \pi_j}} (\lambda_i-\lambda_j-1).
    \end{equation*}
    We do not give a precise description of $L_{\pi,w}$ for general $\pi$, as it will be easily computable in the cases we are interested in.
\end{rem}

\begin{proof}
    Using the description of $L$-factors for discrete automorphic representations given in \eqref{eq:discrete_L}, we see that we have the equality 
    \begin{equation}
    \label{eq:global_equality}
        n_{\pi_w}(w,\lambda+\nu_{w})=n_{\sigma_\pi}(w,\lambda+\nu_{\pi}).
    \end{equation}
    Therefore we may assume that $P_w=P_\pi$. We use the notation of \S\ref{subsubsec:residual_blocks}. For each $i$, set $D_i=\sum_{j=1}^{i} d_j$. Then $w$ acts by permuting the blocks of $M_{P_\pi}$ so that we may identify it with an element of $\fS_{D_m}$. Let $1 \leq i<j \leq m$. For each $1 \leq a \leq d_i$, let $b_a$ be the greatest integer such that $1\leq b_a \leq d_j$ and $w(b_a+D_{j-1})<w(a+D_{i-1})$ (if no such integer exists, set $b_a=0$). Set
    \begin{align}
    \label{eq:n_i,j}
            n_{i,j}(\lambda)&=\prod_{a=1}^{d_i} \prod_{b=1}^{b_a} \frac{L\left(\lambda_i-\lambda_j+a-b+\frac{d_j-d_i}{2},\sigma_i \times \sigma_j^\vee \right)}{L\left(\lambda_i-\lambda_j+a-(b-1)+\frac{d_j-d_i}{2},\sigma_i \times \sigma_j^\vee \right)} \nonumber \\
            &=\prod_{a=1}^{d_i}\frac{L\left(\lambda_i-\lambda_j+a-b_a+\frac{d_j-d_i}{2},\sigma_i \times \sigma_j^\vee \right)}{L\left(\lambda_i-\lambda_j+a+\frac{d_j-d_i}{2},\sigma_i \times \sigma_j^\vee \right)}.
    \end{align}
    There is another description of $n_{i,j}$: if $1 \leq b \leq d_j$, let $a_b$ be the smallest integer such that $1 \leq a_b \leq d_i$ and $w(b+D_{j-1})<w(a_b+D_{i-1})$ (if no such integer exists, set $a_b=d_i+1$). Then we have 
    \begin{align}
        \label{eq:n_i,j_2}
            n_{i,j}(\lambda)&=\prod_{b=1}^{d_j} \prod_{a=a_b}^{d_i} \frac{L\left(\lambda_i-\lambda_j+a-b+\frac{d_j-d_i}{2},\sigma_i \times \sigma_j^\vee \right)}{L\left(\lambda_i-\lambda_j+(a+1)-b+\frac{d_j-d_i}{2},\sigma_i \times \sigma_j^\vee \right)} \nonumber \\
            &=\prod_{b=1}^{d_j}\frac{L\left(\lambda_i-\lambda_j+a_b-b+\frac{d_j-d_i}{2},\sigma_i \times \sigma_j^\vee \right)}{L\left(\lambda_i-\lambda_j+1-b+\frac{d_j+d_i}{2},\sigma_i \times \sigma_j^\vee \right)}.
    \end{align}
    Then by \eqref{eq:n_pi_formula} and the functional equation of $L$-functions (Theorem~\ref{thm:L}), we have
    \begin{equation}
        \label{eq:n_sigma_formula}
        n_\sigma(w,\lambda+\nu_{\pi})=\epsilon(\lambda) \prod_{i<j} n_{i,j}(\lambda),
    \end{equation}
    where $\epsilon$ is some product of $\epsilon$-factors which are entire.

     Take $k$ and $c_J$ given by Theorem~\ref{thm:L}. If $d_j \geq d_i$, for every $a \geq 1$ we see that the function $L\left(\lambda_i-\lambda_j+a+\frac{d_j-d_i}{2},\sigma_i \times \sigma_j^\vee \right)$ doesn't have any zero in the region $\cR_{\pi,k,c_J}$ by Theorem~\ref{thm:L}. If $d_j \leq d_i$, we see that for every $b \leq d_j$ it is $L\left(\lambda_i-\lambda_j+1-b+\frac{d_j+d_i}{2},\sigma_i \times \sigma_j^\vee \right)$ which doesn't have any zero in $\cR_{\pi,k,c_J}$. By \eqref{eq:n_sigma_formula}, the only poles of $n_{\sigma_\pi}(w,\lambda+\nu_{\pi})$ come from poles of $L(\cdot,\sigma_i \times \sigma_j^\vee)$. By Theorem~\ref{thm:L}, they are all simple and located along a finite collection of affine root hyperplanes (even as $\pi$ varies), so that there exists $L_{\pi,w} \in \cc[\fa_{P,\cc}^*]$ such that $\lambda \mapsto L_{\pi,w}(\lambda)n_{\sigma_\pi}(w,\lambda+\nu_{\pi})$ is regular on $\cR_{\pi,k,c_J}$ as claimed.
\end{proof}

We also note that the polynomial $L_{\pi,w}$ can be ignored in a neighborhood of $i \fa_P^*$ if $w \in W(P)$.

\begin{lem}
    \label{lem:blocks_n_regular}
    Assume that $w \in W(P,Q)$. There exist $k>0$ such that for every level $J$ there exists $c_J>0$ such that for every $J$-pair $(P,\pi)$ the map
    \begin{equation*}
        \lambda \mapsto n_{\pi}(w,\lambda).
    \end{equation*}
    is regular in the region $\cS_{\pi,k,c_J}$ (see \eqref{eq:S_defi}).
\end{lem}

\begin{proof}
    By Lemma~\ref{lem:n_regular}, we only have to show that $n_{\sigma_\pi}(w,\lambda+\nu_\pi)$ doesn't have any poles in the region $i \fa_P^*$. We keep the notation from the previous proof. Assume that $w$ switches the blocks $i$ and $j$ of $M_P$ with $i <j$, and that $d_j \geq d_i$. By \eqref{eq:n_i,j_2}, the associated factor is
    \begin{equation*}
        n_{i,j}(\lambda)=\prod_{a=1}^{d_i} \frac{L\left(\lambda_i-\lambda_j +\frac{d_i-d_j}{2}-a+1,\sigma_i \times \sigma_j^\vee \right)}{L\left(\lambda_i-\lambda_j +\frac{d_i+d_j}{2}-a+1 ,\sigma_i \times \sigma_j^\vee\right)}.    \end{equation*}
    By Theorem~\ref{thm:L}, the only possible pole of $n_{i,j}$ on $i \fa_P^*$ is located along the hyperplane $\lambda_i=\lambda_j$ and only occurs if the terms associated to $\frac{d_i-d_j}{2}-a+1 \in \{0,1\}$ appear in the product. For the $0$ case, we obtain $a=1+\frac{d_i-d_j}{2}$ which is only possible if $d_i=d_j$. This term is compensated in the denominator by the term corresponding to $\frac{d_i+d_j}{2}-a+1=1$, i.e. $a=d_i$. For the $1$ case, we get $a=\frac{d_i-d_j}{2}$ which is excluded. If $d_i \geq d_j$, we use the same procedure with \eqref{eq:n_i,j} instead, which concludes the proof.
\end{proof}

\subsubsection{Regularity of local normalized operators}
We keep the data defined at the top of \S\ref{subsec:reg_intertwining}. 

\begin{lem}
\label{lem:holo_compo}
    For every place $v$ the composition
    \begin{equation}
    \label{eq:composition_operator}
        N_{\sigma_{\pi,v}}(w,\lambda+\nu_{\pi}) N_{\sigma_{\pi,v}}(w^*_\pi,\lambda-\nu_{\pi})
    \end{equation}
    is holomorphic in the region
     \begin{equation}
     \label{eq:conv_region}
            \left\{ \lambda \in \fa_{P,\cc}^* \; | \; \forall \alpha \in \Sigma_{P}, \; \langle \Re(\lambda),\alpha^\vee \rangle > -e(\sigma_{\pi,v}) \right\}.
        \end{equation}

    More precisely, let $\lambda$ be an element in the region $\eqref{eq:conv_region}$. Let $w' \in W(P_\pi)$ be an element such that, if we set $P_\pi'=(w')^{-1} .P_\pi$, for all $\alpha \in \Delta_{P_\pi'}$ we have $\langle (w')^{-1}(\lambda-\nu_\pi),\alpha^\vee \rangle >-e(\sigma_{\pi,v})$. Set $\sigma_\pi'=(w')^{-1} \sigma_\pi'$, $\lambda'=(w')^{-1} \lambda$ and $\nu_\pi'=(w')^{-1} \nu_\pi$. Then we have the commutative diagram 
    \begin{equation*}
\begin{tikzcd}
{I_{P_\pi'}^G \sigma_{\pi,\lambda'-\nu_\pi',v}'} \arrow[rr, "{N(w',\lambda'-\nu_\pi')}"] \arrow[rd, "{N(w_\pi^*w',\lambda'-\nu_\pi')}" description, two heads] \arrow[rdd, "{N(ww_\pi^*w',\lambda'-\nu_\pi')}" description, bend right=49] &                                                                  & {I_{P_\pi}^G \sigma_{\pi,\lambda-\nu_\pi,v}} \arrow[ld, "{N(w_\pi^*,\lambda-\nu_\pi)}" description] \arrow[ldd, "{N(ww_\pi^*,\lambda-\nu_\pi)}" description, bend left=49] \\
                                                                                                                                                                                                                                           & {I_{P}^G\pi_{v,\lambda}} \arrow[d, "{N(w,\lambda)}" description] &                                                                                                                                                                            \\
                                                                                                                                                                                                                                           & {I_{w.P}^Gw\pi_{v,\lambda}}                                          &                                                                                                                                                                           
\end{tikzcd}
    \end{equation*}
    where all the arrows are holomorphic in a neighborhood of $\lambda$, $N_{\sigma_\pi',v}(w_\pi^*w',\lambda'-\nu_\pi')$ is surjective and we identify $I_{P}^G \pi_{v}$ and $I_{wP}^G(w.\pi_v)$ as subrepresentations of $I_{P_\pi}^G \sigma_{\pi,\nu_\pi,v}$ and $I_{w.P_\pi}^G w\sigma_{\pi,\nu_\pi,v}$ respectively.
\end{lem}

\begin{proof}
    This is \cite[Proposition~I.11]{MW89}. As the terminology there differs from ours, some comments are in order. Our element $w_\pi^*$ is denoted $w$, and our $w$ is $\sigma$. The condition $w \in {}_Q W_P$ implies condition (b) of \cite[Proposition~I.11]{MW89}. The parameter $\lambda$ is $\underline{s}$. That $\lambda$ belongs to the region \eqref{eq:conv_region} is condition (c). The representation $\sigma_v$ is denoted $\pi$ and is assumed to be generic (which is the case here as $\sigma_v$ is the local component of a cuspidal representation). Finally, the element $w'$ from \cite[Proposition~I.11]{MW89} is the one that we use.
\end{proof}

\begin{cor}
    \label{cor:holo_compo_must}
    We keep the notation from Lemma~\ref{lem:holo_compo}. For every place $v$, the composition $ N_{\sigma_{\pi,v}}(w,\lambda+\nu_{\pi}) N_{\sigma_{\pi,v}}(w^*_\pi,\lambda-\nu_{\pi})$ is holomorphic in the region 
    \begin{equation*}
        \lambda+\nu_w \in \bigcap_{\substack{\alpha \in \Sigma_{P_w} \\ w \alpha<0 }}  \left\{ \mu \in \fa_{P_w}^* \; \middle| \; \langle \Re(\mu),\alpha^\vee \rangle > -e(\sigma_{\pi,v})\right\}.
    \end{equation*}
\end{cor}

\begin{proof}
    If $w \in W(P)$ (so that $\nu_w=0$ and $P_w=P$), this follows from Lemma~\ref{lem:holo_compo} by decomposing $w$ as a product of adjacent transpositions. In general, decompose $w_\pi^*=w_{\pi_w}^* w^{\pi_w}$. Then we have and $w^{\pi_w}\nu_\pi=-\nu_w+\nu_{\pi_w}$. Note that $\sigma_{\pi_w}=\sigma_\pi$. We have for $\lambda \in \fa_{P,\cc}^*$ in general position
    \begin{equation}
        N(w,\lambda+\nu_{\pi}) N(w^*_\pi,\lambda-\nu_{\pi})=N(w,(\lambda+\nu_w)+\nu_{\pi_w})N(w_{\pi_w}^*,(\lambda+\nu_w)-\nu_{\pi_w})N(w^{\pi_w},\lambda-\nu_\pi).
    \end{equation}
    By Theorem~\ref{thm:N}, $N(w^{\pi_w},\lambda-\nu_\pi)$ is regular for $\lambda \in \fa_{P,\cc}^*$. Because $w \in W(P_w)$, we conclude by the first case.
\end{proof}

From Lemma~\ref{lem:holo_compo} we easily get the regularity of $ N_{\pi_w,v}$ for general $w \in {}_Q W_P$ such that $P_\pi \subset P_w$. 

\begin{lem}
\label{lem:N_regular}
    Let $\varphi \in \cA_{P,\pi}(G)$ and assume that $\varphi_{w}$ is factorizable. Then for every place $v$ the local intertwining operator
    \begin{equation}
    \label{eq:local_reg}
        N_{\pi_w,v}(w,\lambda+\nu_{w})\varphi_{w,-\nu_{w},v}
    \end{equation}
    is holomorphic in the region
     \begin{equation}
     \label{eq:conv_region_exp}
            \left\{ \lambda \in \fa_{P,\cc}^* \; | \; \forall \alpha \in \Sigma_{P}, \; \langle \Re(\lambda),\alpha^\vee \rangle > -2/(n^2+1) \right\}.
        \end{equation}
\end{lem}

\begin{proof}
    We know by \cite[Corollary~3.3]{BoiZ} that we can identify globally $\cA_{P,\pi}(G)$ as a subrepresentation of $\cA_{P_\pi,\sigma_\pi,\nu_\pi}(G)$ by applying the constant term $\varphi \mapsto \varphi_{P_\pi}$, which amounts locally to identifying $I_P^G \pi_v$ with the image of the intertwining operator $N_{\sigma_{\pi},v}(\cdot,-\nu_\pi) : I_{P_\pi}^G \sigma_{\nu_{\pi},-\nu_\pi,v} \to I_{P_\pi}^G \sigma_{\nu_{\pi},\nu_\pi,v}$. With the same notation as in Corollary~\ref{cor:holo_compo_must}, we have
    \begin{equation*}
        I_{P_\pi}^G \sigma_{\pi,-\nu_\pi,v} \xrightarrow{N_{\sigma_\pi,v}(w^{\pi_w},-\nu_\pi)} I_{P_\pi}^G \sigma_{\pi,-\nu_{\pi_w}+\nu_w,v} \xrightarrow{N_{\sigma_\pi,v}(w_{\pi_w}^*,-\nu_{\pi_w}+\nu_w)}  I_{P_\pi}^G \sigma_{\pi,\nu_\pi,v},
    \end{equation*}
    so that $I_{P}^G \pi_v$ is identified with the image of this composition in $I_{P_\pi}^G \sigma_{\nu_{\pi},\nu_\pi,v}$. This is a subrepresentation of the image of the last arrow which is $I_{P_w}^G \pi_{w,\nu_w,v}$. Therefore the regularity of \eqref{eq:local_reg} is equivalent to the regularity of $N_{\sigma_{\pi,v}}(w,\lambda+\nu_{\pi}) N_{\sigma_{\pi,v}}(w^*_\pi,\lambda-\nu_{\pi})$, which is known by Lemma~\ref{lem:holo_compo}.

\end{proof}

\subsubsection{Regularity of global intertwining operators} We can now describe the singularities of $M(w,\lambda)$.

\begin{prop}
\label{prop:M_regular}
    There exists $k>0$ such that for every level $J$ there exists $c_J>0$ such that for every $J$-pair $(P,\pi)$ with $P_\pi \subset P_w$ the map
    \begin{equation}
    \label{eq:regularized_operator}
        \lambda \mapsto L_{\pi,w}(\lambda)M(w,\lambda)\varphi_{w}
    \end{equation}
    is regular in $\cR_{\pi,k,c_J}$ for any $\varphi \in \cA_{P,\pi}(G)$, where $L_{\pi,w} \in \cc[\fa_{P,\cc}^*]$ is the product of affine root linear forms defined in Lemma~\ref{lem:n_regular}.
\end{prop}

\begin{rem}
    By Lemma~\ref{lem:N_regular}, all the singularities of $M(w,\lambda)$ in $\cR_{\pi,k,c_J}$ come from the scalar factor $n_{\pi_w}(\lambda+\nu_{w})$. But $L_{\pi,w}$ can have zeros in $i \fa_P^*$ and the operator $M(w,\lambda)$ can indeed have poles in this region. This can already be seen for $G=\GL_4$ in the unramified case (see \cite{Hegde}).
\end{rem}

\begin{proof}
    This follows from the factorization \eqref{eq:normalization_twisted} and Lemmas~\ref{lem:n_regular} and~\ref{lem:N_regular}.
\end{proof}

\begin{cor}
\label{cor:wider_holo}
      The map $\lambda \in \fa_{P}^* \mapsto L_{\pi,w}(\lambda)M(w,\lambda)\varphi_{w}$ is regular in the region $ \lambda+\nu_w \in \cR_{\pi,k,c_J}(w)$, where we recall that $\cR_{\pi,k,c_J}(w)$ was defined in \eqref{eq:R(w)_defi}.
\end{cor}

\begin{proof}
    We have
    \begin{equation*}
        M(w,\lambda)\varphi_w=M(w,\lambda+\nu_w) \varphi_{P_w,-\nu_w},
    \end{equation*}
    and we know by Lemma~\ref{lem:contant_term} that $\varphi_{P_w,-\nu_w} \in \cA_{P_w,\pi_w}(G)$. The result now follows from decomposing the intertwining operator and applying successively Proposition~\ref{prop:M_regular}. Note that here we use the same argument as in \cite[Section~3.4.7]{Ch} to make $\Lambda_\pi^{M_P}$ appear and not $\Lambda_{\pi_w}^{M_{P_w}}$ (which is a consequence of the fact that $\pi$ and $\pi_w$ are obtained as residues of Eisenstein series built from the same cuspidal representations).
\end{proof}

We also obtain a stronger version if $w \in W(P,Q)$. In this case, we already knew that the operator $M(w,\lambda)\varphi$ was regular on $i \fa_P^*$ by \cite[Theorem~7.2]{Art05}. The novelty is to extend this regularity property to a neighborhood of the unitary axis.

\begin{prop}
\label{prop:M_regular_blocks}
     Assume that $w \in W(P,Q)$. There exists $k>0$ such that for every level $J$ there exists $c_J>0$ such that for every $J$-pair $(P,\pi)$ the map
    \begin{equation*}
        \lambda \mapsto M(w,\lambda)\varphi
    \end{equation*}
    is regular in $\cS_{\pi,k,c_J}$ for any $\varphi \in \cA_{P,\pi}(G)$
\end{prop}

\begin{proof}
     This follows from the factorization \eqref{eq:normalization_twisted} and Lemmas~\ref{lem:blocks_n_regular} and~\ref{lem:N_regular}. 
\end{proof}

We now record a generalization of a well-known lemma on the behavior of intertwining operators along root-hyperplanes. We use the notation of \S\ref{subsubsec:residual_blocks}.

\begin{lem}
\label{lem:local_behavior}
    Assume that $n=2m$ and let $P$ be the standard parabolic subgroup of $G$ with standard Levi factor $\GL_m \times \GL_m$. Let $\pi \in \Pi_{\disc}(M_P)$ be of the form $\pi=\pi_1 \boxtimes \pi_1$. Let $w$ be the only non-trivial element in $W(P,P)$. Then for every place $v$ of $F$ we have $N_{\pi,v}(w,0)=\mathrm{Id}$.
\end{lem}

\begin{proof}
    Assume that $\pi_1=\Speh(\sigma,d)$ with $\sigma$ cuspidal. There exists $w' \in W(P_\pi)$ such that $-w' \nu_\pi$ is positive. Set $P_\pi'=(w')^{-1}.P_\pi$ and set $\sigma_\pi'=(w')^{-1} \sigma_\pi$ and $\nu_\pi'=(w')^{-1} \nu_\pi$. Set $w_0=(w')^{-1}w_\pi^* ww_\pi^* w'$. By Theorem~\ref{thm:N} and Lemma~\ref{lem:holo_compo}, for $\lambda \in \fa_{P,\cc}^{G,*}$ in a neighborhood of $0$ we have the commutative diagram with surjective vertical maps
    \begin{equation*}
\begin{tikzcd}
{I_{P_\pi'}^G \sigma_{\pi,\lambda'-\nu_\pi',v}'} \arrow[d, "{N(w_\pi^*w',\lambda'-\nu_\pi')}"', two heads] \arrow[rr, "{N(w_0,\lambda'-\nu_\pi')}"] &  & {I_{P_\pi'}^G \sigma_{\pi,-\lambda'-\nu_\pi',v}'} \arrow[d, "{N(w_\pi^*w',\lambda'-\nu_\pi')}", two heads] \\
{I_P^G \pi_{\lambda,v}} \arrow[rr, "{N(w,\lambda)}"]                                                                                                &  & {I_P^G \pi_{-\lambda,v}}                                                                       
\end{tikzcd}
    \end{equation*}
    It therefore suffices to prove that $N_{\sigma_\pi'}(w_0,\nu_\pi')=\mathrm{Id}$. But we can decompose $\sigma_\pi'=(\sigma \boxtimes \sigma)^{\boxtimes d}$, and by Theorem~\ref{thm:N} we see that $N_{\sigma_\pi'}(w_0,\nu_\pi')$ is the a product of operators $N_{\sigma \boxtimes \sigma,v}(\iota,0)$, with $\iota$ the involution switching the two blocks of $\sigma \boxtimes \sigma$. By \cite[Proposition~6.3]{KS}, they are equal to $\mathrm{Id}$. This concludes the proof.
\end{proof}

\begin{lem}
    \label{lem:M_zeros} 
    Let $P$ be standard parabolic subgroup of $G$, let $\pi \in \Pi_{\disc}(M_P)$. Let $w$ be the element that acts by blocks on $P$ corresponding to the transposition $(i \; j)$ with $i<j$. If $\pi_i \simeq \pi_j$, then for $\lambda \in \fa_{P,\cc}^*$ in general position in the hyperplane $\lambda_i=\lambda_j$ we have for every $\varphi \in \cA_{P,\pi}(G)$
    \begin{equation}
    \label{eq:M_zeros}
        M(w,\lambda) \varphi =- \varphi_\lambda.
    \end{equation}
\end{lem}

\begin{proof}
    Let us keep the notation of the proof of Lemma~\ref{lem:n_regular}. If $l \neq i$, we easily see that $n_{i,l}(\lambda)$ (or $n_{l,i}(\lambda)$) is regular along the hyperplane $\lambda_i=\lambda_j$, and the same goes for $n_{l,j}(\lambda)$ (or $n_{j,l}(\lambda)$). As $d_i=d_j$ and $b_a=d_j$ for every $1 \leq a \leq d_i$, we compute using \eqref{eq:n_i,j}
    \begin{equation*}
        n_{i,j}(\lambda)=\frac{L(\lambda_i-\lambda_j,\sigma_i \times \sigma_j^\vee)}{L(\lambda_i-\lambda_j+1,\sigma_i \times \sigma_j^\vee)} \prod_{a=1}^{d_i-1}L(\lambda_i-\lambda_j+a-d_i,\sigma_i \times \sigma_j^\vee)\prod_{a=2}^{d_i}\frac{1}{L(\lambda_i-\lambda_j+a,\sigma_i \times \sigma_j^\vee)}.
    \end{equation*}
    By Theorem~\ref{thm:L}, $n_{i,j}(\lambda)$ is regular if $\lambda_i-\lambda_j=0$. As $N_{\pi}(w,\lambda)$ is regular for $\lambda$ in general position in this hyperplane by Lemma~\ref{lem:N_regular}, so is $M(w,\lambda) \varphi$. 
    
    Let $\tau \in W(P,P)$ be the permutation $(i+1\;\hdots\;j-1\; j)$. Then if $\lambda_i=\lambda_j$ we have $w \lambda=\lambda$ and therefore by \eqref{eq:M_equation}
    \begin{equation*}
        M(w,\lambda)=M(\tau^{-1},\tau \lambda)M(\tau w \tau^{-1},\tau \lambda) M(\tau,\lambda).
    \end{equation*}
    By the previous discussion, $M(\tau^{-1},\tau \lambda)$, $M(\tau w \tau^{-1},\tau \lambda)$ and $M(\tau,\lambda)$ are regular for $\lambda$ in general position along the hyperplane $\lambda_i=\lambda_j$. 
    
    It follows that it is enough to prove that \eqref{eq:M_zeros} holds in the case $m=2$, i.e. $P=\GL_{n_1} \times \GL_{n_1}$, $\pi=\pi_1 \boxtimes \pi_1$, $P_\pi=\GL_r^d \times \GL_r^d$, $\sigma_\pi=\sigma^{\boxtimes d} \boxtimes \sigma^{\boxtimes d}$ and $w=(1 \; 2)$. By \eqref{eq:normalization_twisted} and Lemma~\ref{lem:local_behavior}, we have to prove that $n_{\pi}(w,\lambda)=-1$ along our hyperplane. Set $s=\lambda_1-\lambda_2$. By \eqref{eq:global_equality} and Theorem~\ref{thm:L}, it is easy to see that $n_\pi(w,s)$ is regular at $s=0$ and that $n_{\pi}(w,0)=(-1)^d (-1)^{d-1}=-1$. This concludes.
\end{proof}

\subsection{Analytic properties of discrete Eisenstein series}

The main result of this section is Theorem~\ref{thm:analytic_Eisenstein} below which describes the analytic behavior of discrete Eisenstein series in a neighborhood $\cR_{\pi,k,c_J}$ of the positive Weyl chambers. 

\subsubsection{Bernstein--Zelevinsky segments} \label{subsubsec:BZ_segments} To write our result, we use the notion of segments from \cite{BZ}. We define a segment $S$ to be a $r$-tuple of complex numbers $(s_1,\hdots,s_r)$ with
\begin{equation*}
    s_1-s_2=-1, \; s_2-s_3=-1, \hdots, s_{r-1}-s_r=-1.
\end{equation*}
The set $\{s_1,\hdots,s_r\}$ completely determines the tuple $(s_1,\hdots,s_r)$, so that we may think of $S$ as a set of complex numbers. If $S$ and $T$ are two segments, we say that they are linked if $S$ and $T$ are not included in one another and if $S \cup T$ is still a segment. 

We use the notation of \S\ref{subsubsec:residual_blocks}. To any $\pi \in \Pi_{\disc}(M_P)$ we associate the $m$ Bernstein--Zelevinsky segments $S_\pi^1, \hdots, S_\pi^m$ such that each $i$, $S_\pi^i=\nu_{P_{\pi_i}}$. In coordinates, $S_\pi^i$ is the $d_i$-tuple  
\begin{equation*}
    \left(S_\pi^i \right)_j=\frac{2j-1-d_i}{2}, \quad 1 \leq j \leq d_i.
\end{equation*}
If $\lambda \in \fa_{P,\cc}^*$ with coordinates $(\lambda_1,\hdots,\lambda_m)$, let $S_\pi^i(\lambda)$ be the segment $S_\pi^i+\lambda_i$. For each $i<j$, the set of $\lambda \in \fa_{P,\cc}^*$ such that $S_\pi^i(\lambda)$ and $S_\pi^j(\lambda)$ are linked is a finite union of affine root hyperplanes of $\fa_{P,\cc}^*$. Let $L_{\pi,i,j} \in \cc[\fa_{P,\cc}^*]$ be the corresponding product of affine root linear forms. Define  
\begin{equation}
\label{eq:f_pi_defi}
    L_{\pi,E}(\lambda)=\prod_{\substack{i<j \\ \sigma_i \simeq \sigma_j}}L_{\pi,i,j}(\lambda),
\end{equation}
where $\sigma_\pi=\sigma_1^{\boxtimes d_1} \boxtimes \hdots \boxtimes \sigma_m^{\boxtimes d_m}$. Note that the set $\{L_{\pi,E}\}$ is finite as $\pi$ varies.

\subsubsection{Analytic properties of discrete Eisenstein series}
We can now state the theorem.

\begin{theorem}
\label{thm:analytic_Eisenstein}
The following properties hold.
\begin{enumerate}
        \item Regularity: there exists $k>0$ such that for every level $J$ there exists $c_J>0$ such that for every $J$-pair $(P,\pi)$ the meromorphic function
        \begin{equation*}
            L_{\pi,E}(\lambda) E(\varphi,\lambda)
        \end{equation*}
        is holomorphic in the region $\lambda \in \cR_{\pi,k,c_J}$ for every $\varphi \in \cA_{P,\pi}(G)$. 
        \item Zeros: let $P$ be a standard parabolic subgroup of $G$, $\pi \in \Pi_{\disc}(M_P)$ and $\varphi \in \cA_{P,\pi}(G)$. For every $\alpha_{i,j} \in \Sigma_P$ which is associated to two blocks $\GL_{n_i}$ and $\GL_{n_j}$ of $M_P$ with $\pi_i \simeq \pi_j$, the Eisenstein series $E(\varphi,\lambda)$ has a zero along the root hyperplane $\langle \lambda, \alpha_{i,j}^\vee \rangle=0$.
    \end{enumerate}
\end{theorem}

\begin{rem}
    Recall that we have defined a polynomial $L_{\pi,w}$ in Lemma~\ref{lem:n_regular}. Then $L_{\pi,E}$ divides the product $\prod_Q \prod_{w \in {}_Q W_P} L_{\pi,w}$, but in general is not equal to it except if $\pi$ is cuspidal. Theorem~\ref{thm:analytic_Eisenstein} therefore shows that there are some subtle compensations of residues in the constant term $E_{P_\pi}(\varphi,\lambda)$. Moreover, Theorem~\ref{thm:analytic_Eisenstein} is sharp in the sense that all the zeros of $L_{\pi,E}$ compensate poles of $E(\varphi,\lambda)$ as can already be seen in the unramified case \cite{Hegde}.
\end{rem}

\begin{proof}
    Let us prove 1. We first assume that $\varphi$ is $K_\infty$-finite. By \cite[Lemma~I.4.10]{MW95}, the singularities of $L_{\pi,E}(\lambda)E(\varphi,\lambda)$ are the same as those of its cuspidal exponents. Let $\phi \in \cA_{P_\pi,\sigma_\pi}(G)$ be such that $\varphi=E^{P,*}(\phi)$. By the computation of the constant term of Eisenstein series \eqref{eq:constant_term} and by Lemma~\ref{lem:contant_term}, for every standard parabolic subgroup $Q$ we have
    \begin{equation}
        \label{eq:regular_discrete}
        L_{\pi,E}(\lambda) E_Q(\varphi,\lambda)^\cusp=c_\pi \sum_{\substack{w \in {}_Q W_P \\ P_\pi = P_w, \; Q_w=Q}}  L_{\pi,E}(\lambda) M(w,\lambda+\nu_{\pi})N_\sigma(w_\pi^*,\lambda-\nu_{\pi})\phi,
    \end{equation}
    where $c_\pi$ is the constant from \eqref{eq:regular_is_local}. By \cite[Lemme~III.2]{MW89}, \eqref{eq:regular_discrete} is regular in a neighborhood of $\lambda \in \fa_{P,\cc}^*$ as long as for all $\alpha \in \Sigma_P$ and all $i$ and $j$ we have $L(1+\langle \lambda,\alpha^\vee \rangle,\sigma_i \times \sigma_j^\vee) \neq 0$ and moreover $\langle \lambda,\alpha^\vee \rangle >-e(\sigma_{\pi,v})$ for any place $v \in V_F$. More precisely, note that \eqref{eq:regular_discrete} is the first equation of \cite[652]{MW89} (see the proof of Lemma~\ref{lem:holo_compo} for some comments on the notation of \cite{MW89}). By Theorem~\ref{thm:L} and Theorem~\ref{thm:N} we may choose $k$ and $c_J$ such that these conditions are satisfied in the region $\cR_{\pi,k,c_J}$. The case of general $\varphi \in \cA_{P,\pi}(G)$ follows from the $K_\infty$-finite case by the same argument as \cite[Remark~11.8]{BK}.
    
    For 2, note that by 1. the hyperplane $\langle \lambda, \alpha_{i,j}^\vee \rangle=0$ is not singular for $E(\varphi,\lambda)$. Let $w \in W(P) \simeq \fS_m$ be the transposition $(i \; j)$. By the functional equation of Eisenstein series from \cite[Theorem~2.3.4]{BL} and Lemma~\ref{lem:M_zeros}, for any $\lambda$ such that $\langle \lambda, \alpha_{i,j}^\vee \rangle=0$ we have
    \begin{equation*}
        E(\varphi,\lambda)=E(M(w,\lambda)\varphi,w\lambda)=-E(\varphi,\lambda).
    \end{equation*}
    This proves 3.
\end{proof}

For every standard parabolic subgroup $P$ of $G$ and every $\pi \in \Pi_\disc(M_P)$, set
\begin{equation}
\label{eq:f_pi_0_defi}
    L_{\pi,0}(\lambda)=\prod_{\pi_i \simeq \pi_j} \langle \lambda, \alpha_{i,j}^\vee \rangle.
\end{equation}
This polynomial controls the zeros of $E(\varphi,\lambda)$. Using Theorem~\ref{thm:analytic_Eisenstein}, we obtain the following corollary.
\begin{cor}
    \label{cor:zeros_of_E}
    There exists $k>0$ such that for every level $J$ there exists $c_J>0$ such that for every $J$-pair $(P,\pi)$ the meromorphic function
    \begin{equation*}
        \frac{L_{\pi,E}(\lambda)}{L_{\pi,0}(\lambda)} E(\varphi,\lambda)
    \end{equation*}
    is holomorphic in the region $\cR_{\pi,k,c_J}$ for every $\varphi \in \cA_{P,\pi}(G)$.
\end{cor}

\subsubsection{The case of generalized Eisenstein series}
\label{subsubsec:partial_Eisenstein}
Let $P, Q$ be standard parabolic subgroups of $G$. Let $\pi \in \Pi_{\disc}(M_P)$. Let $w \in {}_Q W_P$ such that $P_\pi \subset P_w$. Let $\pi_w \in \Pi_{\disc}(M_{P_w})$ be the representation defined in \S\ref{subsubsection:constant_terms_discrete}. Write $M_Q=\prod_{i=1}^{m'} \GL_{n'_i}$ and set $Q_{w,i}=\GL_{n'_i} \cap Q_w$. Write $w\pi_w=(w\pi_w)_1 \boxtimes \hdots \boxtimes (w \pi_w)_{m'}$ accordingly. We have $\fa_{Q_w,\cc}^*=\oplus_i \fa_{Q_{w,i},\cc}^*$, and for $\lambda \in \fa_{Q_w,\cc}^*$ we denote by $p_i(\lambda)$ its projection on $\fa_{Q_{w,i},\cc}^*$. Set
\begin{equation}
\label{eq:partial_divisor}
    L_{w \pi_w,E}^Q(\lambda)=\prod_{i=1}^{m'} L_{(w\pi_w)_i,E}\left(p_i(\lambda)\right),
\end{equation}
where $L_{(w\pi_w)_i,E}$ is defined in \eqref{eq:f_pi_defi}. Moreover, for any $c$ and $k$ let $\cR_{(w\pi_w)_i,k,c} \subset \fa_{Q_{w,i},\cc}^*$ be the subset defined in \eqref{eq:R_defi} for $(w \pi_w)_i$.

\begin{prop}
\label{prop:partial_regular}
    There exists $k>0$ such that for every level $J$ there exists $c_J>0$ such that for any $J$-pair $(P,\pi)$ and $\varphi \in \cA_{P,\pi}(G)$ the Eisenstein series
    \begin{equation*}
        L_{\pi,w}(\lambda)L_{w\pi_w,E} ^Q(w(\lambda+\nu_{w})) E^Q(M(w,\lambda)\varphi_{w},w\lambda)
    \end{equation*}
    is regular if $\lambda+\nu_w \in \cR_{\pi,k,c_J}(w)$ and if $p_i(w(\lambda+\nu_{w})) \in \cR_{(w\pi_w)_i,k,c_J}$ for each $i$.
\end{prop}

\begin{rem}
\label{rem:cuspi_case}
    If $\pi \in \Pi_{\cusp}(M_P)$ (in which case $P_w=P_\pi=P$ and $\nu_{w}=0$), it is enough to ask that $\lambda \in \cR_{\pi,k,c_J}$. Indeed, this already implies that $\lambda+\nu_w \in \cR_{\pi,k,c_J}(w)$ and that $p_i(w\lambda) \in \cR_{(w\pi_w)_i,k,c_J}$ for each $i$. Moreover for every $w$ the product $L_{\pi,w}(\lambda) L_{w \pi_w,E}^Q(w\lambda)$ divides $L_{\pi,E}$.
\end{rem}

\begin{proof}
    By Corollary~\ref{cor:wider_holo}, we may choose the constants $k$ and $c_J>0$ such that the map $L_{\pi,w}(\lambda) M(w,\lambda)\varphi_{w}$ is regular in $\cR_{\pi,k,c_J}$. By the Iwasawa decomposition, a generalized Eisenstein series $E^Q(\psi,\lambda)$ is regular at some $\lambda \in \fa_{P,\cc}^*$ if and only if its restriction to $M_Q(\bA)$ is. By Lemma~\ref{lem:contant_term} we have
    \begin{equation*}
        \left(\left(M(w,\lambda)\varphi_{w}\right)_{|M_Q}\right)_{-w\nu_{w}-\rho_Q} \in \cA_{Q_w,w \pi_{w},w\lambda}(M_Q).
    \end{equation*}
    The result therefore follows from Theorem~\ref{thm:analytic_Eisenstein}.
\end{proof}

\subsection{Bounds for residual Eisenstein series}
\label{subsec:bounds_Eisenstein}

We state some results on growth of Eisenstein series along orthonormal sums. For our purposes, we need these bounds on the right of the unitary axis, or rather on a neighborhood of this region. In \cite[Section~3]{Ch}, such estimates are proved in a neighborhood of the unitary axis and it turns out that the method used there extends to our region. Indeed, the key input needed is the regularity of the intertwining operators $M(w,\lambda)$, which is available to us by Proposition~\ref{prop:M_regular}. Therefore, in what follows we will only sketch the main modifications to adapt \cite[Section~3]{Ch} to our setting. The reader can also refer to \cite[Section~10]{BoiPhD} where we produced an independent proof for our bounds at a time where \cite{Ch} was not yet released, but we emphasize that the arguments we used there are essentially the same.

\subsubsection{Bound for the global operator}

We first bound the global intertwining operators.

\begin{prop}
    \label{prop:bound_global_operator}
    Let $D$ be a holomorphic differential operator on $\fa_{P,\cc}^*$. There exists $k>0$ such that for every level $J$ and every $C>0$, there exist $c_J$ and $C_J>0$ such that for every $J$-triple $(P,\pi,\tau)$ with $P_\pi \subset P_w$ we have
    \begin{equation}
    \label{eq:bound_global_operator}
        \norm{D(L_{\pi,w}(\lambda)M(w,\lambda)\varphi_{w})}_{Q_w,\Pet} \leq C_J \left( 1+\norm{\lambda}^2+\lambda_\tau^2+(\Lambda_\pi^{M_P})^2 \right)^{k} \norm{\varphi_{w}}_{P_w,\Pet},
    \end{equation}
    for every $\varphi \in \cA_{P,\pi}(G)^{\tau,J}$ and $\lambda \in \cR_{\pi,k,c_J}^{C}$, where $L_{\pi,w}$ is the polynomial from Lemma~\ref{lem:n_regular}.
\end{prop}

\begin{proof}
    By Proposition~\ref{prop:M_regular}, we know that $\lambda \mapsto L_{\pi,w}(\lambda)M(w,\lambda) \varphi_w$ is regular. If $\lambda$ belongs to some region $\cS_{\pi,k,c_J}$, the bound \eqref{eq:bound_global_operator} follows from the factorization of $M(w,\lambda)$ in \eqref{eq:normalization_twisted} and from \cite[Proposition~3.4.1.1]{Ch} and \cite[Proposition~3.6.1.1]{Ch}. In general, one readily checks that the proof goes through in the larger region $\cR_{\pi,k,c_J}^{C}$.
\end{proof}

\subsubsection{Bounds for Eisenstein series}
\label{subsubsec:bounds_Eisenstein}

The key estimate we need is the following.

\begin{prop}
\label{prop:bound_Eisenstein}
    There exists $k>0$ such that for all $C>0$ there exists $N>0$ such that for all levels $J$, all $q>0$ and all $X \in \cU(\fg_\infty)$ there exist $c_J>0$ and a continuous semi-norm $\norm{\cdot}_{J,q,X}$ on $\cS(G(\bA))^J$ such that for all $J$-triples $(P,\pi,\tau)$, all $f \in \cS(G(\bA))^J$ and $\varphi \in \cA_{P,\pi}(G)^{\tau,J}$ we have
    \begin{equation}
    \label{eq:good_bound}
         \norm{\frac{L_{\pi,E}(\lambda)}{L_{\pi,0}(\lambda)} E(I_{P}(\lambda,f)\varphi,\lambda)}_{-N,X} \leq \frac{\norm{f_\tau}_{J,q,X} \norm{\varphi}_{P,\Pet}}{(1+\norm{\lambda}^2)^q(1+\lambda_\pi^2+\lambda_\tau^2)^q},
    \end{equation}
    in the region $\lambda \in \cR_{\pi,k,c_J}^{C}$.
\end{prop}

\begin{rem}
\label{rem:m_bound}
    As the proof shows, of Proposition~\ref{prop:bound_Eisenstein} can be strengthened by asking that for all $q>0$ there exists $m>0$ such that for all $f \in C_c^m(G(\bA))^J$ the bound \eqref{eq:good_bound} remains valid, the other quantifiers remaining the same.
\end{rem}

\begin{proof}
    This is a modified version of \cite[Lemma~3.9.1.1]{Ch}. By replacing $f$ by $L(X)f$, we can assume that $X=0$. Let $m>0$ and $g \in C_c^m(G(\bA))$. We first show that \eqref{eq:good_bound} holds if we replace $f$ by $g*f$ in the LHS. The proof of \cite[Lemma~3.9.1.1]{Ch} relies on an estimate of the scalar product of truncated Eisenstein series (\cite[Proposition~3.7.1.1]{Ch}) which itself uses bounds for the intertwining operators $M(w,\lambda)$. By Proposition~\ref{prop:bound_global_operator}, these bounds are available to us in some region $\cR_{\pi,k,c_J}^{C}$. More precisely, by reproducing \cite[Proposition~3.7.1.1]{Ch} we obtain a bound for the truncated scalar product
    \begin{equation*}
    \frac{L_{\pi,E}(\lambda)L_{\pi,E}(\lambda')}{L_{\pi,0}(\lambda)L_{\pi,0}(\lambda')}\langle \Lambda^{T,\mathrm{Art}} E(\varphi,\lambda), \overline{E(\varphi',\lambda')} \rangle_{G,\Pet},
\end{equation*}
which we know is regular by the properties of the truncation operator from \cite[Theorem~3.9]{Zydor}, the continuity of Eisenstein series from \cite[Theorem~2.2]{Lap}, and Theorem~\ref{thm:analytic_Eisenstein}.
    
The proof of Proposition~\ref{prop:bound_Eisenstein} is now the same as \cite[Lemma~3.9.1.1]{Ch} up to two minor differences. The first is that \cite[Lemma~3.9.1.1]{Ch} is written for $x \in G(\bA)^1$. However, for any $a \in A_G^\infty$ one has 
    \begin{align*}
        \Val{E(ax,I_{P}(\lambda,g*f)\varphi,\lambda)}=&\Val{\exp(\langle \lambda ,H_G(a) \rangle)E(x,I_{P}(\lambda,g*f)\varphi,\lambda)} \\
        \leq &\norm{a}^N \Val{E(x,I_{P}(\lambda,g*f)\varphi,\lambda)}
    \end{align*}
    for some $N$ as long as $\norm{\Re(\lambda)}<C$. Moreover, for $x \in G(\bA)^1$ and $a \in A_G^\infty$ we have $\norm{x}_G^N \ll \norm{a}^{N} \norm{ax}_G^{N}$ by \eqref{eq:mult_norm}, and $\norm{a} \ll \norm{ax}_G$ by \cite[(2.4.1.6)]{BPCZ}. Therefore, for our purposes we may replace $\norm{\cdot}_N$ by the sup-norm of our regularized Eisenstein series times $\norm{x}_G^{-N}$.

    The second difference is that \cite[Lemma~3.9.1.1]{Ch} assumes that $\lambda \in \fa_P^{G,*}$ while we work with $\lambda \in \fa_P^{*}$. This comes into play in the computation of the kernel function $k_g$ which has to be replaced by
    \begin{equation*}
        k_g(x,y)=\sum_{\gamma \in G(F)} \int_{A_G^\infty} g(x^{-1}a \gamma y) \exp(\langle H_G(a),\lambda \rangle) da.
    \end{equation*}
    By \cite[Lemma~I.2.4]{MW95}, the integral over $A_G^\infty$ takes place inside a compact subset of $A_G^\infty$ which only depends on the support of $g$. As $\norm{\Re(\lambda)}<C$, we see that \cite[(3.9.1.2)]{Ch} holds, namely that there exist $c_2$ and $N>0$ such that
    \begin{equation*}
        \sup_{y \in G(\bA)^1} \Val{k_g(x,y)} \leq c_2 \norm{x}_G^N, \quad x \in G(\bA)^1.
    \end{equation*}

    With these modifications in mind, \cite[Lemma~3.9.1.1]{Ch} yields for every $q>0$ a bound
      \begin{equation}
      \label{eq:with_g_bound}
         \norm{\frac{L_{\pi,E}(\lambda)}{L_{\pi,0}(\lambda)} E(I_{P}(\lambda,g*f)\varphi,\lambda)}_{-N} \leq \frac{\norm{f_\tau}_{J,q} \norm{\varphi}_{P,\Pet}}{(1+\norm{\lambda}^2)^q(1+\lambda_\pi^2+\lambda_\tau^2)^q}.
    \end{equation}
    But according to \cite[Corollary~4.2]{Art78}, for any level $J$ and any $m \geq 1$ large enough there exist $Z \in \cU(\fg_\infty)$, $g_1 \in C_c^\infty(G(\bA))$ and $g_2 \in C_c^m(G(\bA))$ such that $Z$ is invariant under $K_\infty$-conjugation, $g_1$ and $g_2$ are invariant under $K$-conjugation and are $J$-biinvariant, and 
    for any $f \in \cS(G(\bA))^J$ we have:
    \begin{equation}
    \label{eq:the_trick_with_schwartz}
        f=g_1 * f + g_2 * (Z*f).
    \end{equation}
Proposition~\ref{prop:bound_Eisenstein} is now a direct consequence of \eqref{eq:with_g_bound}.
\end{proof}

From there, one can obtain bounds for Eisenstein series in orthonormal sums.

\begin{theorem}
\label{thm:bound_Eisenstein}
    There exists $k>0$ such that for all $C>0$ there exists $N>0$ such that for all $q>0$ all levels $J$ and all $X \in \cU(\fg_\infty)$ there exist $c_J>0$ and a continuous semi-norm $\norm{\cdot}_{J,q,X}$ on $\cS(G(\bA))^J$ such that for all $J$-pairs $(P,\pi)$ and all $f \in \cS(G(\bA))^J$ we have
    \begin{equation}
    \label{eq:big_bound1}
        \sum_{\varphi \in \cB_{P,\pi}(J)}\norm{\frac{L_{\pi,E}(\lambda)}{L_{\pi,0}(\lambda)}E(I_{P}(\lambda,f)\varphi,\lambda)}_{-N,X} \leq \frac{ \norm{f}_{J,q,X}}{(1+\norm{\lambda}^2)^q(1+\Lambda_\pi^2)^q},
    \end{equation}
    in the region $\lambda \in \cR_{\pi,k,c_J}^{C}$. Moreover, we also have for $\Re(\lambda) \in \overline{\fa_{P}^{*,+}}$ and $\norm{\Re(\lambda)} \leq C$, up to changing the semi-norm,
     \begin{equation}
    \label{eq:big_bound2}
        \sum_{\pi \in \Pi_\disc(M_P)}\sum_{\varphi \in \cB_{P,\pi}(J)}\norm{\frac{L_{\pi,E}(\lambda)}{L_{\pi,0}(\lambda)}E(I_{P}(\lambda,f)\varphi,\lambda)}_{-N,X} \leq \frac{ \norm{f}_{J,q,X}}{(1+\norm{\lambda}^2)^q},
    \end{equation}
\end{theorem}

\begin{proof}
This is a direct consequence of Proposition~\ref{prop:bound_Eisenstein} and of Lemma~\ref{lem:spectral_sum_bound} below.
\end{proof}

\begin{lem}
\label{lem:spectral_sum_bound}
    Let $J$ be a level. For $q>0$ large enough, for any continuous semi-norm $\norm{\cdot}$ on $\cS(G(\bA))^J$ the sum
    \begin{equation*}
        \sum_{P_0 \subset P} \sum_{\pi \in \Pi_{\disc}(M_P)^J}
    \sum_{\tau \in \widehat{K}_\infty}\frac{ \norm{f_\tau} \dim(\cA_{P,\pi}(G)^{\tau,J})} {(1+\lambda_\pi^2+\lambda_\tau^2)^q}
    \end{equation*}
    is absolutely convergent and defines a continuous semi-norm on $\cS(G(\bA))^J$.
\end{lem}

\begin{proof}
We first bound
\begin{equation*}
    \sum_{\tau \in \widehat{K}_\infty}\frac{ \norm{f_\tau}_{J,q} \dim(\cA_{P,\pi}(G)^{\tau,J})} {(1+\lambda_\pi^2+\lambda_\tau^2)^q}.
\end{equation*}
According to \cite[Equation (6.4)]{Mu2}, there exist $A>0$ and $k' \in \nn$ such that
\begin{equation*}
    \dim(\cA_{P,\pi}(G)^{\tau,J}) \leq A(1+\lambda_\pi^2+\lambda_\tau^2)^{k'}, \quad \pi \in \Pi_\disc(M_P).
\end{equation*}
By \cite[931]{Art78}, we know that $(\sum_{\tau} \norm{f_\tau}_{J,q}^2)^{1/2}$ is a continuous semi-norm on $\cS(G(\bA))^J$. By the Cauchy-Schwarz inequality, we are reduced to bounding $\sum_{\tau} (1+\lambda_\pi^2+\lambda_\tau^2)^{-q}$ for $q$ large enough. But by the same argument as in \cite[Proof of Proposition~3.8.3.1]{Ch} there exists $B>0$ such that for all $P$ and all $\pi \in \Pi_\disc(M_P)$ we have
\begin{equation*}
    \sum_{\tau \in \widehat{K}_\infty}(1+\lambda_\pi^2+\lambda_\tau^2)^{-q} \leq B(1+\Lambda_\pi^2)^{-q/2} \sum_{\tau \in \widehat{K}_\infty}(1+\lambda_\tau^2)^{-q/2},
\end{equation*}
and $ \sum_{\tau \in \widehat{K}_\infty}(1+\lambda_\tau^2)^{-q/2}<\infty$ for $q$ large enough.

    We now need to show that for $q>0$ large enough we have
     \begin{equation}
        \label{eq:Muller_spectral}
        \sum_{\pi \in \Pi_\disc(M_P)^J} \frac{1}{(1+\Lambda_\pi^2)^q} < \infty.
    \end{equation}
    This is \cite[line (6.17) p. 711 and below]{Mu2}.
\end{proof}

\subsubsection{Application to relative characters}

We use Theorem~\ref{thm:bound_Eisenstein} to prove convergence properties on relative characters defined by summing along the bases $\cB_{P,\pi}(J)$. 

\begin{prop}
\label{prop:convergence_strong}
    There exists $k>0$ such that for all $C>0$ and all levels $J$, there exists $c_J>0$ such that for all $J$-pairs $(P,\pi)$ and all $F \in \cS([G])^J$ the series
    \begin{equation*}
        \sum_{\varphi \in \cB_{P,\pi}(J)} \left\langle F,\frac{L_{\pi,E}(\lambda)}{L_{\pi,0}(\lambda)} E(\varphi,\lambda) \right\rangle_G  \varphi
    \end{equation*}
    is absolutely convergent in $\cA_{P,\pi}(G)^J$ (embedded with the topology of some $\cT_{-N}([G]_P)^J$) and $L^2([G]_{P,0})^{J,\infty}$, for any $\lambda \in \cR_{\pi,k,c_J}^C$. 
    
    More precisely, for any $N>0$ large enough, any $X \in \cU(\fg_\infty)$ and any $q>0$, there exist $N'$ and a continuous semi-norm $\norm{\cdot}_{J,q,X}$ on $\cT_{-N'}([G])^J$ such that for any $J$-pair $(P,\pi)$ we have
    \begin{equation}
    \label{eq:abs_conv}
        \sum_{\varphi \in \cB_{P,\pi}(J)} \Val{\left\langle F,\frac{L_{\pi,E}(\lambda)}{L_{\pi,0}(\lambda)} E(\varphi,\lambda) \right\rangle_G}  \norm{\varphi}_{-N,X} \leq \frac{\norm{F}_{J,q,X}}{(1+\norm{\lambda}^2)^q(1+\Lambda_\pi^2)^q}, 
    \end{equation}
    for $\lambda \in \cR_{\pi,k,c_J}^C$ in our region and $F \in \cT_{-N'}([G])^J$. 
\end{prop}

\begin{proof}
    By Lemma~\ref{lem:sobolev_and_co}, for any $N$ large enough there exist $c>0$ and $r>0$ such that
    \begin{equation*}
        \norm{\varphi}_{-N,X} \leq c\sum_{j=1}^r \norm{R(\Delta^j)\varphi}_{P,\Pet}.
    \end{equation*}
    As in the proof of \cite[Proposition~3.8.2.3]{Ch}, we obtain some constants $c'>0$ and $d$ such that for any $\varphi \in \cA_{P,\pi}(G)^{\tau,J}$ we have for all $1 \leq j \leq r$
    \begin{equation*}
        \norm{R(\Delta^j)\varphi}_{P,\Pet} \leq c'(1+\lambda_\pi^2+\lambda_\tau^2)^{d} \norm{\varphi}_{P,\Pet}.
    \end{equation*}
    As in \eqref{eq:the_trick_with_schwartz}, there exist for any $m \geq 1$ large enough some elements $g_1,g_2 \in C_c^m(G(\bA))$ and $Z \in \cU(\fg_\infty)$ such that for any $F \in \cT_{-N'}([G])^J$ we have $F=g_1*F+g_2*(Z*F)$. We now conclude by Lemma~\ref{lem:height_properties}, Proposition~\ref{prop:bound_Eisenstein} (more precisely Remark~\ref{rem:m_bound}) and by repeating the proof of Theorem~\ref{thm:bound_Eisenstein}. The $L^2([G]_{P,0})^{J,\infty}$ version follows from Lemma~\ref{lem:sobolev_and_co}.
    
\end{proof}

\subsubsection{Bounds for individual Eisenstein series}

We now generalize Proposition~\ref{prop:bound_Eisenstein} for non necessarily $K_\infty$ finite Eisenstein series.

\begin{prop}
\label{prop:bound_Eisenstein_non_finite}
    There exists $k>0$ such that for all $C>0$ there exists $N>0$ such that for all levels $J$, all $q>0$ and all $X \in \cU(\fg_\infty)$ there exist $c_J>0$ and a continuous semi-norm $\norm{\cdot}_{J,q,X}$ on $\cS(G(\bA))^J$ such that for all $J$-pairs $(P,\pi)$, all $f \in \cS(G(\bA))^J$ and $\varphi \in \cA_{P,\pi}(G)^{J}$ we have
    \begin{equation}
    \label{eq:good_bound_non_finite}
         \norm{\frac{L_{\pi,E}(\lambda)}{L_{\pi,0}(\lambda)} E(I_{P}(\lambda,f)\varphi,\lambda)}_{-N,X} \leq \frac{\norm{f}_{J,q,X} \norm{\varphi}_{P,\Pet}}{(1+\norm{\lambda}^2)^q},
    \end{equation}
    in the region $\lambda \in \cR_{\pi,k,c_J}^{C}$. 

    Moreover, for all $N'$ large enough, there exist $d>0$ and $Y_1, \hdots, Y_r \in \cU(\fg_\infty)$ such that for all $J$-pairs $(P,\pi)$ and all $\varphi \in \cA_{P,\pi}(G)^J$ we have
    \begin{equation}
    \label{eq:individual_Eisenstein}
         \norm{\frac{L_{\pi,E}(\lambda)}{L_{\pi,0}(\lambda)} E(\varphi,\lambda)}_{-N,X} \leq (1+\norm{\lambda}^2)^d \sum_{i=1}^r \norm{\varphi}_{-N',Y_i}.
    \end{equation}
        in the region $\lambda \in \cR_{\pi,k,c_J}^{C}$, where $k$, $C$, $N$, $X$ and $c_J$ are as before.
\end{prop}

\begin{proof}
     Let $k$, $N$ and $c_J$ be given by Proposition~\ref{prop:bound_Eisenstein}. By \cite[931]{Art78}, there exists $m_0 \geq 1$ such that for all $m \geq m_0$ and all $f \in C_c^m(G(\bA))^J$ we have $f = \sum_{\tau \in \widehat{K}_\infty} f_\tau$ in $C_c^m(G(\bA))^J$. By Proposition~\ref{prop:bound_Eisenstein}, or more precisely by Remark~\ref{rem:m_bound}, for any $q >0$ and any $m$ large enough, there exists a continuous semi-norm $\norm{\cdot}_{J,q,X}$ on $C_c^m(G(\bA))$ such that for any $f \in C_c^m(G(\bA))^J$, we have
     \begin{equation*}
         \norm{\frac{L_{\pi,E}(\lambda)}{L_{\pi,0}(\lambda)} E(I_{P}(\lambda,f) \varphi,\lambda)}_{-N,X} \leq \sum_{\tau \in \widehat{K}_\infty} \frac{\norm{f_\tau}_{J,q,X} \norm{e_\tau*\varphi}_{P,\Pet}}{(1+\norm{\lambda}^2)^q(1+\lambda_\pi^2+\lambda_\tau^2)^q},
     \end{equation*}
     for any $\lambda \in \cR_{\pi,k,c_J}^C$. By the Cauchy--Schwarz inequality and because $(\sum_\tau \norm{e_\tau*\varphi}_{P,\Pet}^2)^{1/2} = \norm{\varphi}_{P,\Pet}$, we obtain \eqref{eq:good_bound_non_finite}.

     We now take $g_1,g_2 \in C_c^m(G(\bA))^J$ and $Z \in \cU(\fg_\infty)$ such that $g_1+g_2*Z$ is the Dirac distribution at identity as in \eqref{eq:the_trick_with_schwartz}. Because $\norm{I_{P}(Z,\lambda) \varphi}_{P,\Pet} \leq (1+\norm{\lambda}^2)^{d'} \sum_i \norm{R(Z_i)\varphi}_{P,\Pet}$ for some $d'>0$ and $Z_1, \hdots, Z_{r'} \in \cU(\fg_\infty)$, by plugging $g_1$ and $g_2$ in \eqref{eq:good_bound_non_finite} we obtain constants $c$ and $d$ such that
     \begin{equation*}
         \norm{\frac{L_{\pi,E}(\lambda)}{L_{\pi,0}(\lambda)} E(\varphi,\lambda)}_{-N,X} \leq c(1+\norm{\lambda}^2)^d \sum_{i=1}^{r'} \norm{R(Z_i)\varphi}_{P,\Pet}.
    \end{equation*}
    We conclude by Lemma~\ref{lem:sobolev_and_co}.
     
\end{proof}

As noted in \cite[Section~3.9.3]{Ch}, we may apply the methods used in \S\ref{subsubsec:bounds_Eisenstein} to bound generalized Eisenstein series. We only state the analogue of Proposition~\ref{prop:bound_Eisenstein_non_finite} in this setting, although it is clear that the other results can be adapted as well. We keep the notation of \S\ref{subsubsec:partial_Eisenstein}.

\begin{prop}
    \label{prop:partial_Eisenstein}
      Let $P$ and $Q$ be standard parabolic subgroups of $G$. Let $w \in {}_Q W_P$. There exists $k>0$ such that for all $C>0$ there exist $N>0$ such that for all levels $J>0$, all $X \in \cU(\fg_\infty)$ and all $N'>0$ large enough there exists $c_J>0$, $d>0$ and $Y_1, \hdots, Y_r \in \cU(\fg_\infty)$ such that for all $J$-pairs $(P,\pi)$ and all $\varphi \in \cA_{P,\pi}(G)^J$ we have
    \begin{equation*}
         \norm{L_{\pi,w}(\lambda) L_{w \pi_w,E}^Q(w(\lambda+\nu_{w}))E^Q(M(w,\lambda)\varphi_{P_w},w\lambda)}_{-N,X}   \leq (1+\norm{\lambda}^2)^d \sum_{i=1}^r \norm{\varphi}_{-N',Y_i}.
    \end{equation*}
    in the region $\lambda+\nu_w \in \cR^C_{\pi,k,c_J}(w)$ and $p_i(w(\lambda+\nu_{w})) \in \cR_{(w\pi_w)_i,k,c_J}$ for all $i$. Here $L_{w \pi_w,E}^Q$ is defined in \eqref{eq:partial_divisor} and $\norm{\cdot}_{-N,X}$ is a semi-norm on $\cT_{-N}([G]_Q)$.
\end{prop}

\begin{proof}
    This is an adaptation of the proof of Proposition~\ref{prop:bound_Eisenstein_non_finite} (see also \cite[Theorem~3.9.3.1]{Ch}). It follows from the bounds for intertwining operators (Corollary~\ref{cor:wider_holo}) and the regularity of generalized Eisenstein series (Proposition~\ref{prop:partial_Eisenstein}).
\end{proof}

\subsection{An extension of Langlands spectral decomposition theorem}

We now use Theorem~\ref{thm:bound_Eisenstein} to extend the Langlands spectral decomposition Theorem~\ref{thm:Langlands_spectral} to functions on $[G]$ of rapid enough decay. 

\begin{prop}
\label{prop:extended_Langlands}
    There exists $N >0$ such that for all level $J$ of $G$ and all $F_1,F_2 \in \cT_{-N}([G])^J$ we have
     \begin{equation}
     \label{eq:Langlands_extend}
        \langle F_1,F_2 \rangle_G 
        = \sum_{P_0 \subset P} \Val{\cP(M_P)}^{-1} \sum_{\pi \in \Pi_\disc(M_P)} \int_{i \fa_P^*}  \sum_{\varphi \in \cB_{P,\pi}(J)} \langle F_1,E(\varphi,\lambda) \rangle_G \langle E(\varphi,\lambda),F_2 \rangle_G d\lambda,
    \end{equation}
    where this expression is absolutely convergent.
\end{prop}

\begin{proof}
    For all pairs $(P,\pi)$, the polynomial $L_{\pi,E}$ of Theorem~\ref{thm:bound_Eisenstein} is non-zero and bounded above on $i \fa_P^*$. By Theorem~\ref{thm:bound_Eisenstein}, we see that there exists $N'>0$ such that both sides of \eqref{eq:Langlands_extend} define separately continuous linear forms on $\cT_{-N'}([G])^J$. Moreover, by Theorem~\ref{thm:Langlands_spectral}, it holds if $F_1$ and $F_2$ are pseudo Eisenstein series. By Lemma~\ref{lem:approx_infini} and Lemma~\ref{lem:pseudo_dense}, the closure in $\cT_{-N'}([G])$ of the vector space spanned by these functions contains $\cT_{-N'-1}([G])$. We conclude that $N=N'+1$ works.
\end{proof}

\begin{cor}
    \label{cor:Langlands_spectral_extended}
    There exists $N>0$ such that for all level $J$ of $G$ and all $F \in \cT_{-N}([G])^J$ we have
    \begin{equation*}
        F=\sum_{P_0 \subset P} \Val{\cP(M_P)}^{-1} \sum_{\pi \in \Pi_\disc(M_P)} \int_{i \fa_P^*} \sum_{\varphi \in \cB_{P,\pi}(J)} \langle F,E(\varphi,\lambda) \rangle_G E(\varphi,\lambda) d\lambda.
    \end{equation*}
\end{cor}

\section{Ichino--Yamana--Zydor regularized periods}
\label{chap:IYZ_periods}
For the rest of the paper, $G$ is $\GL_n \times \GL_{n+1}$. We embed $\GL_n$ in $\GL_{n+1}$ by
\begin{equation}
\label{eq:n_embedding}
    g \mapsto  \begin{pmatrix} g & \\ & 1 \end{pmatrix}.
\end{equation}
Let $H \simeq \GL_n$ be the diagonal subgroup in $G$. The goal is this section is to introduce the regularized period $\cP$ from \cite{Zydor} of the period integral along $[H]$, and to study its analytic properties. 

For the rest of this text, we will use the following conventions. Let $P_0 \subset G$ be the product of the subgroups of upper triangular matrices in $\GL_n$ and $\GL_{n+1}$, and $T_0 \subset G$ to be the product of the standard diagonal tori. If $J$ is a subgroup of $G$, we write $J=J_n \times J_{n+1}$. Similarly we have $\fa_0=\fa_{0,n} \oplus \fa_{0,n+1}$. 

All the constructions done relatively to $H$ will be decorated by a subscript ${}_H$. In particular, a subgroup of $H$ will typically be denoted $J_H$, and if $J$ is a subgroup of $G$ we set $J_H=J \cap H$. The pair $(T_{0,H},P_{0,H})$ is standard in $H$. Set $\fa_{0,H}=\fa_{T_{0,H}}$. It embeds into $\fa_0$. We will often identify the group $\GL_n$ either with $H$, either with its two copies in $G$ from the left and right coordinates, using the embedding \eqref{eq:n_embedding} for the latter.

\subsection{Rankin--Selberg parabolic subgroups}
\label{sec:RS_subgroup}

We start by defining a subset $\cF_{\RS}$ of the set of semi-standard parabolic subgroups of $G$ which appears in the definition of the regularized period $\cP$.

\subsubsection{Definition of the set of Rankin--Selberg parabolic subgroups}
\label{subsubsec:RS_defi}
We define $\cF_{\RS}$ the set of \emph{Rankin--Selberg parabolic subgroups} of $G$. This is the set of semi-standard parabolic subgroups $P=P_n \times P_{n+1}$ of $G$ such that $P_n$ is standard and, with respect to the embedding \eqref{eq:n_embedding}, $P_n=P_{n+1} \cap \GL_n$. In particular, if $P \in \cF_{\RS}$ then $P_H$ is a standard parabolic subgroup of $H$ isomorphic to $P_n$. Conversely, if $P_H$ is a standard parabolic subgroup of $H$, set
\begin{equation*}
    \cP_{\RS}(P_H)=\{ Q \in \cF_{\RS} \; | \; Q_H=P_H \}.
\end{equation*}
Then we have
\begin{equation}
\label{eq:F_decompo}
    \cF_{\RS}= \bigsqcup_{P_H \subset H \; \mathrm{standard}} \cP_{\RS}(P_H).
\end{equation}

Recall that in Section~\ref{chap:poles} we have associated to any standard parabolic subgroup $P$ of $\GL_n$ or $\GL_{n+1}$ a tuple $\underline{n}(P)$. Let $P_H$ be a standard parabolic subgroup of $H$. Write $\underline{n}(P_H)=(n_1,\hdots,n_m)$ with $\sum n_i=n$. Let $\cP_n^{n+1}(P_H)$ be the set of couples $(P^{\std}_{n+1},i_0)$ where $P^{\std}_{n+1}$ is a standard parabolic subgroup of $\GL_{n+1}$ and $i_0$ is an integer such that one of the two following alternatives is satisfied:
\begin{enumerate}
    \item $\underline{n}(P^{\std}_{n+1})=(n_1, \hdots, n_{i_0-1}, n_{i_0}+1,n_{i_0+1}, \hdots ,n_m)$ and $1 \leq i_0 \leq m$;
    \item $\underline{n}(P^{\std}_{n+1})=(n_1, \hdots, n_{i_0-1}, 1,n_{i_0}, \hdots ,n_m)$ and $1 \leq i_0 \leq m+1$.
\end{enumerate}
In the first case, set $N=\sum_{i=1}^{i_0} n_i$, and in the second set $N=\sum_{i=1}^{i_0-1} n_i$. Let $w(P^{\std}_{n+1},i_0) \in \fS_{n+1}$ be the cycle 
\begin{equation}
\label{eq:w_defi}
    w(P^{\std}_{n+1},i_0)=( N+2 \quad \hdots \quad n \quad n+1 \quad N+1).
\end{equation}
We identify it with an element in $W_{n+1}$ the Weyl group of $\GL_{n+1}$. The following is \cite[Corollary~4.2]{BoiZ}.

\begin{prop}
\label{prop:param_RS}
    We have a bijection
    \begin{equation*}
        \cF_{\RS}\simeq \left\{ (P_{n+1}^{\std},i_0) \; \middle| \; \begin{array}{l}
            P^{\std}_{n+1} \subset \GL_{n+1} \; \text{standard}, \\ M_{P^{\std}_{n+1}}=\prod_{i=1}^m \GL_{n_i}, \\
             1 \leq i_0 \leq m.
        \end{array} \right\}=\bigsqcup_{P_H\subset H \; \mathrm{standard}} \cP_n^{n+1}(P_H).
    \end{equation*}
\end{prop}

We set $\fz_P=\fa_P \cap \fa_{0,H}$. If $P \subset Q \in \cF_{\RS}$, let $\fz_P^Q$ be the orthogonal of $\fz_Q$ in $\fz_P$. Note that this is consistent as $\fz_G=\{0\}$. If $T \in \fz_{P}$, let $T_Q$ be its projection on $\fz_Q$ and $T_P^Q$ be its projection on $\fz_{P}^Q$, both with respect to $\fz_P=\fz_{P}^Q \oplus \fz_Q$. This applies in particular to $T \in \fa_{0,H}=\fz_{P_0}$ and we simply write $T^Q$ for $T_{P_0}^Q$. We have a notion of "sufficiently positive" for elements in $\fa_{0,H}$ (see \cite[Section~4.3]{BoiZ}).

\subsubsection{Standardization}

Let $P \in \cF_{\RS}$. Let $(P_{n+1}^{\std},i_P)$ be the inverse image of $P$ under the isomorphism of Proposition~\ref{prop:param_RS}. Let $w_P^\std=w(P^{\std}_{n+1},i_P)$ be the element defined in \eqref{eq:w_defi}. We have
\begin{equation}
\label{eq:standardized_P}
    P_{n+1}=w_P^\std.P^{\std}_{n+1}.
\end{equation}

Set $P^{\std}=P_n \times P^{\std}_{n+1}$, which is a standard parabolic subgroup of $G$. This is the \emph{standardization} of $P$. Write $(n_i)=\underline{n}(P^{\std}_{n+1})$. We may decompose the standard Levi factor $M_P^{\std}$ of $P^\std$ as
\begin{equation}
\label{eq:standard_decompo}
    M_{P}^{\std}=M_{P,n}^{\std}\times M_{P,n+1}^{\std}=\left( \bfM_{P,+} \times \cM_{P,n} \times \bfM_{P,-} \right) \times \left( \bfM_{P,+} \times \cM^{\std}_{P,n+1} \times \bfM^{\std}_{P,-} \right),
\end{equation}
where
\begin{equation}
\label{eq:ss_decompo}
   \bfM_{P,+} =   
            \prod_{i < i_P} \GL_{n_i}, \quad \bfM_{P,-}\simeq \bfM^{\std}_{P,-} =   
            \prod_{i > i_P} \GL_{n_i}, \quad \cM_{P,n} =  
            \GL_{n_{i_P}-1}, \quad \cM^{\std}_{P,n+1} =   
            \GL_{n_{i_P}}.
\end{equation}
We add a ${}^{\std}$ on $\bfM^{\std}_{P,-}$ to emphasize that, although they are isomorphic, the groups $\bfM_{P,-}$ and $\bfM^{\std}_{P,-}$ are not identified by the embedding \eqref{eq:n_embedding}. We also set
\begin{equation*}
    \bfM_{P,+}^{\std,2}=\bfM_{P,+}^2, \quad \bfM_{P,-}^{\std,2}=\bfM_{P,-} \times \bfM_{P,-}^\std, \quad 
    \bfM_{P}^{\std,2}=\bfM_{P,+}^{\std,2} \times \bfM_{P,-}^{\std,2}, \quad \cM_P^\std = \cM_{P,n} \times \cM_{P,n+1}^\std,
\end{equation*}
which all naturally embed into $G$. By composing with $w_P^\std$, we get a decomposition
\begin{equation}
\label{eq:M_P_levi}
     M_{P} \simeq \left( \bfM_{P,+} \times \cM_{P,n} \times \bfM_{P,-} \right) \times \left( \bfM_{P,+} \times \cM_{P,n+1} \times \bfM_{P,-} \right).
\end{equation}
The two copies of $\bfM_{P,-}$ in \eqref{eq:M_P_levi} are now identified by the embedding \eqref{eq:n_embedding}. The groups $\cM^{\std}_{P,n+1}$ and $\cM_{P,n+1}$ are isomorphic but in general embedded in two different ways in $\GL_{n+1}$. Set
\begin{equation*}
    \bfM_P=\bfM_{P,+} \times \bfM_{P,-}, \quad \bfM_P^2=\bfM_P \times \bfM_P, \quad \cM_P=\cM_{P,n} \times \cM_{P,n+1}.
\end{equation*}
The restriction of the diagonal embedding $H \subset G$ gives $\cM_{P,H} \subset \cM_P$ and we have $\cM_{P,H} \simeq \cM_{P,n}$. Note that $M_{P_H}=M_{P,H}=\bfM_P \times \cM_{P,H}$. The group $\cM_P$ is isomorphic to $\GL_{n_{i_P}-1} \times \GL_{n_{i_P}}$, and we can define embedding of $\GL_{n_{i_P}-1}$ in $\cM_P$ as in \eqref{eq:n_embedding}. This is compatible with the inclusion $H \subset G$ in the sense that $\cM_{P,H}=\cM_P \cap \GL_{n_{i_P}-1}$. 

Set
\begin{equation*}
    \underline{\rho}_P=(\rho_P-2 \rho_{P_H})_{| \fz_{P}} \in \fz_P^*.
\end{equation*}
In coordinates, we have $(\underline{\rho}_P)_i=\frac{1}{2}$ if $i<i_P$, and $(\underline{\rho}_P)_i=-\frac{1}{2}$ if $i>i_P$.

\subsection{Regularized Rankin--Selberg periods à la Zydor}
\label{sec:reg_RS_periods_Z}

In this section, we fix $Q \in \cF_{\RS}$. We write $M^\std_{Q,n+1}=\prod_{i=1}^m \GL_{n_i}$. We recall the definition of the regularized period $\cP^Q$ from \cite{Zydor}.

\subsubsection{Iwasawa decomposition and measures} \label{subsubsec:global_measure}
Set
\begin{equation*}
    M_{Q_H}(\bA)^{Q,1}=\left\{ m \in M_{Q_H}(\bA) \; \middle| \; \left(H_{Q_H}(m)\right)_Q=0 \right\}, \quad Z_Q^\infty=A_Q^\infty \cap A_{Q_H}^\infty.
\end{equation*}
The restriction of $H_{Q_H}$ to $Z_Q^\infty$ is an isomorphism with image $\fz_Q$. This gives $Z_Q^\infty$ a Haar measure. We have a direct product decomposition of commuting groups
\begin{equation}
\label{eq:M_P_H_decompo}
    M_{Q_H}(\bA)=Z_Q^{\infty} M_{Q_H}(\bA)^{Q,1}.
\end{equation}
By the Iwasawa decomposition, there is a unique Haar measure on $M_{Q_H}(\bA)^{Q,1}$ such that for all $f \in C_c(H(\bA))$ we have
\begin{equation}
\label{eq:facto_measure}
    \int_{H(\bA)} f(h)dh=\int_{K_H} \int_{N_{Q_H}(\bA)} \int_{M_{Q_H}(\bA)^{Q,1}} \int_{Z_Q^\infty} \exp(\langle -2 \rho_{Q_H},H_{Q_H}(am) \rangle) f(nmak) da dm dn dk.
\end{equation}
We set
\begin{equation*}
    [M_{Q_H}]^{Q,1}=M_{Q_H}(F) \backslash M_{Q_H}(\bA)^{Q,1}
\end{equation*}
which is given the quotient by the counting measure. Note that we have
\begin{equation*}
    [M_{Q,H}]^{Q,1} \simeq \left(\prod_{i \neq i_Q} [\GL_{n_i}]^1\right) \times [\GL_{n_{i_Q}-1}]=[\bfM_{Q}]^1 \times [\cM_{Q,H}].
\end{equation*}
Then the measure $dh$ coincides the product of the ones on each $[\GL_{n_i}]^1$ and on $[\GL_{n_{i_Q}-1}]$ described in \S\ref{subsubsec:first_measure}. Moreover, if $Q=G$ we have $M_{G,H}(\bA)^{G,1}=H(\bA)$ and $Z_G^\infty=\{1\}$. 

\subsubsection{Truncated periods} Let $T \in \fa_{0,H}$. In \cite{Zydor}, Zydor introduces a truncated operator $\Lambda^{T,Q}$ defined on the space of locally integrable function on $Q(F) \backslash G(\bA)$. If $Q=G$ we simply write $\Lambda^T$. Its main property is the following.

\begin{theorem}
\label{thm:truncation_operator}
    Let $T \in \fa_{0,H}$ be sufficiently positive. Let $J$ be a level of $G$. For any $N,N' >0 $ there exists a finite family $(X_i)_{i \in I}$ of elements in $\cU(\fg_\infty)$ such that for any smooth and right $J$-invariant function $\phi$ on $[G]_Q$, the function $\Lambda^{T,Q} \phi$ is a function on $[H]_{Q_H}$ and we have
    \begin{equation*}
        \sup_{m \in M_{Q,H}(\bA)^{Q,1}} \norm{m}^N_{M_{Q,H}} \Val{\Lambda^{T,Q} \phi(mk)} \leq \sum_{i \in I} \norm{\phi}_{-N',X_i}.
    \end{equation*}
\end{theorem}

\begin{proof}
    This is \cite[Theorem~3.9]{Zydor}. Note that the statement in ibid. is weaker, but our version can easily be extracted from the proof.
\end{proof}

Let $\phi \in \cT([G])$. We have $\phi_Q \in \cT([G]_Q)$. For $T$ sufficiently positive, we define the truncated period of $\phi$ relative to $Q$ to be
\begin{equation*}
    \cP^{T,Q} (\phi)=\int_{K_H} \int_{[M_{Q,H}]^{Q,1}} \exp(-\langle 2 \rho_{P_H}, H_{P_H}(m) \rangle) \Lambda^{T,Q} \phi(mk) dmdk.
\end{equation*}
This is absolutely convergent by Lemma~\ref{lem:height_properties} and Theorem~\ref{thm:truncation_operator}. Note that if $Q=G$ it reduces to  
\begin{equation*}
    \cP^{T} (\phi)= \int_{[H]} \Lambda^{T} \phi(h) dh.
\end{equation*}

\subsubsection{Regularized periods}
In \cite[Section~4.5]{Zydor}, a subspace of \emph{regular} automorphic forms $\cA_Q(G)^{\mathrm{reg}} \subset \cA_Q(G)$ is defined. More precisely, a $\varphi \in \cA_Q(G)$ belongs to $\cA_Q(G)^{\mathrm{reg}}$ if for every Rankin--Selberg parabolic subgroup $P \subset Q$ and $h \in H(\bA)$, the exponents of the $Z_P^\infty$ finite function $z \mapsto \varphi(zh)$ belong to a dense open subset of $\fa_{P,\cc}^*$. The space $\cA_Q(G)^{\mathrm{reg}}$ is stable under right-translations by $G(\bA)$. We then have the following theorem.

\begin{theorem}
\label{thm:reg}
    For $\varphi \in \cA_Q(G)^{\mathrm{reg}}$ there exists a unique exponential polynomial on $\fa_{0,H}$ that coincides with $T \mapsto \cP^{T,Q}(\varphi)$ for $T$ sufficiently positive and whose purely polynomial part is constant. We denote it by $\cP^Q(\varphi)$.
\end{theorem}

The number $\cP^Q(\varphi)$ is the \emph{regularized period à la Zydor} along $Z_Q^\infty M_{Q_H}(F) N_{Q_H}(\bA) \backslash H(\bA)$.

\subsubsection{Parabolic descent} The period $\cP^Q$ can be described by parabolic descent. Note that $M_{Q,H}$ identifies as $\bfM_Q \times \cM_{Q,H} \subset \bfM_Q^2 \times \cM_Q$, where the map $\bfM_Q \subset \bfM_Q^2$ is diagonal and $\cM_{Q,H} \subset \cM_Q$ is the Rankin--Selberg embedding. Using \cite{Zydor}, one can associate to these two pairs of groups and subgroups two regularized periods $\cP^{\bfM_Q^2}$ and $\cP^{\cM_Q}$ on some spaces of automorphic forms $\cA(\bfM_Q^2)^{\mathrm{reg}}$ and $\cA(\cM_Q)^{\mathrm{reg}}$ respectively. The following lemma is then proved in \cite[Lemma~4.12]{BoiZ}.

\begin{lem}
\label{lem:parabolic_descent}
    Let $\varphi \in \cA_Q(G)^{\mathrm{reg}}$. Set
    \begin{equation*}
        \varphi_{M_Q}=\left(R(e_{K_H})\varphi\right)_{|M_Q(\bA),-\rho_Q}.
    \end{equation*}
    Then we have
    \begin{equation}
    \label{eq:both_orders}
        \cP^Q(\varphi)=\cP^{\bfM_Q^2} \left(m \in [\bfM^2_Q] \mapsto \cP^{\cM_Q}(R(m)\varphi_{M_Q}) \right)=\cP^{\cM_Q} \left(m \in [\cM_Q] \mapsto \cP^{\bfM_Q^2}(R(m)\varphi_{M_Q}) \right),
    \end{equation}
    where all the automorphic forms belong to the relevant $\cA^{\mathrm{reg}}$ spaces.
\end{lem}

With the notation of Lemma~\ref{lem:parabolic_descent}, we write
\begin{equation}
\label{eq:tensor_product}
    \cP^Q(\varphi)=(\cP^{\bfM_Q^2} \otimes \cP^{\cM_Q}) (\varphi_{M_Q}),
\end{equation}
where the tensor product notation means that the regularized periods can be taken in any order as in \eqref{eq:both_orders}. If $\bfM_Q^2= \prod_{i=1}^m \GL_{n_i}^2$ and if $\phi \in \cA(\bfM_Q^2)^{\mathrm{reg}}$, we may further decompose
\begin{equation}
\label{eq:gln_product}
    \cP^{\bfM_Q^2}(\phi)=\bigotimes_{i=1}^m \cP^{\GL_{n_i}^2}(\phi).
\end{equation}
We call $\cP^{\GL_{n_i}^2}$ the \emph{regularized diagonal Arthur period on $\GL_{n_i}$}.

\subsubsection{Regularized periods of Eisenstein series} 
Let $P$ be a standard parabolic subgroup of $G$, let $\pi \in \Pi_\disc(M_P)$. Let $\varphi \in \cA_{P,\pi}(G)$. Let $w \in {}_{Q^\std} W_P$. For $\lambda \in \fa_{P_w,\cc}^*$ in general position, set
\begin{equation}
    \label{eq:truncated_period_defi}
    \cP^{T,Q}(\varphi,\lambda,w)=\cP^{T,Q}(E^{Q^\std}(w_Q^{\std,-1} \cdot,M(w,\lambda)\varphi_{P_w},w \lambda)),
\end{equation}

\begin{rem}
    The element $w_Q^\std$ was defined in \eqref{eq:standardized_P}. We prefer to make it appear in \eqref{eq:truncated_period_defi} to deal with the standard parabolic subgroup $Q^\std$ rather than $Q$.
\end{rem}

The truncated period is well defined for $\lambda$ in general position by \cite[Theorem~2.2]{Lap} (which ensures that generalized Eisenstein series are of uniform moderate growth), and if $P_\pi \not\subset P_w$ it is zero. In \cite[Lemma~4.15]{BoiZ}, we show that the generalized Eisenstein series $E^{Q^\std}(w_Q^{\std,-1} \cdot,M(w,\lambda)\varphi_{P_w},w \lambda)$ belongs to $\cA_Q(G)^{\mathrm{reg}}$ for $\lambda \in \fa_{P,\cc}^*$ in general position. We can therefore define 
    \begin{equation}
    \label{eq:PQE}
    \cP^Q(\varphi,\lambda,w)=\cP^Q(E^{Q^\std}(w_Q^{\std,-1} \cdot,M(w,\lambda)\varphi_{P_w},w \lambda)).
\end{equation}
The following property follows directly from the definition of regularized periods in \cite{Zydor}.

\begin{prop}
\label{prop:unfold_periods}
    For $\lambda \in \fa_{P,\cc}^*$ in general position and $T$ sufficiently positive we have
    \begin{equation}
    \label{eq:P_unfolding}
        \cP^Q(\varphi,\lambda,w)=\sum_{\substack{ R \in \cF_{\RS} \\ R \subset Q}} \varepsilon_{R}^Q \sum_{w' \in {}_{R^\std} W_{Q^\std_w}^{Q^\std}w} \cP^{T,R}(\varphi,\lambda,w') \cdot \frac{\exp(\langle w_R^\std w'(\lambda+\nu_{w'})+\underline{\rho}_R,T_R^Q \rangle)}{\hat{\theta}_R^Q(w_R^\std w' (\lambda + \nu_{w'})+\underline{\rho}_R)},
    \end{equation}
    where $\varepsilon_{R}^Q$ is a sign, $\hat{\theta}_R^Q$ is the product of affine linear forms defined in \cite[Section~4.2]{BoiZ}, and we understand that the summands are zero unless $P_\pi \subset P_{w'}$ (whether or not the denominator is identically zero). A similar expression holds for $\cP^{\bfM_Q^2}$ and $\cP^{\cM_Q}$. 
\end{prop}

\subsection{Regularized periods as Zeta integrals and inner-products}

In this section, we recall the results of \cite{BoiZ} which compute the regularized periods $\cP^Q$ in terms of global Zeta integrals and Petersson inner products. They deeply rely on \cite{IY}, and will have important analytic consequences in \S\ref{sec:analytic_properties}.

\subsubsection{Rankin--Selberg periods and Zeta functions}
\label{subsec:RS_and_Zeta}

Let $\psi$ be a generic character of $[N_0]$ which is trivial on $N_{0,H}(\bA)$. For any automorphic form $\Phi \in \cA(G)$, we may consider the global Whittaker function
\begin{equation}
\label{eq:global_whitt}
    W^{\psi}(g,\Phi)=\int_{[N_0]} \Phi(ng) \overline{\psi}(n) dn, \quad g \in [G].
\end{equation}
If $P$ be a standard parabolic subgroup of $G$ and $\pi \in \Pi_\disc(M_P)$, for $\lambda \in \fa_{P,\cc}^*$ in general position we can consider the global Zeta integral
\begin{equation}
    \label{eq:global_zeta}
    Z_\pi(\varphi,\lambda):=\int_{N_{0,H}(\bA) \backslash H(\bA)} W^\psi(h,E(\varphi,\lambda)) dh, \quad \varphi \in \cA_{P,\pi}(G).
\end{equation}
By \cite[Lemma~7.1.1.1]{BPCZ}, this integral is absolutely convergent for $\Re(\lambda)$ in some open subset of $\fa_{P}^*$, and extends to a meromorphic function in $\lambda$ by \cite[Corollary~5.4]{IY} and \cite[Equation~(4.2)]{IY}. Note that because residual representations are not generic, it is zero as soon as $\pi$ is not cuspidal.

We now assume that $\pi$ is cuspidal. Write $M_P=\prod_{i=1}^{m_n} \GL_{n_{n,i}} \times \prod_{i=1}^{m_{n+1}} \GL_{n_{n+1,i}}$ and $\pi=\boxtimes_{i=1}^{m_n} \pi_{n,i} \boxtimes_{i=1}^{m_{n+1}} \pi_{n+1,i}$ accordingly. Consider the products of completed Rankin--Selberg $L$ functions
\begin{equation}
    b(\lambda,\pi)=\prod_{i<j}L(1+\lambda_{n,i}-\lambda_{n,j},\pi_{n,i} \times \pi_{n,j}^\vee)\prod_{i<j}L(1+\lambda_{n+1,i}-\lambda_{n+1,j},\pi_{n+1,i} \times \pi_{n+1,j}^\vee).
\end{equation}
and
\begin{equation}
    L\left(\lambda+\frac{1}{2},\pi_{n} \times \pi_{n+1}\right):=\prod_{i=1}^{m_n} \prod_{j=1}^{m_{n+1}} L(1/2+\lambda_{n,i}+\lambda_{n+1,j},\pi_{n,i} \times \pi_{n+1,j}).
\end{equation}
Then it is shown in \cite[Corollary~5.7]{IY} that if $\varphi=\otimes_v \varphi_v \in \cA_{P,\pi}(G)$ is factorizable, there exists a finite set $\tS$ of places of $F$ such that 
\begin{equation}
    \label{eq:Zeta_facto_global}
     Z_\pi(\varphi,\lambda)=\frac{L(\lambda+\frac{1}{2},\pi_n \times \pi_{n+1})}{b(\lambda,\pi)} \times \prod_{v \in \tS} Z_{\pi_v}^\natural(\varphi_v,\lambda),
\end{equation}
where the $Z_{\pi_v}^\natural$ are local Zeta integrals normalized by the local version of the quotient of $L$-functions. Moreover, we know by \cite[Theorem~1.1]{IY} and \cite[Proposition~4.20]{BoiZ} that $\cP$ computes the global Zeta integral.

\begin{prop}
    \label{prop:P=Z}
    For $\varphi \in \cA_{P,\pi}(G)$ and $\lambda \in \fa_{P,\cc}^*$ in general position, we have $\cP(\varphi,\lambda)=Z_\pi(\varphi,\lambda)$.
\end{prop}

Set
 \begin{equation}
    \label{eq:f_pi_Z_defi}
        L_{\pi,Z}(\lambda)=\prod_{\substack{i,j \\ \pi_{n,i} \simeq \pi_{n+1,j}^\vee}} \left(\lambda_{n,i}+\lambda_{n+1,j}\pm\frac{1}{2}\right)\prod_{\substack{i<j \\ \pi_{n,i} \simeq \pi_{n,j}}}(\lambda_{n,i}-\lambda_{n,j})^{-1}\prod_{\substack{i<j \\ \pi_{n+1,i} \simeq \pi_{n+1,j}}}(\lambda_{n+1,i}-\lambda_{n+1,j})^{-1}.
    \end{equation}

\begin{cor}
\label{cor:regularity_P}
    There exists $k>0$ such that for every level $J$ of $G$ there exists $c_J>0$ such that for every $J$-pair $(P,\pi)$ with $\pi\in \Pi_\cusp(M_P)$ and every $\varphi \in \cA_{P,\pi}(G)$ the map
    \begin{equation*}
        \lambda \in \fa_{P,\cc}^* \mapsto L_{\pi,Z}(\lambda)\cP(\varphi,\lambda)
    \end{equation*}
    is regular on $\cR_{\pi,k,c_J}$.
\end{cor}

\begin{proof}
    By \cite[Lemma~4.18]{BoiZ}, for every place $v$ the local factor $\lambda \in \fa_{P,\cc}^* \mapsto Z_{\pi_v}^\natural(\varphi_v,\lambda)$ is a meromorphic function which is regular for $\Re(\lambda) \in \overline{\fa_{P}^{*,+}}$. 
    By Theorem~\ref{thm:L} the factorization of the Zeta function from \eqref{eq:Zeta_facto_global}, we know that $\lambda \mapsto \cP(\varphi,\lambda)$ has simple zeros along $\lambda_{n,i}-\lambda_{n,j}=0$ and $\lambda_{n+1,i}-\lambda_{n+1,j}=0$ if $\pi_{n,i}\simeq \pi_{n,j}$ and $\pi_{n+1,i}\simeq \pi_{n+1,j}$ respectively. Indeed, note that these zeros can neither be compensated by poles of $L(\lambda+\frac{1}{2},\pi_n \times \pi_{n+1})$ nor of $\prod_{v \in \tS} Z_{\pi_v}^\natural$. Set
    \begin{equation*}
        L_{\pi,Z}^{\mathrm{num}}(\lambda)=\prod_{\substack{i,j \\ \pi_{n,i} \simeq \pi_{n+1,j}^\vee}} \left(\lambda_{n,i}+\lambda_{n+1,j}\pm\frac{1}{2}\right).
    \end{equation*}
    Then it is enough to show that $\lambda \mapsto L_{\pi,Z}^{\mathrm{num}}(\lambda) \cP(\varphi,\lambda)$ is regular on $\cR_{\pi,k,c_J}$. We know that it is regular at least on $\Re(\lambda) \in \overline{\fa_{P}^{*,+}}$.
    
    To obtain regularity on the bigger region $\cR_{\pi,k,c_J}$, we use the localization of the poles of the cuspidal Eisenstein series in Theorem~\ref{thm:analytic_Eisenstein} and of the generalized cuspidal Eisenstein series in Proposition~\ref{prop:partial_regular} (see also Remark~\ref{rem:cuspi_case}). By Proposition~\ref{prop:unfold_periods}, we conclude that there exist $k$ and $c_J$ as in the statement of Corollary~\ref{cor:regularity_P} such that
    \begin{equation*}
        \lambda \mapsto \left(\prod_{Q \in \cF_{\RS}} \prod_{w \in W(P;Q^\std)} \hat{\theta}_Q(w_Q^\std w\lambda + \underline{\rho}_Q)\right) L_{\pi,E}(\lambda) \cP(\varphi,\lambda)
    \end{equation*}
    is regular on $\cR_{\pi,k,c_J}$, where $L_{\pi,E}$ is defined in \eqref{eq:f_pi_defi}. Because the intersection of any hyperplane cut out by $(\prod \hat{\theta}_Q(w_Q^\std w\lambda + \underline{\rho}_Q)) L_{\pi,E}(\lambda)/L_{\pi,Z}^{\mathrm{num}}(\lambda)$ with the region $\Re(\lambda) \in \overline{\fa_{P}^{*,+}}$ is of codimension $1$, we conclude that $L_{\pi,Z}^{\mathrm{num}}(\lambda) \cP(\varphi,\lambda)$ must be regular on $\cR_{\pi,k,c_J}$.
\end{proof}

\subsubsection{The diagonal Arthur period}
We now investigate the diagonal Arthur period $\cP^{\GL_k^2}$ for $k \geq 1$ of \eqref{eq:gln_product}. This is the regularized period of \cite{Zydor} associated to the diagonal subgroup $\GL_k \subset \GL_k^2$. It is well defined for all automorphic forms in some subspace $\cA(\GL_k^2)^\mathrm{reg}$.

\begin{prop}
\label{prop:diagonal_Arthur}
    Let $P$ be a standard parabolic subgroup of $\GL_k^2$. Let $\pi \in \Pi_\disc(M_P^2)$. Let $\varphi=\varphi_1 \otimes \varphi_2 \in \cA_{P,\pi}(\GL_k^2)$. We have the following alternative.
    \begin{enumerate}
        \item If $P=\GL_k^2$, then $\varphi \in \cA(\GL_k^2)^{\mathrm{reg}}$ and
        \begin{equation*}
            \cP^{\GL_k^2}(\varphi)=\langle \varphi_1,\overline{\varphi_2} \rangle_{\GL_k,\Pet}. 
        \end{equation*}
        \item If $P \neq \GL_k^2$, then for $\lambda \in \fa_{P,\cc}^*$ in general position we have $E(\varphi,\lambda) \in \cA(\GL_k^2)^{\mathrm{reg}}$ and
        \begin{equation*}
            \cP^{\GL_k^2}(E(\varphi,\lambda))=0.
        \end{equation*}
    \end{enumerate}
\end{prop}

\begin{proof}
    The first point follows by \cite[Theorem~4.6]{Zydor}. For the second, let $P \neq \GL_k^2$. It is shown in \cite[Lemma~3.1.5.1]{Ch} that for $\lambda \in \fa_{P,\cc}^*$ in general position we have $E(\varphi,\lambda) \in \cA(\GL_k^2)^{\mathrm{reg}}$ and the map $\lambda \mapsto \cP^{\GL_k^2}(E(\varphi,\lambda))$ is meromorphic. Write $P=P_1 \times P_2$, $\pi=\pi_1 \boxtimes \pi_2$ and $\lambda=\lambda_1+\lambda_2$. By \cite[Theorem~4.1]{Zydor}, the regularized period yields for $\lambda$ in general position a $\GL_k(\bA)^1$-invariant pairing 
    \begin{equation*}
         (\varphi_1,\varphi_2) \in \cA_{P_1,\pi_1,\lambda_1}(\GL_k) \times \cA_{P_2,\pi_2,\lambda_2}(\GL_k)  \mapsto \cP^{\GL_k^2}(E(\varphi,\lambda)) \in \cc.
    \end{equation*}
    Such a pairing must be zero for $\lambda$ in general position (globally and locally) by Bernstein's principle (\cite[208]{JLR}). It follows that $\cP^{\GL_k^2}(E(\varphi,\lambda))$ is identically zero.
\end{proof}

\subsection{Analytic properties of regularized periods}
\label{sec:analytic_properties}

In this section, we fix $Q \in \cF_{RS}$. We take $P$ a standard parabolic subgroup of $G$ and $\pi \in \Pi_{\disc}(M_P)$. By combining Lemma~\ref{lem:parabolic_descent}, Proposition~\ref{prop:P=Z} and Proposition~\ref{prop:diagonal_Arthur}, we obtain a complete description of the periods $\cP^Q(\varphi,\lambda,w)$ and of their analytic properties on some open set of $\fa_{P,\cc}^{*}$.  

\subsubsection{Poles of regularized periods}

We first assume that $P \subset Q^\std$. We have a decomposition $\pi_\bfP \boxtimes \pi_\cP$ corresponding to the decomposition of the Levi factor $M_P$. We further decompose $\pi_{\bfP}=\pi_{\bfP,n} \boxtimes \pi_{\bfP,n+1}$. For every $\lambda \in \fa_{P,\cc}^*$, denote by $\lambda_\cP$ the restriction of $\lambda$ to $\fa_{\cP,\cc}$. Moreover, recall that $\bfM_{Q,+}^{\std,2}$ is a standard Levi subgroup of some $\GL_{n_+}^2$ (embedded in the "upper-left corner" of $G$) and that $\bfM_{Q,-}^{\std,2}$ is a standard Levi subgroup of some $\GL_{n_-}^2$ (embedded in the "lower-right" corner of $G$). Set $\bfP=(\GL_{n_+} \times \GL_{n_-})^2 \cap P$. This is a standard parabolic subgroup of $(\GL_{n_+} \times \GL_{n_-})^2$ with Levi factor $\bfM_P$. Moreover, $\bfM_{Q}^{\std,2}$ is a Levi subgroup of $(\GL_{n_+} \times \GL_{n_-})^2$ that contains $\bfM_P$. Let $R$ be the standard parabolic subgroup of $G$ with Levi factor $(\GL_{n_+} \times \cM_{Q,n}^\std \times \GL_{n_-}) \times (\GL_{n_+} \times \cM_{Q,n+1}^\std \times \GL_{n_-})$. 

Let $\varphi \in \cA_{P,\pi}(G)$. Set 
\begin{equation*}
    \varphi_{M_{Q}^\std}=\left(R(w_Q^{\std,-1})R(e_{K_H})\varphi\right)_{|M_{Q}^\std(\bA),-\rho_{Q}^\std}, \quad \varphi_{M_R}=\left(R(w_Q^{\std,-1})R(e_{K_H})\varphi\right)_{|M_{R}(\bA),-\rho_{R}}.
\end{equation*}
Then $ \varphi_{M_{Q}^\std} \in \cA_{\bfP,\pi_\bfP}(\bfM_{Q}^{\std,2}) \otimes \cA_{\cP,\pi_\cP}(\cM_{Q}^\std)$ and $ \varphi_{M_{R}} \in \cA_{\bfP,\pi_\bfP}((\GL_{n_+} \times \GL_{n_-})^2) \otimes \cA_{\cP,\pi_\cP}(\cM_{Q}^\std)$. Because $\cM_Q^\std$ is isomorphic to $\GL_r \times \GL_{r+1}$ for some $r$, we can consider the Zeta function $Z_{\pi_\cP}^{\cM_{Q}^\std}(\cdot,\lambda_\cP)$ which is defined on $\cA_{\cP,\pi_\cP}(\cM_{Q}^\std)$. 

\begin{prop}
\label{prop:parabolic_descent}
    Let $\pi \in \Pi_\disc(M_P)$. We have the following alternative: 
    \begin{itemize}
        \item If $\bfM_P=\bfM_{Q}^{^\std,2}$, $\pi_{\bfP,n} \simeq \pi_{\bfP,n+1}^\vee$ and if $\pi_\cP$ is cuspidal, then for $\lambda \in \fa_{P,\cc}^*$ in general position and every $\varphi \in \cA_{P,\pi}(G)$ we have
        \begin{equation}
        \label{eq:cP^Q_explicit}
            \cP^Q(\varphi,\lambda)=\left( \langle \cdot,\overline{\cdot} \rangle_{\bfM_{P,H},\Pet} \otimes Z_{\pi_\cP}^{\cM_{Q}^\std}(\cdot,\lambda_\cP) \right) (\varphi_{M_{Q}^\std})=\left( \langle \cdot,\overline{\cdot} \rangle_{\bfP_H,\Pet} \otimes Z_{\pi_\cP}^{\cM_{Q}^\std}(\cdot,\lambda_\cP) \right) (\varphi_{M_R}),
        \end{equation}
        where $\bfP \simeq \bfP_H^2$ and the tensor product notation is used as in \eqref{eq:tensor_product}.
        \item Otherwise, for every $\varphi \in \cA_{P,\pi}(G)$ the map $\lambda \mapsto \cP^Q(\varphi,\lambda)$ is zero.
    \end{itemize}
\end{prop}

\begin{proof}
    This is a direct consequence of parabolic descent (Lemma~\ref{lem:parabolic_descent}) and the description of $\cP^{\cM_Q}$ in Proposition~\ref{prop:P=Z} and of $\cP^{\bfM_Q^2}$ in Proposition~\ref{prop:diagonal_Arthur} and \eqref{eq:gln_product}. The only thing that we have to prove is the final equality in \eqref{eq:cP^Q_explicit}. Note that $\bfK=w_Q^{\std,-1} K_H w_Q^\std \cap (\GL_{n_+}(\bA) \times \GL_{n_-}(\bA))$ is a good maximal compact subgroup of this group. We equip it with the probability Haar measure. Write $\varphi=\varphi_n \otimes \varphi_{n+1}$. Then by the Iwasawa decomposition we have
    \begin{equation*}
        \langle \varphi_{M_R,n} ,\overline{\varphi}_{M_R,n+1} \rangle_{\bfP_H,\Pet}=\int_{\bfK}\langle R(k)\varphi_{M_{Q}^\std,n},\overline{R(k)\varphi_{M_{Q}^\std,n+1}} \rangle_{\bfM_{P,H},\Pet}dk.
    \end{equation*}
    The second equality follows.
\end{proof}

We now lift the hypothesis that $P \subset Q^\std$ and take $w \in {}_{Q^\std} W_P$. We study the regularity of the period $\cP^Q(\varphi,\lambda,w)$ for $\varphi \in \cA_{P,\pi}(G)$. By Lemma~\ref{lem:contant_term}, there exist a discrete automorphic representation $\pi_w \in \Pi_{\disc}(M_{P_w})$ and an unramified character $\nu_w \in \fa_{P_w}^*$ such that for any $\varphi \in \cA_{P,\pi}(G)$ we have $\varphi_{P_w} \in \cA_{P_w,\pi_w,\nu_w}(G)$. We can apply the above notation to this representation, for example by writing $\lambda_{\cP}$ for the restriction of any $\lambda \in \fa_{Q^\std_w,\cc}^*$ to $\fa_{\cQ_w,\cc}$ where $\cQ_w=Q_w \cap \cM_{Q}^\std$.

Recall that we have defined in \S\ref{subsubsec:reg_scalar} a polynomial $L_{\pi,w}$ which controls the poles of $M(w,\lambda)$ in the region $\cR_{\pi,k,c_J}(w)$ (see \eqref{eq:R(w)_defi}). We have the following description of the poles of $\cP^Q(\varphi,\lambda,w)$.

\begin{prop}
\label{prop:reg_P}
     There exists $k$ such that for every level $J$ there exists $c_J>0$ such that for every $J$-pair $(P,\pi)$ and every $\varphi \in \cA_{P,\pi}(G)$ the map
     \begin{equation}
     \label{eq:normalized_P}
        \lambda \mapsto L_{(w\pi_w)_\cP,Z}\left((w(\lambda+\nu_w))_\cP\right)  L_{\pi,w}(\lambda)\cP^Q(\varphi,\lambda,w)
    \end{equation}
    is regular in the region 
    \begin{equation}
    \label{eq:region_holo_P}
        \left\{ \lambda \in \fa_{P,\cc}^* \; \middle| \; \lambda+\nu_w \in \cR_{\pi,k,c_J}(w), \; (w(\lambda+\nu_w))_\cP \in \cR_{(w\pi_w)_\cP,k,c_J} \right\}.
    \end{equation}
\end{prop}

\begin{proof}
    By Corollary~\ref{cor:wider_holo}, we know that $\lambda \mapsto L_{\pi,w}(\lambda)M(w,\lambda) \varphi_{P_w}$ is regular in the region \eqref{eq:region_holo_P}. By Theorem~\ref{thm:truncation_operator} and \cite[Theorem~2.2]{Lap}, we are reduced to the case $w=1$ and $P=P_w$.

     We now want to prove that $\lambda \mapsto L_{\pi_\cP,Z}(\lambda_\cP)\cP^Q(\varphi,\lambda)$ is regular in the region $\lambda_{\cP} \in \cR_{\pi_{\cP},k,c_J}$. We can assume that we are in the first case of Proposition~\ref{prop:parabolic_descent}. By parabolic descent (Lemma~\ref{lem:parabolic_descent}), we need to study the regularity of 
    \begin{equation}
    \label{eq:double_variable_map}
        (\lambda_\bfP,\lambda_\cP) \in \fa_{\bfP,\cc}^* \times \fa_{\cP,\cc}^* \mapsto \cP^{\bfM^2_Q}\left(m \in \bfM_Q^2(\bA) \mapsto L_{\pi,Z}(\lambda_\cP)\cP^{\cM_Q}(R(m)\varphi_{M_Q},\lambda_{\bfP}+\lambda_{\cP})\right).
    \end{equation}
    By Proposition~\ref{prop:diagonal_Arthur}, because $\bfM_P=\bfM_{Q}^{\std,2}$ the period $\cP^{\bfM_Q^2}$ is just a Petersson inner-product so that \eqref{eq:double_variable_map} is constant in $\lambda_{\bfP}$. It is therefore enough to show that it is holomorphic in the variable $\lambda_{\cP} \in \cR_{\pi_\cP,k,c_J}$ for $\lambda_{\bfP}$ fixed by Hartogs' theorem. By Corollary~\ref{cor:regularity_P}, for every $m$ the map $\lambda_{\cP} \mapsto L_{\pi,Z}(\lambda_\cP)\cP^{\cM_Q}(R(m)\varphi_{M_Q},\lambda_{\bfP}+\lambda_{\cP})$ is regular in the desired region. Moreover, by Corollary~\ref{cor:regularity_P} below (which for this purpose is independent from this argument), one can easily upgrade this to a regular map $\lambda_{\cP} \mapsto L_{\pi,Z}(\lambda_\cP)\cP^{\cM_Q}(R(\cdot)\varphi_{M_Q},\lambda_{\bfP}+\lambda_{\cP})$ valued in $\cT([\bfM_Q^2])$ (see \cite[Section~A.0.3]{BPCZ} for this notion). Because $\cP^{\bfM_Q^2}$ is a sum of truncated periods by Proposition~\ref{prop:unfold_periods}, we infer that \eqref{eq:double_variable_map} is also regular in the region $\lambda_\cP \in  \cR_{\pi_\cP,k,c_J}$ by Theorem~\ref{thm:truncation_operator}. This concludes. 
\end{proof}

\subsubsection{Bounds for regularized periods}

We now bound $\cP^Q(\varphi,\lambda,w)$. 

\begin{prop}
\label{prop:individual_bound_P}
    There exists $k>0$ such that for any level $J$ and $C>0$ there exist $c_J>0$, $d>0$, $N>0$ and $X_1, \hdots, X_r \in \cU(\fg_\infty)$ such that for all $J$-pair $(P,\pi)$ and $\varphi \in \cA_{P,\pi}(G)^J$ we have
    \begin{equation}
    \label{eq:optimal_bound_P}
       \Val{L_{(w\pi_w)_\cP,Z}\left((w(\lambda+\nu_w))_\cP\right)  L_{\pi,w}(\lambda)\cP^Q(\varphi,\lambda,w)} \leq (1+\norm{\lambda}^2)^d \sum_{i=1}^r \norm{\varphi}_{-N,X_i},
    \end{equation}
    in the region
    \begin{equation}
    \label{eq:big_intersection}
        \left\{ \lambda \in \fa_{P,\cc}^* \; \middle| \; \lambda+\nu_w \in \cR^C_{\pi,k,c_J}(w), \; (w(\lambda+\nu_w))_\cP \in \cR_{(w\pi_w)_\cP,k,c_J} \right\}.
    \end{equation}
   
\end{prop}

\begin{proof}
    Let $(P,\pi)$ be a $J$-pair. We can assume that the discrete automorphic representation $w \pi_w$ of $M_{Q^\std_w}$ satisfies the first condition of Proposition~\ref{prop:parabolic_descent}, as otherwise the period is zero. This implies that $Q^\std_w \cap \bfM_{Q}^{\std,2}=\bfM_{Q}^{\std,2}$.  
    
    Let $R \in \cF_{\RS}$ such that $R \subset Q$, and let $w' \in {}_{R^\std}W_{Q^\std_w}^{Q^\std}w$ with $P_\pi \subset P_{w'}$. The set ${}_{R^\std}W_{Q^\std_w}^{Q^\std}$ decomposes as ${}_{R^\std}W_{Q^\std_w}^{\bfM_{Q^\std}^{\std,2}} \times {}_{R^\std}W_{Q^\std_w}^{\cM_{Q}^\std}$, so that we can write $w'=(w'_\bfM,w'_\cM)w$ under this decomposition. We have $w'_\bfM=1$ and moreover $\bfM_{Q}^{\std,2} \cap R^\std=\bfM_{Q}^{\std,2} \cap R^\std_{w'}$, $P_w=P_{w'}$, $\pi_w=\pi_{w'}$ and $\nu_w=\nu_{w'}$. We now want to use Proposition~\ref{prop:partial_Eisenstein} to bound the generalized Eisenstein series $E^{R^\std}(M(w',\lambda)\varphi_{P_{w}},w'\lambda)$. Recall that in \S\ref{subsubsec:partial_Eisenstein} we have defined some projectors $p_i$ that appear in the statement of Proposition~\ref{prop:partial_Eisenstein}. Then it is easily checked by the previous discussion and because $(w'\pi_{w})_\cP$ is cuspidal that the condition $(w(\lambda+\nu_w))_\cP \in \cR_{(w\pi_w)_\cP,k,c_J}$ implies $p_i(w'(\lambda+\nu_{w'})) \in \cR_{(w'\pi_w)_i,k,c_J}$ for all $i$ (see Remark~\ref{rem:cuspi_case}). By Proposition~\ref{prop:partial_Eisenstein}, there exists a finite product of linear forms $L_{w'}(\lambda)$ (which can be chosen independently of $\pi$) and $N'>0$ such that for all $Y \in \cU(\fg_\infty)$ and all $N>0$ large enough we have $X_1, \hdots, X_r \in \cU(\fg_\infty)$ and $d>0$ such that
    \begin{equation}
    \label{eq:bound_partial_proof}
         \norm{L_{w'}(\lambda)E^{R^\std}(M(w',\lambda)\varphi_{P_w},w'\lambda)}_{-N',Y}   \leq (1+\norm{\lambda}^2)^d \sum_{i=1}^r \norm{\varphi}_{-N,X_i},
    \end{equation}
    for $\lambda$ in the region \eqref{eq:big_intersection}. This estimate is uniform in $\varphi \in \cA_{P,\pi}(G)^J$ and $(P,\pi)$.
    
    Let us denote by $L(\lambda)$ the product of all the $L_{w'}$ with all the $\hat{\theta}_R^Q(w_R^\std w'(\lambda+\nu_{w'})+\underline{\rho}_R)$. By Theorem~\ref{thm:truncation_operator} and Proposition~\ref{prop:unfold_periods}, we see that for any $N'>0$ large enough we have $Y_1,\hdots,Y_{r'} \in \cU(\fg_\infty)$ such that for all $J$-pairs $(P,\pi)$, all $\varphi \in \cA_{P,\pi}(G)^J$ and all $\lambda$ in \eqref{eq:big_intersection} we have
    \begin{equation*}
        \Val{L(\lambda)\cP^Q(\varphi,\lambda,w)} \leq \sum_{\substack{R \in \cF_{\RS} \\ R \subset Q}} \sum_{w' \in {}_{R^\std}W_{Q^\std_w}^{Q^\std}w} \sum_{i=1}^{r'} \norm{ L(\lambda) E^{R^\std}(M(w',\lambda)\varphi_{P_w},w'\lambda)}_{-N',Y_i}.
    \end{equation*}
    By \eqref{eq:bound_partial_proof}, we conclude that \eqref{eq:optimal_bound_P} holds for $L(\lambda) \cP^Q(\varphi,\lambda,w)$. But by Proposition~\ref{prop:reg_P}, we know that $\lambda \mapsto L_{(w\pi_w)_\cP,Z}\left((w(\lambda+\nu_w))_\cP\right)  L_{\pi,w}(\lambda) \cP^Q(\varphi,\lambda,w)$ is holomorphic in our region. It remains to use \cite[Lemma~2.4.2.1]{Ch} to eliminate the superfluous linear forms.
\end{proof}

\subsection{Residues of Rankin--Selberg periods}
\label{sec:residues_RS}

In this section, we recall the results of \cite{BoiZ} on residues of the regularized period $\lambda \mapsto \cP(\varphi,\lambda)$. 

\subsubsection{A naive notion of residues} \label{subsubsec:naive_residue} Let $m \geq 1$, let $f$ be a meromorphic function on $\cc^m$. Let $\Lambda$ be a non-zero affine linear form on $\cc^m$. Write $\cH_{\Lambda}$ for the affine hyperplane $\{ v \in \cc^m \; | \; \Lambda(v)=0\}$. The map $ v \mapsto \Lambda(v) f(v)$ is a meromorphic function on $\cc^m$. Assume that $\cH_{\Lambda}$ is not contained in its polar divisor (i.e. $\cH_{\Lambda}$ is at most a simple polar divisor of $f$). Then its restriction to $\cH_{\Lambda}$ is a meromorphic function on $\cH$, and we set
\begin{equation*}
    \underset{\Lambda}{\Res} f:= \left( \Lambda f\right)_{| \cH_{\Lambda}}.
\end{equation*}
Let $\Lambda_1, \hdots, \Lambda_r$ be a family of affine linear forms such that the underlying family of linear forms is linearly disjoint. We consider the iterated residue
\begin{equation*}
    \underset{\Lambda_{r} \leftarrow \Lambda_{1}}{\Res} f:= \underset{\Lambda_{r}}{\Res} \hdots \underset{\Lambda_{1}}{\Res} f,
\end{equation*}
provided each residue is defined in the above sense. This is a meromorphic function on $\cH:=\bigcap \cH_{\Lambda_i}$. Note that the iterated residue a priori depends on the order of the affine linear forms. 

\subsubsection{Residues as regularized periods}
Let $P$ be a standard parabolic subgroup of $G$ and $\pi \in \Pi_\cusp(M_P)$. Write the Levi factor $M_P=\left( \prod_{i=1}^{m_n} \GL_{n_{n,i}} \right) \times \left( \prod_{j=1}^{m_{n+1}} \GL_{n_{n+1,j}} \right)$ and $\pi=\boxtimes \pi_{n,i} \boxtimes \pi_{n+1,j}$ accordingly. The next proposition summarizes the results of \cite[Proposition~4.23]{BoiZ} and \cite[Lemma~5.1]{BoiZ}.

\begin{prop}
\label{prop:order_residues}
    Let $1 \leq i_{+,1}, \hdots, i_{+,m_+} \leq m_n$, $1 \leq i_{-,1}, \hdots, i_{-,m_-} \leq m_n$, $1 \leq j_{+,1}, \hdots, j_{+,m_+} \leq m_{n+1}$ and $1 \leq j_{-,1}, \hdots, j_{-,m_-} \leq m_{n+1}$. Assume that
    \begin{itemize}
        \item The indices $i_{+,1}, \hdots, i_{+,m_+}, i_{-,1}, \hdots, i_{-,m_-}$ are distinct;
        \item The indices $j_{+,1}, \hdots, j_{+,m_+}, j_{-,1}, \hdots, j_{-,m_-}$ are distinct;
        \item For every $l$ we have $\pi_{n,i_{+,l}}=\pi_{n+1,j_{+,l}}^\vee$ and $\pi_{n,i_{-,l}}=\pi_{n+1,j_{-,l}}^\vee$.
    \end{itemize}
    For $l$ in the suitable range, consider the affine linear forms on $\fa_{P,\cc}^*$ defined by $\Lambda_{+,l}(\lambda)=\lambda_{n,i_{+,l}}+\lambda_{n+1,j_{+,l}}+\frac{1}{2}$, and $\Lambda_{-,l}(\lambda)=\lambda_{n,i_{-,l}}+\lambda_{n+1,j_{-,l}}-\frac{1}{2}$. Let $Q^\std_{n+1}$ be the standard parabolic subgroup of $\GL_{n+1}$ with standard Levi factor
    \begin{equation}
        \label{eq:explicit_levi}
        \prod_{l=1}^{m_+} \GL_{n_{i_{+,l}}} \times \GL_{k+1} \times \prod_{l=m_-}^{1} \GL_{n_{i_{-,l}}},
    \end{equation}
    and let $Q \in \cF_{\RS}$ be the element corresponding to $(Q^\std_{n+1},m_++1)$ under the bijection of Proposition~\ref{prop:param_RS}. Let $w \in W(P;Q^\std)$ be the only element such that $w_n(i_{+,l})=l$, $w_n(i_{-,l})=m_n-l+1$, $w_{n+1}(j_{+,l})=l$ and $w_{n+1}(j_{-,l})=m_{n+1}-l+1$. Then for every $\varphi \in \cA_{P,\pi}(G)$ we have
    \begin{equation}
    \label{eq:residue_hard}
        \underset{\Lambda_{m_-,-} \leftarrow \Lambda_{-,1}}{\Res} \; \underset{\Lambda_{+,m_+} \leftarrow \Lambda_{+,1}}{\Res} \cP(\varphi,\lambda)=(-1)^{m_+}\cP^Q(\varphi,\lambda,w).
    \end{equation}
    Moreover, the iterated residue may be taken in any order in the set $\{\Lambda_{1,\pm}, \hdots ,\Lambda_{m_{\pm},\pm}\}$.
\end{prop}

\section{Extension of the Rankin--Selberg period}
\label{chap:RS_non_tempered}

\subsection{An extension of \cite[Section~5]{BoiZ}}
\label{sec:relevant_discrete}

In this section, we define a set of \emph{relevant inducing data} $\Pi_H$. They parametrize inductions $\cA_{P,\pi}(G)$ on which we can build the regularized period $\cP_\pi$ from \cite{BoiZ}. In fact, for the purpose of the fine spectral expansion of the Rankin--Selberg period, we need to work with slightly more general inductions than in \cite{BoiZ}. 

\subsubsection{Relevant inducing data}
\label{subsubsec:relevant_inducing}

In the rest of the text, we will use the following convention. If $\underline{n}(1), \hdots ,\underline{n}(k)$ are tuples of integers in $\zz_{\geq 0}^{m_1}, \hdots, \zz_{\geq 0}^{m_k}$ respectively, we write $\underline{n}(i,j)$ for their elements with $1 \leq i \leq k$ and $1 \leq j \leq m_i$. We also write $(\underline{n}(1),\hdots,\underline{n}(k))$ for the tuple in $\zz_{\geq 0}^{m_1+\hdots+m_k}$ obtained by concatenating them. Recall that any standard parabolic subgroup $P$ of $\GL_n$ is determined by the tuple of integers $\underline{n}(P)$ defined in Section~\ref{chap:poles}. We naturally extend this notion to standard parabolic subgroups of $G$. We also allow entries of $\underline{n}(P)$ to be zero.

Let $k \geq 1$ and $\pi \in \Pi_\disc(\GL_k)$. Write $\sigma_\pi=\sigma^{\boxtimes d}$ with $\sigma \in \Pi_\cusp(\GL_r)$ so that $\pi=\Speh(\sigma,d)$. By analogy with the local notion of derivatives introduced in \cite{Zel}, we define the automorphic derivative of $\pi$ to be
\begin{equation*}
    \pi^-=\Speh(\sigma,d-1) \in \Pi_{\disc}(\GL_{r(d-1)}).
\end{equation*}
Note that if $\pi$ is cuspidal, i.e. if $d=1$, the representation $\pi^-$ is the trivial representation of the trivial group.

We now define the set $\Pi_H$ of relevant inducing data. It is the set of triples $(I,P,\pi)$ such that the following conditions hold.
\begin{itemize}
    \item $I \in \zz_{\geq 0}^4$ with $I=(n_+,n_1,n_2,n_-)$.
    \item We have tuples of integers $\underline{n}(+)\in \zz_{\geq 1}^{m_+}$, $\underline{n}(1) \in \zz_{\geq 1}^{m_1}$, $\underline{n}'(1) \in \zz_{\geq 0}^{m_1}$, $\underline{n}(2)\in \zz_{\geq 1}^{m_2}$, $\underline{n}'(2) \in \zz_{\geq 0}^{m_2}$ and $\underline{n}(-)\in \zz_{\geq 1}^{m_-}$ such that $\sum \underline{n}(+,i)=n_+$, $\sum \underline{n}(1,i)=n_1$, $\sum \underline{n}(2,i)=n_2$ and $\sum \underline{n}(-,i)=n_-$ and $P$ is the standard parabolic subgroup of $G$ such that
    \begin{equation}
        \label{eq:M_P}
        \underline{n}(P)=\left( (\underline{n}(+),\underline{n}(1),\underline{n}'(2),\underline{n}(-)),(\underline{n}(+),\underline{n}'(1),\underline{n}(2),\underline{n}(-)) \right).
    \end{equation}
    
    \item $\pi \in \Pi_\disc(M_P)$ is a discrete automorphic representation (with trivial central character on $A_P^\infty$) such that, with respect to \eqref{eq:M_P}, $\pi=\pi_n \boxtimes \pi_{n+1}$ decomposes as
     \begin{align}
    \label{eq:pi_n}
        \pi_n&=\boxtimes_{i=1}^{m_+} \pi_{+,i} \boxtimes_{i=1}^{m_1} \pi_{1,i} \boxtimes_{i=1}^{m_2} \pi_{2,i}^{-,\vee} \boxtimes_{i=1}^{m_-} \pi_{-,i}, \\
    \label{eq:pi_n+1}
         \pi_{n+1}&=\boxtimes_{i=1}^{m_+} \pi_{+,i}^\vee \boxtimes_{i=1}^{m_1} \pi_{1,i}^{-,\vee} \boxtimes_{i=1}^{m_2} \pi_{2,i} \boxtimes_{i=1}^{m_-} \pi_{-,i}^\vee.
    \end{align}    
\end{itemize} 
In this situation, for $1 \leq i \leq m_1$ we can write for $\pi_{1,i}=\Speh(\sigma_{1,i},\underline{d}(1,i))$, for some representation $\sigma_{1,i} \in \Pi_\cusp(\GL_{\underline{r}(1,i)})$ and $\underline{d}(1,i)$ and $\underline{r}(1,i)$ some positive integers. Similarly, for $1 \leq i \leq m_2$ we have $\pi_{2,i}=\Speh(\sigma_{2,i},\underline{d}(2,i))$. In particular, we have for $i \in \{1,2\}$ and $1 \leq j \leq m_i$ the formulae $\underline{n}(i,j)=\underline{d}(i,j)\underline{r}(i,j)$ and $\underline{n}'(i,j)=(\underline{d}(i,j)-1)\underline{r}(i,j)$.

Let $(I,P,\pi) \in \Pi_H$. With the choices of coordinates made in Section~\ref{chap:poles}, 
$\fa_{P}^*$ is realized as a subspace
\begin{equation}
\label{eq:a_P_explicit coordinates}
    \fa_{P}^* \subset \left( \rr^{m_+} \times \rr^{m_1} \times \rr^{m_2} \times \rr^{m_-} \right) \times \left( \rr^{m_+} \times \rr^{m_1} \times \rr^{m_2} \times \rr^{m_-} \right).
\end{equation}
A similar decomposition holds for $\fa_{P,\cc}^*$. If $\lambda \in \fa_{P}^*$, we write 
\begin{equation}
\label{eq:notation_P}
    \lambda=\left( (\lambda(+)_n, \lambda(1)_n, \lambda(2)_n, \lambda(-)_n),(\lambda(+)_{n+1}, \lambda(1)_{n+1}, \lambda(2)_{n+1}, \lambda(-)_{n+1}) \right)
\end{equation}
according to this decomposition. Note that we have $\lambda(2)_{n,i}=0$ if $\underline{d}(2,i)=1$, and $\lambda(1)_{n+1,i}=0$ if $\underline{d}(1,i)=1$. Let $\underline{\rho}_\pi$ be the element of $\fa_P^*$ defined as
\begin{equation}
    \label{eq:underline_rho_defi} \underline{\rho}_\pi:=\left(\left(\underbrace{\frac{1}{4},\hdots,\frac{1}{4}}_{m_+},\underbrace{0,\hdots,0}_{m_1+m_2},\underbrace{-\frac{1}{4},\hdots,-\frac{1}{4}}_{m_-}\right),\left(\underbrace{\frac{1}{4},\hdots,\frac{1}{4}}_{m_+},\underbrace{0,\hdots,0}_{m_1+m_2},\underbrace{-\frac{1}{4},\hdots,-\frac{1}{4}}_{m_-}\right)\right) \in \fa_P^*.
\end{equation}

We define the anti-diagonal subspace $\fa_\pi^* \subset \fa_P^*$ to be
\begin{equation}
\label{eq:a_pi_defi}
    \fa_{\pi}^*=\left\{ \lambda \in \fa_{P}^* \; \middle| \; \begin{array}{ll}
        \lambda(+)_n=-\lambda(+)_{n+1}, &  \\
        \lambda(1)_{n,i}=-\lambda(1)_{n+1,i}, & 1 \leq i \leq m_1, \; \text{if } \underline{d}(1,i) \neq 1, \\
        \lambda(2)_{n,i}=-\lambda(2)_{n+1,i}, & 1 \leq i \leq m_2, \; \text{if } \underline{d}(2,i) \neq 1, \\
         \lambda(-)_n=-\lambda(-)_{n+1}, & 
    \end{array} \right\}.
\end{equation}
We have an isomorphism 
\begin{equation}
\label{eq:coord_a_pi}
    \lambda \in \fa_{\pi}^* \mapsto \left( \lambda(+),\lambda(1), \lambda(2),\lambda(-) \right):=\left( \lambda(+)_n,\lambda(1)_n, \lambda(2)_{n+1},\lambda(-)_n \right) \in \rr^{m_+} \times \rr^{m_1} \times \rr^{m_2} \times \rr^{m_-}
\end{equation}
We also have an anti-diagonal subspace $\fa_{\pi,\cc}^* \subset \fa_{P,\cc}^*$ defined by the same equations. Note that $i\fa_{\pi}^{*}$ is exactly the subspace of $\lambda \in \fa_{P,\cc}^*$ such that $(P,\pi_\lambda) \in \Pi_H$ if we lift the requirement that the central character is trivial on $A_P^\infty$, and ask that it is unitary instead.

\begin{rem}
    The set of relevant inducing pairs $\Pi_H$ introduced in \cite[Section~5.1]{BoiZ} corresponds to the special case where $n_+=n_-=0$. We limited ourselves to this setting in \cite{BoiZ} as our goal there was to prove the non-tempered Gan--Gross--Prasad conjecture from \cite{GGP2}. 
\end{rem}

\subsubsection{The regularized Rankin--Selberg period}

Let $(I,P,\pi) \in \Pi_H$. Let $\sigma_\pi$ and $\nu_\pi$ be respectively the cuspidal automorphic representation of $M_{P_\pi}$ and the element of $\fa_{P_\pi}^*$ such that $\cA_{P,\pi}(G)$ is obtained by taking residues of Eisenstein series on the induction $\cA_{P_\pi,\sigma_\pi,-\nu_\pi}(G)$ (see \S\ref{subsubsec:residual_blocks}). For $\lambda \in \fa_{P_\pi,\cc}^*$ in general position, we have the global Zeta function $Z_{\sigma_\pi}(\cdot,\lambda)$ from \S\ref{subsec:RS_and_Zeta}. By Proposition~\ref{prop:P=Z}, we know that it is equal to the regularized period $\cP(\cdot,\lambda)$. Moreover, it has the Euler product expansion from \eqref{eq:Zeta_facto_global} which involves a global $L$-factor $ L\left(\lambda+\frac{1}{2},\sigma_{\pi,n} \times \sigma_{\pi,n+1}\right)$. We can identify $\fa_{\pi,\cc}^*-\nu_\pi-\underline{\rho}_\pi$ as an affine subspace of $\fa_{P_\pi,\cc}^*$. By Theorem~\ref{thm:L}, we know that all the singularities of $Z_{\sigma_\pi}$ that contain this subspace are affine hyperplanes coming from singularities of $ L\left(\lambda+\frac{1}{2},\sigma_{\pi,n} \times \sigma_{\pi,n+1}\right)$. We can consider the residue $\Res \; Z_{\sigma_\pi}(\cdot,\lambda)$ defined by multiplying by all the corresponding affine linear forms (normalized as in Proposition~\ref{prop:order_residues}) and evaluating on $\fa_{\pi,\cc}^*-\nu_\pi-\underline{\rho}_\pi$. This a priori depends on the chosen order. We now state a slight generalization of \cite[Theorem~5.2]{BoiZ}.

\begin{theorem}
    \label{thm:P_pi_quotient}
    For every $\phi \in \cA_{P_\pi,\sigma_\pi}(G)$, the residue $\Res \; Z_{\sigma_\pi}(\phi,\lambda)$ is well-defined and independent on the order. For $\lambda \in \fa_{\pi,\cc}^*-\underline{\rho}_\pi$ in general position, the linear form $\phi \mapsto \Res \; Z_{\sigma_\pi}(\phi,\mu)$ factors through the quotient $\cA_{P_\pi,\sigma_\pi,\mu}(G) \twoheadrightarrow \cA_{P,\pi,\lambda}(G)$, where $\mu=\lambda-\nu_\pi$.
\end{theorem}

The linear form thus obtained on $\cA_{P,\pi,\lambda}(G)$ will be denoted by $\cP_\pi(\cdot,\lambda)$. It is $H(\bA)$-invariant.

\begin{proof}
    That the residue is independent of the order of affine linear forms follows from the same argument as \cite[Lemma~5.1]{BoiZ}. If $n_+=n_-=0$, the rest of Theorem~\ref{thm:P_pi_quotient} is \cite[Theorem~5.2]{BoiZ}. For the general case, we proceed as follows. We have a natural identification of $\fa_{P_\pi,\cc}^*$ with 
     \begin{equation*}
       \left(\prod_{i=1}^{m_+} \cc^{\underline{d}(+,i)} \prod_{i=1}^{m_1} \cc^{\underline{d}(1,i)}  \prod_{i=1}^{m_2} \cc^{\underline{d}(2,i)-1}  \prod_{i=1}^{m_-} \cc^{\underline{d}(-,i)} \right) \times \left(\prod_{i=1}^{m_+} \cc^{\underline{d}(+,i)}  \prod_{i=1}^{m_1} \cc^{\underline{d}(1,i)-1}  \prod_{i=1}^{m_2} \cc^{\underline{d}(2,i)}  \prod_{i=1}^{m_-} \cc^{\underline{d}(-,i)}\right).
    \end{equation*}
    We write the coordinates of any $\lambda \in \fa_{P_\pi,\cc}^*$ with respect to these identifications. More precisely, if $\lambda=(\lambda_n,\lambda_{n+1}) \in \fa_{P_\pi,\cc}^*$ we write $\lambda_n$ as
    \begin{equation*}
        ( \lambda(+,1)_n, \hdots, \lambda(+,m_+)_n,\lambda(1,1)_n, \hdots,\lambda(1,m_1)_n, \lambda(2,1)_n, \hdots ,\lambda(2,m_2)_n, \lambda(-,1)_n, \hdots, \lambda(-,m_-)_n )
    \end{equation*}
    where for example, $\lambda(+,1)_n \in \cc^{\underline{d}(+,1)}$ with coordinates $\lambda(+,1)_{n,1}, \hdots ,\lambda(+,1)_{n,\underline{d}(+,1)}$. We now consider the set $\cL$ of affine linear forms on $\fa_{P_\pi,\cc}^*$ defined as 
     \begin{equation*}
        \left\{
            \begin{array}{ll}
                \Lambda(+,i,j)(\lambda)=-(\lambda(+,i)_{n,j}+\lambda(+,i)_{n+1,\underline{d}(+,i)-j+1}+1/2), &  \quad \left\{\begin{array}{ll}
                    1 \leq i \leq m_+, \\
                    1 \leq j \leq \underline{d}(+,i),
                \end{array}\right.  \\
                \Lambda(-,i,j)(\lambda)=\lambda(-,i)_{n,j}+\lambda(-,i)_{n+1,\underline{d}(-,i)-j+1}-1/2, &  \quad \left\{\begin{array}{ll}
                    1 \leq i \leq m_-,  \\
                    1 \leq j \leq \underline{d}(-,i).
                \end{array}\right.  
            \end{array}
        \right.
    \end{equation*}
    We also have the set $\cL'$ of linear forms defined by the following equations.
    \begin{equation*}
        \left\{
            \begin{array}{ll}
                \Lambda'(+,i,j)(\lambda)=\lambda(+,i)_{n,j}+\lambda(+,i)_{n+1,\underline{d}(+,i)-j}-1/2, &  \quad \left\{\begin{array}{ll}
                    1 \leq i \leq m_+,  \\
                    1 \leq j \leq \underline{d}(+,i)-1,
                \end{array}\right.    \\
                \Lambda'(-,i,j)(\lambda)=-(\lambda(-,i)_{n,j}+\lambda(-,i)_{n+1,\underline{d}(-,i)-j+2}+1/2), &  \quad \left\{\begin{array}{ll}
                    1 \leq i \leq m_-,  \\
                    2 \leq j \leq \underline{d}(-,i).
                \end{array}\right.   
            \end{array}
        \right.
    \end{equation*}
    All these linear forms direct singular affine hyperplanes of $Z_{\sigma_\pi}$. If we write $\cH$ and $\cH'$ for their respective intersections, then $\fa_{\pi,\cc}^*-\nu_\pi-\underline{\rho}_\pi \subset \cH \cap \cH'$. We now order the set $\cL$ by  
      \begin{equation}
        \label{eq:order_1}
        \left( \xleftarrow[i=m_-]{1} \xleftarrow[j=1]{\underline{d}(-,i)} \Lambda(-,i,j) \right)  \leftarrow
         \left( \xleftarrow[i=1]{m_+} \xleftarrow[j=\underline{d}(+,i)]{1} \Lambda(+,i,j) \right),
    \end{equation}
     where for example the notation $\xleftarrow[i=1]{m_+} \xleftarrow[j=\underline{d}(+,i)]{1} \Lambda(+,i,j)$ means
    \begin{equation*}
     \Lambda(+,m_+,1) \leftarrow \hdots \leftarrow \Lambda(+,2,\underline{d}(+,2)) \leftarrow \Lambda(+,1,1) \leftarrow \hdots \leftarrow \Lambda(+,1,\underline{d}(+,1)).
    \end{equation*}

    Let $Q^\std$ be the standard parabolic subgroup of $\GL_{n+1}$ with standard Levi of the form
    \begin{equation*}
        \prod_{i=1}^{m_+} \GL_{\underline{r}(+,i)}^{\underline{d}(+,i)} \times \GL_{r+1} \times \prod_{i=1}^{m_-} \GL_{\underline{r}(-,i)}^{\underline{d}(-,i)}.
    \end{equation*}
    Let $Q$ be the Rankin--Selberg parabolic subgroup of $G$ associated to the pair $(Q^\std,\sum_{i=1}^{m_+} \underline{d}(+,i))$ under the bijection of Proposition~\ref{prop:param_RS}. Moreover, set $\pi_+=\boxtimes \pi_{+,i}$ and $\pi_-=\boxtimes \pi_{-,i}$. They are representations of standard Levi subgroups of $\GL_{n_+}$ and $\GL_{n_-}$. For $m \in \{n,n+1\}$, we have the elements $w_{\pi_+,m}^*$ and $w_{\pi_-,m}^*$, which we identify as elements of $\GL_{n_+}$ and $\GL_{n_-}$ embedded in $M_{P,m}$. Then by applying Proposition~\ref{prop:order_residues}, we see that if we take the residues in the order prescribed above we have for $\lambda \in \cH$ in general position
    \begin{equation*}
        \Res_{\cL} \; Z_{\sigma_\pi}(\phi,\lambda)=\cP^{Q}(\phi,\lambda,w_{\pi_+,n}^* w_{\pi_-,n}^*).
    \end{equation*}
    Note that, when restricted to $\cH$, the linear forms in $\cL'$ can be written, for $i$ and $j$ in the suitable range, as
     \begin{equation}
    \label{eq:L'_alternative}
        \left\{
            \begin{array}{l}
                \Lambda'(+,i,j)(\lambda)=\lambda(+,i)_{n,j}-\lambda(+,i)_{n,j+1}-1 =\lambda(+,i)_{n+1,\underline{d}(+,i)-j}-\lambda(+,i)_{n+1,\underline{d}(+,i)-j+1}-1 \\
                \Lambda'(-,i,j)(\lambda)=\lambda(-,i)_{n,j-1}-\lambda(-,i)_{n,j}+1 =\lambda(-,i)_{n+1,\underline{d}(-,i)-j+1}-\lambda(-,i)_{n+1,\underline{d}(-,i)-j+2}+1.
            \end{array}
        \right.
    \end{equation}
    It follows as in \cite[Lemma~5.5]{BoiZ} that we have for $\lambda \in \cH \cap \cH'$ in general position
    \begin{equation*}
        \Res_{\cL'} \Res_{\cL} \; Z_{\sigma_\pi}(\phi,\lambda)=\cP^{Q}(M^*(w_{\pi_+,n}^* w_{\pi_-,n}^*,\lambda)\phi,w_{\pi_+,n}^* w_{\pi_-,n}^* \lambda).
    \end{equation*}
    By reversing the order in \eqref{eq:order_1}, we see that the same relation holds with $M^*(w_{\pi_+,n+1}^* w_{\pi_-,n+1}^*,\lambda)$ instead. By parabolic descent (Proposition~\ref{prop:parabolic_descent}), we can now input the result of \cite[Theorem~5.2]{BoiZ}. This shows that the residue factors through $M^*(w_{\pi,n}^*,\lambda)$ and $M^*(w_{\pi,n+1}^* ,\lambda)$ for $\lambda \in \fa_{\pi,\cc}^*-\nu_\pi-\underline{\rho}_\pi$. But we know by \cite[Lemma~3.2]{BoiZ} that these operators realize the global quotient maps, which concludes the proof.
\end{proof}

    \subsubsection{The residue-free construction}
    \label{subsubsec:residue_free}
    We now explain an alternative construction of $\cP_{\pi}$ without residues. The idea is to realize $\cA_{P,\pi}(G)$ as a subrepresentation of some parabolic induction from the cuspidal spectrum rather than as a quotient. It is more suited to study the analytic properties of $\cP_\pi$ which we will prove in \S\ref{subsec:analytic_period} below.

    For $i \in \{1,2\}$, let $\underline{n}^-(i) \in \zz_{\geq 0}^{2m_i}$ be the tuple such that $\underline{n}^-(i,2j-1)=(\underline{d}(i,j)-1)\underline{r}(i,j)$ and $\underline{n}^-(i,2j)=\underline{r}(i,j)$ for $1 \leq j \leq m_i$. Let $P_{\pi,+}$ be the standard parabolic subgroup of $G$ such that
    \begin{equation*}
        \underline{n}(P_{\pi,+})=\left( (\underline{n}(+),\underline{n}^-(1),\underline{n}'(2),\underline{n}(-)),(\underline{n}(+),\underline{n}'(1),\underline{n}^-(2),\underline{n}(-)) \right).
    \end{equation*}
    Then $P_\pi \subset P_{\pi,+} \subset P$. Let $w_+ \in W(P_{\pi,+})$ be the shortest element such that, if we set $Q_{\pi,+}=w_+ P_{\pi,+}$, then
    \begin{equation*}
         \underline{n}(Q_{\pi,+})=\left( (\underline{n}(+),\underline{n}'(1),\underline{n}'(2),\underline{r}(1),\underline{n}(-)),(\underline{n}(+),\underline{n}'(1),\underline{n}'(2),\underline{r}(2),\underline{n}(-)) \right).
    \end{equation*}
    In words, on the $\GL_n$ component $w_+$ sends the last $\GL_{\underline{r}(1,i)}$ block in each product $\GL_{(\underline{d}(1,i)-1)\underline{r}(1,i)}\times \GL_{\underline{r}(1,i)}$ after the product $\prod_{i=1}^{m_2} \GL_{(\underline{d}(2,i)-1)\underline{r}(2,i)}$, while preserving the order. The description on the $\GL_{n+1}$ component is the same.
    
    Let $P_{+,n+1}^\std$ be the standard parabolic subgroup of $\GL_{n+1}$ such that 
     \begin{equation}
        \label{eq:M_P_+}
         \underline{n}(P_{+,n+1}^\std)=\left( \underline{n}(+),\underline{n}'(1),\underline{n}'(2),k+1,\underline{n}(-) \right).
    \end{equation}
    where we set $k+1=\sum_i \underline{r}(2,i)$. Let $P_+$ be the Rankin--Selberg parabolic subgroup of $G$ corresponding to the pair $(P^\std_{+,n+1},m_++m_1+m_2+1)$ under the bijection of Proposition~\ref{prop:param_RS}. Here we allow for blocks of size zero, so that the $(m_++m_1+m_2+1)^{\text{th}}$ block is $\GL_{k+1}$. Note that $w_{+}\in {}_{P^\std_+} W_P$ and that $P_\pi \subset P_{w_{+}}=P_{\pi,+}$. 
    
    We now consider the regularized period $\lambda \in \fa_{P,\cc}^* \mapsto \cP^{P_+}(\varphi,\lambda,w_{+})$ which is well-defined for $\lambda \in \fa_{P,\cc}^*$ in general position. In fact, it is also meromorphic for $\lambda$ in the smaller subspace $\fa_{\pi,\cc}^*-\underline{\rho}_\pi$ as shown in the following proposition.

    \begin{prop}
    \label{prop:alternative_construction}
        Let $\varphi \in \cA_{P,\pi}(G)$. The map
        \begin{equation}
        \label{eq:constant_term_construction}
            \lambda \in \fa_{\pi,\cc}^*-\underline{\rho}_\pi \mapsto \cP^{P_+}\left(\varphi,\lambda,w_{+}\right)
        \end{equation}
        is a well-defined meromorphic function on $\fa_{\pi,\cc}^*-\underline{\rho}_\pi$. Moreover, for $\lambda$ in general position we have
        \begin{equation}
        \label{eq:two_defi_coincide}
            \cP^{P_+}\left(\varphi,\lambda,w_{+}\right)=\cP_{\pi}(\varphi,\lambda).
        \end{equation}
    \end{prop}

    \begin{proof}
        This is the same proof as \cite[Proposition~5.9]{BoiZ}, taking into account the additional affine linear forms appearing in the proof of Theorem~\ref{thm:P_pi_quotient}.
    \end{proof}

    \subsubsection{Functional equations of \texorpdfstring{$\cP_\pi$}{the period}}
    \label{sec:functional_equation}

    Note that any $w \in W(P)$ can be identified with a couple $(w_n,w_{n+1}) \in \fS(m_++m_1+m_2+m_-)^2$, where we recall that we allow blocks of size zero in the case $\underline{d}(1,i)=1$ or $\underline{d}(2,j)=1$. We define $W(\pi)$ to be the set of $w \in W(P)$ which belong to the subgroup $(\fS(m_+) \times \fS(m_1+m_2) \times \fS(m_-))^2$ and which satisfy
    \begin{equation}
    \label{eq:W_Delta_defi}
        w_n=w_{n+1}=(\sigma_+,\sigma,\sigma_-), \; \sigma_+ \in \fS(m_+), \; \sigma \in \fS(m_1) \times \fS(m_2) \subset \fS(m_1+m_2), \; \sigma_- \in \fS(m_-).
    \end{equation}
    Therefore, with the definitions in \S\ref{subsubsec:relevant_inducing}, $W(\pi)$ is exactly the subset of $w \in W(P)$ such that $(I,w. P,w \pi) \in \Pi_H$. Moreover, we have $w(\fa_{\pi,\cc}^*-\underline{\rho}_\pi)=\fa_{w \pi,\cc}^*-\underline{\rho}_{w \pi}$.  

    To any element $((\sigma_+,\sigma,\sigma_-),(\sigma_+,\sigma,\sigma_-)) \in (\fS(m_+) \times \fS(m_1+m_2) \times \fS(m_-))^2$ we may associate $w_{\bfM} \in W(M_{P^\std_+})$ which acts by blocks on $M_{P_+}^\std$ in a natural way on \eqref{eq:M_P_+} by stabilizing the $\GL_k$ and $\GL_{k+1}$ blocks on the $\GL_n$ and $\GL_{n+1}$ components respectively. We denote by $W_{\bfM}(\pi)$ the subset of $w$ that arise this way, identified with a subset of $W(P_+^\std) \subset W$. By conjugating by $w_{P_+}^\std$, we can identify $w_{\bfM}$ with an element of $\bfM_{P_+}^2$, which we still denote by $w_{\bfM}$. We then have the Rankin--Selberg parabolic subgroup $w_{\bfM}.P_+$. Moreover, the standard parabolic subgroup $(w_+. P_{\pi,+}) \cap \cM_{P_+}^\std$ has standard Levi $\prod_{i=1}^{m_1} \GL_{\underline{r}(1,i)} \times \prod_{i=1}^{m_2} \GL_{\underline{r}(2,i)}$. Write $W_{\cM}(\pi)$ for $W((w_+. P_{\pi,+}) \cap \cM_{P_+}^\std)$ of elements that act by blocks on it. It is isomorphic to $\fS(m_1) \times \fS(m_2)$. If we identify it with a subgroup of $W$, its elements commute with $W_{\bfM}(\pi)$. 

    The functional equations satisfied by $\cP_{\pi}$ are summarized in the following proposition.

    \begin{lem}
        \label{lem:functional_equation_1}
            Let $\varphi \in \cA_{P,\pi}(G)$. Let $w_\bfM \in W_{\bfM}(\pi)$, let $w_\cM \in W_{\cM}(\pi)$. Then we have for $\lambda \in \fa_{\pi,\cc}^*-\underline{\rho}_\pi$ in general position
            \begin{equation}
            \label{eq:functional_equation}
                \cP_{\pi}(\varphi,\lambda)=\cP^{w_{\bfM}.P_+}\left(M(w_\bfM w_\cM w_{+},\lambda)\varphi_{P_{\pi,+}},w_\bfM w_\cM w_{+} \lambda\right).
            \end{equation}
        \end{lem}
    
        \begin{proof}
            It is easy to show using the same arguments as in the proof of \cite[Lemma~5.5]{BoiZ} that the intertwining operator and the regularized period in the RHS of \eqref{eq:functional_equation} are well-defined for $\lambda \in \fa_{\pi,\cc}^*-\underline{\rho}_\pi$, so that it yields a meromorphic function in $\lambda$. By the functional equation of Eisenstein series from \cite[Theorem~2.3.4]{BL}, we have for $\lambda \in \fa_{P,\cc}^*$ in general position the equality
            \begin{equation*}
                E^{P_+^\std}(M(w_{+},\lambda)\varphi_{P_{\pi,+}},w_{+}\lambda)=E^{P_+^\std}(M(w_{\cM} w_{+},\lambda)\varphi_{P_{\pi,+}},w_{\cM}w_{+}\lambda).
            \end{equation*}
            Let us now assume that $\lambda \in i\fa_{\pi}^*$. By Proposition~\ref{prop:parabolic_descent}, we know that $\cP^{P_+}$ decomposes as a product of a (bilinear) inner-product on $\bfM_{P_+}^2$ and a Zeta function on $\cM_{P_+}$. For the former, we have the have the unitarity of global intertwining operators on the unitary axis from \cite[Theorem~7.2]{Art05}, and for the latter we have the functional equations of Eisenstein series. Using the expression of $\cP_{\pi}(\varphi,\lambda)$ from Proposition~\ref{prop:alternative_construction}, we see that
            \begin{equation*}
                 \cP_{\pi}(\varphi,\lambda)=\cP^{w_{\bfM}.P_+}\left(\left(M(w_\bfM,w_\cM w_{+}\lambda)E^{P_+^\std}(M(w_{\cM} w_{+},\lambda)\varphi_{P_{\pi,+}},w_{\cM}w_{+}\lambda)\right)(w_{w_\bfM.P_+}^{\std,-1} \cdot)\right).
            \end{equation*}
            Note that here we use $w_{P_+}^\std=w_{w_\bfM.P_+}^\std$. We now conclude by Lemma~\ref{lem:ME=EM} that we can switch the intertwining operator with the partial Eisenstein series. This shows that \eqref{eq:functional_equation} holds for $\lambda \in i \fa_{\pi}^*-\underline{\rho}_\pi$, and we conclude for $\lambda$ in general position by analytic continuation.
        \end{proof}
    
    \begin{cor}
    \label{cor:independence_choice_couple}
        Let $w \in W(\pi)$. For $\varphi \in \cA_{P,\pi}(G)$ and $\lambda \in \fa_{\pi,\cc}^*-\underline{\rho}_\pi$ in general position we have
         \begin{equation}
    \label{eq:independence_choice_couple}
        \cP_{\pi}(\varphi,\lambda)=\cP_{w \pi}(M(w,\lambda) \varphi,w\lambda).
    \end{equation}
    \end{cor}

    \begin{proof}
       This follows from Proposition~\ref{prop:alternative_construction} and Lemma~\ref{lem:functional_equation_1}.
    \end{proof}   
    
   \subsubsection{Analytic properties of \texorpdfstring{$\cP_\pi$}{the period}}
   \label{subsec:analytic_period}
    We now investigate the singularities of the regularized period $\cP_{\pi}(\cdot,\lambda)$. Let $(P,\pi) \in \Pi_H$. Set 
    \begin{equation*}
        \pi_{1,2}=\left(\boxtimes_{i=1}^{m_1} \pi_{1,i} \boxtimes_{i=1}^{m_2} \pi_{2,i}^{-,\vee} \right)\boxtimes \left( \boxtimes_{i=1}^{m_1} \pi_{1,i}^{-,\vee} \boxtimes_{i=1}^{m_2} \pi_{2,i}  \right).
    \end{equation*}
    This is a discrete representation of a subgroup $M_{1,2}\subset M_P$. Moreover, if $\lambda \in \fa_{\pi,\cc}^*$, let $\lambda_{1,2}$ be its restriction to $\fa_{M_{1,2},\cc}^*$. In the coordinates of \eqref{eq:coord_a_pi}, we have the expression $\lambda_{1,2}=((\lambda(1),-\lambda(2)),(-\lambda(1),\lambda(2)))$. Let $L_{\pi,\cP}(\lambda)$ be the product of linear forms on $\fa_{\pi,\cc}^*-\nu_\pi$ defined by

    \begin{equation}
    \label{eq:f_pi_basic}
        L_{\pi,\cP}(\lambda)=\prod_{\substack{i,j\\  \pi_{1,i} \simeq \pi_{2,j}^{-,\vee}}} \left(\lambda(1)_i+\lambda(2)_j\right) \prod_{\substack{i,j \\  \pi_{1,i}^{-,\vee} \simeq \pi_{2,j}}} \left(\lambda(1)_i+\lambda(2)_j\right).
    \end{equation}
    It only depends on $\lambda_{1,2}$.

    \begin{prop}
    \label{prop:reg_P_pi}
        There exists $k>0$ such that for every level $J$ of $G$ there exists $c_J>0$ such that for every $J$-pair $(P,\pi) \in \Pi_H$ and every $\varphi \in \cA_{P,\pi}(G)$ the meromorphic function
        \begin{equation*}
            \lambda \in \fa_{\pi,\cc}^*-\underline{\rho}_\pi \mapsto L_{\pi,\cP}(\lambda)\cP_{\pi}(\varphi,\lambda)
        \end{equation*}
        is regular in the region $\lambda_{1,2} \in \cS_{\pi_{1,2},k,c_J}$ (see \eqref{eq:S_defi}).
    \end{prop}

    \begin{proof}
        Let $w \in W(\pi)$. By Proposition~\ref{prop:M_regular_blocks}, there exists $k$ and $c_J$ as in the proposition such that the operator $M(w,\lambda)$ is regular in the region $\cS_{\pi,k,c_J}$. Moreover, this region is stable by $w$. It follows from Corollary~\ref{cor:independence_choice_couple} that we can replace $\pi$ by $w \pi$. In particular, we can and will assume that $\underline{d}(1,1) \geq \hdots \geq \underline{d}(1,m_1)$ and $\underline{d}(2,1) \geq \hdots \geq \underline{d}(2,m_2)$. 

        By Proposition~\ref{prop:alternative_construction}, we have to study the poles of $\lambda \mapsto \cP^{P_+}(\varphi,\lambda,w_{+})$. To do this, we want to use Proposition~\ref{prop:reg_P}. By Lemma~\ref{lem:contant_term}, we have $\varphi_{P_{\pi,+}} \in \cA_{\pi_+,\nu_+}(G)$ where $\pi_+ \in \Pi_{\disc}(M_{P_{\pi,+}})$ and $\nu_+ \in \fa_{P_{\pi,+}}^*$. In coordinates, we have
        \begin{align}
            \nu_{+,n}&=\left(\underbrace{0,\hdots,0}_{m_+},\underbrace{-\frac{1}{2},\frac{\underline{d}(1,1)-1}{2},\hdots,-\frac{1}{2},\frac{\underline{d}(1,m_1)-1}{2}}_{m_1},\underbrace{0,\hdots,0}_{m_2}, \underbrace{0,\hdots,0}_{m_-}\right), \label{eq:nu_+_n} \\
            \nu_{+,n+1}&=\left(\underbrace{0,\hdots,0}_{m_+},\underbrace{0,\hdots,0}_{m_1},\underbrace{-\frac{1}{2},\frac{\underline{d}(2,1)-1}{2},\hdots,-\frac{1}{2},\frac{\underline{d}(2,m_2)-1}{2}}_{m_2}, \underbrace{0,\hdots,0}_{m_-}\right), \label{eq:nu_+_n+1}
        \end{align}
        Going back to the description of \S\ref{subsubsec:residue_free}, we see that 
        \begin{equation}
            \label{eq:w_+_positive}
            \left( \alpha \in \Delta_{P_{\pi,+}}, \quad w_+ \alpha<0 \right) \implies \langle \nu_+,\alpha^\vee \rangle \geq 0.
        \end{equation}
       Moreover, $w_+$ only acts on $M_{1,2}$. It follows that the condition $\lambda_{1,2} \in \cS_{\pi_{1,2},k,c_J}$ implies that $\lambda+\nu_+ \in \cR_{\pi,k,c_J}(w_+)$. We now compute $n_{\pi_+}(w_+,\lambda+\nu_+)=n_{\pi_+,n}(w_+,\lambda+\nu_+)n_{\pi_+,n+1}(w_+,\lambda+\nu_+)$. On the $\GL_n$ side, using the fact that $(I,P,\pi) \in \Pi_H$, it follows from \eqref{eq:nu_+_n} and the computations of Lemma~\ref{lem:n_regular} that
        \begin{align*}
            n_{\pi_+,n}(w_+,\lambda+\nu_+)=&\prod_{1 \leq i < j \leq m_1} \frac{L\left(\lambda(1)_i-\lambda(1)_j+\frac{\underline{d}(1,i)-\underline{d}(1,j)}{2}+1,\sigma_{1,i} \times \sigma_{1,j}^\vee \right)}{L\left(\lambda(1)_i-\lambda(1)_j+\frac{\underline{d}(1,i)+\underline{d}(1,j)}{2},\sigma_{1,i} \times \sigma_{1,j}^\vee \right)} \\
            & \times\prod_{\substack{1 \leq i \leq m_1 \\ 1 \leq j \leq m_2}} \frac{L\left(\lambda(1)_i+\lambda(2)_j+\frac{\underline{d}(1,i)-\underline{d}(2,j)+1}{2},\sigma_{1,i} \times \sigma_{2,j} \right)}{L\left(\lambda(1)_i+\lambda(2)_j+\frac{\underline{d}(1,i)+\underline{d}(2,j)-1}{2},\sigma_{1,i} \times \sigma_{2,j} \right)},
        \end{align*}
    and on $\GL_{n+1}$ we have
    \begin{equation*}
        n_{\pi_+,n}(w_+,\lambda+\nu_+)=\prod_{1 \leq i < j \leq m_2} \frac{L\left(\lambda(2)_i-\lambda(2)_j+\frac{\underline{d}(2,i)-\underline{d}(2,j)}{2}+1,\sigma_{2,i} \times \sigma_{2,j}^\vee \right)}{L\left(\lambda(2)_i-\lambda(2)_j+\frac{\underline{d}(2,i)+\underline{d}(2,j)}{2},\sigma_{2,i} \times \sigma_{2,j}^\vee \right)}.
    \end{equation*}
   By Theorem~\ref{thm:L}, we can localize the poles of these expressions. In some region $\lambda_{1,2} \in \cS_{\pi_{1,2},k,c_J}$, all the denominators are non-zero except $L\left(\lambda(1)_i+\lambda(2)_j+\frac{\underline{d}(1,i)+\underline{d}(2,j)-1}{2},\sigma_{1,i} \times \sigma_{2,j} \right)$ with $\underline{d}(1,i)=\underline{d}(2,j)=1$. However, this term is compensated by the corresponding numerator so that all the poles come from the numerators. Moreover, the possible pole of  $L\left(\lambda(1)_i+\lambda(2)_j+\frac{\underline{d}(1,i)-\underline{d}(2,j)+1}{2},\sigma_{1,i} \times \sigma_{2,j} \right)$ for $\underline{d}(1,i)=2$ and $\underline{d}(2,j)=1$, or for $\underline{d}(1,i)=1$ and $\underline{d}(2,j)=2$, is always compensated by a pole of the corresponding denominator. Finally we see that we can take
    \begin{align}
    \label{eq:pole_n_+}
        L_{\pi,w_+}(\lambda)=\prod_{\substack{i<j \\ \pi_{1,i} \simeq \pi_{1,j} }} (\lambda(1)_i-\lambda(1)_j )  &\prod_{\substack{i<j \\ \pi_{2,i} \simeq \pi_{2,j} }} (\lambda(2)_i-\lambda(2)_j ) \nonumber \\
         &\times \prod_{\substack{\pi_{1,i} \simeq \pi_{2,j}^{-,\vee} \\ \underline{d}(2,j) \neq 2}} (\lambda(1)_i+\lambda(2)_j )  \prod_{\substack{\pi_{1,i}^{-,\vee} \simeq \pi_{2,j} \\ \underline{d}(1,i) \neq 2}} (\lambda(1)_i+\lambda(2)_j ).
    \end{align}
    This product of affine linear forms controls the poles of $M(w_+,\lambda) \varphi_{P_{\pi,+}}$ in our region.

    On the other hand, with the notation of Proposition~\ref{prop:parabolic_descent}, we have
    \begin{equation*}
        (w_+ \nu_+)_{\cP}=\left( \left( \frac{\underline{d}(1,1)-1}{2},\hdots,\frac{\underline{d}(1,m_1)-1}{2}\right),\left(\frac{\underline{d}(2,1)-1}{2},\hdots,\frac{\underline{d}(2,m_2)-1}{2}\right) \right).
    \end{equation*}
    Because of our hypothesis on the $\underline{d}(1,i)$ and $\underline{d}(2,j)$, this elements belongs to $\overline{\fa_{\cP}^{*,+}}$. In particular, $\lambda_{1,2} \in \cS_{\pi_{1,2},k,c_J}$ implies that $(w_+(\lambda+\nu_+))_\cP \in \cR_{(w_+\pi_+)_\cP,k,c_J}$. With the notation of \eqref{eq:f_pi_Z_defi}, we see that $L_{(w_+ \pi_+)_{\cP},Z}(w_+(\lambda+\nu_+)_{\cP_+})$ is
    \begin{align}
        &\prod_{\sigma_{2,j}\simeq \sigma_{1,i}^\vee} \left(\lambda(1)_i+\lambda(2)_j+\frac{\underline{d}(1,i)+\underline{d}(2,j)- 1}{2}\right)\prod_{\sigma_{2,j}\simeq \sigma_{1,i}^\vee} \left(\lambda(1)_i+\lambda(2)_j+\frac{\underline{d}(1,i)+\underline{d}(2,j)- 3}{2}\right) \nonumber \\
        & \prod_{\substack{1 \leq i<j \leq m_1 \\ \sigma_{1,i} \simeq \sigma_{1,j}}}\left(\lambda(1)_i-\lambda(1)_j+\frac{\underline{d}(1,i)-\underline{d}(1,j)}{2}\right)^{-1}  
        \prod_{\substack{1 \leq i<j \leq m_2 \\ \sigma_{2,i} \simeq \sigma_{2,j}}}\left(\lambda(2)_i-\lambda(2)_j+\frac{\underline{d}(2,i)-\underline{d}(2,j)}{2}\right)^{-1}. \label{eq:L_pi_P}
    \end{align}
    The first term is always non-zero for $\lambda_{1,2} \in \cS_{\pi_{1,2},k,c_J}$. The second has the same zeros as 
    \begin{equation*}
        \prod_{\substack{\pi_{1,i} \simeq \pi_{2,j}^{-,\vee} \\ \underline{d}(2,j) = 2}} (\lambda(1)_i+\lambda(2)_j )  \prod_{\substack{\pi_{1,i}^{-,\vee} \simeq \pi_{2,j} \\ \underline{d}(1,i) = 2}} (\lambda(1)_i+\lambda(2)_j ).
    \end{equation*}
    Finally, the last two products compensate the first two terms in \eqref{eq:pole_n_+}. Putting everything together, we conclude that the result follows from Proposition~\ref{prop:reg_P}.
    \end{proof}

    We finally bound the regularized period.

    \begin{prop}
\label{prop:individual_bound_P_pi}
    There exists $k>0$ such that for any level $J$ and $C>0$ there exist $c_J>0$, $d>0$, $N>0$ and $X_1, \hdots, X_r \in \cU(\fg_\infty)$ such that for any $J$-pair $(P,\pi) \in \Pi_H$ and any $\varphi \in \cA_{P,\pi}(G)^J$ we have
    \begin{equation}
    \label{eq:optimal_bound_P_pi}
       \Val{L_{\pi,\cP}(\lambda)\cP_\pi(\varphi,\lambda)} \leq (1+\norm{\lambda}^2)^d \sum_{i=1}^r \norm{\varphi}_{-N,X_i},
    \end{equation}
    in the region
    \begin{equation}
    \label{eq:big_intersection_P_pi}
        \left\{ \lambda \in \fa_{\pi,\cc}^*-\underline{\rho}_\pi \; \middle| \; \lambda_{1,2} \in \cS_{\pi_{1,2},k,c_J}, \; \norm{\Re(\lambda)}<C \right\}.
    \end{equation}
\end{prop}

\begin{proof}
    Given all the explicit computations of Proposition~\ref{prop:reg_P_pi}, this is a direct consequence of Proposition~\ref{prop:individual_bound_P}.
\end{proof}

\subsubsection{Proof of Theorem~\ref{thm:GGP_global_intro}}
\label{subsubsec:proof_main_theorem}

We now end the proof of Theorem~\ref{thm:GGP_global_intro} which described the main properties of $\cP_\pi$. Its first point was the fact that the residue of the Rankin--Selberg Zeta integrals $\Res \; Z_{\sigma_\pi}(\phi,\lambda)$ factored through $\cA_{P_\pi,\sigma_\pi,\lambda-\nu_pi}(G) \twoheadrightarrow \cA_{P,\pi,\lambda}(G)$ for $\lambda \in \fa_{\pi,\cc}^*-\underline{\rho}_\pi$ in general position. This was the content of Theorem~\ref{thm:P_pi_quotient}. The second point was that it defined a continuous linear form in $\varphi$, which was shown in Proposition~\ref{prop:individual_bound_P_pi}. Finally, we have to prove that $\cP_\pi$ admits an Euler product expansion. If we set 
\begin{equation*}
    \cL(\lambda,\pi)=\underset{\fa_{\pi,\cc}^*-\underline{\rho}_\pi}{\Res} \frac{L(\lambda-\nu_\pi+\frac{1}{2},\sigma_{\pi,n} \times \sigma_{pi,n+1})}{b(\lambda-\nu_\pi,\sigma_\pi)},
\end{equation*}
then by computing the residues on the Euler product expansion of \eqref{eq:Zeta_facto_global}, we see as in \cite[Theorem~5.2]{BoiZ} that for $\varphi=E^{P,*}(\phi,-\nu_\pi)$ and $\phi=\otimes_v \phi_v \in \cA_{P_\pi,\sigma_\pi}(G)$ we have for $\lambda \in \fa_{\pi,\cc}^*-\underline{\rho}_\pi$ in general position
\begin{equation*}
    \cP_\pi(\varphi,\lambda)=\cL(\lambda,\pi) \prod_{v \in \tS} Z^\natural_{\sigma_\pi,v}(\phi_v,\lambda-\nu_\pi),
\end{equation*}
for some finite set of places $\tS$. This was exactly the content of Theorem~\ref{thm:GGP_global_intro}, which therefore concludes the proof.

    \subsection{Increasing inducing data}
    \label{subsec:increasing}
    For the purpose of our proof of the fine spectral expansion of the Rankin--Selberg period, we define a set $\Pi_H^\uparrow$ of \emph{increasing inducing data}. More precisely, if $(I,P,\pi) \in \Pi_H$, we want to choose an induction $\cA_{Q,\pi'}(G)$ isomorphic to $\cA_{P,\pi}(G)$ such that the singularities of the regularized period $\cP_{\pi'}(\cdot,\lambda)$ are controlled in a larger region than $\lambda_{1,2} \in \cS_{\pi'_{1,2},k,c_J}$. This will prove to be crucial in our shift of contours arguments. The definition of $\Pi_H^\uparrow$ is rather involved, and we invite the reader to come back to it when necessary in the course of Section~\ref{sec:pseudo_spectral}.
    
    \subsubsection{A set of combinatorial gadgets}
    \label{subsubsec:condition_up} 

    We define $\Pi_H^\uparrow$ to be the set of tuples $(I,P,\pi,I_1,I_2)$ satisfying the following conditions. 
    \begin{itemize}
        \item $I \in \zz_{\geq 0}^{6}$ is a tuple of the form $I=\left(n_+,n_1,n_{\mathrm{c},1},n_2,n_{\mathrm{c},2},n_-\right)$.
    \item $I_1$ and $I_2$ are subset of $\{1, \hdots, m_1 \}$ and $\{1, \hdots, m_2 \}$ for some integers $m_1$ and $m_2$ respectively, with $\Val{I_1}=\Val{I_2}$. We set $m=\Val{I_1}$ and write  $I_1=\{i_1(1) < \hdots <i_1(m)\}$ and $I_2=\{i_2(1)<\hdots<i_2(m)\}$.
    \item We have tuples of integers $\underline{n}(+)\in \zz_{\geq 1}^{m_+}$, $\underline{n}(1) \in \zz_{\geq 1}^{m_1}$, $\underline{n}'(1) \in \zz_{\geq 0}^{m_1-m}$, $\underline{n}_{\mathrm{c}}(1) \in \zz_{\geq 1}^{m_{\mathrm{c},1}}$, $\underline{n}(2)\in \zz_{\geq 1}^{m_2}$, $\underline{n}'(2) \in \zz_{\geq 0}^{m_2-m}$, $\underline{n}_{\mathrm{c}}(2) \in \zz_{\geq 1}^{m_{\mathrm{c},2}}$ and $\underline{n}(-)\in \zz_{\geq 1}^{m_-}$ such that $\sum \underline{n}(+,i)=n_+$, $\sum \underline{n}(1,i)=n_1$, $\sum \underline{n}_{\mathrm{c}}(1,i)=n_{\mathrm{c},1}$, $\sum \underline{n}(2,i)=n_2$, $\sum \underline{n}_{\mathrm{c}}(2,i)=n_{\mathrm{c},2}$ and $\sum \underline{n}(-,i)=n_-$ and $P$ is the standard parabolic subgroup of $G$ such that
    \begin{equation}
        \label{eq:M_P_increasing}
        \underline{n}(P)=\left( (\underline{n}(+),\underline{n}'(2),\underline{n}(1),\underline{n}_{\mathrm{c}}(1),\underline{n}(-)),(\underline{n}(+),\underline{n}'(1),\underline{n}(2),\underline{n}_{\mathrm{c}}(2),\underline{n}(-)) \right).
    \end{equation}
    \item $\pi=\pi_n \boxtimes \pi_{n+1} \in \Pi_{\disc}(M_P)$ is of the form 
    \begin{align}              
      \pi_n&=\boxtimes_{i=1}^{m_+} \pi_{+,i}  \boxtimes_{\substack{i=1 \\ i \notin I_2}}^{m_2} \pi_{2,i}^{-,\vee} \boxtimes_{i=1}^{m_1} \pi_{1,i} \boxtimes_{i=1}^{m_{\mathrm{c},1}} \pi_{\mathrm{c},1,i}\boxtimes_{i=1}^{m_-} \pi_{-,i}, \label{eq:pi_up_n} \\
         \pi_{n+1}&=\boxtimes_{i=1}^{m_+} \pi_{+,i}^\vee \boxtimes_{\substack{i=1 \\ i \notin I_1}}^{m_1} \pi_{1,i}^{-,\vee}  \boxtimes_{i=1}^{m_2} \pi_{2,i} \boxtimes_{i=1}^{m_{\mathrm{c},2}} \pi_{\mathrm{c},2,i} \boxtimes_{i=1}^{m_-} \pi_{-,i}^\vee, \label{eq:pi_up_n+1}
    \end{align}
    where
    \begin{itemize}[label=\tiny$\bullet$]
        \item all the $\pi_{\hdots}$ are discrete automorphic representations of the corresponding block of $M_P$,
        \item for the relevant $i$, we have $\pi_{1,i}=\Speh(\sigma_{1,i},\underline{d}(1,i))$, $\pi_{2,i}=\Speh(\sigma_{2,i},\underline{d}(2,i))$ for some representations $\sigma_{1,i} \in \Pi_\cusp(\GL_{\underline{r}(1,i)})$ and $\sigma_{2,i} \in \Pi_\cusp(\GL_{\underline{r}(2,i)})$,
        \item $\underline{d}(1,1) \geq \hdots \geq \underline{d}(1,m_1) \geq 2$ and $\underline{d}(2,1) \geq \hdots \geq \underline{d}(2,m_2) \geq 2$.
        \item for the relevant $i$, we have $\pi_{\mathrm{c},1,i} \in \Pi_\cusp(\GL_{\underline{n}_\mathrm{c}(1,i)})$, $\pi_{\mathrm{c},2,i} \in \Pi_\cusp(\GL_{\underline{n}_\mathrm{c}(2,i)})$,
        \item for $1 \leq j \leq m$, we have $\pi_{1,i_1(j)}\simeq \pi_{2,i_2(j)}^\vee$ (and therefore $\underline{d}(1,i_1(j))=\underline{d}(2,i_2(j))$ and $\underline{r}(1,i_1(j))=\underline{r}(2,i_2(j))$).
    \end{itemize}
    \end{itemize}
    We identify $\underline{n}'(1)$ and $\underline{n}'(2)$ with tuples indexed by $\{1, \hdots, m_1\} \setminus I_1$ and $\{1, \hdots, m_2\} \setminus I_2$ respectively. In particular, we have for $i \in \{1,2\}$ and $1 \leq j \leq m_i$ the formulae $\underline{n}(i,j)=\underline{d}(i,j)\underline{r}(i,j)$, and for $j \notin I_i$, $\underline{n}'(i,j)=(\underline{d}(i,j)-1)\underline{r}(i,j)$. Note that all the $\pi_{1,j}$ and $\pi_{2,j}$ are residual. If $I_1=I_2=\emptyset$, we simply write $(I,P,\pi) \in \Pi_H^\uparrow$. 

    The elements in $\fa_{P,\cc}^*$ decompose as 
    \begin{equation}
    \label{eq:notation_P_up}
    \lambda=\left( (\lambda(+)_n, \lambda(2)_n,\lambda(1)_n, \lambda(1)_{\mathrm{c}},\lambda(-)_n),(\lambda(+)_{n+1}, \lambda(1)_{n+1},\lambda(2)_{n+1},\lambda(2)_{\mathrm{c}},\lambda(-)_{n+1}) \right).
    \end{equation}
    We define
    \begin{equation}
\label{eq:a_pi_up_defi}
    \fa_{\pi}^*=\left\{ \lambda \in \fa_{P}^* \; \middle| \; \begin{array}{ll}
        \lambda(+)_n=-\lambda(+)_{n+1}, &  \\
        \lambda(1)_{n,i}=-\lambda(1)_{n+1,i}, & 1 \leq i \leq m_1, \; \text{if } \underline{d}(1,i) \neq 1 \text{ and } i \notin I_1, \\
        \lambda(2)_{n,i}=-\lambda(2)_{n+1,i}, & 1 \leq i \leq m_2, \; \text{if } \underline{d}(2,i) \neq 1\text{ and } i \notin I_2, \\
        \lambda(1)_{n,i_1(j)}=-\lambda(2)_{n,i_2(j)}, & 1 \leq j \leq m, \\
        
         \lambda(-)_n=-\lambda(-)_{n+1}, & 
    \end{array} \right\}.
\end{equation}
This notation is somewhat abusive as this space really depends on the data of $I$, $I_1$ and $I_2$. The dependence should be clear in context and we use similar simplifications throughout this section.

    We have an element $\underline{\rho}_\pi \in \fa_P^*$ defined in the coordinates of \eqref{eq:M_P_increasing} by
     \begin{equation}
    \label{eq:rho_P_down}
        \underline{\rho}_\pi=\left(\left(\underbrace{1/4,\hdots,1/4}_{m_+},0,0,0,\underbrace{-1/4,\hdots,-1/4}_{m_-}\right),\left(\underbrace{1/4,\hdots,1/4}_{m_+},0,0,0,\underbrace{-1/4,\hdots,-1/4}_{m_-}\right)\right).
    \end{equation}
    We will also need the variation
     \begin{equation}
    \label{eq:rho_P_up}
        \underline{\rho}_\pi^{\uparrow}=\left(\left(0,\underbrace{1/4,\hdots,1/4}_{m_2-m},\underbrace{-1/4,\hdots,-1/4}_{m_1},0,0\right),\left(0,\underbrace{1/4,\hdots,1/4}_{m_1-m},\underbrace{-1/4,\hdots,-1/4}_{m_2},0,0\right)\right) \in \fa_P^*.
    \end{equation}
    We emphasize that it does not belong to $\fa_\pi^*$ except if $I_1=I_2=\emptyset$.

    \subsubsection{Construction of the period}
        \label{subsubsec:regularized_increasing}

    Let $(I,P,\pi,I_1,I_2) \in \Pi_H^{\uparrow}$. We now build a regularized period $\cP_{\pi}^{\uparrow}$ on $\cA_{P,\pi}(G)$. For $i \in \{1,2\}$, let $\underline{n}^-(i) \in \zz_{\geq 0}^{2m_i}$ be the tuple such that $\underline{n}^-(i,2j-1)=(\underline{d}(i,j)-1)\underline{r}(i,j)$ and $\underline{n}^-(i,2j)=\underline{r}(i,j)$ for $1 \leq j \leq m_i$. Let $P_{\pi}^\uparrow$ be the standard parabolic subgroup of $G$ such that
    \begin{equation*}
        \underline{n}(P_{\pi}^\uparrow)=\left( (\underline{n}(+),\underline{n}'(2),\underline{n}^-(1),\underline{n}_{\mathrm{c}}(1),\underline{n}(-)),(\underline{n}(+),\underline{n}'(1),\underline{n}^-(2),\underline{n}_{\mathrm{c}}(2),\underline{n}(-)) \right).
    \end{equation*}
    For $j \in \{1,2\}$, we also define the tuple $\underline{n}_{I_i}^-(i) \in \zz_{\geq 0}^{I_i}$ by $\underline{n}_{I_i}^-(i,j)=(\underline{d}(i,j)-1)\underline{r}(i,j)$ for $j \in I_i$. Note that if we identify $I_1$ and $I_2$ with $\{1, \hdots, m\}$ via the maps $i_1$ and $i_2$, then $\underline{n}_{I_1}^-(1)=\underline{n}^-_{I_2}(2)$. We then let $w^\uparrow$ be the shortest element in $W(P_{\pi}^\uparrow)$ such that, if we set $Q_{\pi}^\uparrow=w^\uparrow. P_\pi^\uparrow$, we have     
  \begin{equation*}
        \underline{n}(Q_{\pi}^\uparrow)=\left( (\underline{n}(+),\underline{n}'(2),\underline{n}'(1),\underline{n}^-_{I_1}(1),\underline{r}(1),\underline{n}_{\mathrm{c}}(1),\underline{n}(-)),
        (\underline{n}(+),\underline{n}'(2),\underline{n}'(1),\underline{n}^-_{I_2}(2),\underline{r}(2),\underline{n}_{\mathrm{c}}(2),\underline{n}(-))\right).
    \end{equation*}
    In words, $w^\uparrow$ only acts on the $\underline{n}^-(1)$ component on the $\GL_n$-side, and on the $(\underline{n}'(1),\underline{n}^-(2))$ component on the $\GL_{n+1}$-side. It sends all the $\GL_{\underline{r}(1,i)}$ and $\GL_{\underline{r}(2,i)}$ blocks at the bottom, and then sorts the $\GL_{(\underline{d}(1,i)-1)\underline{r}(1,i)}$ on $\GL_n$ side, and $\GL_{(\underline{d}(2,i)-1)\underline{r}(2,i)}$ on $\GL_{n+1}$ by bringing those coming from $I_1$ and $I_2$ at the bottom. Finally, it exchanges the resulting $\underline{n}'(1)$ and $\underline{n}'(2)$ blocks on the $\GL_{n+1}$-side.

    Let $P^{\uparrow,\std}_{n+1}$ be the standard parabolic subgroup of $\GL_{n+1}$ such that
     \begin{equation*}
        \underline{n}(P_{n+1}^{\uparrow,\std})=\left( 
        \underline{n}(+),\underline{n}'(2),\underline{n}'(1),\underline{n}^-_{I_2}(2),k+1,\underline{n}(-)\right).
    \end{equation*}
        where
    \begin{equation*}
        k+1=\sum_{i=1}^{m_2} \underline{r}(2,i)+\sum_{i=1}^{m_{\mathrm{c},2}} \underline{n}_{\mathrm{c}}(2,i).
    \end{equation*}
    Let $P^\uparrow$ be the Rankin--Selberg parabolic subgroup of $G$ associated the pair $(P_{n+1}^{\uparrow,\std},m_++m_1+m_2-m+1)$ under the bijection of Proposition~\ref{prop:param_RS}, where once again the $(m_++m_1+m_2-m+1)^\text{th}$ block is $\GL_{k+1}$. We have $w^\uparrow \in {}_{P^{\uparrow,\std}} W_P$ and $P_{w^\uparrow}=P_{\pi}^\uparrow$.

    We set for $\lambda \in \fa_{P,\cc}^*$ in general position and $\varphi \in \cA_{P,\pi}(G)$
    \begin{equation}
        \label{eq:P_up} \cP_\pi^\uparrow(\varphi,\lambda):=\cP^{P^{\uparrow}}(\varphi,\lambda,w^\uparrow).
    \end{equation}
    Note that contrary to the case of $\cP_\pi$, we prefer to view $\cP_\pi^\uparrow$ as a meromorphic function on the whole space $\fa_{P,\cc}^*$ rather than solely on a translate of $\fa_{\pi,\cc}^*$. This is possible by \cite[Lemma~4.15]{BoiZ}. The trade-off is that the linear form $\varphi \mapsto \cP_\pi^\uparrow(\varphi,\lambda)$ is not $H(\bA)$-invariant in general. By reproducing the proof of Proposition~\ref{prop:alternative_construction}, the restriction of $\cP_{\pi}^\uparrow(\varphi,\lambda)$ to $\fa_{\pi,\cc}^*-\underline{\rho}_\pi-\underline{\rho}_\pi^\uparrow$ is well-defined (see also Proposition~\ref{prop:up_is_down} below).

    \subsubsection{Relation with the regularized period} 
    \label{subsubsec:connect_periods}
    
    We connect the regularized period $\cP_\pi^\uparrow$ from \eqref{eq:P_up} to the linear forms $\cP_\pi$ built in Theorem~\ref{thm:P_pi_quotient}. First, for $i \in \{1,2\}$ we write $\underline{n}_{I_i}(i)$ and $\underline{n}_{\setminus I_i}(i)$ for the tuples in $\zz_{\geq 1}^{I_i}$ and $\zz_{\geq 1}^{\{1,\hdots,m_i\} \setminus I_i}$ such that for $1 \leq j \leq m_i$, we have $\underline{n}_{I_i}(i,j)=\underline{n}(i,j)$ if $j \in I_i$, and $\underline{n}_{\setminus I_i}(i,j)=\underline{n}(i,j)$ otherwise. We now let $w^{\downarrow} \in W(P)$ be the shortest element such that, if we set $P^\downarrow=w^\downarrow .P$, then
    \begin{equation*}
        \underline{n}(P^\downarrow)=\left((\underline{n}(+),\underline{n}_{\setminus I_1}(1),\underline{n}_{\mathrm{c}}(1),\underline{n}'(2),\underline{n}_{I_1}(1),\underline{n}(-)),(\underline{n}(+),\underline{n}'(1),\underline{n}_{\setminus I_2}(2),\underline{n}_{\mathrm{c}}(2),\underline{n}_{I_2}(2),\underline{n}(-))\right).
    \end{equation*}
    Set $\pi^\downarrow:=w^\downarrow \pi$. Define
    \begin{equation*}
        N_1=\sum_{j \notin I_1} \underline{n}(1,j) + \sum_{j=1}^{m_{{\mathrm{c},1}}} \underline{n}_{\mathrm{c}}(1,j), \quad N_2=\sum_{j \notin I_2} \underline{n}(2,j) + \sum_{j=1}^{m_{{\mathrm{c},2}}} \underline{n}_{\mathrm{c}}(2,j), \quad N_-=\sum_{j \in I_1} \underline{n}(1,j) +n_-.
    \end{equation*}
    Then we have $((n_+,N_1,N_2,N_-),P^\downarrow,\pi^\downarrow) \in \Pi_H$ and we can therefore consider the regularized period $\cP_{\pi^\downarrow}$ built in Theorem~\ref{thm:P_pi_quotient}. Note that we have
    \begin{equation}
        \label{eq:up_is_down_vector}
        w^\downarrow (\fa_{\pi,\cc}^*-\underline{\rho}_\pi-\underline{\rho}_\pi^\uparrow)=\fa_{\pi^\downarrow,\cc}^*-\underline{\rho}_{\pi^\downarrow}
    \end{equation}

    \begin{prop}
    \label{prop:up_is_down}
        Let $\varphi \in \cA_{P,\pi}(G)$. For $\lambda \in \fa_{\pi,\cc}^*-\underline{\rho}_\pi-\underline{\rho}_\pi^{\uparrow}$ in general position we have
        \begin{equation}
        \label{eq:up_is_down}
            \cP_{\pi}^{\uparrow}(\varphi,\lambda)=\cP_{\pi^\downarrow}(M(w^\downarrow ,\lambda)\varphi,w^\downarrow\lambda).
        \end{equation}
    \end{prop}

    \begin{proof}
        We begin by expressing each side of \eqref{eq:up_is_down} as a residue of a regularized period induced from $\sigma_\pi$. By parabolic descent (Lemma~\ref{lem:parabolic_descent}), it is enough to deal with the case $m_+=m_-=0$. We decompose the standard Levi $M_{P_\pi}$ of $P_\pi$ with respect to $M_{P_\pi} \subset M_P$. This yields an identification 
        \begin{equation*}
            \fa_{P_\pi,n,\cc}^* =\prod_{i \notin I_2} \cc^{\underline{d}(2,i)-1}  \times \prod_{i=1}^{m_1} \cc^{\underline{d}(1,i)} \times  \prod_{i=1}^{m_{\mathrm{c},1}} \cc, \quad   \fa_{P_\pi,n+1,\cc}^* = \prod_{i \notin I_1} \cc^{\underline{d}(1,i)-1} \times \prod_{i=1}^{m_2} \cc^{\underline{d}(2,i)} \times \prod_{i=1}^{m_{\mathrm{c},2}} \cc.
        \end{equation*}
        We now define affine linear forms as in the proof of Theorem~\ref{thm:P_pi_quotient}. We use the same convention for coordinates. We begin with the set $\Lambda^\uparrow_+$ of affine linear forms described by
        \begin{equation*}
            \left\{
                \begin{array}{ll}
                    \Lambda_+(1,i,j)(\lambda)=-(\lambda(1,i)_{n,\underline{d}(1,i)-j+1}+\lambda(1,i)_{n+1,j}+1/2), &  \quad  \left\{ \begin{array}{l} 
                        1 \leq i \leq m_1, \; i \notin I_1, \\
                         1 \leq j \leq \underline{d}(1,i)-1,
                            \end{array} \right. \\
                    \Lambda_+(2,i,j)(\lambda)=-(\lambda(2,i)_{n,j}+\lambda(2,i)_{n+1,\underline{d}(2,i)-j+1}+1/2), &  \quad  \left\{ \begin{array}{l} 
                                1 \leq i \leq m_2, \; i \notin I_2, \\
                                 1 \leq j \leq \underline{d}(2,i)-1,
                                    \end{array} \right. 
                \end{array}
            \right.
        \end{equation*}
        as well as the set $\Lambda^{\uparrow}_-$ of
        \begin{equation*}
            \left\{
                \begin{array}{ll}
                    \Lambda_-(1,i,j)(\lambda)=\lambda(1,i)_{n,\underline{d}(1,i)-j}+\lambda(1,i)_{n+1,j}-1/2, &  \quad  \left\{ \begin{array}{l} 
                        1 \leq i \leq m_1, \; i \notin I_1, \\
                         1 \leq j \leq \underline{d}(1,i)-1,
                            \end{array} \right. \\
                    \Lambda_-(2,i,j)(\lambda)=\lambda(2,i)_{n,j}+\lambda(2,i)_{n+1,\underline{d}(2,i)-j}-1/2, &  \quad  \left\{ \begin{array}{l} 
                                1 \leq i \leq m_2, \; i \notin I_2, \\
                                 1 \leq j \leq \underline{d}(2,i)-1.
                                    \end{array} \right. \\
                \end{array}
            \right.
        \end{equation*}
        We then add the set $\Lambda_+^{\uparrow,'}$ defined by
        \begin{equation*}
            \left\{
                \begin{array}{ll}
                    \Lambda_+'(i,j)(\lambda)=-(\lambda(1,i_1(i))_{n,j}+\lambda(2,i_2(i))_{n+1,\underline{d}(2,i)-j+2}+1/2), &  \quad \left\{ \begin{array}{l}
                        1 \leq i \leq m, \\
                        2 \leq j \leq \underline{d}(1,i_1(i)),
                         \end{array} \right.  
                \end{array}
            \right.
        \end{equation*}
        as well as $\Lambda_-^{\uparrow,'}$ given by
        \begin{equation*}
            \left\{
                \begin{array}{ll}
                    \Lambda_-'(i,j)(\lambda)=\lambda(1,i_1(i))_{n,j}+\lambda(2,i_2(i))_{n+1,\underline{d}(2,i)-j+1}-1/2, &  \quad \left\{ \begin{array}{l}
                        1 \leq i \leq m, \\
                        1 \leq j \leq \underline{d}(1,i_1(i)),
                         \end{array} \right.  
                \end{array}
            \right.
        \end{equation*}
        Let $\cH$ be the intersection of all the zero sets of the affine linear forms in $\Lambda_+^\uparrow$, $\Lambda_-^{\uparrow}$, $\Lambda_+^{\uparrow,'}$ and $\Lambda_-^{\uparrow,'}$. Then we have 
        \begin{equation*}
            \cH=\fa^*_{\pi,\cc}-\underline{\rho}_\pi-\underline{\rho}_\pi^\uparrow-\nu_\pi=(w^{\downarrow})^{-1}\left( \fa_{\pi^\downarrow,\cc}^*-\underline{\rho}_{\pi^\downarrow}-\nu_{\pi^\downarrow}\right).
        \end{equation*}
        Let $\varphi \in \cA_{P,\pi}(G)$. Assume that $\varphi=E^*(\phi,-\nu_\pi)$ with $\phi \in \cA_{P_\pi,\sigma_\pi}(G)$. By reproducing the proof of Theorem~\ref{thm:P_pi_quotient} (see also \cite[Theorem~5.2]{BoiZ}), we see that for $\lambda \in \fa_{\pi,\cc}^*-\underline{\rho}_\pi-\underline{\rho}_{\pi}^\uparrow$ in general position we have
        \begin{equation}
            \label{eq:first_residue_up}
            \left(\underset{\Lambda_-^{\uparrow,'}}{\Res} \; \underset{\Lambda^\uparrow_-}{\Res} \; \underset{\Lambda_+^{\uparrow,'}}{\Res}\; \underset{\Lambda_+^\uparrow}{\Res}\right)Z_{\sigma_\pi}(\phi)(\lambda-\nu_\pi)=\cP_\pi^\uparrow(\varphi,\lambda).
        \end{equation}
        Moreover, $\Lambda_+^\uparrow \cup \Lambda_-^{\uparrow,'}$ and $\Lambda_-^{\uparrow} \cup \Lambda_+^{\uparrow,'}$ are the sets used in Theorem~\ref{thm:P_pi_quotient} and \cite[Theorem~5.2]{BoiZ} to define the period $\cP_{\pi^\downarrow}$ for $(P^\downarrow,\pi^\downarrow)$. By reproducing once again the argument, we end up at
        \begin{equation}
            \label{eq:second_residue_up}
            \left(\underset{\Lambda_+^{\uparrow,'}}{\Res} \; \underset{\Lambda_-^{\uparrow}}{\Res} \;  \underset{\Lambda_-^{\uparrow,'}}{\Res}\; \underset{\Lambda_+^\uparrow}{\Res}\right)Z_{\sigma_\pi}(\phi)(\lambda-\nu_\pi)=\cP_{\pi^\downarrow}(M(w^\downarrow,\lambda)\varphi,w^\downarrow \lambda).
        \end{equation}
        By the same argument as in \cite[Lemma~5.1]{BoiZ}, we can compute these residues in any order. Therefore, \eqref{eq:first_residue_up} and \eqref{eq:second_residue_up} are equal, which concludes the proof.
    \end{proof}

    \begin{rem}
        We spell out the content of Proposition~\ref{prop:up_is_down} in a simple example. We take $n=2$, so that $G=\GL_2 \times \GL_3$. We consider the tuple $((0,2,0,2,1,0),P,\pi,\{1\},\{1\})$ where $P$ is the standard parabolic subgroup of $G$ with standard Levi $(\GL_2) \times (\GL_2 \times \GL_1)$ and $\pi$ is the trivial representation of $M_P$. We have $\fa_{P,\cc}^* \simeq \cc \times \cc^2$, and with respect to these coordinates, $\fa_{\pi,\cc}^*=\{(a,(-a,b)) \; | \; a,b \in \cc \}$ and $\underline{\rho}_\pi+\underline{\rho}_\pi^\uparrow=(-1/4,(-1/4,0))$. Here $P^\uparrow$ is the Borel subgroup, and $w_+=1$. We take $\varphi^\circ \in \cA_{P,\pi}(G)$ the unique $K$-invariant vector such that $\varphi^\circ(1)=1$. We assume that $F=\qq$, and denote by $\zeta$ the completed zeta function of $\qq$. Using \cite{CS} and the Archimedean computations of \cite{Sta}, and taking into account the computation of $\varphi^\circ_{P_0}$ in Lemma~\ref{lem:contant_term}, we see that for $\lambda=(x,(y,z)) \in \fa_{P,\cc}^*$ 
        \begin{equation*}
            \cP^{P^\uparrow}(\varphi^\circ,\lambda)=\vol([\GL_1]_0) \times \frac{\zeta(x+y+3/2)\zeta(x+z+1)}{\zeta(y-x+3/2)},
        \end{equation*}
        where the volume of $[\GL_1]_0$ is taken with respect to the measure of \S\ref{subsubsec:first_measure}. This meromorphic function is well defined for $\lambda=(a+1/4,(-a+1/4,b)) \in \fa_{\pi,\cc}^*-\underline{\rho}_\pi-\underline{\rho}_\pi^\uparrow$ in general position and for such $\lambda$ we get 
        \begin{equation*}
            \cP_{\pi}^\uparrow(\varphi^\circ,\lambda)=\vol([\GL_1]_0) \times \frac{\zeta(2)\zeta(a+b+5/4)}{\zeta(-b-a+7/4)},
        \end{equation*}
        On the other hand, $P^\downarrow$ has standard Levi $(\GL_2) \times (\GL_1 \times \GL_2)$, and $w_+$ exchanges the $\GL_2$ and $\GL_1$ blocks on the $\GL_{3}$-side. Because of the factorization of $M(w^\uparrow,\lambda)$ in \eqref{eq:normalization_non_twisted}, we get 
        \begin{equation*}
            \cP_{\pi^\downarrow}(M(w^\downarrow ,\lambda)\varphi^\circ,w^\downarrow\lambda)=\vol([\GL_2]_0) \times \frac{\zeta(-b-a-1/4)}{\zeta(-b-a+7/4)}.
        \end{equation*}
        Therefore, Proposition~\ref{prop:up_is_down} amounts to the two equalities
        \begin{equation*}
            \vol([\GL_1]_0) \zeta(2)=\vol([\GL_2]_0) \quad \text{and} \quad \zeta(a+b+5/4)=\zeta(-a-b-1/4).
        \end{equation*}
        The first equation is equivalent to the computation of the Tamagawa volume of $\GL_n$ by \cite{Langlands65}, and the second is the functional equation of the $\zeta$ function.
    \end{rem}

    \subsubsection{Functional equations}
    \label{subsubsec:functional_up}

    We now describe some subsets of the Weyl group $W$ of $G$. We first define $W_{\bfM}(\pi)$. This is the subset of the set of $w=(w_n,w_{n+1}) \in W$ that act by blocks on $\bfM_{P^{\uparrow}}^{\std,2}$ and correspond to permutations
    \begin{equation*}
        ((\sigma_+,\sigma_{1,2},\sigma_-),(\sigma_+,\sigma_{1,2},\sigma_-)), \quad \sigma_+ \in \fS(m_+), \quad \sigma_{1,2} \in \fS(m_1+m_2-|I_1|), \quad \sigma_- \in \fS(m_-).
    \end{equation*}
    This set really depends on $I$ and $I_1$, $I_2$ but we drop the reference from the notation. We can also identify any such $w$ with an element of $\bfM_{P^\uparrow}^2$ (by conjugating by $w_{P^\uparrow}^\std$), and we will again denote it by $w$. In particular, we have the Rankin--Selberg parabolic subgroup $w. P^\uparrow$.

    We then define $W_{\cM}(\pi)$ to be the subset of elements acting by blocks on the standard parabolic subgroup $(w^\uparrow.P^{\uparrow}_\pi) \cap \cM_{P^{\uparrow,\std}}$. Its elements are identified with couples $(\sigma_n,\sigma_{n+1}) \in \fS(m_{c,1}+m_1) \times \fS(m_{c,2}+m_2)$. The next result follows from Lemma~\ref{lem:functional_equation_1} and Proposition~\ref{prop:up_is_down}.

     \begin{lem}
    \label{lem:functional_up}
         Let $\varphi \in \cA_{P,\pi}(G)$. Let $w_\bfM \in W_{\bfM}(\pi)$, let $w_\cM \in W_{\cM}(\pi)$. Then we have for $\lambda \in \fa_{\pi,\cc}^*-\underline{\rho}_\pi-\underline{\rho}_\pi^{\uparrow}$ in general position
        \begin{equation*}
            \cP_{\pi}^\uparrow(\varphi,\lambda)=\cP^{w_{\bfM}.P^{\uparrow}}\left(M(w_\bfM w_\cM w^\uparrow,\lambda)\varphi_{P_{\pi}^\uparrow},w_\bfM w_\cM w^\uparrow \lambda\right).
        \end{equation*}
    \end{lem}

    \subsubsection{Analytic properties of $\cP_\pi^\uparrow$}
    \label{subsubsec:poles_of_P_pi_up} 
    Consider the map
    \begin{equation*}
        \lambda \in \fa_{\pi,\cc}^*-\underline{\rho}_\pi-\cc \underline{\rho}_\pi^\uparrow \mapsto \left( \lambda(+)_n,\lambda(1)_n, \lambda(1)_{\mathrm{c}},\lambda(2)_{n+1},\lambda(2)_{\mathrm{c}},\lambda(-)_n \right).
    \end{equation*}
    It follows from the definition that it is injective. Indeed, for $i \notin I_1$ we have $\lambda(1)_{n,i}=-\lambda(1)_{n+1,i}$, and for $i \notin I_2$ we have $\lambda(2)_{n,i}=-\lambda(2)_{n+1,i}$. We set
     \begin{equation}
    \label{eq:lambda_reg_pi}
        \left( \lambda(+),\lambda(1),\lambda(1)_{\mathrm{c}},\lambda(2),\lambda(2)_{\mathrm{c}},\lambda(-) \right)=\left( \lambda(+)_n,\lambda(1)_n, \lambda(1)_{\mathrm{c}},\lambda(2)_{n+1},\lambda(2)_{\mathrm{c}},\lambda(-)_n \right).
    \end{equation}
    In these coordinates, $-\underline{\rho}_\pi^\uparrow$ is $(0,1/4,0,1/4,0,0)$. We now set 
    \begin{equation*}
        \lambda_{1,2}=(\lambda(1),\lambda(2)).
    \end{equation*}
    Let $\pi^\uparrow \in \Pi_\disc(M_{P_\pi^\uparrow})$ and $\nu^\uparrow \in \fa_{P_\pi^\uparrow}^*$ such that for any $\varphi \in \cA_{P,\pi}(G)$ we have $\varphi_{P_\pi^\uparrow} \in \cA_{P_\pi^\uparrow,\pi^\uparrow,\nu^\uparrow}(G)$. If $\mu \in w^\uparrow \fa_{P_\pi^\uparrow}^*$, we write $\mu=\mu_{\cP^\uparrow}+\mu_{\bfP^\uparrow}$ for its decomposition with respect to $M_{P^{\uparrow}}^\std=\cM_{P^{\uparrow}}^\std \times \bfM_{P^{\uparrow}}^{\std,2}$. Then we have
    \begin{align*}
        (w^\uparrow (\lambda +\nu^\uparrow))_{\cP^\uparrow,n}&=\left(\lambda(1)_1+\frac{\underline{d}(1,1)-1}{2},\hdots,\lambda(1)_{m_1}+\frac{\underline{d}(1,m_1)-1}{2} ,\lambda(1)_\mathrm{c}\right), \\
         (w^\uparrow (\lambda +\nu^\uparrow))_{\cP^\uparrow,n+1}&=\left(\lambda(2)_1+\frac{\underline{d}(2,1)-1}{2},\hdots,\lambda(2)_{m_2}+\frac{\underline{d}(2,m_2)-1}{2} ,\lambda(2)_\mathrm{c}\right).
    \end{align*}
    With the notation of \eqref{eq:pi_up_n} and \eqref{eq:pi_up_n+1}, we set
    \begin{align*}
        \pi_{1,2}&=\boxtimes_{i=1}^{m_1} \pi_{1,i} \boxtimes_{i=1}^{m_2} \pi_{2,i}, \\
        \pi_{\cP}&=\left(\boxtimes_{i=1}^{m_1} \sigma_{1,i}  \boxtimes_{i=1}^{m_{\mathrm{c},1,i}} \pi_{\mathrm{c},1,i} \right) \boxtimes \left(  \boxtimes_{i=1}^{m_2} \sigma_{2,i}\boxtimes_{i=1}^{m_{\mathrm{c},2}} \pi_{\mathrm{c},2,i}  \right),
    \end{align*}
    where we recall that $\pi_{1,i}=\Speh(\sigma_{1,i},\underline{d}(1,i))$ and $\pi_{2,i}=\Speh(\sigma_{2,i},\underline{d}(2,i))$. Finally, we set
    \begin{align}
          L_{\pi,\cP}^{\uparrow}&(\lambda)=\prod_{  \pi_{\mathrm{c},1,i} \simeq \pi_{\mathrm{c},2,j}^{\vee}} \left(\lambda(1)_{\mathrm{c},i}+\lambda(2)_{\mathrm{c},j}\pm\frac{1}{2}\right)\nonumber \\          
          & \times \prod_{  \pi_{\mathrm{c},1,i} \simeq \sigma_{2,j}^{\vee}} \left(\lambda(1)_{\mathrm{c},i}+\lambda(2)_j+\frac{\underline{d}(2,j)-1\pm1}{2}\right) 
           \prod_{  \sigma_{1,i}^{\vee} \simeq \pi_{\mathrm{c},2,j}} \left(\lambda(1)_{i}+\lambda(2)_{\mathrm{c},j}+\frac{\underline{d}(1,i)-1\pm1}{2}\right) \nonumber \\
          &\times \prod_{\substack{\pi_{1,i} \simeq \pi_{2,j}^{-,\vee} \\ i \notin I_1, j \notin I_2 }} \left(\lambda(1)_i+\lambda(2)_j\right) \prod_{ \substack{\pi_{1,i}^{-,\vee} \simeq \pi_{2,j} \\  i \notin I_1, j \notin I_2 }} \left(\lambda(1)_i+\lambda(2)_j\right). \label{eq:f_pi_regular}
    \end{align}

    \begin{prop}
        \label{prop:reg_P_pi_up} 
         There exists $k>0$ such that for every level $J$ of $G$ there exists $c_J>0$ such that for every $J$-pair $(P,\pi)$ with $(I,P,\pi,I_1,I_2) \in \Pi_H^\uparrow$, and every $\varphi \in \cA_{P,\pi}(G)$ the meromorphic function
        \begin{equation*}
            \lambda \mapsto L_{\pi,\cP}^\uparrow(\lambda)\cP^\uparrow_{\pi}(\varphi,\lambda)
        \end{equation*}
        is regular in the region
        \begin{equation}
        \label{eq:weird_region}
           \left\{ \lambda \in \fa_{\pi,\cc}^* -\underline{\rho}_\pi-t\underline{\rho}_\pi^\uparrow  \; \middle| \;  (\lambda+t\underline{\rho}_\pi^\uparrow)_{1,2} \in \cS_{\pi_{1,2},k,c_J}, \; (w^\uparrow(\lambda+\nu^\uparrow))_\cP \in \cR_{\pi_\cP,k,c_J}, \; 0 \leq t \leq 1 \right\}.
        \end{equation}
    \end{prop}

    \begin{rem}
    \label{rem:right_order} We make the following remarks on Proposition~\ref{prop:reg_P_pi_up}.
    \begin{itemize}
        \item The region \eqref{eq:weird_region} is non-empty because $\underline{d}(1,1) \geq \hdots \geq \underline{d}(1,m_1)$ and $\underline{d}(2,1) \geq \hdots \geq \underline{d}(2,m_2)$.
        \item The last two factors of \eqref{eq:f_pi_regular} can only be zero if $t$ is close to zero.
        \item If $\lambda=\mu-\underline{\rho}_\pi-t \underline{\rho}_\pi^{\uparrow}$ with $0 \leq t \leq 1$, then $(w^\uparrow(\mu+\nu^\uparrow))_\cP \in \cR_{\pi_\cP,k,c_J}$ implies that $(w^\uparrow(\lambda+\nu^\uparrow))_\cP \in \cR_{\pi_\cP,k,c_J}$.
        \item The gain of Proposition~\ref{prop:reg_P_pi_up} in comparison with Proposition~\ref{prop:reg_P_pi} (poles of $\cP_\pi$) is that we have more freedom on $\lambda(1)_\mathrm{c}$ and $\lambda(2)_\mathrm{c}$.
    \end{itemize}
    \end{rem}

    \begin{proof}
       The proof follows the same pattern as Proposition~\ref{prop:reg_P_pi} so we will be brief. Let $\lambda$ be in the region \eqref{eq:weird_region}. Write $\lambda=\mu-\underline{\rho}_\pi-t \underline{\rho}_\pi^\uparrow$, so that $\mu_{1,2} \in \cS_{\pi_{1,2},k,c_J}$. By the description of \S\ref{subsubsec:regularized_increasing}, we see that for any $\alpha \in \Delta_{P_{\pi}^\uparrow}$, $w^\uparrow \alpha<0$ implies that $\langle \nu^\uparrow-\underline{\rho}_\pi-t\underline{\rho}_{\pi}^\uparrow,\alpha^\vee \rangle \geq 0$ as long as $0 \leq t \leq 1$. Therefore, the condition $ \mu_{1,2} \in \cS_{\pi_{1,2},k,c_J}$ implies that $\mu-\underline{\rho}_\pi-t \underline{\rho}_{\pi}^\uparrow+\nu^\uparrow \in \cR_{\pi,k,c_J}(w^\uparrow)$. By direct computation, we see that we can take
       \begin{equation*}
           L_{\pi,w^\uparrow}(\lambda)= \prod_{\substack{i<j \\ \pi_{1,i} \simeq \pi_{1,j} }} (\lambda(1)_i-\lambda(1)_j )  \prod_{\substack{i<j \\ \pi_{2,i} \simeq \pi_{2,j} }} (\lambda(2)_i-\lambda(2)_j ) \prod_{\substack{\pi_{1,i} \simeq \pi_{2,j}^{-,\vee} \\ \underline{d}(2,j) \neq 2  \\ i \notin I_1, j \notin I_2}} (\lambda(1)_i+\lambda(2)_j )  \prod_{\substack{\pi_{1,i}^{-,\vee} \simeq \pi_{2,j} \\ \underline{d}(1,i) \neq 2  \\ i \notin I_1, j \notin I_2}} (\lambda(1)_i+\lambda(2)_j ).
       \end{equation*}
       Note that the conditions $\underline{d}(2,j) = 2$ and $\underline{d}(1,i) = 2$ are superfluous as they would imply $\pi_{1,i}$ or $\pi_{2,j}$ cuspidal, which is not possible by the definition of $\Pi^\uparrow_H$. It remains to compute the factor $L_{(w^\uparrow \pi^\uparrow)_{\cP},Z}$ of \eqref{eq:f_pi_Z_defi}. But this is the same calculation as \eqref{eq:L_pi_P} in the proof of Proposition~\ref{prop:reg_P_pi}. The only difference is that we a priori see the factor 
       \begin{equation*}
         \prod_{\substack{\pi_{1,i} \simeq \pi_{2,j}^{-,\vee} \\ \underline{d}(2,j) = 2 }} (\lambda(1)_i+\lambda(2)_j )  \prod_{\substack{\pi_{1,i}^{-,\vee} \simeq \pi_{2,j} \\ \underline{d}(1,i) =2  }} (\lambda(1)_i+\lambda(2)_j ).
       \end{equation*}
       But as we just seen this condition is never met. Putting everything together, we obtain Proposition~\ref{prop:reg_P_pi_up}.
    \end{proof}
    
    We can also state a bound for the regularized period $\cP_{\pi}^\uparrow(\varphi,\lambda)$.  

\begin{prop}
\label{prop:bound_P_pi_up}
    There exists $k>0$ such that for any level $J$ and $C>0$ there exist $c_J>0$, $d>0$, $N>0$ and $X_1, \hdots, X_r \in \cU(\fg_\infty)$ such that for any $J$-pair $(P,\pi) \in \Pi$ with $(I,P,\pi,I_1,I_2) \in \Pi_H^\uparrow$ and any $\varphi \in \cA_{P,\pi}(G)^J$ we have
    \begin{equation}
    \label{eq:individual_P_pi_bound_up}
        \Val{L_{\pi,\cP}^\uparrow(\lambda) \cP_\pi^\uparrow(\varphi,\lambda)} \leq (1+\norm{\lambda}^2)^d \sum_{i=1}^r \norm{\varphi}_{-N,X_i}.
    \end{equation}
    in the region
      \begin{equation}
        \label{eq:weird_region2}
            \left\{  \lambda \in \fa_{\pi,\cc}^* -\underline{\rho}_\pi-t\underline{\rho}_\pi^\uparrow  \; \middle| \;  (\lambda+t\underline{\rho}_\pi^\uparrow)_{1,2} \in \cS_{\pi_{1,2},k,c_J}, \; (w^\uparrow(\lambda+\nu^\uparrow))_\cP \in \cR_{\pi_\cP,k,c_J}, \; 0 \leq t \leq 1, \; \norm{\Re(\lambda)}< C \right\}.
        \end{equation}
\end{prop}

\begin{proof}
    This is exactly the same proof as Proposition~\ref{prop:individual_bound_P_pi}.
\end{proof}

    \subsubsection{The $\emptyset$-transformation}
    \label{subsubsec:empty_transformation}
    We finally explain how to associate to $(I,P,\pi,I_1,I_2)$ another element $(I_\emptyset,P_\emptyset,\pi_\emptyset) \in \Pi_H^\uparrow$ that will be relevant in our proof of the fine spectral expansion of the Rankin--Selberg period. We recall that this implies that $I_{\emptyset,1}=I_{\emptyset,2}=\emptyset$, hence the notation. As in \S\ref{subsubsec:connect_periods} we define an element $w_\emptyset \in W(P)$ such that $P_\emptyset:=w_\emptyset .P$ satisfies
     \begin{equation*}
        \underline{n}(P_\emptyset)=\left((\underline{n}(+),\underline{n}'(2),\underline{n}_{\setminus I_1}(1),\underline{n}_{\mathrm{c}}(1),\underline{n}_{I_1}(1),\underline{n}(-)),(\underline{n}(+),\underline{n}'(1),\underline{n}_{\setminus I_2}(2),\underline{n}_{\mathrm{c}}(2),\underline{n}_{I_2}(2),\underline{n}(-))\right).
    \end{equation*}
    The standard Levi factor is very close to $M_{P^\downarrow}$, except that the second and third product of blocks on the $\GL_n$ side are put in the appropriate order to match the definition of \eqref{eq:M_P_increasing}. We then set 
    \begin{equation*}
        I_\emptyset=\left(n_+,\left(\sum_{j \in \{1,\hdots,m_1\} \setminus I_1} \underline{n}(1,j)\right),n_{\mathrm{c},1},\left(\sum_{j \in \{1,\hdots,m_2\} \setminus I_2}\underline{n}(2,j)\right),n_{\mathrm{c},2},\left(\sum_{j \in I_1} \underline{n}(1,j)\right) + n_- \right).
    \end{equation*}
    We have
    \begin{equation}
        \label{eq:empty_relation}
        w_\emptyset\left( \underline{\rho}_\pi+\underline{\rho}_\pi^\uparrow\right)=\underline{\rho}_{\pi_\emptyset}+\underline{\rho}_{\pi_\emptyset}^\uparrow.
    \end{equation}
    Note that because $I_{\emptyset,1}=I_{\emptyset,2}=\emptyset$, the element $\underline{\rho}_{\pi_\emptyset}^\uparrow$ actually belongs to $\fa_{\pi_\emptyset}^*$.

    As in Proposition~\ref{prop:up_is_down} we can relate the two regularized periods.

      \begin{lem}
    \label{lem:up_is_empty}
        Let $\varphi \in \cA_{P,\pi}(G)$. For $\lambda \in \fa_{\pi,\cc}^*-\underline{\rho}_\pi-\underline{\rho}_\pi^{\uparrow}$ in general position we have
        \begin{equation*}
            \cP_{\pi}^{\uparrow}(\varphi,\lambda)=\cP_{\pi_\emptyset}^\uparrow(M(w_\emptyset ,\lambda)\varphi,w_\emptyset\lambda).
        \end{equation*}
    \end{lem}

    \begin{proof}
        This can be proved as Proposition~\ref{prop:up_is_down}.
    \end{proof}

    Finally, let us note that $\pi_\emptyset$ is closely related to $\pi^\downarrow$ defined in \S\ref{subsubsec:connect_periods}. We sum up their relations in the following lemma.

    \begin{lem}
        \label{lem:down_is_empty}
        We have $\pi^\downarrow=(\pi_\emptyset)^\downarrow$. Moreover, let $w_\emptyset^\downarrow$ be the element "$w^\downarrow$" built for $\pi_\emptyset$ in \S\ref{subsubsec:connect_periods}. Then we have $w^\downarrow=w_\emptyset^\downarrow w_\emptyset$ and $ w_\emptyset^\downarrow\underline{\rho}_{\pi_\emptyset}=\underline{\rho}_{\pi^\downarrow}$.
    \end{lem}

\section{Expansion of the Rankin--Selberg period}
\label{sec:pseudo_spectral}
The goal of this section is to compute the fine spectral expansion of the Rankin--Selberg period of Theorem~\ref{thm:spectral_expansion_intro}. 

\subsection{The spectral expansion}

We first precisely write the result that we prove in \S\ref{sec:pseudo_spectral}, and then prove that it is equivalent to that of Theorem~\ref{thm:spectral_expansion_intro}.

\subsubsection{Measures} \label{subsubsec:i_a_pi_measure} Let $(I,P,\pi) \in \Pi_H$. We give the $\rr$-vector space $i \fa_\pi^*$ a Haar measure. Recall that in \S\ref{subsubsec:first_measure} for any algebraic group $G'$ we have equipped $\fa_{G'}$ with the Haar measure giving covolume $1$ to $\Hom(X^*(G'),\zz)$, and $i \fa_{G'}^*$ with the dual Haar measure. With the notation of \S\ref{subsubsec:relevant_inducing}, let $G'$ be the algebraic group
\begin{equation*}
    G'=\prod_{i=1}^{m_+} \GL_{\underline{n}(+,i)} \prod_{i=1}^{m_1} \GL_{\underline{d}(1,i)\underline{r}(1,i)} \prod_{i=1}^{m_2} \GL_{\underline{d}(2,i)\underline{r}(2,i)} \prod_{i=1}^{m_-} \GL_{\underline{n}(-,i)}.
\end{equation*}
By choosing the canonical basis of $X^*(G')$ as a basis for $\fa_{G'}^*$, \eqref{eq:coord_a_pi} becomes an isomorphism $\fa_\pi^* \simeq \fa_{G'}^*$. We equip $i \fa_\pi^*$ with the pushforward of the measure from $i \fa_{G'}^*$. Note that if $m_+=m_-=0$ and if $\pi$ is cuspidal, so that $\fa_\pi^*=\fa_P^*$, we get back the Haar measure that we equipped $i \fa_P^*$ with in \S\ref{subsubsec:first_measure}.

We also take the opportunity to fix measure for the increasing inducing data. Let $(I,P,\pi,I_1,I_2) \in \Pi_H^\uparrow$ be an increasing distinguished datum. Recall that we have defined in \S\ref{subsubsec:connect_periods} an associated triple $(I^\downarrow,P^\downarrow,\pi^\downarrow) \in \Pi_H$, which is obtained by conjugating by the element $w^\downarrow$. We equip $i \fa_{\pi}^*$ with the pushforward of the measure on $i \fa_{\pi^\downarrow}^*$ we juste described.

\subsubsection{The main result}

The goal of this section is to prove the following theorem. Recall that for each $(P,\pi) \in \Pi_H$ we have defined in \S\ref{sec:functional_equation} a subset $W(\pi) \subset W(P)$. 

\begin{theorem}
\label{thm:pseudo_spectral}
    Let $J$ be a level of $G$. For every $f \in \cS([G])^J$ be have
    \begin{equation}
    \label{eq:pseudo_spectral}
        \int_{[H]} f(h)dh=\sum_{(I,P,\pi) \in \Pi_H} \frac{1}{|W(\pi)|}  \int_{\lambda \in i \fa_\pi^*} \sum_{\varphi \in \cB_{P,\pi}(J)} \cP_\pi(\varphi,\lambda-\underline{\rho}_\pi) \langle f,E(\varphi,\lambda+\underline{\rho}_\pi) \rangle_G d\lambda,
    \end{equation}
    where this integral is absolutely convergent.
\end{theorem}

As noted in \cite[Remark~3.8.2.1]{Ch}, this expansion is independent from the choice of the level $J$.

\subsubsection{An alternative version using relative characters}

Before starting the proof of Theorem~\ref{thm:pseudo_spectral}, we reformulate it as in Section~\ref{sec:intro}.

Let $\Pi_H^J$ be the set of $(I,P,\pi) \in \Pi_H$ such that $(P,\pi)$ is a $J$-pair. Let $(I,P,\pi) \in \Pi_H^J$. For every $F \in \cS([G])^J$, and $\lambda \in \cS_{\pi,k,c_J}$ (for $k$ and $c_J$ as in Theorem~\ref{thm:bound_Eisenstein}), set
\begin{equation}
    \label{eq:relative_character}
     \cI_{(I,P,\pi)}(F,\lambda)=\sum_{\varphi \in \cB_{P,\pi}(J)} \langle F, E(\varphi,-\overline{\lambda}+\underline{\rho}_\pi) \rangle_G \cP_{\pi}(\varphi,\lambda-\underline{\rho}_\pi).
\end{equation}
The relative character $\cI_{(I,P,\pi)}(F,\lambda)$ is independent of the choice of $J$. 

\begin{lem}
\label{lem:bound_relative_character}
    The sum in \eqref{eq:relative_character} is absolutely convergent and defines an analytic function on $\lambda$. For fixed level $J$, for all $q>0$ there exists a continuous semi-norm $\norm{\cdot}_{J,q}$ on $\cS([G])^J$ such that for $(I,P,\pi) \in \Pi_H^J$ and every $f \in \cS([G])^J$ we have
    \begin{equation}
    \label{eq:bound_relative_char}
         \cI_{(I,P,\pi)}(F,\lambda) \leq \frac{\norm{F}_{J,q}}{(1+\norm{\lambda}^2)^q(1+\Lambda_\pi^2)^q}, \quad \lambda \in  i \fa_\pi^*.
    \end{equation}
    Moreover, $ \cI_{(I,P,\pi)}(F,\lambda)$ is $H(\bA)$-invariant in the sense that for every $h \in H(\bA)$ we have $ \cI_{(I,P,\pi)}(R(h)F,\lambda)= \cI_{(I,P,\pi)}(F,\lambda)$.
\end{lem}

\begin{proof}
    The bound \eqref{eq:bound_relative_char} follows from Proposition~\ref{prop:convergence_strong} and Proposition~\ref{prop:individual_bound_P_pi} by noting that the zeros of $L_{\pi,\cP}(\lambda)$ in a neighborhood of $i \fa_\pi^*-\underline{\rho}_\pi$ are contained in those of $L_{\pi,0}(-\overline{\lambda})$. 

    For the second assertion, for fixed $\lambda$ away from the zeros of $L_{\pi,\cP}(\lambda-\underline{\rho}_\pi)$ we know by Proposition~\ref{prop:convergence_strong} and Proposition~\ref{prop:individual_bound_P_pi} again that
    \begin{equation*}
         \cI_{(I,P,\pi)}(R(h)F,\lambda)=\cP_\pi \left( \sum_{\varphi \in \cB_{P,\pi}(J)} \langle R(h)F,E(\varphi,-\overline{\lambda}+\underline{\rho}_\pi) \rangle_G \varphi,\lambda -\underline{\rho}_\pi \right),
     \end{equation*}
     where the sum converges in $\cA_{P,\pi}(G)^J$. But this is $ \cI_{(I,P,\pi)}(F,\lambda)$ because $\cP_{\pi}$ is $H(\bA)$-invariant by Theorem~\ref{thm:P_pi_quotient}. We conclude by analytic continuation.
\end{proof}

We can now state the alternative version of Theorem~\ref{thm:pseudo_spectral}.

\begin{theorem}
    \label{thm:fine_spectral_proof}
    For any $F \in \cS([G])^J$ we have
    \begin{equation}
    \label{eq:fine_spectral_proof}
        \int_{[H]}F(h)dh=\sum_{(I,P,\pi) \in \Pi_H^J} \frac{1}{\Val{W(\pi)}} \int_{i \fa_\pi^*}  \cI_{(I,P,\pi)}(F,\lambda) d \lambda.
    \end{equation}
    where the integral is absolutely convergent. 
\end{theorem}

\begin{proof}
    That the integral is absolutely convergent follows from Lemma~\ref{lem:bound_relative_character} and \eqref{eq:Muller_spectral}. Theorem~\ref{thm:fine_spectral_proof} is now simply a reformulation of Theorem~\ref{thm:pseudo_spectral}.
\end{proof}

Now let $f \in \cS(G(\bA))$ and $g \in G(\bA)$. Set $F=K_f(g,\cdot)$, where $K_f$ is the kernel function from \eqref{eq:kernel_function}. By \cite[Lemma~2.10.1.1]{BPCZ}, we have $K_f(g,\cdot) \in \cS([G])$. Theorem~\ref{thm:spectral_expansion_intro} is now exactly Theorem~\ref{thm:fine_spectral_proof} up to a change of variables once we realize that
\begin{equation*}
    \langle F,E(\varphi,-\overline{\lambda}-\underline{\rho}_\pi) \rangle=E(g,I_P(f,\lambda-\underline{\rho}_\pi)\overline{\varphi},\overline{\lambda}-\underline{\rho}_\pi), \quad \varphi \in \cB_{P,\pi}(J), \quad \lambda \in i \fa_\pi^*.
\end{equation*}

\subsection{Unfolding of the Rankin--Selberg period in terms of partial Zeta functions}
\label{subsec:unfolding}

In the rest of this section, we fix a level $J$. We prove a first spectral expansion for the integral $\int_{[H]} f(h) dh$ using Rankin--Selberg unfolding.

\subsubsection{Statement of the result}
\label{subsubsec:statement_schwartz}

To write our first formula, we need some notation. For every integer $0 \leq r \leq n$, let $P_r$ be the standard parabolic subgroup of $G$ with Levi subgroup $M_r:=(\GL_r \times \GL_{n-r}) \times (\GL_r \times \GL_{n+1-r})$. Note that it is a standard Rankin--Selberg parabolic subgroup of $G$, so that $P_r^\std=P_r$. We simply write $\bfM_r^2$ for $\bfM_{P_r}^2$, $\bfM_r$ for $\bfM_r^2 \cap H$ and $\cM_r$ for $\cM_{P_r}$. These groups are respectively isomorphic to $\GL_r^2$, $\GL_r$ and $\GL_{n-r} \times \GL_{n+1-r}$. With respect to the decomposition of $M_r$ given above, set 
\begin{equation}
    \label{eq:underline_z} \underline{z}_{r}=\left((0,1/4),(0,1/4)\right) \in \fa_{P_r}^*.
\end{equation}
Let $I_r$ be the tuple
\begin{equation}
    \label{eq:I_r_defi}
    I_r=(r,0,n-r,0,n+1-r,0).
\end{equation}
We have the set of tuples $(I_r,P,\pi) \in \Pi_H^\uparrow$ (see \S\ref{subsubsec:condition_up}). Recall that this means that $I_1=I_2=\emptyset$ in the notation of \S\ref{subsubsec:condition_up}. By unfolding the definition, $(I_r,P,\pi) \in \Pi_H^\uparrow$ is the datum of a standard parabolic subgroup $P$ of $G$ and of $\pi \in \Pi_\disc(M_P)$ such that the following conditions are satisfied.
\begin{itemize}
    \item We have $P \subset P_r$. We set $\cP=P \cap \cM_r$ and $\bfP^2=P \cap \bfM_r^2$.
    \item Under this decomposition, we have $\pi=(\pi_\bfP \boxtimes \pi_\bfP^\vee) \boxtimes \pi_\cP$, where $\pi_\bfP \in \Pi_\disc(M_\bfP)$ and $\pi_\cP \in \Pi_\cusp(M_\cP)$.
\end{itemize}

Let $(I_r,P,\pi) \in \Pi_H^\uparrow$. We have the space $\fa_\pi^*$ defined in \eqref{eq:a_pi_up_defi}. We can describe it as $\fa_{\pi}^*=\fa_{\bfP}^* \oplus \fa_{\cP}^*$ where $\fa_{\bfP}^*$ is anti-diagonally embedded in $\fa_{\bfP^2}^* \subset \fa_{P}^*$. We have the regularized period $\cP_{\pi}^\uparrow$ from \eqref{eq:P_up}. By Proposition~\ref{prop:up_is_down}, this is simply $\cP_{\pi}$ in this case as $P^\downarrow=P$. We also have the element $\underline{\rho}_\pi \in \fa_P^*$ from \eqref{eq:rho_P_down}. It has coordinates $((1/4,0),(1/4,0))$. We write $\cP(M_\bfP)$ for the set of semi-standard parabolic subgroups of $\bfM_{r}$ with semi-standard Levi factor $M_{\bfP}$, and $\cP(M_\cP)$ for the those of $\cM_r$ with semi-standard Levi $M_\cP$. We have the set $W(\pi)$ defined in \S\ref{sec:functional_equation} and it satisfies
\begin{equation*}
        \Val{W(\pi)}=\Val{\cP(M_\bfP)}\Val{\cP(M_\cP)} .
\end{equation*}
The goal of this section is to prove the following proposition. 

\begin{prop}
\label{prop:enter_Eisenstein}
    There exists $c>0$ such that for all $t>c$ and $f \in \cS([G])^J$ we have
    \begin{equation}
        \label{eq:enter_Eisenstein}
            \int_{[H]}f(h)dh
            =\sum_{r=0}^{n} \sum_{(I_r,P,\pi) \in \Pi_H^\uparrow} \frac{1}{|W(\pi)|} \int_{i \fa_{\pi}^*-\underline{\rho}_\pi+t\underline{z}_{r}} \sum_{\varphi \in \cB_{P,\pi}(J)}\cP_{\pi}^\uparrow(\varphi,\lambda) \langle f,E(\varphi,-\overline{\lambda}) \rangle_{G} d\lambda.
    \end{equation}
\end{prop}

\begin{rem}
\label{rem:rapid_decay}
    Note that the sum only runs through the $(I_r,P,\pi)$ such that $(P,\pi)$ is a $J$-pair. It now follows from Theorem~\ref{thm:analytic_Eisenstein} and Proposition~\ref{prop:reg_P_pi_up} that the integrand in Proposition~\ref{prop:enter_Eisenstein} is holomorphic, and it is of rapid decay by Proposition~\ref{prop:convergence_strong} and Proposition~\ref{prop:bound_P_pi_up}. Therefore, \eqref{eq:enter_Eisenstein} is well-defined. 
\end{rem}

\subsubsection{Unfolding of the Rankin--Selberg period in terms of partial Zeta integrals}

Let $0 \leq r \leq n$. If $f \in \cS([G])$, we can consider the restriction of the constant term $f_{P_r}$ to $[M_r]$. It is a priori an element in $\cT([M_r])$, but its growth can be more finely controlled.

\begin{lem}
    \label{lem:constant_growth}
    The following assertions hold.
    \begin{enumerate}
        \item For every $g_r \in \cM_r (\bA)$ and $f \in \cS([G])$, the map $m_r \in [\bfM_r^2] \mapsto f_{P_r}(m_r g_r)$ belongs to $\cS([\bfM_r^2])$.
        \item Let $N>0$. There exists $c_N>0$ such that for all $t>c_N$ there exists a continuous semi-norm $\norm{\cdot}_t$ on $\cS([G])$ such that
        \begin{equation*}
            \Val{f_{P_r}(m_r g_r)}\Val{\det g_r}^{-t} \leq \norm{f}_{t} \norm{m_r}_{\bfM_r}^{-N} \norm{g_r}_{\cM_r}^{-N}, \quad m_r \in [\bfM_r], \quad g_r \in [\cM_r], \quad f \in \cS([G]).
        \end{equation*}
    \end{enumerate}
    
\end{lem}

\begin{proof}
    The first assertion is a direct consequence of \cite[Lemma~2.5.13.1]{BPCZ}. For the second, note that for every $m_r \in \bfM_r (\bA)$ and $g_r \in \cM_r (\bA)$ we have
    \begin{equation*}
        \delta_{P_r}^{1/2}(m_r g_r)=\Val{m_{r}}^{n-r+1/2} \Val{g_r}^{-r},
    \end{equation*}
    where we write $\Val{\cdot}$ for $\Val{ \det( \cdot)}$ and the determinant is taken in $G$. By \cite[Equation~(2.4.1.4)]{BPCZ} and \cite[Lemma~2.5.13.1]{BPCZ}, there exists $c>0$ such that for every $N'>0$ we have a continuous semi-norm $\norm{\cdot}_{N'}$ on $\cS([G])$ such that
    \begin{equation}
        \label{eq:constant_term_estimate}
        \Val{f_{P_r}(m_r g_r)} \leq \norm{f}_{N'} \norm{m_r}_{\bfM_r}^{-N'} \norm{g_r}_{\cM_r}^{-N'} \Val{m_{r}}^{-c(n-r+1/2)N'} \Val{g_r}^{crN'},
    \end{equation} 
    for every $m_r \in [\bfM_r]$, $g_r \in [\cM_r]$ and $f \in \cS([G])$. Set $c_N=c(n+1/2)N$ and take $t>c_N$. By applying \eqref{eq:constant_term_estimate} to $N'=t/{c_N} \times N$ and to the Schwartz-function $f \Val{\cdot}^{-c(n-r+1/2)N'}$, we get
    \begin{equation*}
        \Val{f_{P_r}(m_r g_r)}\Val{\det g_r}^{-t} \leq \norm{f \Val{\cdot}^{-c(n-r+1/2)N'}}_{N'} \norm{m_r}_{\bfM_r}^{-N} \norm{g_r}_{\cM_r}^{-N}, 
    \end{equation*}
    once again for every $m_r \in [\bfM_r]$, $g_r \in [\cM_r]$ and $f \in \cS([G])$. This concludes.
\end{proof}

For any $f \in \cS([G])$, we can therefore consider the twisted partial diagonal period
\begin{equation}
    \label{eq:twisted_partial}
    g_r \in \cM_r (\bA) \mapsto P_{r}(f)(g_r):=\delta_{P_r}^{-1/2}(g_r)\int_{[\bfM_r]} f_{P_r}(m_rg_r) \delta_{P_{r,H}}^{-1}(m_r)dm_r,
\end{equation}
It is absolutely convergent by Lemma~\ref{lem:constant_growth} and \cite[Proposition~A.1.1.(vi)]{Beu}.

We now recall the construction of Zeta integrals on $\cM_r$. Let $\cN_r$ be the unipotent radical of the standard Borel of upper triangular matrices in $\cM_r$. Let $\psi$ be a generic automorphic character of $N_0(\bA)$ trivial on $N_{0,H}(\bA)$. Denote again by $\psi$ its restriction to $\cN_r (\bA)$. For $s \in \cc$ and $N>0$, consider the map
\begin{equation*}
    W_{r,f_r} : g_r \in [\cM_r] \mapsto \int_{[\cN_r]} f_r (ng_r) \overline{\psi}(n) dn, \quad f_r \in \cT_N([\cM_r]),
\end{equation*}
and 
\begin{equation}
    \label{eq:Z_r}
    Z_r (f_r,s) := \int_{\cN_{r,H}(\bA) \backslash \cM_{r,H}(\bA)} W_{r,f_r}(h_r) \Val{\det h_r}^s dh_r, \quad f_r \in \cT_N([\cM_r]).
\end{equation}
The first is always defined by an absolute convergent integral, while the other may not be. We have the following result from \cite{BPCZ}.

\begin{lem}[{\cite[Lemma~7.1.1.1.]{BPCZ}}]
    \label{lem:zeta_convergence}
    Let $N>0$. There exists $c_N >0$ such that for every $s \in \cc$ with $\Re(s) \geq c_N$ the integral \eqref{eq:Z_r} is absolutely convergent for every $f_r \in \cT_N([\cM_r])$. Moreover, the map $f_r \in \cT_N([\cM_r]) \mapsto Z_r (f_r,s)$ is continuous.
\end{lem}

We now unfold the Rankin--Selberg integral using the partial Zeta functions. Recall that we write $R$ for the action by right-translation on spaces of functions.

\begin{prop} 
    \label{prop:zeta_unfold}
    For every $f \in \cS([G])$ we have
    \begin{equation}
        \label{eq:zeta_unfold}
        \int_{[H]} f(h) dh = \sum_{r=0}^n \int_{K_H} Z_r (P_r (R(k)f),0)dk.
    \end{equation}
\end{prop}
\begin{proof}
    Note that the $Z_r (P_r (R(k)f),0)$ are all well-defined by Lemma~\ref{lem:constant_growth} and Lemma~\ref{lem:zeta_convergence}. We now claim that \eqref{eq:zeta_unfold} is proved in the course of the proof of \cite[Proposition~7.2.0.2]{BPCZ}. Indeed, \cite{BPCZ} shows that there exists $c>0$ such that for every $\Re(s)\geq c$ and $f \in \cS([G])$ we have the equality
    \begin{equation*}
        \int_{[H]} f(h) \Val{h}^s dh = \sum_{r=0}^n  \int_{\cN_{r,H}(\bA) \backslash \cM_{r,H}(\bA)} \int_{K_H} \int_{[\bfM_r]} \delta_{P_{r,H}}^{-1}(h_r m_r)\Val{h_r m_r}^s  W_{r,R(m_r k)f_{P_r}}(h_r)  d m_r d h_r dk, 
    \end{equation*}
    where the integrals are absolutely convergent. For all $m_r \in \cM_r (\bA) \cap H(\bA)$ we have $\delta_{P_r,H}(m_r)=\delta_{P_r}^{1/2}(m_r)$. Therefore, \eqref{eq:zeta_unfold} follows by applying this result to the Schwartz function $g \mapsto f(g) \Val{g}^{-c}$. Note that \cite{BPCZ} assumes that $f$ belongs to $\cS_\chi([G])$ for some regular cuspidal datum $\chi$ (see \cite[Section~2.9.7]{BPCZ}). However, this assumption is only used to show that the terms attached to $r>0$ vanish in \eqref{eq:zeta_unfold}, which we do not claim here.
\end{proof}

\subsubsection{Sectral unfolding for partial periods of constant terms}

We keep $0 \leq r \leq n$ and write the spectral expansion of the twisted partial diagonal period $P_r$ from \eqref{eq:twisted_partial}. We consider $\Pi_r$ the set of couples $(P,\pi)$ such that the following conditions are satisfied.
\begin{itemize}
    \item $P$ is a standard parabolic subgroup of $G$ contained in $P_r$ such that $\bfM_r^2 \cap P$ is of the form $\bfP^2$. We write $\cP=\cM_r \cap P$.
    \item $\pi \in \Pi_\disc(M_P)$ decomposes as
    \begin{equation*}
        \pi=(\pi_\bfP \boxtimes \pi_\bfP^\vee) \boxtimes \pi_\cP.
    \end{equation*}
\end{itemize}
We embed $\fa_{\bfP}^*$ diagonally into $\fa_{\bfP^2}$ and further into $\fa_{P}^*$. We write $\fa_{\bfP}^{*,-}$ for the anti-diagonal copy of $\fa_{\bfP}^*$ in $\fa_{\bfP^2}^*$. As in \S\ref{subsubsec:statement_schwartz}, relatively to $\fa_P^*=\fa_{\bfP^2}^* \oplus \fa_{\cP}^*$ we define
\begin{equation*}
    \underline{\rho}_\pi=((1/4,0),(1/4,0)) , \quad \underline{z}_r = ((0,1/4),(0,1/4)).
\end{equation*}
Then $\underline{\rho}_\pi \in \fa_{\bfP}^*$.

For every $g_r=(g_{r,n},g_{r,n+1}) \in \cM_r (\bA)$, we have the continuous linear form
\begin{equation*}
    \varphi=\varphi_n \otimes \varphi_{n+1} \in \cA_{P \cap M_r,\pi}(M_r) \mapsto \left(\langle \cdot,\overline{\cdot} \rangle_{\bfP,\Pet} \otimes \mathrm{ev}_{g_r} \right)(\varphi) := \langle \varphi_{n}(\cdot \; g_{r,n}),\overline{\varphi}_{n+1}(\cdot \; g_{r,n+1}) \rangle_{\bfP,\Pet},
\end{equation*}
where the notations $\varphi_n(\cdot \; g_{r,n})$ and $\varphi_{n+1}(\cdot \; g_{r,n+1})$ mean that we consider the restrictions of these automorphic forms to the two copies of $M_\bfP$. For $\varphi \in \cA_{P \cap M_r,\pi}(M_r)$, we can consider the automorphic form
\begin{equation*}
    \left(\langle \cdot,\overline{\cdot} \rangle_{\bfP,\Pet}\right)(\varphi) : g_r \mapsto \left(\langle \cdot,\overline{\cdot} \rangle_{\bfP,\Pet} \otimes \mathrm{ev}_{g_r} \right)(\varphi).
\end{equation*}
It belongs to $\cA_{\cP,\pi_\cP}(\cM_r)$. 

If now $\varphi \in \cA_{P,\pi}(G)$, we set 
\begin{equation*}
    \varphi_M : m \in M_r (\bA) \mapsto \delta_{P_r}^{-1/2}(m) \varphi(m).
\end{equation*}
Then $\varphi_M \in \cA_{P \cap M_r,\pi}(M_r)$. By composing with the previous map, we obtain
\begin{equation*}
   \varphi \in \cA_{P,\pi}(G) \mapsto \left(\langle \cdot,\overline{\cdot} \rangle_{\bfP,\Pet}\right)(\varphi_M)  \in \cA_{\cP,\pi_\cP}(\cM_r)
\end{equation*}

\begin{lem}
    \label{lem:partial_diagonal}
    There exists $c>0$ such that for every $t>c$ and $f \in \cS([G])^J$ we have 
    \begin{align}
        \label{eq:partial_diagonal_expansion}
        P_r (f)(g_r)= &\sum_{(P,\pi) \in \Pi_r} \Val{\cP(M_\bfP)}^{-1}\Val{\cP(M_\cP)}^{-1}\nonumber \\
        \times &\int_{i\fa_{\bfP}^{*,-} \oplus i \fa_{\cP}^*-\underline{\rho}_\pi+t \underline{z}_r} \sum_{\varphi \in \cB_{P,\pi}(J)} \langle f, E(\varphi,-\overline{\lambda}) \rangle_G E^{\cM_r}\left(g_r, \left(\langle \cdot,\overline{\cdot} \rangle_{\bfP,\Pet}\right)(\varphi_M)  ,\lambda_\cP\right) d\lambda,
    \end{align}
    for every $g_r \in \cM_r (\bA)$.
\end{lem}

\begin{proof}
    Without loss of generality, we can assume that $f=f_n \otimes f_{n+1}$ as both sides of \eqref{eq:partial_diagonal_expansion} define continuous linear forms on $\cS([G])$ by Lemma~\ref{lem:constant_growth} for the LHS, and Propositions~\ref{prop:convergence_strong} and \ref{prop:bound_Eisenstein_non_finite} as well as \eqref{eq:Muller_spectral} for the RHS. Indeed, because $t$ is large the region of integration does not meet the possible singularities of Eisenstein series by Theorem~\ref{thm:analytic_Eisenstein}. We denote by $\rho_{\bfP_r,H}$ the restriction of $\rho_{P_r,H}$ to $\fa_{\bfP_r}$ which we identify with an element of $\fa_{\bfP_r}^*$. Let $g_r \in \cM_r (\bA)$. By the first assertion of Lemma~\ref{lem:constant_growth} and Langlands' spectral decomposition of the inner product of Schwartz functions (Proposition~\ref{prop:extended_Langlands}), we have 
    \begin{align*}
        P_r (f)(g_r)= \delta_{P_r}^{-1/2}(g_r)\sum_{\bfP \subset \bfP_r} &\sum_{\pi_\bfP \in \Pi_\disc(M_\bfP)}\sum_{\phi \in \cB_{\bfP^2,\pi_\bfP \boxtimes \pi_\bfP^\vee}(J)}\\
        \times & \int_{i \fa_{\bfP}^{*,-}+2 \rho_{\bfP_r,H}} \langle R(g_r)f_{P_r},E^{\bfM_r}(\phi,-\overline{\lambda_\bfP})\rangle_{\bfM_r} \langle \phi_n,\overline{\phi}_{n+1} \rangle_{\bfP,\Pet} d \lambda_\bfP,
    \end{align*}
    where we write again $J$ for the level instead of $J \cap \bfM_r^2$. Note that here we artificially add the Petersson inner product $\langle \phi_n,\overline{\phi}_{n+1} \rangle_{\bfP,\Pet}$ to sum over $\cB_{\bfP^2,\pi_\bfP \boxtimes \pi_\bfP^\vee}(J)$ rather than $\cB_{\bfP,\pi_\bfP}(J)$. 

    Let $N>0$ be such that Langlands' spectral decomposition holds for functions in $\cT_{-N}([\cM_r])$ (Corollary~\ref{cor:Langlands_spectral_extended}). Up to enlarging it, we can also assume that the theorem holds for functions in $\cT_{-N}([\bfM_r^2])$. Take $c_N>0$ as in Lemma~\ref{lem:constant_growth} as well as $t>4c_N$. For every $\bfP$, $\pi_\bfP$ and $\phi$ we can therefore write the spectral expansion of $g_r \mapsto \langle R(g_r)f_{P_r},E^{\bfM_r}(\phi,-\overline{\lambda_\bfP})\rangle_{\bfM_r} \Val{g_r}^{-t/4}$. By absorbing this twist in the integral, we see that 
    \begin{align}
        \label{eq:spectral_partial_proof}
        &P_r (f)(g_r)= \delta_{P_r}^{-1/2}(g_r)\sum_{\bfP \subset \bfP_r} \sum_{\pi_\bfP \in \Pi_\disc(M_\bfP)}\sum_{\phi \in \cB_{\bfP^2,\pi_\bfP \boxtimes \pi_\bfP^\vee}(J)} \sum_{\cP \subset \cP_r} \sum_{\pi_\cP \in \Pi_\disc(M_{\cP})} \sum_{\phi' \in \cB_{\cP,\pi_\cP}(J)} \nonumber\\
        \times & \int_{i \fa_{\bfP}^{*,-}\oplus i \fa_{\cP}^*+2 \rho_{\bfP_r,H} + t \underline{z}_r} \langle f_{P_r},E^{M_r}(\phi \otimes \phi',-\overline{\lambda})\rangle_{\bfM_r} \langle \phi_n,\overline{\phi}_{n+1}  \rangle_{\bfP,\Pet} E^{\cM_r}(g_r,\phi',\lambda_\cP)d \lambda,
    \end{align}
    where we write $\phi \otimes \phi'$ for the product of $\phi$ and $\phi'$ seen as an automorphic form on $M_r$. Note that thanks to our bounds on Eisenstein series from Propositions~\ref{prop:convergence_strong} and \ref{prop:bound_Eisenstein_non_finite} and \eqref{eq:Muller_spectral}, we know that this integral is absolutely convergent and therefore that we can switch the order of the sums. 

    Take $\bfP$, $\pi_\bfP$, $\cP$ and $\pi_\cP$ as in \eqref{eq:spectral_partial_proof}. Set $P=(\bfP^2 \times \cP)N_{P_r}$ and $\pi=(\pi_\bfP \boxtimes \pi_\bfP^\vee) \boxtimes \pi_\cP \in \Pi_\disc(M_P)$. Then $(P,\pi) \in \Pi_r$. Take $\lambda \in i \fa_{\bfP}^{*,-}\oplus i \fa_{\cP}^*+2 \rho_{\bfP_r,H} + t \underline{z}_r$. By Proposition~\ref{prop:convergence_strong} and Lemma~\ref{lem:constant_growth}, we see that the series
    \begin{equation}
        \label{eq:f_P_r_1}
        f_{P_r,\pi_\lambda}:=\sum_{\varphi \in \cB_{P \cap M_r,\pi}(J)}\langle f_{P_r},E^{M_r}(\varphi,-\overline{\lambda})\rangle_{\bfM_r} \varphi_\lambda 
    \end{equation}
    is absolutely convergent in $\cA_{P \cap M_r,\pi_\lambda}(M_r)^J$. Because $g \in [G] \mapsto (R(g)f)_{P_r,\pi_\lambda}$ belongs to $\cA_{P,\pi_\lambda,-\rho_{P_r}}(G)$, we have
    \begin{equation*}
        f_{P_r,\pi_\lambda}=\sum_{\varphi \in \cB_{P,\pi}(J)} \langle F,\varphi \rangle_{P,\Pet} \varphi_{M,\lambda}.
    \end{equation*} 
    For fixed $\varphi \in \cB_{P,\pi}(J)$, we check by undolding the integrals, taking into account the twist by $\rho_{P_r}$ and using \eqref{eq:f_P_r_1}, that 
    \begin{equation*}
        \langle F,\varphi \rangle_{P,\Pet}=\langle f_{P_r},E^{M_r}(\varphi,-\overline{\lambda}+\rho_{P_r})\rangle_{P_r}.
    \end{equation*}
    It follows that
    \begin{equation}
        \label{eq:f_P_r_pi_lambda}
        f_{P_r,\pi_\lambda}=\sum_{\varphi \in \cB_{P,\pi}(J)}\langle f_{P_r},E^{M_r}(\varphi,-\overline{\lambda}+\rho_{P_r})\rangle_{P_r} \varphi_{M,\lambda}
    \end{equation}
    Because $t \gg 0$, the Eisenstein series over $P_r (F) \backslash G(F)$ of $E^{M_r}(\varphi,-\overline{\lambda}+\rho_{P_r})$ is absolutely convergent. Therefore, by adjunction of the constant term and Eisenstein series, we may rewrite \eqref{eq:f_P_r_pi_lambda} as
    \begin{equation}
        \label{eq:f_P_r_2}
        f_{P_r,\pi_\lambda}=\sum_{\varphi \in \cB_{P,\pi}(J)}\langle f,E(\varphi,-\overline{\lambda}+\rho_{P_r})\rangle_{G} \varphi_{M,\lambda}
    \end{equation}
    Once again, by Proposition~\ref{prop:convergence_strong} this sum is absolutely convergent in $\cA_{P \cap M_r,\pi_\lambda}(M_r)^J$. By Proposition~\ref{prop:bound_Eisenstein_non_finite}, we know that the linear form
    \begin{equation*}
        \cA_{P \cap M_r,\pi_\lambda}(M_r)^J \mapsto  E^{\cM_r}\left(g_r, \left(\langle \cdot,\overline{\cdot} \rangle_{\bfP,\Pet}\right)(\varphi)  ,\lambda_\cP\right)
    \end{equation*}
    is continuous. By comparing the two expressions of $f_{P_r,\pi_\lambda}$ obtained in \eqref{eq:f_P_r_1} and \eqref{eq:f_P_r_2}, we conclude that \eqref{eq:partial_diagonal_expansion} holds once we note that
    \begin{equation*}
        \underline{\rho}_\pi=(\rho_{P_r})_{|\fa_\bfP}-2 \rho_{\bfP_r,H},
    \end{equation*}
    which follows from the definition.
\end{proof}

\subsubsection{End of the proof of Proposition~\ref{prop:enter_Eisenstein}}

We can now end the proof of Proposition~\ref{prop:enter_Eisenstein}. Let $f \in \cS([G])^J$. Let $c>0$ be given by Lemma~\ref{lem:partial_diagonal}. By Propositions~\ref{prop:convergence_strong} and \ref{prop:bound_Eisenstein_non_finite} as well as \eqref{eq:Muller_spectral}, there exists $N>0$ such that for every $t>4c$ the integral 
 \begin{align*}
       g_r \in [\cM_r] \mapsto \Val{\det g_r}^{-t/4} &\sum_{(P,\pi) \in \Pi_r} \Val{\cP(M_\bfP)}^{-1}\Val{\cP(M_\cP)}^{-1} \\
       \times &\int_{i\fa_{\bfP}^{*,-} \oplus i \fa_{\cP}^*-\underline{\rho}_\pi+t \underline{z}_r} \sum_{\varphi \in \cB_{P,\pi}(J)} \langle f, E(\varphi,-\overline{\lambda}) \rangle_G E^{\cM_r}\left(g_r, \left(\langle \cdot,\overline{\cdot} \rangle_{\bfP,\Pet}\right)(\varphi_M)  ,\lambda_\cP\right) d\lambda,
\end{align*}
is absolutely convergent in $\cT_{-N}(\cM_r)$. Assume moreover that $t>4c_N$ where $c_N$ is given by Lemma~\ref{lem:zeta_convergence}. By Lemma~\ref{lem:zeta_convergence}, Proposition~\ref{prop:zeta_unfold} and a change of variables, we conclude that 
\begin{align}
    \label{eq:big_spectral}
    &\int_{[H]} f(h)dh=\sum_{r=0}^n \sum_{(P,\pi) \in \Pi_r} \Val{\cP(M_\bfP)}^{-1}\Val{\cP(M_\cP)}^{-1} \nonumber \\
    \times &\int_{i\fa_{\bfP}^{*,-} \oplus i \fa_{\cP}^*-\underline{\rho}_\pi+t \underline{z}_r} \sum_{\varphi \in \cB_{P,\pi}(J)} \langle f, E(\varphi,-\overline{\lambda}) \rangle_G \int_{K_H}Z_r (\left(E^{\cM_r}\left( \left(\langle \cdot,\overline{\cdot} \rangle_{\bfP,\Pet}\right)(R(k)\varphi_M),\lambda_\cP\right),0\right) dk d\lambda.
\end{align}
By Proposition~\ref{prop:P=Z}, we know that the Zeta integral $Z_r$ is zero on automorphic representations induced from the residual spectrum. Therefore, the term attached to $(P,\pi)$ in \eqref{eq:big_spectral} is zero as soon as $\pi_\cP \notin \Pi_\cusp(M_\cP)$ so that the sum takes place over triples $(I_r,P,\pi) \in \Pi_H^\uparrow$. 

For $(I_r,P,\pi) \in \Pi_H^\uparrow$, by parabolic descent (Proposition~\ref{prop:parabolic_descent}) and by the definition of $\cP_\pi^\uparrow$ in terms of regularized period in \eqref{eq:P_up}, we have for every $\varphi \in \cB_{P,\pi}(J)$ the formula
\begin{equation*}
    \cP_{\pi}^\uparrow(\varphi,\lambda)=\int_{K_H}Z_r (\left(E^{\cM_r}\left( \left(\langle \cdot,\overline{\cdot} \rangle_{\bfP,\Pet}\right)(R(k)\varphi_M),\lambda_\cP\right),0\right)dk.
\end{equation*}
Thanks to the description given in \S\ref{subsubsec:statement_schwartz}, we also know that $i\fa_{\bfP}^{*,-} \oplus i \fa_{\cP}^*-\underline{\rho}_\pi+t \underline{z}_r$ is nothing but $i \fa_\pi^* -\underline{\rho}_\pi+t \underline{z}_r$ and that $\Val{W(\pi)}=\Val{\cP(M_\bfP)}\Val{\cP(M_\cP)}$. This concludes the proof of Proposition~\ref{prop:enter_Eisenstein}.

\subsection{Additional residues from discrete Eisenstein series}
\label{subsec:additional}

For the rest of the proof of Theorem~\ref{thm:pseudo_spectral} we fix a Schwartz function $f \in \cS([G])^J$. We now compute a second spectral expansion for $\int_{[H]} f(h) dh$, building on the one from Proposition~\ref{prop:enter_Eisenstein}. It is obtained by shifting the contour in the integrals of \eqref{eq:enter_Eisenstein}. Doing so, we will cross some singularities of the integrand which all come from the Eisenstein series $E(\varphi,-\overline{\lambda})$.

\subsubsection{Statement of the result}
\label{subsubsec:statement_result}

In this section, we take $0 \leq r \leq n$ and $(I_r,P,\pi) \in \Pi_H^\uparrow$. Our goal is to find a new intermediate expression for 
\begin{equation}
\label{eq:I_pi_defi}
    I_\pi:=\int_{i \fa_{\pi}^*-\underline{\rho}_\pi+t\underline{z}_{r}} \sum_{\varphi \in \cB_{P,\pi}(J)}\cP_{\pi}^\uparrow(\varphi,\lambda) \langle f,E(\varphi,-\overline{\lambda}) \rangle_{G} d\lambda.
\end{equation}
We have the following functional equation for $I_\pi$.

\begin{lem}
\label{lem:functional_I_pi}
    For any $w \in W(\pi)$, we have $I_\pi=I_{w \pi}$.
\end{lem}

\begin{proof}
    Note that $w$ sends $i \fa_\pi^*-\underline{\rho}_\pi+t \underline{z}_{r}$ to $i \fa_{w\pi}^*-\underline{\rho}_{w \pi}+t \underline{z}_{r}$. The result now follows from the functional equation of $\cP_\pi^\uparrow$ (Lemma~\ref{lem:functional_up}) and of Eisenstein series (Theorem~\ref{thm:analytic_Eisenstein}) using a change of variables in $\varphi$ and $\lambda$.
\end{proof}

Write $\pi=\pi_n \boxtimes \pi_{n+1}$ as
\begin{equation}
    \label{eq:pi_defi}
    \pi_n=\boxtimes_{i=1}^{m_+} \pi_{+,i}\boxtimes_{i=1}^{m_{\mathrm{c},1}} \pi_{\mathrm{c},1,i} , \quad \pi_{n+1}=\boxtimes_{i=1}^{m_+} \pi_{+,i}^\vee\boxtimes_{i=1}^{m_{\mathrm{c},2}} \pi_{\mathrm{c},2,i} ,
\end{equation}
where the $\pi_{\mathrm{c},1,i}$ and $\pi_{\mathrm{c},2,i}$ are cuspidal and the $\pi_{+,i}$ discrete. We furthermore write $\pi_{+,i}=\Speh(\sigma_{+,i},\underline{d}(+,i))$ for each $i$. By Lemma~\ref{lem:functional_I_pi}, we may assume that $\underline{d}(+,1) \geq \hdots \geq \underline{d}(+,m_+)$. We also define $\pi_{+,i}^+:=\Speh(\sigma_{+,i},\underline{d}(+,i)+1)$.

In order to state our intermediate formula, we will need some combinatorial gadgets. We denote by $\cG(\pi)$ the set of undirected simple graphs $\Gamma$ such that
\begin{itemize}
    \item the vertices of $\Gamma$ are $ \pi_{+,1}, \hdots, \pi_{+,m_+}$, $\pi_{\mathrm{c},1,1}, \hdots, \pi_{\mathrm{c},1,m_{\mathrm{c},1}}, \pi_{\mathrm{c},2,1}, \hdots, \pi_{\mathrm{c},2,m_{\mathrm{c},2}}$ (with multiplicity);
    \item the edges of $\Gamma$ are of the form $\{ \pi_{+,i},\pi_{\mathrm{c},1,j} \}$ with $\sigma_{+,i} \simeq \pi_{\mathrm{c},1,j} $, or $\{ \pi_{+,i} ,\pi_{\mathrm{c},2,j}\}$ with $\sigma_{+,i}^\vee \simeq \pi_{\mathrm{c},2,j}$;
    \item each $\pi_{\mathrm{c},1,i}$ and $\pi_{\mathrm{c},2,i}$ has degree at most one;
    \item each $\pi_{+,1}, \hdots, \pi_{+,m_+}$ has degree at most two, and if it is two then the neighbors are some $\pi_{\mathrm{c},1,i}$ and $\pi_{\mathrm{c},2,j}$ (they can not be $\pi_{\mathrm{c},1,i}$ and $\pi_{\mathrm{c},1,j}$ or $\pi_{\mathrm{c},2,i}$ and $\pi_{\mathrm{c},2,j}$).
\end{itemize}
For each $i$, we denote by $\deg(\mathrm{c},1,i)$ (resp. $\deg(\mathrm{c},2,i)$) the degree of $\pi_{\mathrm{c},1,i}$ (resp. $\pi_{\mathrm{c},2,i}$), and by $\deg(+,1,i)$ (resp.$\deg(+,2,i)$) the number of neighbors of the form $\pi_{\mathrm{c},1,j}$ (resp. $\pi_{\mathrm{c},2,j}$) of $\pi_{+,i}$. These integers are all either zero or one.

If $\Gamma \in \cG(\pi)$, we can define a discrete representation $\pi_\Gamma$ of some standard Levi $M_{Q_{\Gamma}}$ by
\begin{align}
    \pi_{\Gamma,n}&=\underset{\substack{\deg(+,1,i)=0 \\ \deg(+,2,i)=0}}{\boxtimes} \pi_{+,i} \underset{\substack{\deg(+,2,i)=1 \\ \deg(+,1,i)=0}}{\boxtimes} \pi_{+,i}\underset{\deg(+,1,i)=1}{\boxtimes} \pi_{+,i}^+  \underset{\deg(\mathrm{c},1,i)=0}{\boxtimes} \pi_{\mathrm{c},1,i} , \label{eq:pi_gamma_n} \\
    \pi_{\Gamma,n+1}&= \underset{\substack{\deg(+,1,i)=0 \\ \deg(+,2,i)=0}}{\boxtimes} \pi_{+,i}^\vee  \underset{\substack{\deg(+,1,i)=1 \\ \deg(+,2,i)=0}}{\boxtimes} \pi_{+,i}^\vee \underset{\deg(+,2,i)=1}{\boxtimes} \pi_{+,i}^{+,\vee}. \underset{\deg(\mathrm{c},2,i)=0}{\boxtimes} \pi_{\mathrm{c},2,i} \label{eq:pi_gamma_n+1}
\end{align}
Here we impose that the representations appear in the order $i=1,2,\hdots$ for $i$ in each set. 
Set
\begin{align*}
    I_{\Gamma}&=\left( \sum_{\substack{\deg(+,1,i)=0 \\ \deg(+,2,i)=0}} \underline{n}(+,i),\sum_{\deg(+,1,i)=1} (\underline{n}(+,i)+\underline{r}(+,i)),\sum_{\deg(\mathrm{c},1,i)} \underline{n}(\mathrm{c},1,i), \right. \\
    & \left. \sum_{\deg(+,2,i)=1} (\underline{n}(+,i)+\underline{r}(+,i)),\sum_{\deg(\mathrm{c},2,i)} \underline{n}(\mathrm{c},2,i),0 \right).
\end{align*}
Finally, set 
\begin{equation*}
    I_{\Gamma,1}=I_{\Gamma,2}=\{ i \; | \; \deg(+,1,i)=\deg(+,2,i)=1 \}
\end{equation*}
Then up to some evident identifications, we have $(I_\Gamma,Q_\Gamma,\pi_\Gamma,I_{\Gamma,1},I_{\Gamma,2}) \in \Pi_H^\uparrow$. This holds because we have assumed that $\underline{d}(+,1) \geq \hdots \geq \underline{d}(+,m_+)$, and because the tensor products in \eqref{eq:pi_gamma_n} and \eqref{eq:pi_gamma_n+1} are taken with respect to the natural order on the index $i$. 

We define $\Pi_H^\uparrow(\pi)$ to be the image of the map $\Gamma \in \cG(\pi) \mapsto (I_\Gamma,R_\Gamma,\pi_\Gamma,I_{\Gamma,1},I_{\Gamma,2}) \in \Pi_H^\uparrow$. Note that its fibers can be of cardinal strictly bigger than $1$. Moreover, $(I_r,P,\pi) \in \Pi_H^\uparrow(\pi)$ with the null graph being its sole preimage. If $(I,Q,\tau,I_1,I_2) \in \Pi_H^\uparrow(\pi)$, we can decompose $\tau$ as in \eqref{eq:pi_up_n} and \eqref{eq:pi_up_n+1}. To differentiate from $\pi$, we write $\tau_{1,i}=\Speh(\sigma_{1,i}(\tau),\underline{d}(\tau,1,i))$, $\tau_{2,i}=\Speh(\sigma_{2,i}(\tau),\underline{d}(\tau,2,i))$ and $\tau_{+,i}=\Speh(\sigma_{+,i}(\tau),\underline{d}(\tau,+,i))$. The definition of $\Pi_H^\uparrow(\pi)$ imposes that the $\underline{d}(\tau,1,i)$ and $\underline{d}(\tau,2,i)$ are strictly greater than $1$.

Let $(I,Q,\tau,I_1,I_2) \in \Pi_H^\uparrow(\pi)$. We have defined in \S\ref{subsubsec:connect_periods} a triple $(I^\downarrow,Q^\downarrow,\tau^\downarrow) \in \Pi_H$. We can therefore also consider $W(\tau^\downarrow)$ and denote by $\Stab(\tau)$ the stabilizer of $\tau^\downarrow$ in this set. Because we have fixed the order of the representations in \eqref{eq:pi_gamma_n} and \eqref{eq:pi_gamma_n+1}, the fiber of the map 
\begin{equation}
\label{eq:down_map}
    (I',Q',\tau',I_1',I_2') \in \Pi^\uparrow_H(\pi) \mapsto (I^{',\downarrow},Q^{',\downarrow},\tau^{',\downarrow}) \in \Pi_H
\end{equation}
above the set $W(\tau^\downarrow).\tau^\downarrow$ is reduced to $(I,Q,\tau,I_1,I_2)$.

We have $M_Q \subset (\GL_{n-n_{\mathrm{c},1}'}\times\GL_{n_{\mathrm{c},1}'} ) \times ( \GL_{n+1-n_{\mathrm{c},2}'}\times \GL_{n_{\mathrm{c},2}'})$. Under this decomposition, let $\underline{z}_\tau$ be the element
    \begin{equation}
    \label{eq:underline_z_R}
        \underline{z}_\tau=((0,1/4),(0,1/4)) \in \fa_{\tau}^* \subset \fa_Q^*.
    \end{equation}
    Therefore, $\underline{z}_\tau$ only lives above the $\boxtimes \tau_{\mathrm{c},1,i}$ and $\boxtimes \tau_{\mathrm{c},2,i}$. If $Q=P$, this is $\underline{z}_{r}$. Finally, recall that we have defined some elements $\underline{\rho}_\tau$ and $\underline{\rho}_\tau^\uparrow$ in \eqref{eq:rho_P_down} and \eqref{eq:rho_P_up} respectively. We can now state our intermediate result.

    \begin{prop}
    \label{prop:I_pi_exp}
        We have
        \begin{equation}
        \label{eq:I_pi_exp}
            I_\pi=\sum_{(I,Q,\tau,I_1,I_2) \in \Pi_H^\uparrow(\pi)}\frac{|\Stab(\pi)|}{|\Stab(\tau)|} \int_{i \fa_{\tau}^*-\underline{\rho}_\tau-\underline{\rho}_\tau^\uparrow+2 \underline{z}_\tau} \sum_{\varphi \in \cB_{Q,\tau}(J)}\cP_{\tau}^\uparrow(\varphi,\lambda) \langle f,E(\varphi,-\overline{\lambda}) \rangle_{G} d\lambda.
        \end{equation}
    \end{prop}
    We will end the proof of Proposition~\ref{prop:I_pi_exp} in \S\ref{subsubsec:end_proof_prop}.

    \begin{rem}
        The appearance of $2$ is somewhat artificial and we may replace it by any real number $t$ such that $1<t<3$. It will be proved in the course of the proof that \eqref{eq:I_pi_exp} is well-defined, that is that $\sum_{\varphi}\cP_{\tau}^\uparrow(\varphi,\lambda) \langle f,E(\varphi,-\overline{\lambda}) \rangle_{G}$ is regular and of rapid decay in the region of integration.
    \end{rem}

    \subsubsection{A short description of the argument}
    \label{subsubsec:short_description1}

    Before starting the proof of Proposition~\ref{prop:I_pi_exp}, let us give a short survey of the argument. The idea is to use contour shifting in the integral \eqref{eq:I_pi_defi} defining $I_\pi$ to bring the domain of integration from $i \fa_\pi^*-\underline{\rho}_\pi+t \underline{z}_r$ to $i \fa_{\pi}^*-\underline{\rho}_\pi+2 \underline{z}_r$. This will be done by a step by step process to gradually decrease $t$, which amounts to crossing vertical strips. It takes the form of an induction argument which is the content of \S\ref{subsubsec:successive_changes}. Inside each strip we will encounter poles of the integrand. These singularities are described in \S\ref{subsubsec:poles_shift}. It turns out that they all come from the Eisenstein series $E(\varphi,-\overline{\lambda})$. More precisely, they occur when, with the notation of \eqref{eq:pi_defi}, a segment corresponding to a $\pi_{+,i}$ (resp. $\pi_{+,i}^\vee$) can be linked with a $\pi_{\mathrm{c},1,j}$ (resp. $\pi_{\mathrm{c},2,j}$). Because the latter is cuspidal, this can only happen if they are juxtaposed. We therefore obtain a new residual representation $\tau$. The residue of our integrand along such singularity is computed in \S\ref{subsubsec:computation_residues} : it is a relative character $I_\tau$ of $f$ along the induction of $\tau$ to $G$. It then remains to shift the contour in \eqref{eq:I_pi_defi} in a way such that the singularities that our contours of integration cross are always simple. This allows us to carry out the computation only by using the one-dimensional residue theorem. This technical step is done first on the variables coming from the $\GL_n$-side in \S\ref{subsubsec:contour_gln} and then from the $\GL_{n+1}$-side in \S\ref{subsubsec:contour_gl_n+1}. 

    To keep track of all the residues we catch along the way, we use the graph-theoretic formalism described above. Its meaning is the following: each time we cross a singularity coming from a juxtaposition between $\pi_{+,i}$ and $\pi_{\mathrm{c},1,j}$ (or $\pi_{+,i}^\vee$ and $\pi_{\mathrm{c},2,j}$), we draw an edge between the two corresponding vertices. The restraints we impose on our graphs $\Gamma$ mean that such residue can only occur at most once per $\pi_{\mathrm{c},1,j}$ and $\pi_{\mathrm{c},2,j}$, and at most twice per $\pi_{+,i}$/$\pi_{+,i}^\vee$, in which case one singularity comes from $\GL_n$ and the other from $\GL_{n+1}$. At the level of the associated tuple $(I_\Gamma,Q_\Gamma,\pi_\Gamma,I_{\Gamma,1},I_{\Gamma,2})$, these double residues are remembered in $I_{\Gamma,1}$ and $I_{\Gamma,2}$. Using this formalism, we are able to write $I_\pi$ as a weighted sum of relative characters $I_{\pi_\Gamma}$ indexed by graphs in $\Gamma \in \cG(\pi)$. It then remains to use a combinatorial argument to rather express it as sum over tuples in $\Pi_H^\uparrow(\pi)$. This is the content of \S\ref{subsubsec:counting_contributions}. Once this is done, the proof of Proposition~\ref{prop:I_pi_exp} is complete.

    \subsubsection{The successive changes of contours}
    \label{subsubsec:successive_changes}

     The proof of Proposition~\ref{prop:I_pi_exp} is quite involved and will take the reminder of this section. We will freely use the notation for the coordinate of an element $\lambda \in \fa_{\tau,\cc}^*-\underline{\rho}_\delta-\underline{\rho}_\delta^\uparrow$ from \S\ref{subsubsec:poles_of_P_pi_up}. We start from the definition of $I_\pi$ given in \eqref{eq:I_pi_defi}. The integral takes place in the region $i \fa_{\pi}^*-\underline{\rho}_{\pi}+t\underline{z}_{r}$ for some $t>0$ large enough. In particular, $t/2 > \underline{d}(+,i)$ for all $1 \leq i \leq m_+$.

    Let $d$ be a positive integer. We denote by $\Pi_H^\uparrow(\pi,d)$ the subset of $\Pi_H^\uparrow(\pi)$ consisting of the tuples $(I,Q,\tau,I_1,I_2)$ such that we have $\underline{d}(\tau,1,i) > d$ and $\underline{d}(\tau,2,i) > d$ for all $i$. Proposition~\ref{prop:I_pi_exp} now follows from the next proposition applied to $d=1$.

    \begin{prop}
    \label{prop:recu}
        For all $1 \leq d \leq \lceil{t/2}\rceil$, $I_\pi$ is equal to 
         \begin{equation}
         \label{eq:recu_to_prove}
            \sum_{(I,Q,\tau,I_1,I_2) \in \Pi_H^\uparrow(\pi,d)}\frac{|\Stab(\pi)|}{|\Stab(\tau)|} \int_{i \fa_{\tau}^*-\underline{\rho}_\tau-\underline{\rho}_\tau^\uparrow+2d \underline{z}_\tau} \sum_{\varphi \in \cB_{Q,\tau}(J)}\cP_{\tau}^\uparrow(\varphi,\lambda) \langle f,E(\varphi,-\overline{\lambda}) \rangle_{G} d\lambda.
        \end{equation}
    \end{prop}

    We will prove Proposition~\ref{prop:recu} by decreasing induction starting at $d= \lceil{t/2}\rceil$. The proof will take us until the end of the section and will be broken down in several lemmas. It will end in \S\ref{subsubsec:end_proof_prop}. We begin with the initialization step of the induction.

    \begin{lem}
        Proposition~\ref{prop:recu} holds for $d=\lceil{t/2}\rceil$.
    \end{lem}

    \begin{proof}
    By definition, the only element in $\Pi_H^\uparrow(\pi,\lceil{t/2}\rceil)$ is $(I_r,P,\pi)$. To conclude, we have to move the contour from $i \fa_{\pi}^*-\underline{\rho}_{\pi}+t\underline{z}_{r}$ to $i \fa_{\pi}^* -\underline{\rho}_\pi+2\lceil{t/2}\rceil \underline{z}_{r}$, which is possible by Proposition~\ref{prop:enter_Eisenstein}.
    \end{proof}

    In what follows we explain how to prove the induction step. The relevant integral will be the following. For any $(I,Q,\tau,I_1,I_2) \in \Pi_H^\uparrow(\pi,d+1)$, set
    \begin{equation}
    \label{eq:I_tau_defi}
        I_{\tau,d+1}:=\int_{i \fa_{\tau}^*-\underline{\rho}_\tau-\underline{\rho}_\tau^\uparrow+2(d+1) \underline{z}_\tau} \sum_{\varphi \in \cB_{Q,\tau}(J)} \cP_{\tau}^\uparrow(\varphi,\lambda) \langle f,E(\varphi,-\overline{\lambda}) \rangle_{G} d\lambda.
    \end{equation}
    That this integral is well defined will follow from Lemma~\ref{lem:poles_of_the_product}. 

    \subsubsection{Poles}
    \label{subsubsec:poles_shift}

    We want to move the contour in the integral \eqref{eq:I_tau_defi}. To achieve this, we begin by studying the poles of the integrand. However, it will be useful to do this for a larger class of representations. For now, we let $(I,Q,\tau,I_1,I_2) \in \Pi_H^\uparrow(\pi,d)$. Let $\varphi \in \cB_{Q,\tau}(J)$. We will work with the following region:
    \begin{equation}
    \label{eq:our_region}
       \cR_{\tau,d+1}^\uparrow := \left\{ \lambda -\underline{\rho}_\tau-\underline{\rho}_\tau^\uparrow  + \nu \; \middle| \; \lambda \in \fa_{\tau,\cc}^* \cap \cS_{\tau,k,c_J}, \; \nu \in [2d \underline{z}_\tau,2(d+1) \underline{z}_\tau]
        \right\},
    \end{equation}
    where we take $k$ and $c_J$ given by Theorem~\ref{thm:analytic_Eisenstein} and Proposition~\ref{prop:reg_P_pi_up}. Moreover, we may assume that $\cS_{\tau,k,c_J} \subset \{ \lambda \; | \; \norm{\Re(\lambda)}<\varepsilon \}$ for $\varepsilon$ small.  

    \begin{lem}
    \label{lem:poles_of_the_product}
        For $\varphi \in \cB_{Q,\tau}(J)$, the possible poles of $\cP_{\tau}^\uparrow(\varphi,\lambda) \langle f,E(\varphi,-\overline{\lambda}) \rangle_{G}$ in the region $\cR_{\tau,d+1}^\uparrow$ of \eqref{eq:our_region} are along the zeros of the polynomial  
            \begin{align}
       &\prod_{\substack{ \sigma_{+,i}(\tau) \simeq \tau_{\mathrm{c},1,j}\\ \underline{d}(\tau,+,j)=d}} \left( \lambda(+)_i-\lambda(1)_{\mathrm{c},j}-\frac{d+1}{2}\right) \prod_{\substack{\sigma_{+,i}(\tau)^\vee \simeq \tau_{\mathrm{c},2,j}   \\ \underline{d}(\tau,+,i)=d}} \left( -\lambda(+)_i-\lambda(2)_{\mathrm{c},j}-\frac{d+1}{2}\right) \nonumber \\
        \times&  \prod_{\substack{\sigma_{2,i}(\tau)^\vee \simeq \tau_{\mathrm{c},1,j} \\  \underline{d}(\tau,2,i)=d+1 \\i \notin I_2}} \left( -\lambda(2)_i-\lambda(1)_{\mathrm{c},j}-\frac{d+1}{2}\right) \prod_{\substack{\sigma_{1,i}(\tau)^\vee \simeq \tau_{\mathrm{c},2,j}  \\ \underline{d}(\tau,1,i)=d+1 \\  i \notin I_1}} \left(-\lambda(1)_i -\lambda(2)_{\mathrm{c},j}-\frac{d+1}{2}\right).     \label{eq:thepoles}
    \end{align}
    Moreover, it is of rapid decay in the sense that if we denote by $L$ the polynomial in \eqref{eq:thepoles}, for all $N>0$ there exists $C>0$ such that for all $\lambda$ in the region $\cR_{\tau,d+1}^\uparrow$ we have
    \begin{equation}
        \label{eq:rapid_decay_product} \sum_{\varphi \in \cB_{Q,\tau}(J)}\Val{L(\lambda)\cP_{\tau}^\uparrow(\varphi,\lambda) \langle f,E(\varphi,-\overline{\lambda}) \rangle_{G}} \leq \frac{C}{(1+\norm{\lambda}^2)^N}.
    \end{equation}
    In particular, the map $\lambda \mapsto \sum_{\varphi} L(\lambda)\cP_{\tau}^\uparrow(\varphi,\lambda) \langle f,E(\varphi,-\overline{\lambda}) \rangle_{G}$ is holomorphic in $\cR_{\tau,d+1}^\uparrow$.
    \end{lem}

    \begin{proof}
    
    By Corollary~\ref{cor:zeros_of_E}, we investigate the behavior of $L_{\tau,0}(-\lambda)\cP_{\tau}^\uparrow(\varphi,\lambda) \langle f,L_{\tau,0}^{-1}(-\overline{\lambda})E(\varphi,-\overline{\lambda}) \rangle_{G}$, where we recall that $L_{\tau,0}$ is the polynomial cutting out the zeros of the Eisenstein series.

    We begin with $L_{\tau,0}(-\lambda)\cP_{\tau}^\uparrow(\varphi,\lambda)$. First, we claim that our region is contained in the set \eqref{eq:weird_region} of Proposition~\ref{prop:reg_P_pi_up}. With the notation of \S\ref{subsubsec:poles_of_P_pi_up}, we have to show that for any $\nu$ as in \eqref{eq:our_region}, $(w^\uparrow (-\underline{\rho}_\tau-\underline{\rho}_\tau^\uparrow+\nu) +\nu^\uparrow))_{\cP^\uparrow}$ is positive. To prove this, say on the $\GL_n$-side, it is enough to note that the vector
    \begin{equation*}
        \left(\frac{1}{4}+\frac{\underline{d}(\tau,1,1)-1}{2},\hdots,\frac{1}{4}+\frac{\underline{d}(\tau,1,m_1)-1}{2},\frac{d+1}{2},\hdots,\frac{d+1}{2}\right)
    \end{equation*}
    is positive. This holds because $\underline{d}(\tau,1,1)\geq \hdots \geq \underline{d}(\tau,1,m_1)>d$, and the same argument works equally well on the $\GL_{n+1}$-side.
    
    Therefore, the poles of $\cP_{\tau}^\uparrow(\varphi,\lambda)$ are controlled by $L_{\tau,\cP}^{\uparrow}$. By comparing \eqref{eq:f_pi_0_defi} (definition of $L_{\tau,0}$) and \eqref{eq:f_pi_regular} (definition of $L_{\tau,\cP}$), we see that the poles of the product are contained in the zeros of 
    \begin{align*}
        &\prod_{\tau_{\mathrm{c},1,i} \simeq \tau_{\mathrm{c},2,j}^{\vee}} \left(\lambda(1)_{\mathrm{c},i}+\lambda(2)_{\mathrm{c},j}\pm\frac{1}{2}\right) \times \prod_{  \sigma_{2,i}(\tau)^\vee \simeq \tau_{\mathrm{c},1,j}} \left(\lambda(2)_i+\lambda(1)_{\mathrm{c},j}+\frac{\underline{d}(\tau,2,i)-1\pm1}{2}\right)  \\
        \times &\prod_{\sigma_{1,i}(\tau)^\vee \simeq \tau_{\mathrm{c},2,j}} \left(\lambda(1)_{i}+\lambda(2)_{\mathrm{c},j}+\frac{\underline{d}(\tau,1,i)-1\pm1}{2}\right).
    \end{align*}
    We claim that this product is non zero in our region $\cR_{\tau,d+1}^\uparrow$. Indeed, for the first factor we always have $\Re(\lambda(1)_{\mathrm{c},i}+\lambda(2)_{\mathrm{c},j})\geq d-2\varepsilon \geq 1-2\varepsilon>0$. For the second, we have $\Re(\lambda(2)_i+\lambda(1)_{\mathrm{c},j}) \geq d/2+1/4-2\varepsilon\geq 3/4-2\varepsilon$, while $(\underline{d}(\tau,2,i)-1\pm 1)/2 \geq -1/2$, and the same argument works for the third. Therefore, $L_{\tau,0}(\lambda) \cP_{\tau}^\uparrow(\varphi,\lambda)$ is regular in $\cR_{\tau,d+1}^\uparrow$.

    We now study $L_{\tau,0}^{-1}(-\lambda) E(\varphi,-\lambda)$. It is enough to investigate the meromorphic functions $L_{\tau_n,0}^{-1}(-\lambda_n) E(\varphi_n,-\lambda_n)$ and $L_{\tau_{n+1},0}^{-1}(-\lambda_{n+1}) E(\varphi_{n+1},-\lambda_{n+1})$ separately, so that we only explain the first case. Because we always have $\underline{\rho}_\tau+\underline{\rho}_\tau^\uparrow-\nu \in \overline{\fa_{Q}^{*,+}}$, the region $\cR_{\tau,d+1}^\uparrow$ is contained in $\cR_{\tau,k,c_J}$ and we can use Corollary~\ref{cor:zeros_of_E}. It follows that the poles of $L_{\tau_n,0}^{-1}(-\lambda_n) E(\varphi_n,-\lambda_n)$ are controlled by $L_{\tau_n,E}$. 
    
    By going back to the definition in \eqref{eq:f_pi_defi}, we see that most factors will immediately be non-zero in our region. Indeed, two cases can be easily excluded. The first is the possible poles coming from isomorphisms $\sigma_{+,i}(\tau) \simeq \sigma_{2,j}(\tau)^\vee $. Then the shifted segments associated to the representations $\tau_{+,i}$ and $\tau_{2,j}^{-,\vee}$ can not be linked for $\lambda$ in the region $\cR_{\tau,d+1}^\uparrow$ in the sense of \S\ref{subsubsec:BZ_segments} because 
    $\Re(\lambda(+)_{n,i}) \approx \Re(\lambda(2)_{n,j}) \approx 1/4$ (up to $\varepsilon$). It is indeed straightforward that two segments with almost the same mean can not be linked. The other easy situation is $\sigma_{+,i}(\tau) \simeq \sigma_{1,j}(\tau)$ or $\sigma_{2,k}(\tau)^\vee \simeq \sigma_{1,j}(\tau)$. Indeed, in both cases we have $\Re(\lambda(+)_{n,i}) \approx \Re(\lambda(2)_{n,k}) \approx 1/4$ and $\Re(\lambda(1)_{n,j}) \simeq -1/4$. But it is straightforward that two segments whose means differ by almost $1/2$ can not be linked. 
    
    Therefore, the only possible poles are contained in the zeros of
    \begin{align}
        &\prod_{\sigma_{1,i}(\tau) \simeq \tau_{\mathrm{c},1,j}} \left( \lambda(1)_i-\lambda(1)_{\mathrm{c},j}-\frac{\underline{d}(\tau,1,i)+1}{2}\right)\nonumber \\
        & \times \prod_{\substack{\sigma_{2,i}(\tau)^\vee \simeq \tau_{\mathrm{c},1,j} \\ i \notin I_2}} \left( -\lambda(2)_i-\lambda(1)_{\mathrm{c},j}-\frac{\underline{d}(\tau,2,i)}{2}\right) \prod_{\sigma_{+,i}(\tau) \simeq \tau_{\mathrm{c},1,j}} \left( \lambda(+)_i-\lambda(1)_{\mathrm{c},j}-\frac{\underline{d}(\tau,+,i)+1}{2}\right). \label{eq:last_possible_poles}
    \end{align}
    Note that the change of signs in the second equation is due to the fact that in our coordinates we have $\lambda(2)_{n,i}=-\lambda(2)_i$ (see \S\ref{subsubsec:condition_up}). The first product is non-zero in the region $\cR_{\tau,d+1}^\uparrow$ because $\Re(\lambda(1)_i-\lambda(1)_{\mathrm{c},j}) \leq -1/4+(d+1)/2<(\underline{d}(\tau,1,i)+1)/2$. For the second, the same argument shows that the zeros come from the terms with $\underline{d}(\tau,2,i)=d+1$, and for the third from $\underline{d}(\tau,+,i)=d$.

    By repeating the argument on the $\GL_{n+1}$ side, we finally see that the only possible poles of $\cP_{\tau}^\uparrow(\varphi,\lambda) \langle f,E(\varphi,-\overline{\lambda}) \rangle_G$ in the region $\cR_{\tau,d+1}^\uparrow$ are indeed along the zeros of the polynomial in \eqref{eq:thepoles}.

    To conclude the proof of Lemma~\ref{lem:poles_of_the_product}, it remains to prove that $\cP_{\tau}^\uparrow(\varphi,\lambda) \langle f,E(\varphi,-\overline{\lambda}) \rangle_G$ is of rapid decay in the sense of \eqref{eq:rapid_decay_product}. This is a direct consequence of Propositions~\ref{prop:convergence_strong} and~\ref{prop:bound_P_pi_up}.
    \end{proof}

    \subsubsection{Computation of residues}
    \label{subsubsec:computation_residues}

    Assume now that $(I,Q,\tau,I_1,I_2) \in \Pi_H^{\uparrow}(\pi,d)$ comes from a graph $\Gamma \in \cG(\pi)$. We now explain how to assign to each affine linear form $\Lambda$ appearing in \eqref{eq:thepoles} a graph $\Gamma_\Lambda \in \cG(\pi)$. 
    
    Let us explain the case $\Lambda(\lambda)=\lambda(+)_{i_+}-\lambda(1)_{\mathrm{c},i_1}-(d+1)/2$. By definition of $\Pi_H^\uparrow(\pi)$, we have indices $j_+$ and $j_1$ such that $\tau_{+,i_+}=\pi_{+,j_+}$ and $\tau_{\mathrm{c},1,i_1}=\pi_{\mathrm{c},1,j_1}$. Note that these indices $j_+$ and $j_1$ are uniquely determined if we ask that they correspond exactly to those given by the presentation $\tau$ as $\pi(\Gamma)$ given in \eqref{eq:pi_gamma_n}. Let $\Gamma_\Lambda$ be the graph obtained from $\Gamma$ by adding the edge $\{ \pi_{+,j_+},\pi_{\mathrm{c},1,j_1}\}$. Then we have $\Gamma_\Lambda \in \cG(\pi)$. For simplicity, we will write $Q_\Lambda$ instead of $Q_{\Gamma_\Lambda}$, $\pi_\Lambda$ instead of $\pi_{\Gamma_\Lambda}$ and so on.

    For $\Lambda(\lambda)= \lambda(+)_{i_+}-\lambda(2)_{\mathrm{c},i_2}-(d+1)/2$, the graph $\Gamma_\Lambda$ is built the same way. Let us explain the last case that will be relevant to us which is $\Lambda(\lambda)=-\lambda(1)_{i_1}-\lambda(2)_{\mathrm{c},i_2}-(d+1)/2$ with $\sigma_{1,i_1}(\tau)^\vee \simeq \tau_{\mathrm{c},2,i_2}$, $\underline{d}(\tau,1,i_1)=d+1$ and $i_1 \notin I_1$. As before, define $j_1$ and $j_2$ such that $\tau_{1,i_1}^{-,\vee}=\pi_{+,j_1}^\vee$ and $\tau_{\mathrm{c},2,i_2}=\pi_{\mathrm{c},2,j_2}$. Then we add the edge $\{\pi_{+,j_1},\pi_{\mathrm{c},2,j_2}\}$. Note that $\pi_{+,j_1}$ already had degree $1$ as it was connected to the cuspidal representation $\pi_{\mathrm{c},1,l}$ used to obtain $\tau_{1,i_1}$ on the $\GL_n$-side.

    We now describe the residues obtained along the zeros of \eqref{eq:thepoles}. Note that they are all at most simple and that we may use the naive notion of residue described in \S\ref{subsubsec:naive_residue}. Once again, we begin with the case $\Lambda(\lambda)= \lambda(+)_{i_+}-\lambda(2)_{\mathrm{c},i_2}-(d+1)/2$. There exists a unique integer $1 \leq k \leq m_{1}(\tau)$ such that 
        \begin{align*}
        \pi_{\Lambda,n}&=\boxtimes_{\substack{i=1 \\ i \neq i_+}}^{m_+(\tau)} \tau_{+,i}\boxtimes_{\substack{i=1 \\ i \notin I_2}}^{m_2(\tau)}\tau_{2,i}^{-,\vee} \boxtimes \left(  \boxtimes_{i=1}^{k-1} \tau_{1,i} \boxtimes \pi_{+,j_+}^+ \boxtimes_{i=k}^{m_{1}(\tau)} \tau_{1,i} \right)  \boxtimes_{\substack{i=1 \\ i \neq i_1}}^{m_{\mathrm{c},1}(\tau)} \tau_{\mathrm{c},1,i}  , \\     
        \pi_{\Lambda,n+1}&=\boxtimes_{\substack{i=1 \\ i \neq i_+}}^{m_+(\tau)} \tau_{+,i}^\vee \boxtimes \left(\boxtimes_{\substack{i=1 \\ i \notin I_1}}^{k-1} \tau_{1,i}^{-,\vee} \boxtimes \pi_{+,j_+}^\vee \boxtimes_{\substack{i=k \\ i \notin I_1}}^{m_1(\tau)} \tau_{1,i}^{-,\vee}  \right) \boxtimes_{i=1}^{m_2(\tau)} \tau_{2,i} \boxtimes_{i=1}^{m_{\mathrm{c},2}(\tau)} \tau_{\mathrm{c},2,i}.
    \end{align*}
    Here we need the integer $k$ because we have imposed that the representations in \eqref{eq:pi_gamma_n} and \eqref{eq:pi_gamma_n+1} appear in a certain order. Associated to $\Gamma_\Lambda$ is a regularized period $\cP_{\pi_\Lambda}^\uparrow$.
    
    Let $w \in W(Q)$ such that $w \tau$ is the representation with $(w\tau)_{n+1}=\pi_{\Lambda,n+1}$ and 
    \begin{equation}
    \label{eq:w_tau_explicit}
        (w \tau)_n=\boxtimes_{\substack{i=1 \\ i \neq i_+}}^{m_+(\tau)} \tau_{+,i}\boxtimes_{\substack{i=1 \\ i \notin I_2}}^{m_2(\tau)}\tau_{2,i}^{-,\vee} \boxtimes \left(  \boxtimes_{i=1}^{k-1} \tau_{1,i} \boxtimes (\tau_{+,i_+} \boxtimes \tau_{\mathrm{c},1,i_1}) \boxtimes_{i=k}^{m_{1}(\tau)} \tau_{1,i} \right)  \boxtimes_{\substack{i=1 \\ i \neq i_1}}^{m_{\mathrm{c},1}(\tau)} \tau_{\mathrm{c},1,i}.
    \end{equation}
    Note that
    \begin{equation}
    \label{eq:w_iso}
        w\left(\Lambda^{-1}(\{0\}) \cap (\fa_{\tau,\cc}^*-\underline{\rho}_\tau-\underline{\rho}_\tau^\uparrow)\right)=\fa_{\pi_\Lambda,\cc}^*-\underline{\rho}_{\pi_\Lambda}-\underline{\rho}_{\pi_\Lambda}^\uparrow-\nu_{w. Q,\pi_\Lambda}^{Q_\Lambda},
    \end{equation}
    where $\nu_{w.Q,\pi_\Lambda}^{Q_\Lambda}$ is the twist defined in \eqref{eq:nu_exp_defi}. Moreover, up to shrinking the constants we have
    \begin{equation}
    \label{eq:stability_region}
        \lambda \in \Lambda^{-1}(\{0\}) \cap \cR_{\tau,d+1}^\uparrow \implies w \lambda + \nu_{w.Q,\pi_\Lambda}^{Q_\Lambda} \in \cR_{\pi_\Lambda,d+1}^\uparrow.
    \end{equation}

    If $\Lambda(\lambda)= \lambda(+)_{i_+}-\lambda(2)_{\mathrm{c},i_2}-(d+1)/2$, the construction is exactly the same by on the $\GL_{n+1}$-side. In the case $\Lambda(\lambda)=-\lambda(1)_{i_1}-\lambda(2)_{\mathrm{c},i_2}-(d+1)/2$, we have $\pi_{\Lambda,n}=\tau_n$ and 
      \begin{equation*}
        \pi_{\Lambda,n+1}=\boxtimes_{i=1}^{m_+(\tau)} \tau_{+,i}\boxtimes_{i \notin I_1 \cup \{i_1\}} \tau_{1,i}^{-,\vee}\boxtimes \left( \boxtimes_{i=1}^{k-1} \tau_{2,i} \boxtimes \tau_{1,i_1}^\vee  \boxtimes_{i=k}^{m_2(\tau)} \tau_{2,i} \right)  \boxtimes_{\substack{i=1 \\ i \neq i_2}}^{m_{\mathrm{c},2}(\tau)} \tau_{\mathrm{c},2,i},
      \end{equation*}
    for some index $k$. In this case, let $w \in W(R_{n+1})$ such that 
      \begin{equation*}
        (w\tau)_{n+1}=\boxtimes_{i=1}^{m_+(\tau)} \tau_{+,i}\boxtimes_{i \notin I_1 \cup \{i_1\}} \tau_{1,i}^{-,\vee}\boxtimes \left( \boxtimes_{i=1}^{k-1} \tau_{2,i} \boxtimes \left(\tau_{1,i_1}^{-,\vee} \boxtimes \tau_{\mathrm{c},2,i_2}\right)  \boxtimes_{i=k}^{m_2(\tau)} \tau_{2,i} \right)\boxtimes_{\substack{i=1 \\ i \neq i_2}}^{m_{\mathrm{c},2}(\tau)} \tau_{\mathrm{c},2,i}
    \end{equation*}
    Then the relations \eqref{eq:w_iso} and \eqref{eq:stability_region} still hold. 
    
    \begin{lem}
    \label{lem:compute_residues}
        Let $\Lambda$ be any affine linear form appearing in \eqref{eq:thepoles}. For $\lambda \in \Lambda^{-1}(\{0\}) \cap (\fa_{\tau,\cc}^*-\underline{\rho}_\tau-\underline{\rho}_\tau^\uparrow)$ in general position we have
        \begin{equation}
        \label{eq:the_residue}
            \underset{\Lambda}{\Res}\left(\sum_{\varphi \in \cB_{Q,\tau}(J)} \cP_{\tau}^\uparrow(\varphi,\mu) \langle f,E(\varphi,-\overline{\mu}) \rangle_G \right)(\lambda)=-\sum_{\varphi \in \cB_{Q_\Lambda,\pi_\Lambda}(J)} \cP_{\pi_\Lambda}^\uparrow(\varphi,\lambda') \langle f,E(\varphi,-\overline{\lambda'}) \rangle_G,
        \end{equation}
        where $\lambda'=w\lambda + \nu_{w. Q,\pi_\Lambda}^{Q_\Lambda}$.
    \end{lem}

    \begin{proof}
        We only deal with the case $\Lambda(\lambda)=\lambda(+)_{i_+}-\lambda(1)_{\mathrm{c},i_1}-(d+1)/2$, the others being exactly the same. For simplicity, write $R=Q_\Lambda$ and $\delta=\pi_\Lambda$.
        
        That the residue is well-defined in our naive sense follows from Lemma~\ref{lem:poles_of_the_product}. By the unitarity of the intertwining operators and the functional equation of Eisenstein series (Theorem~\ref{thm:analytic_Eisenstein}), we see using a change of variabless and analytic continuation that for $\lambda \in \fa_{Q,\cc}^*$ in general position we have
        \begin{equation*}
            \sum_{\varphi \in \cB_{Q,\tau}(J)} \cP_{\tau}^\uparrow(\varphi,\lambda) \langle f,E(\varphi,-\overline{\lambda}) \rangle_G =\sum_{\varphi \in \cB_{w.Q,w\tau}(J)} \cP_\tau^{\uparrow}\left(M(w^{-1},w\lambda)\varphi,\lambda\right) \langle f,E(\varphi,-w\overline{\lambda}) \rangle_G.
        \end{equation*}
        Because we know the poles of $M(w^{-1},w\lambda)$ (Proposition~\ref{prop:M_regular}) and of $\cP_{\tau}^\uparrow$ (Proposition~\ref{prop:reg_P_pi_up}), we conclude that $\Lambda^{-1}(\{0\})$ is not contained in any of the singularities of $\lambda \mapsto \cP_\tau^{\uparrow}\left(M(w^{-1},w\lambda)\varphi,\lambda\right)$ in $\fa_{\tau,\cc}^*-\underline{\rho}_\tau-\underline{\rho}_\tau^\uparrow$. Moreover, by Lemma~\ref{lem:poles_of_the_product}, using \cite[Theorem~2.2]{Lap} we can compute the residue inside the sum $\sum_\varphi$ and the inner-product $\langle \cdot,\cdot \rangle_G$. It follows that the residue in \eqref{eq:the_residue} is 
        \begin{equation*}
            \sum_{\varphi \in \cB_{w.Q,w\tau}(J)} \cP_\tau^{\uparrow}\left(M(w^{-1},w\lambda)\varphi,\lambda\right) \langle f,\underset{\overline{\Lambda}}{\Res}\left(E(\varphi,-w\overline{\mu})\right)(\lambda) \rangle_G.
        \end{equation*}
        Let $\varphi \in \cA_{w.Q,w\tau}(G)$. By yet another use of \cite[Theorem~2.2]{Lap}, we see that we have
         \begin{equation*}
            \underset{\overline{\Lambda}}{\Res}\left(E(\varphi,-w\overline{\mu})\right)(\lambda)=-E\left(E^{R,*}(\varphi,-\overline{\lambda'}-\nu_{w. Q,\delta}^{R})\right). 
        \end{equation*}
        where we have by definition of $\lambda'$
        \begin{equation*}
           E^{R,*}(\varphi,-\overline{\lambda'}-\nu_{w. Q,\delta}^{R})=-\underset{\overline{\Lambda}}{\Res}( E^{R}(\varphi,-w\overline{\mu}))(\lambda).
        \end{equation*}
        Note that $(w.Q)_{w\tau}=R_{\delta}$. Let $\sigma_{w\tau}$ be the cuspidal representation of $M_{R_{\delta}}$ defined in \S\ref{subsubsec:residual_blocks}, so that there exists $\phi \in \cA_{\sigma_{w\tau},R_{\delta}}(G)$ with $\varphi=E^{w.Q,*}(\phi,-\nu_{w\tau})$. Then by the definition of $E^{R,*}$ given in \eqref{eq:residual}, and because $\nu_{w\tau}+\nu_{w. Q,\delta}^{R}=\nu_{\delta}$, we have
        \begin{equation*}
             E^{R,*}(\varphi,-\nu_{w.Q,\delta}^{R})=E^{R,*}(\phi,-\nu_{\delta}).
        \end{equation*}
        In particular, for every $\varphi' \in \cA_{Q,\delta}(G)$ by Proposition~\ref{prop:adjunt} we get
        \begin{equation*}
            \langle \varphi',E^{R,*}(\varphi,-\nu_{w. Q,\delta}^{R}) \rangle_{R,\Pet}=\langle \varphi',E^{R,*}(\phi,-\nu_{\delta}) \rangle_{R,\Pet}=\langle \varphi'_{R_{\delta}},\phi \rangle_{R_{\delta},\Pet},
        \end{equation*}
        but also by transitivity of constant term
        \begin{equation*}
             \langle \varphi'_{w. Q},\varphi \rangle_{w.Q,\Pet}=\langle \varphi'_{w.Q},E^{w.Q,*}(\phi,-\nu_{w \tau}) \rangle_{w.Q, \Pet}=\langle \varphi'_{R_{\delta}},\phi \rangle_{R_{\delta},\Pet}.
        \end{equation*}
        It follows that the regularized Eisenstein series $E^{R,*}$ (defined on $\cA_{w. Q,w\tau}$) and the constant term along $w. Q$ (defined on $\cA_{R,\delta}(G)$) are adjoint. This implies that the residue in \eqref{eq:the_residue} is
        \begin{equation*}
           - \sum_{\varphi \in \cB_{R,\delta}(J)} \cP^{\uparrow}_{\tau}\left(M(w^{-1},\lambda')\varphi_{w. Q},w^{-1} \lambda'\right) \langle f,E(\varphi,-\overline{\lambda'}) \rangle_G.
        \end{equation*}
        
        Let $P_{\delta}^\uparrow$, $P_{\tau}^\uparrow$ and $w_{\delta}^\uparrow$, $w_{\tau}^\uparrow$ respectively be the Rankin--Selberg parabolic subgroup and element of the Weyl group defined for $\delta$ and $\tau$ (or rather the associated tuples) in \S\ref{subsubsec:condition_up} and \S\ref{subsubsec:regularized_increasing}. By going back to their definitions, we see that there exist $w_\cM \in W_\cM(\tau)$ and $w_\bfM \in W_\bfM(\tau)$ such that $w_\cM w_\bfM w^{\uparrow}_\tau=w^{\uparrow}_{\delta}w$ (see \S\ref{subsubsec:functional_up}) and that moreover $w_{\bfM}.P_{\tau}^{\uparrow}=P_{\delta}^\uparrow$. By \eqref{eq:constant_inter}, the functional equation of regularized period proved in Lemma~\ref{lem:functional_up} and the definition of $\cP_{\delta}^\uparrow$ in \eqref{eq:P_up}, we obtain 
        \begin{equation*}
            \cP^{\uparrow}_{\tau}\left(M(w^{-1},\lambda')\varphi_{w. Q},w^{-1} \lambda'\right)=\cP^{w_{\bfM}.P_{\tau}^{\uparrow}}\left(M(w_{\delta}^\uparrow,\lambda)\varphi_{P_{\delta}^\uparrow},w_{\delta}^{\uparrow} \lambda' \right)=\cP_{\delta}^\uparrow(\varphi,\lambda').
        \end{equation*}
        This concludes the proof of the lemma.        
    \end{proof}

    \begin{rem}
        To prove Lemma~\ref{lem:compute_residues}, we crucially use the fact that, up to a change of variabless, our singularities comes from a juxtaposition of two segments of $\tau$ (in the sense of \S\ref{subsubsec:BZ_segments}). In this case, the image spanned by the residues of Eisenstein series is the space of Eisenstein series induced from the discrete automorphic representation obtained by juxtaposing the two corresponding components of $\tau$. For the other singularities, computing the residues is a hard problem (see e.g. \cite{GS}).
    \end{rem}

    \subsubsection{Change of contours on the $\GL_n$-side}
    \label{subsubsec:contour_gln}

    We now fix $(I,Q,\tau,I_1,I_2) \in \Pi_H^\uparrow(\pi,d+1)$ and assume that it comes from a graph $\Gamma \in \cG(\pi)$. We begin our change of contours in the integral $I_{\tau,d+1}$ of \eqref{eq:I_tau_defi} with the variables on the $\GL_n$-side. We define $\cG(\Gamma,n,d)$ to be the subset of graphs $\Gamma' \in \cG(\pi)$ such that 
    \begin{itemize}
        \item $\Gamma$ is a subgraph of $\Gamma'$;
        \item all the edges in $\Gamma'$ that do not belong to $\Gamma$ are of the form $\{ \pi_{+,i},\pi_{\mathrm{c},1,j} \}$ with $\underline{d}(+,i)=d$.
    \end{itemize}
    In words, the graphs $\Gamma' \in \cG(\Gamma,n,d)$ are obtained by adding edges to $\Gamma$ exactly as we did to describe the residue of $\sum \cP_{\tau}^\uparrow(\varphi,\lambda) \langle f,E(\varphi,-\overline{\lambda})\rangle_G$ along $\Lambda(\lambda)=\lambda(+)_{i}-\lambda(1)_{\mathrm{c},j}-(d+1)/2$ in Lemma~\ref{lem:compute_residues}. For short, we set $\underline{z}_{\Gamma'}:=\underline{z}_{\tau_{\Gamma'}}$ and so on.

    Recall that our goal is to compute the integral \begin{equation*}
        I_{\tau,d+1}=\int_{i \fa_{\tau}^*-\underline{\rho}_\tau-\underline{\rho}_\tau^\uparrow+2(d+1) \underline{z}_\tau} \sum_{\varphi \in \cB_{Q,\tau}(J)} \cP_{\tau}^\uparrow(\varphi,\lambda) \langle f,E(\varphi,-\overline{\lambda}) \rangle_{G} d\lambda.
    \end{equation*}

    \begin{lem}
    \label{lem:residue_n}
        We have
        \begin{equation}
        \label{eq:residue_n}
           I_{\tau,d+1}=\sum_{\Gamma' \in \cG(\Gamma,n,d)} \int_{i \fa_{\pi_{\Gamma'}}^*+\kappa_{d,d+1}(\Gamma')} \sum_{\varphi \in \cB_{Q_{\Gamma'},\pi_{\Gamma'}}(J)} \cP_{\pi_{\Gamma'}}^\uparrow(\varphi,\lambda) \langle f,E(\varphi,-\overline{\lambda}) \rangle_{G} d\lambda,
        \end{equation}
        where
        \begin{equation*}           
        \kappa_{d,d+1}(\Gamma')=-\underline{\rho}_{\Gamma'}-\underline{\rho}_{\Gamma'}^\uparrow+2d\underline{z}_{\Gamma',n}+2(d+1) \underline{z}_{\Gamma',n+1}.
        \end{equation*}
    \end{lem}

    \begin{proof}
        Note that the set $\cG(\Gamma,n,d)$ is finite. For every $\Gamma' \in \cG(\Gamma,n,d)$, let $L_{\Gamma'}$ be the polynomial defined in \eqref{eq:thepoles} which controls the poles of $\cP_{\pi_{\Gamma'}}^\uparrow(\varphi,\lambda) \langle f,E(\varphi,-\overline{\lambda}) \rangle_{G}$ in the region $\cR_{\pi_{\Gamma'},d+1}^\uparrow$ of \eqref{eq:our_region}. Denote by $(I',Q',\tau',I_1',I_2')$ the image of $\Gamma'$. By the definition of $\cS_{\tau',k,c_J}$ in \eqref{eq:S_defi}, there exist $k'>0$ and $c>0$ such that for all $T>0$ we have
       \begin{equation}
         \label{eq:truncated_integration}
            \left\{ \lambda \in \fa_{\tau',\cc}^* \; | \; \max \Val{\Im(\lambda_i)} \leq T, \; \max \Val{\Re(\lambda_i)} \leq c(1+T)^{-k'} \right\}\subset \cS_{\tau',k,c_J}.
        \end{equation}
        Note that $k$, $c_J$, $k'$ and $c$ can be chosen independently of $\tau'$ by finiteness. 
        
        Let $D$ be the maximum of the degrees of the $L_{\Gamma'}$ and take $N \geq Dk'+1$. Let $\varepsilon >0$. By Lemma~\ref{lem:poles_of_the_product}, there exists $T>0$ such that for every $\tau'$ and every index $i$ we have
        \begin{equation}
        \label{eq:est_sum_1}
            \sup_{\substack{\lambda \in \cR_{\tau',d+1}^\uparrow \\  \Val{\Im(\lambda_{i})} \geq T}} \left((1+\norm{\lambda}^2)^{N}\sum_{\varphi \in \cB_{R',\tau'}(J)} \Val{L_{\Gamma'}(\lambda)\cP_{\tau'}^\uparrow(\varphi,\lambda) \langle f,E(\varphi,-\overline{\lambda}) \rangle_{G} }\right) \leq \varepsilon,
        \end{equation}
        and moreover
        \begin{equation}
        \label{eq:est_sum_2}
            (c/4)^{-D} (2T+1)^{Dk'}\int_{\substack{\lambda \in i\fa_{\tau'}^* \\ \Val{\lambda_i} \geq T} } \frac{1}{(1+\norm{\lambda}^2)^N}d\lambda \leq 1, \quad \frac{(c/4)^{-D} (2T+1)^{Dk'}}{(1+T^2)^{N/2}}\int_{\substack{\lambda \in i\fa_{\tau'}^* \\ \lambda_i=0} } \frac{1}{(1+\norm{\lambda}^2)^{N/2}}d\lambda \leq 1.
        \end{equation}
        
        Let $\underline{\varepsilon}_{\tau'} \in \fa_{\tau'}^*$ such that all the coordinates of $\underline{\varepsilon}_{\tau'}(+)$ are distinct and of absolute value less than $c(2T+1)^{-k'}/4$, and all the other are zero. Set $\eta=3c(2T+1)^{-k'}$. Define
         \begin{equation*}
            \kappa_{\tau'}=-\underline{\rho}_{\tau'}-\underline{\rho}_{\tau'}^\uparrow+(2d+1)\underline{z}_{\tau',n}+2(d+1) \underline{z}_{\tau',n+1}.
        \end{equation*}
        This point is the middle of the segment $[\kappa_{d,d+1}(\Gamma'),-\underline{\rho}_{\tau'}-\underline{\rho}_{\tau'}^\uparrow+2(d+1) \underline{z}_{\tau'}]$.       

        Let $(I',Q',\tau',I_1',I_2')$ be the image of some graph $\Gamma'$. By Lemma~\ref{lem:poles_of_the_product}, we can describe the poles of $\cP_{\tau'}^\uparrow(\varphi,\lambda) \langle f,E(\varphi,-\overline{\lambda}) \rangle_{G}$ in the region
        \begin{equation}
        \label{eq:n_region}
            \left\{ \lambda -\underline{\rho}_{\tau'}-\underline{\rho}_{\tau'}^\uparrow  + \nu \; \middle| \; \lambda \in \fa_{\tau',\cc}^* \cap \cS_{\tau',k,c_J}, 
            \begin{array}{l}
                \nu_n \in [2d \underline{z}_{\tau',n},2(d+1) \underline{z}_{\tau',n}], \\
                 \nu_{n+1} =2(d+1) \underline{z}_{\tau',n+1}.
            \end{array}
        \right\},
        \end{equation}
        which is contained in $\cR_{\tau',d+1}^\uparrow $. Because $\tau$ comes from a tuple belonging to $\Pi_H^\uparrow(\pi,d+1)$, by the definition of $\cG(\Gamma,n,d)$ we see that we always have $\underline{d}(\tau',2,j)>d+1$, so that the only possible poles in the region \eqref{eq:n_region} are those of the product
        \begin{equation}
            \label{eq:thepoles_n}
           \prod_{\substack{ \sigma_{+,i}(\tau') \simeq \tau'_{\mathrm{c},1,j}\\ \underline{d}(\tau',+,j)=d}} \left( \lambda(+)_i-\lambda(1)_{\mathrm{c},j}-\frac{d+1}{2}\right).
        \end{equation}
        We can also replace $L_{\Gamma'}$ by this product.

        We now go back to our initial representation $\tau$. By the previous discussion, we can move the contour in the integral to obtain
        \begin{equation}
        \label{eq:first_CS}
            I_{\tau,d+1}=\int_{i \fa_{\tau}^*+\kappa_\tau+\eta \underline{z}_{\tau,n}} \sum_{\varphi \in \cB_{Q,\tau}(J)} \cP_{\tau}^\uparrow(\varphi,\lambda) \langle f,E(\varphi,-\overline{\lambda}) \rangle_{G} d\lambda.
        \end{equation}
        For any $\lambda \in i \fa_{\tau}^*+\kappa_\tau+\eta \underline{z}_{\tau,n}$ we have 
        \begin{equation}
            \label{eq:complicated_bound}
            \Val{L_{\Gamma'}(\lambda)} \geq (\eta/4)^D \geq (c/4)^{D} (2T+1)^{-k'D}.
        \end{equation}
        This implies that for any $i$ we have by \eqref{eq:est_sum_1} and \eqref{eq:est_sum_2}
        \begin{equation*}
            \int_{\substack{i \fa_{\tau}^*+\kappa_\tau+\eta \underline{z}_{\tau,n} \\  \Val{\Im(\lambda_{i})} \geq T}} \sum_{\varphi \in \cB_{Q,\tau}(J)} \Val{\cP_{\tau}^\uparrow(\varphi,\lambda) \langle f,E(\varphi,-\overline{\lambda}) \rangle_{G}} d\lambda \leq \varepsilon,
        \end{equation*}
        Therefore, there exists a constant $C$ (independent of $\varepsilon$) such that
        \begin{equation*}
            \Val{I_{\tau,d+1}-\int_{\substack{i \fa_{\tau}^*+\kappa_\tau+\eta \underline{z}_{\tau,n} \\ \max \Val{\Im(\lambda_i)} \leq T}} \sum_{\varphi \in \cB_{Q,\tau}(J)} \cP_{\tau}^\uparrow(\varphi,\lambda) \langle f,E(\varphi,-\overline{\lambda}) \rangle_{G} d\lambda} \leq C \varepsilon,
        \end{equation*}

        We focus on the second integral. By Fubini's theorem, we can assume that we are integrating last in the variable $\lambda(+)_1$. We change the contour in this last integral, so that all the other variables are fixed. We want to go from $\Re(\lambda(+)_1)=0$ to $\Re(\lambda(+)_1)=\underline{\varepsilon}_{\tau}(+)_1$, still with the condition $\Val{\Im(\lambda_i)}\leq T$. To achieve this we link these two segments with the two additional segments of real part $\Re(\lambda(+)_1) \in [0,\underline{\varepsilon}_\tau(+)_1]$ and imaginary part $\Im(\lambda(+)_1)=\pm T$. Let us denote this contour by $\gamma$. By the description of \eqref{eq:thepoles_n}, our integrand has no poles inside $\gamma$. Indeed, the key is that with our choice of $\eta$, $\eta \underline{z}_{\tau,n}$ has no coordinates with real part lying in $[-\underline{\varepsilon}_\tau(+)_1,\underline{\varepsilon}_\tau(+)_1]$. Moreover, along the two real segments of $\gamma$, as in \eqref{eq:complicated_bound} we have the estimate $ \Val{L_{\Gamma'}(\lambda)} \geq (c/4)^{D} (2T+1)^{-k'D}$. By \eqref{eq:est_sum_1} and \eqref{eq:est_sum_2}, the integrals along these two real segments are bounded by the length of the segments (which can be assumed to be less than $1$) times $\varepsilon$. By repeating this change of contour for each variable $\lambda(+)_i$, up to changing $C$ we arrive at
        \begin{equation}
        \label{eq:I_tau_T_defi}
            \Val{I_{\tau,d+1}-\int_{\substack{i \fa_{\tau}^*+\kappa_\tau+\eta \underline{z}_{\tau,n}+\underline{\varepsilon}_\tau \\ \max \Val{\Im(\lambda_i)} \leq T}} \sum_{\varphi \in \cB_{Q,\tau}(J)} \cP_{\tau}^\uparrow(\varphi,\lambda) \langle f,E(\varphi,-\overline{\lambda}) \rangle_{G} d\lambda} \leq C \varepsilon.
        \end{equation}
        We denote the integral in \eqref{eq:I_tau_T_defi} by $I_{\tau,d+1}(T)$.

        We now change the contour in the variable $\lambda(1)_{\mathrm{c},1}$. We set
        \begin{equation*}
             \underline{z}_1 = \left( \underbrace{0,\hdots\hdots\hdots\hdots\hdots\hdots,0}_{m_{+}(\tau)+m_2(\tau)-|I_2|+m_1(\tau)}, 1/4,\underbrace{0,\hdots,0}_{m_{\mathrm{c},1}}  \right) \in \fa_{Q,n}^*.
        \end{equation*}
        In words, $\underline{z}_1$ is the vector in $\fa_{Q,n}^*$ whose only non-zero coordinates is $1/4$ above $\tau_{\mathrm{c},1,1}$. We want to move from $i \fa_\tau^* + \kappa_\tau+\eta \underline{z}_{\tau,n}+\underline{\varepsilon}_\tau$ to $i \fa_\tau^* +\kappa_\tau+\eta \underline{z}_{\tau,n}-2 \eta \underline{z}_1+\underline{\varepsilon}_\tau$. By Fubini's theorem, we can integrate in $\lambda(1)_{\mathrm{c},1}$ last, so that we may consider that all the other variables are fixed. This means that we want to go from an integral on $\Val{\Im(\lambda(1)_{\mathrm{c},1})} \leq T$ starting from $\Re(\lambda(1)_{\mathrm{c},1})=(d+1)/2+\eta/4$ to $\Re(\lambda(1)_{\mathrm{c},1})=(d+1)/2-\eta/4$. By the description of \eqref{eq:thepoles_n} there is no pole in these two regions, or if $2T \geq \Val{\Im(\lambda(1)_{\mathrm{c,1}})}>T$. We may therefore consider a contour $\gamma$ linking our two segments from above and below, so that non singularities occur along $\gamma$. Note that $\Val{L_{\Gamma'}(\lambda)}$ is bounded from below by $(c/4)^{D} (2T+1)^{-k'D}$ along $\gamma$: along the two imaginary segments it follows from the same argument as in the previous change of contour, while in the two additional curves this condition can easily be met. Therefore, by \eqref{eq:est_sum_1} and \eqref{eq:est_sum_2} the two additional integrals in the region $\Val{\Im(\lambda(1)_{\mathrm{c,1}})}>T$ will be bounded by an absolute constant of $\varepsilon$. The singularities inside $\gamma$ are those prescribed by \eqref{eq:thepoles_n}, and they are simple as the coordinates of $\underline{\varepsilon}_\tau(+)$ are distinct. By the residue theorem and our choices of measures, the integral along $\gamma$ is equal to the sum of the residues. Note that in Lemma~\ref{lem:compute_residues} we had computed this residues with respect to linear forms $\Lambda(\lambda)=\lambda(+)_i-\lambda(1)_{\mathrm{c},j}-(d+1)/2$. However, here we are shifting the contour in the variable $\lambda(1)_{\mathrm{c},1}$ so that a $-$ sign will occur. By the computation from Lemma~\ref{lem:compute_residues}, we conclude that, up to increasing $C$, we have
        \begin{equation}
        \label{eq:big_CS}
         \Val{I_{\tau,d+1}(T)-\left(\int_{\substack{i \fa_\tau^* +\kappa_\tau+\eta \underline{z}_{\tau,n}-2 \eta \underline{z}_1+\underline{\varepsilon}_\tau \\ \max \Val{\Im(\lambda_i)} \leq T}} \sum_{\varphi \in \cB_{Q,\tau}(J)} \cP_{\tau}^\uparrow(\varphi,\lambda) \langle f,E(\varphi,-\overline{\lambda}) \rangle_{G} d\lambda+ \sum_{\Gamma'} I_{\Gamma'}(T)\right)} \leq C\varepsilon.
        \end{equation}
        Here $\Gamma'$ ranges in the elements of $\cG(\Gamma,n,d)$ such that there is exactly one edge that does not belong to $\Gamma$ and such that it is of the form $\{ \pi_{+,i},\tau_{\mathrm{c},1,1}\}$ with $\underline{d}(+,i)=d$. In that case, we set for $\tau'=\pi_\Gamma'$ and $Q'=Q_{\Gamma'}$
        \begin{equation*}
            I_{\Gamma'}(T)=\int_{\substack{i \fa_{\tau'}^*+\kappa_{\tau'}+\eta \underline{z}_{\tau',n}+\underline{\varepsilon}_{\tau}  \\ \max\Val{\Im(\lambda_i)}\leq T }} \sum_{\varphi \in \cB_{Q',\tau'}(J)} \cP_{\tau'}^\uparrow(\varphi,\lambda) \langle f,E(\varphi,-\overline{\lambda}) \rangle_{G} d\lambda.
        \end{equation*}
        Note that here $\underline{\varepsilon}_{\tau}$ is identified with the element of $\fa_{\tau'}^*$ which has $\underline{\varepsilon}_{\tau}(+)_1$ as its coordinates in $\lambda(1)_1$, and $\underline{\varepsilon}_{\tau}(+)_i$ in $\lambda(+)_{i-1}$ for $i \geq 2$. However, we claim that we can move the contour of this integral to $i \fa_{\tau'}^*+\kappa_{\tau'}+\eta \underline{z}_{\tau',n}+\underline{\varepsilon}_{\tau'}$. Indeed, this change will take place in the variable $\lambda(1)_1$ inside the region \eqref{eq:n_region} and will not cross any poles thanks to \eqref{eq:thepoles_n}, so that in particular $L_{\Gamma'}$ remains constant. By \eqref{eq:est_sum_1} and \eqref{eq:est_sum_2}, we conclude that up to increasing the absolute constant $C$ we may assume that
        \begin{equation}
                \label{eq:small_CS}
            \Val{I_{\Gamma'}(T)-I_{\tau',d+1}(T)} \leq C \varepsilon.
        \end{equation}
        Note that throughout this process the constant $C$ is independent of $\varepsilon$.

        We can now conclude the proof of Lemma~\ref{lem:residue_n}. By \eqref{eq:big_CS} and \eqref{eq:small_CS}, we see that $I_{\tau,d+1}(T)$ is, up to $C\varepsilon$, equal to the same integral shifted to $i \fa_\tau^* +\kappa_\tau+\eta \underline{z}_{\tau,n}-2 \eta \underline{z}_1+\underline{\varepsilon}_\tau$ (the absolute values of the imaginary parts still being bounded by $T$), plus a sum of $I_{\tau',d+1}(T)$. Let us explain how to deal with each contribution. 
        
        For the first integral, we continue to change the contour in each variable $\lambda(1)_{\mathrm{c},i}$ by subtracting $2 \eta \underline{z}_i$. By reproducing the exact same argument, we see that each steps adds a finite sum of $I_{\tau',d+1}(T)$ where $\tau'$ is obtained by adding an edge $\{ \pi_{+,i},\pi_{\mathrm{c},1,j}\}$ with $\underline{d}(+,i)=d$ to $\Gamma$. At the end, our integral takes place along $i \fa_\tau^* +\kappa_\tau-\eta \underline{z}_{\tau,n}+\underline{\varepsilon}_\tau$ (with the bound on the imaginary parts). But as in \eqref{eq:first_CS} and \eqref{eq:I_tau_T_defi}, by changes of contours (now generating no residues) we see that this is the integral along $i \fa_{\tau}^*+\kappa_{d,d+1}(\Gamma)$ up to $C \varepsilon$, this time with no bound on the imaginary parts. This is the contribution coming from $\Gamma$ in \eqref{eq:residue_n}.

        We are now left with the remaining $I_{\tau',d+1}(T)$. Because the $\tau'$s are associated to graphs $\Gamma' \in \cG(\Gamma,n,d)$, their poles are also determined by \eqref{eq:thepoles_n} in the region \eqref{eq:n_region} by Lemma~\ref{lem:poles_of_the_product} and we can reproduce the procedure used for $I_{\tau,d+1}(T)$. We conclude by induction because as $m_{\mathrm{c},1}(\tau')<m_{\mathrm{c},1}(\tau)$, so that we have less variables to apply change of contour to. Note that at each step we will add an edge of the form $\{\pi_{+,i}, \pi_{\mathrm{c},1,j}\}$ to $\Gamma'$, so that at the end we get a sum indexed by the graphs $\cG(\Gamma,n,d)$. Moreover, because we change the contour of the variables in a certain order, each graph can only occur once. 
        
        We now see that \eqref{eq:residue_n} holds up to an absolute constant of $\varepsilon$, and we conclude by taking $\varepsilon \to 0$. This ends the proof of Lemma~\ref{lem:residue_n}.
    \end{proof}

    \subsubsection{Changes of contours on the $\GL_{n+1}$-side}
    \label{subsubsec:contour_gl_n+1}

    We keep our element $(I,Q,\tau,I_1,I_2) \in \Pi_H^\uparrow(\pi,d+1)$ coming from the graph $\Gamma$ and the notation from \S\ref{subsubsec:contour_gln}. We now define $\cG(\Gamma,d)$ to be the subset of $\Gamma' \in \cG(\pi)$ such that
    \begin{itemize}
        \item $\Gamma$ is a subgraph of $\Gamma'$,
        \item all the edges in $\Gamma'$ that do not belong to $\Gamma$ are of the form $\{ \pi_{+,i},\pi_{\mathrm{c},1,j} \}$ or $\{ \pi_{+,i},\pi_{\mathrm{c},2,j} \}$ with $\underline{d}(+,i)=d$.
    \end{itemize}

    By changes of contours on the variables on the $\GL_{n+1}$-side, we obtain the following result.

    \begin{lem}
    \label{lem:residue_n+1}
        We have
        \begin{equation}
        \label{eq:residue_n+1}
           I_{\tau,d+1}=\sum_{\Gamma' \in \cG(\Gamma,d)} \int_{i \fa_{\pi_{\Gamma'}}^*-\underline{\rho}_{\Gamma'}-\underline{\rho}_{\Gamma'}^\uparrow+2d\underline{z}_{\Gamma'}} \sum_{\varphi \in \cB_{Q_{\Gamma'},\pi_{\Gamma'}}(J)} \cP_{\pi_{\Gamma'}}^\uparrow(\varphi,\lambda) \langle f,E(\varphi,-\overline{\lambda}) \rangle_{G} d\lambda.
        \end{equation}
    \end{lem}

    \begin{proof}
        The proof follows the same pattern as Lemma~\ref{lem:residue_n}, up to one major difference. By Lemma~\ref{lem:residue_n}, we can start from $(I',Q',\tau',I_1',I_2')$ that comes from a graph $\Gamma' \in \cG(\Gamma,n,d)$ and change the contour in the integral in \eqref{eq:residue_n}. We now have to study the poles in the region 
         \begin{equation}
        \label{eq:n+1_region}
            \left\{ \lambda -\underline{\rho}_{\tau'}-\underline{\rho}_{\tau'}^\uparrow  + \nu \; \middle| \; \lambda \in \fa_{\tau',\cc}^* \cap \cS_{\tau',k,c_J}, 
            \begin{array}{l}
                \nu_n=2d \underline{z}_{\tau',n}, \\
                 \nu_{n+1} \in [2d \underline{z}_{\tau',n+1},2(d+1) \underline{z}_{\tau',n+1}].
            \end{array}
        \right\},
        \end{equation}
        The difference with Lemma~\ref{lem:residue_n} is that we can have some $i$ such that $\underline{d}(\tau',1,i)=d+1$. By Lemma~\ref{lem:poles_of_the_product}, we see that the poles in \eqref{eq:n+1_region} are now along the zeros of 
        \begin{equation*}
            \prod_{\substack{\sigma_{+,i}(\tau')^\vee \simeq \tau_{\mathrm{c},2,j}   \\ \underline{d}(\tau',+,i)=d}} \left( -\lambda(+)_i-\lambda(2)_{\mathrm{c},j}-\frac{d+1}{2}\right) \prod_{\substack{\sigma_{1,i}(\tau')^\vee \simeq \tau'_{\mathrm{c},2,j}  \\ \underline{d}(\tau',1,i)=d+1 \\  i \notin I_1'}} \left(-\lambda(1)_i -\lambda(2)_{\mathrm{c},j}-\frac{d+1}{2}\right).
        \end{equation*}
        We know the residues arising from these singularities thanks Lemma~\ref{lem:compute_residues}. By reproducing the argument in Lemma~\ref{lem:residue_n} (i.e. truncate the integral to assume that the poles are simple), we see that 
        \begin{equation*}
             I_{\tau',d+1}=\sum_{\Gamma'' \in \cG(\Gamma',n+1,d)} \int_{i \fa_{\pi_{\Gamma''}}^*-\underline{\rho}_{\Gamma''}-\underline{\rho}_{\Gamma''}^\uparrow+2d\underline{z}_{\Gamma''}} \sum_{\varphi \in \cB_{Q_{\Gamma''},\pi_{\Gamma''}}(J)} \cP_{\pi_{\Gamma''}}^\uparrow(\varphi,\lambda) \langle f,E(\varphi,-\overline{\lambda}) \rangle_{G} d\lambda,
        \end{equation*}
        where $\cG(\Gamma',n+1,d)$ is the set of graphs $\Gamma'' \in \cG(\pi)$ such that $\Gamma'$ is a subgraph of $\Gamma''$ and all the edges in $\Gamma''$ that do not belong to $\Gamma'$ are of the form $\{ \pi_{+,i},\pi_{\mathrm{c},2,j} \}$ with $\underline{d}(+,i)=d$. To conclude, it remains to note that 
        \begin{equation*}
            \cG(\Gamma,d)=\bigsqcup_{\Gamma' \in \cG(\Gamma,n,d)} \cG(\Gamma',n+1,d).
        \end{equation*}
    \end{proof}

    \subsubsection{Counting the contributions}
    \label{subsubsec:counting_contributions}

    We now end the proof of Proposition~\ref{prop:recu}. We assume that \eqref{eq:recu_to_prove} holds for $d+1 \geq 2$. Let $\cG(\pi,d)$ be the subset of $\Gamma \in \cG(\pi)$ such that $(I_\Gamma,Q_\Gamma,\pi_\Gamma,I_{1,\Gamma},I_{2,\Gamma}) \in \Pi_H^\uparrow(\pi,d)$. We need to sum over elements in $\Pi_H^\uparrow(\pi,d)$ rather than graphs in $\cG(\pi,d)$. This is the content of the following combinatorial result.
    
    \begin{lem}
    \label{lem:counting1}
       We have
        \begin{equation*}
            \bigsqcup_{\Gamma \in \cG(\pi,d+1)}\cG(\Gamma,d) = \cG(\pi,d).
        \end{equation*}
        Moreover, let $\Gamma \in \cG(\pi,d+1)$ and let $(I,Q,\tau,I_1,I_2)$ be its image. Then the fiber of the map
        \begin{equation}
        \label{eq:ugly_map}
            \Gamma' \in \cG(\Gamma,d) \mapsto (I_{\Gamma'},Q_{\Gamma'},\pi_{\Gamma'},I_{1,\Gamma'},I_{2,\Gamma'}) \in \Pi_H^\uparrow(\pi,d)
        \end{equation}
        above a point $(I',Q',\tau',I_1',I_2')$ in the image is of cardinal $|\Stab(\tau)||\Stab(\tau')|^{-1}$.
    \end{lem}

    \begin{proof}
        For the first point, that the union is disjoint follows from the fact that for each $\Gamma' \in \cG(\pi,d)$ there exists a unique graph $\Gamma \in \cG(\pi,d+1)$ such that $\Gamma'$ is obtained by adding edges to $\Gamma$. We also clearly have the equality as any $\Gamma \in \cG(\pi,d)$ can be obtained by this procedure.

        We move to the second assertion. Because there is a single element in the preimage of $W(\tau^\downarrow). \tau^\downarrow$ under the map \eqref{eq:down_map}, we may assume that we have decompositions
        \begin{equation*}
            \tau_{+}=\boxtimes_{i=1}^{k} \tau_{+,i}^{\boxtimes k_{+,i}} \boxtimes \tau_{+, \neq d}, \quad \tau_{\mathrm{c},1}=\boxtimes_{i=1}^{k} \sigma_{i}^{\boxtimes k_{\mathrm{c},1,i}}, \quad \tau_{\mathrm{c},2}=\boxtimes_{i=1}^{k} (\sigma_{i}^\vee)^{\boxtimes k_{\mathrm{c},2,i}},
        \end{equation*}
        for $k$ some integer, where all the $\sigma_i$ are mutually non-isomorphic cuspidal representations of some $\GL_r$'s, for all $i$ we have $\tau_{+,i}=\Speh(\sigma_i,d)$ and $\tau_{+,\neq d}$ is a product of Speh representations $\Speh(\sigma,d')$ with $\sigma$ cuspidal and $d' \neq d$. Note that we allow the $k_{+,i}$, $k_{\mathrm{c},1,i}$ and $k_{\mathrm{c},2,i}$ to be zero.

        Let $\Gamma' \in \cG(\Gamma,d)$. Let $1 \leq i \leq k$. Denote by $k_{1,i}$ the number of vertices $\pi_{+,j} \in \Gamma'$ with $\pi_{+,j} \simeq \tau_{+,i}$, $\deg(+,1,j)=1$ and $\deg(+,2,j)=0$, $k_{2,i}$ the number of these vertices with $\deg(+,1,j)=0$ and $\deg(+,2,j)=1$, and finally $k_{1,2,i}$ the number of those with $\deg(+,1,j)=\deg(+,2,j)=1$. It is readily checked that the number of $\Gamma'' \in \cG(\Gamma,d)$ with the same image as $\Gamma'$ under \eqref{eq:ugly_map} is
        \begin{equation*}
            \prod_{i=1}^k \left(\frac{k_{\mathrm{c},1,i}! k_{\mathrm{c},2,i}! k_{+,i}!}{(k_{\mathrm{c},1,i}-k_{1,i}-k_{1,2,i})!(k_{\mathrm{c},2,i}-k_{2,i}-k_{1,2,i})!k_{1,i}! k_{2,i}!k_{1,2,i}!(k_{+,i}-k_{1,i}-k_{2,i}-k_{1,2,i})!} \right).
        \end{equation*}
        If we denote by $(I',Q',\tau',I_1',I_2')$ the image of $\Gamma'$, this is exactly $|\Stab(\tau)||\Stab(\tau')|^{-1}$. This concludes the proof.
    \end{proof}

    \subsubsection{End of the proof of Propositions~\ref{prop:recu} and~\ref{prop:I_pi_exp}}
    \label{subsubsec:end_proof_prop}

    Proposition~\ref{prop:recu} is now a consequence of Lemma~\ref{lem:residue_n+1} and Lemma~\ref{lem:counting1}. This also concludes the proof of Proposition~\ref{prop:I_pi_exp} by taking $d=1$.

    \subsubsection{Conclusion}
 
    We can now write our new expression of the Rankin--Selberg period.

    \begin{prop}
    \label{prop:almost_final}
        We have 
         \begin{align}
             \int_{[H]}f(h)dh=&\sum_{r=0}^n \sum_{(I_r,P,\pi) \in \Pi_H^\uparrow} \sum_{(I,Q,\tau,I_1,I_2) \in \Pi_H^\uparrow(\pi)}\nonumber \\
        \times&\frac{|\Stab(\pi)|}{|W(\pi)||\Stab(\tau)|} \int_{i \fa_{\tau}^*-\underline{\rho}_\tau-\underline{\rho}_\tau^\uparrow+2 \underline{z}_\tau} \sum_{\varphi \in \cB_{Q,\tau}(J)}\cP_{\tau}^\uparrow(\varphi,\lambda) \langle f,E(\varphi,-\overline{\lambda}) \rangle_{G} d\lambda, \label{eq:almost_final}
        \end{align}
        where $\Pi_H^\uparrow(\pi)$ is the set of tuples obtained from the null graph built from $w \pi$ with $w$ being any element in $W(\pi)$ such that $\underline{d}(+,1) \geq \hdots \geq \underline{d}(+,m_+)$.
    \end{prop}

    \begin{proof}
        This is a direct consequence of Proposition~\ref{prop:enter_Eisenstein} and Proposition~\ref{prop:I_pi_exp}. The only thing that we have to check is that the expression in the last line of \eqref{eq:almost_final} is independent from the choice of $w \in W(\pi)$. But this follows from Lemma~\ref{lem:functional_I_pi}.
    \end{proof}

\subsection{Additional residues from the regularized period}
\label{subsec:this_is_the_end}

In this section, we continue the computation of the spectral expansion of $\int_{[H]} f(h)dh$ by performing a final shift in the contour of the integrals of \eqref{eq:almost_final}. In contrast with that of \S\ref{subsec:additional}, this wave of shifts will produce additional contributions coming from residues of the regularized period $\cP_\tau^\uparrow$. 

\subsubsection{Statement of the result}
\label{subsubsec:result_2}

Throughout this section, we fix $0 \leq r \leq n$ and $(I_r,P,\pi) \in \Pi_H^{\uparrow}$. We then fix $(I,Q,\tau,I_1,I_2) \in \Pi_H^\uparrow(\pi)$. Our goal is to compute the integral 
\begin{equation}
\label{eq:last_I_tau}
    I_\tau:=\int_{i \fa_{\tau}^*-\underline{\rho}_\tau-\underline{\rho}_\tau^\uparrow+2 \underline{z}_\tau} \sum_{\varphi \in \cB_{Q,\tau}(J)}\cP_{\tau}^\uparrow(\varphi,\lambda) \langle f,E(\varphi,-\overline{\lambda}) \rangle_{G} d\lambda.
\end{equation}
As in \S\ref{subsubsec:statement_result}, the integral $I_\tau$ will give rise to several contributions that we will parametrize by graphs. However, the definition is less involved than in \S\ref{subsubsec:statement_result}.

As usual, we can write
\begin{align}              
      \tau_n&=\boxtimes_{i=1}^{m_+} \tau_{+,i} \boxtimes_{\substack{i=1 \\ i \notin I_2}}^{m_2} \tau_{2,i}^{-,\vee}\boxtimes_{i=1}^{m_1} \tau_{1,i}  \boxtimes_{i=1}^{m_{\mathrm{c},1}} \tau_{\mathrm{c},1,i}, \label{eq:defi_tau_n} \\
         \tau_{n+1}&=  \boxtimes_{i=1}^{m_+} \tau_{+,i}^\vee \boxtimes_{\substack{i=1 \\ i \notin I_1}}^{m_1} \tau_{1,i}^{-,\vee} \boxtimes_{i=1}^{m_2} \tau_{2,i} \boxtimes_{i=1}^{m_{\mathrm{c},2}} \tau_{\mathrm{c},2,i} . \label{eq:defi_tau_n+1}
    \end{align}
    Because the reference to $\pi$ will not be needed in our discussion, we simply write $\tau_{1,i}=\Speh(\sigma_{1,i},\underline{d}(1,i))$ and so on. We denote by $\cG_{\mathrm{c}}(\tau)$ the set of undirected simple graphs $\Gamma$ such that
    \begin{itemize}
        \item the vertices of $\Gamma$ are $\tau_{\mathrm{c},1,1},\hdots,\tau_{\mathrm{c},1,m_1}$, $\tau_{\mathrm{c},2,1},\hdots,\tau_{\mathrm{c},2,m_2}$ (with multiplicity);
        \item the edges of $\Gamma$ are of the form $\{ \tau_{\mathrm{c},1,i},\tau_{\mathrm{c},2,j}\}$ for some $i$ and $j$ with $\tau_{\mathrm{c},1,i} \simeq \tau_{\mathrm{c},2,j}^\vee$;
        \item the degree of each vertex is $0$ or $1$.
    \end{itemize}
    For each $i$, we denote by $\deg(\mathrm{c},1,i)$ (resp. $\deg(\mathrm{c},2,i)$) the degree of $\tau_{\mathrm{c},1,i}$ (resp. $\tau_{\mathrm{c},2,i}$). There is a bijection
    \begin{equation}
        \label{eq:bij_I}
        i \in \{1 \leq i \leq m_{\mathrm{c},1} \; | \; \deg(\mathrm{c},1,i)=1 \} \mapsto j(i) \in \{1 \leq j \leq m_{\mathrm{c},2} \; | \; \deg(\mathrm{c},2,j)=1 \},
    \end{equation}
    such that $\tau_{\mathrm{c},2,j(i)}$ is the only neighbor of $\tau_{\mathrm{c},1,i}$. We can then define a representation $\tau_{\Gamma}$ of a standard Levi $M_{R_{\Gamma}}$ of $G$ by
     \begin{align}              
      \tau_{\Gamma,n}&=\boxtimes_{i=1}^{m_+} \tau_{+,i}\boxtimes_{\substack{i=1 \\ i \notin I_2}}^{m_2} \tau_{2,i}^{-,\vee}\boxtimes_{i=1}^{m_1} \tau_{1,i}\underset{\deg(\mathrm{c},1,i)=0}{\boxtimes}  \tau_{\mathrm{c},1,i}  \underset{\deg(\mathrm{c},1,i)=1}{\boxtimes}  \tau_{\mathrm{c},1,i} , \label{eq:tau_gamma_n} \\
        \tau_{\Gamma,n+1}&= \boxtimes_{i=1}^{m_+} \tau_{+,i}^\vee\boxtimes_{\substack{i=1 \\ i \notin I_1}}^{m_1} \tau_{1,i}^{-,\vee} \boxtimes_{i=1}^{m_2} \tau_{2,i} \underset{\deg(\mathrm{c},2,i)=0}{\boxtimes}  \tau_{\mathrm{c},2,i}  \underset{\deg(\mathrm{c},1,i)=1}{\boxtimes}  \tau_{\mathrm{c},2,j(i)} . \label{eq:tau_gamma_n+1}
    \end{align} 
    Associated to the representation $\tau_\Gamma$ is a natural tuple $J_\Gamma=(n_+,n_1,n_{\mathrm{c},1},n_2,n_{\mathrm{c},2},n_-)$. More precisely, we have $n_-=\sum_{\deg(\mathrm{c},1,i)} n_{\mathrm{c},1,i}$ where as usual $\tau_{\mathrm{c},1,i}$ is a cuspidal representation of $\GL_{n_{\mathrm{c},1,i}}$. We then see that \eqref{eq:bij_I} defines two sets $J_{1,\Gamma}$ and $J_{2,\Gamma}$ such that $(J_\Gamma,R_\Gamma,\tau_\Gamma,J_{1,\Gamma},J_{2,\Gamma}) \in \Pi^\uparrow_H$. We denote by $\Pi_{H,\mathrm{c}}^\uparrow(\tau)$ the set of tuples obtained this way. Note that $(I,Q,\tau,I_1,I_2) \in \Pi_{H,\mathrm{c}}^\uparrow(\tau)$ because it is the image of the null graph. To keep track of this starting tuple, we will write the elements of $\Pi_{H,\mathrm{c}}^\uparrow(\tau)$ with the letter $J$ rather than $I$. We hope that this does not cause confusion with our fixed level (which is also denoted by $J$).
    
    Because $\tau$ is fixed, we will typically write $\delta$ for the representation $\tau_\Gamma$. In this situation, we will have $\delta_{1,i}=\Speh(\sigma_{1,i}(\delta),\underline{d}(\delta,1,i))$ and so on. We will need the element
    \begin{equation}
        \label{eq:underline_z_delta}
        \underline{z}_{\delta}:=((0,0,0,1/4,0,0),(0,0,0,1/4,0,0)) \in \fa_{R}^*,
    \end{equation}
    where the $1/4$ only lives above $\underset{\deg(\mathrm{c},1,i)=0}{\boxtimes}  \tau_{\mathrm{c},1,i} $ and $\underset{\deg(\mathrm{c},2,i)=0}{\boxtimes}  \tau_{\mathrm{c},2,i} $.

    Note that the fiber of the map
    \begin{equation*}
        (J',R',\delta',J_1',J_2') \in \Pi_{H,\mathrm{c}}^\uparrow(\tau) \mapsto (R^{',\downarrow},\delta^{',\downarrow}) \in \Pi_H
    \end{equation*}
    above the set $W(\delta^\downarrow). \delta^\downarrow$ is always reduced to $(J,R,\delta,J_1,J_2)$. As in \S\ref{subsubsec:statement_result}, we write $\Stab(\delta)$ for the stabilizer of $\delta^\downarrow$ in $W(\delta^\downarrow)$.

    The goal of \S\ref{subsec:this_is_the_end} is to prove the following proposition.

    \begin{prop}
    \label{prop:this_is_the_end}
        We have
        \begin{equation*}
            I_{\tau}=\sum_{(J,R,\delta,J_1,J_2) \in \Pi_{H,\mathrm{c}}^\uparrow(\tau)}\frac{|\Stab(\tau)|}{|\Stab(\delta)|}\int_{i \fa_{\delta^\downarrow}^*-\underline{\rho}_{\delta^\downarrow}} \sum_{\varphi \in \cB_{R^\downarrow,\delta^\downarrow}(J)} \cP_{\delta^\downarrow}(\varphi,\lambda) \langle f,E(\varphi,-\overline{\lambda}) \rangle_{G} d\lambda.
        \end{equation*}
    \end{prop}

    The proof of Proposition~\ref{prop:this_is_the_end} will run through \S\ref{subsec:this_is_the_end} and end in \S\ref{subsubsec:end_proof_2}. Once again, it will be proved by changing the contour in several steps. Because most arguments are similar to those used in the course of \S\ref{subsec:additional} to prove Proposition~\ref{prop:I_pi_exp}, we shall give less details when reasonable.

    \subsubsection{A short description of the argument}
    As in \S\ref{subsubsec:short_description1}, we give the main ideas behind the proof of Proposition~\ref{prop:this_is_the_end}. Our goal is to bring the domain of integration in $I_\tau$ from $i \fa_{\tau}^*-\underline{\rho}_\tau+\underline{\rho}_\tau^\uparrow + 2 \underline{z}_\tau$ to $i\fa_{\tau^\downarrow}^*-\underline{\rho}_{\tau^\downarrow}$. We do this in three steps.

    This first one is to bring our integral to $i \fa_{\tau}^*-\underline{\rho}_\tau+\underline{\rho}_\tau^\uparrow +  \underline{z}_\tau$. This is a delicate maneuver as our integrand will have singularities in this region. They are computed in \S\ref{subsubsec:poles_final} and all come from the regularized period $\cP^\uparrow$, and more precisely from isomorphisms $\tau_{\mathrm{c},1,i} \simeq \tau_{\mathrm{c},2,j}^\vee$. We determine the associated residues in \S\ref{subsubsec:residues_final}. They can be expressed as relative characters living the induction of $w \tau$ for $w$ an element in the Weyl group acting by blocks on $\tau$. The idea here is rather simple : the two representations $\tau_{\mathrm{c},1,i}$ and $\tau_{\mathrm{c},2,j}^\vee$ will be moved from the "$(\mathrm{c},1)$" and "$(\mathrm{c},2)$" parts of $\tau$ to the "$-$" one. Once again, to keep track of these residues we use graphs and add an edge between the vertices corresponding to our representations if we cross the relevant singularity. Our constraints mean that this can happen at most once for each $\tau_{\mathrm{c},1,i}$ and $\tau_{\mathrm{c},2,j}$, and the representation $\tau_\Gamma$ described in \eqref{eq:tau_gamma_n} and \eqref{eq:tau_gamma_n+1} is the $w\tau$ alluded to above. 

    As noted earlier, our integrand in fact has poles along $i \fa_{\tau}^*-\underline{\rho}_\tau+\underline{\rho}_\tau^\uparrow +  t\underline{z}_\tau$ exactly when $t=1$. Our goal is to integrate in this region with $t<1$, thus collecting the residues. The issue is that this makes $-\overline{\lambda}$ leave the positive Weyl chamber in which we control $E(\varphi,-\overline{\lambda})$. To solve this problem, we use the fact that the regions $\cR_{\tau,k,c}$ described in \S\ref{subsec:R_region} go a little bit beyond this chamber. This allows us to integrate over $i \fa_{\tau}^*-\underline{\rho}_\tau+\underline{\rho}_\tau^\uparrow +  (1-\eta)\underline{z}_\tau$ and $\max \Val{\Im(\lambda_i)} \leq T$ with $\eta$ small and $T$ large, while controlling the size of the tail. This is done in \S\ref{subsubsec:first_change}. We end up with a sum of relative characters indexed by graphs $\Gamma \in \cG_\mathrm{c}(\tau)$, with corresponding $\delta=\tau_\Gamma$, integrated along $i \fa_{\delta}^*-\underline{\rho}_\delta+\underline{\rho}_\delta^\uparrow +  (1-\eta)\underline{z}_\delta$ and $\max \Val{\Im(\lambda_i)} \leq T$.

    At that point, we use a change of variabless to replace $\delta$ with some $w \delta$ in order to bring our integral back in the positive Weyl chamber. Once again, this is possible as the associated $\cR$-regions overlap. It turns out that $w \delta$ is simply the representation $\delta_\emptyset$ described in \S\ref{subsubsec:empty_transformation}. After taking into account our twists, and after an easy manipulation, our integral now takes place in the region $i \fa_{\delta_\emptyset}^*-(1-\eta)(\underline{\rho}_{\delta_\emptyset}+\underline{\rho}_{\delta_\emptyset}^\uparrow-\underline{z}_{\delta_\emptyset})$ with the additional requirement $\max \Val{\Im(\lambda_i)}  \leq T$. We may now shift the contour to $i \fa_{\delta_\emptyset}^*-\underline{\rho}_{\delta_\emptyset}$ and lift the requirement on the imaginary parts as our integrand is regular in the corresponding strip. This is the content of \S\ref{subsubsec:second_change}. It then remains to express our result in terms of $\delta^\downarrow$ rather than $\delta_\emptyset$ (\S\ref{subsubsec:non-increasing}) and to sum over $\delta^\downarrow$ rather than graphs in $\cG_\mathrm{c}(\tau)$ (\S\ref{subsubsec:counting_2}) to conclude the proof of Proposition~\ref{prop:this_is_the_end}.

    \subsubsection{Poles}
    \label{subsubsec:poles_final}

    Let $(J,R,\delta,J_1,J_2) \in \Pi_{H,\mathrm{c}}^\uparrow(\tau)$. Our region of interest will be
      \begin{equation}
    \label{eq:our_region2}
       \cR_{\delta}^\uparrow := \left\{ \lambda -\underline{\rho}_{\delta}-\underline{\rho}_{\delta}^\uparrow+\underline{z}_\delta\; \middle| \; \lambda \in \fa_{\delta,\cc}^* \cap \cS_{\delta,k,c_J}
        \right\},
    \end{equation}
     where we take $k$ and $c_J$ given by Theorem~\ref{thm:analytic_Eisenstein} and Proposition~\ref{prop:reg_P_pi_up}. In the special case where $\delta=\tau$, we also set 
     \begin{equation}
            \label{eq:our_region_tau}
               \cR_{\tau}^{\uparrow,+} := \left\{ \lambda -\underline{\rho}_{\tau}-\underline{\rho}_{\tau}^\uparrow  + \nu \; \middle| \; \lambda \in \fa_{\tau,\cc}^* \cap \cS_{\tau,k,c_J}, \; \nu \in [ \underline{z}_\tau,2\underline{z}_\tau]
                \right\},
     \end{equation}

    \begin{lem}
    \label{lem:poles_of_the_product2}
        For $\varphi \in \cB_{R,\delta}(J)$, the possible poles of $\cP_{\delta}^\uparrow(\varphi,\lambda) \langle f,E(\varphi,-\overline{\lambda}) \rangle_{G}$ in the region $\cR_{\delta}^\uparrow$ (or $\cR_{\tau}^{\uparrow,+}$) are along the zeros of the polynomial
            \begin{equation}
            \prod_{\delta_{\mathrm{c},1,i} \simeq \delta_{\mathrm{c},2,j}^\vee} \left( \lambda(1)_{\mathrm{c},i}+\lambda(2)_{\mathrm{c},j}-\frac{1}{2} \right)
            \label{eq:thepoles2}
    \end{equation}
    Moreover, it is of rapid decay in the sense that if we denote by $L$ the polynomial in \eqref{eq:thepoles2}, for all $N>0$ there exists $C>0$ such that for all $\lambda$ in the region $\cR_{\delta}^\uparrow$ (or $\cR_{\tau}^{\uparrow,+}$) we have
    \begin{equation}
        \label{eq:rapid_decay_product2} \sum_{\varphi \in \cB_{R,\delta}(J)}\Val{L(\lambda)\cP_{\delta}^\uparrow(\varphi,\lambda) \langle f,E(\varphi,-\overline{\lambda}) \rangle_{G}} \leq \frac{C}{(1+\norm{\lambda}^2)^N}.
    \end{equation}
    In particular, the map $\lambda \mapsto \sum_{\varphi} L(\lambda)\cP_{\delta}^\uparrow(\varphi,\lambda) \langle f,E(\varphi,-\overline{\lambda}) \rangle_{G}$ is holomorphic in these regions.
    \end{lem}

    \begin{proof}
        This is proved by reproducing the argument of Lemma~\ref{lem:poles_of_the_product} with $d=0$. The key point is that we always have $\underline{\rho}_\delta+\underline{\rho}_\delta^\uparrow-\underline{z}_\delta \in \overline{\fa_{R}^{*,+}}$ which lets us use Corollary~\ref{cor:zeros_of_E} to locate the poles of the Eisenstein series. Note that contrary to the situation of Lemma~\ref{lem:poles_of_the_product}, the poles in \eqref{eq:thepoles2} come from the period $\cP_{\delta}^\uparrow(\varphi,\lambda)$ rather than the Eisenstein series. 
    \end{proof}

    \begin{rem}
        \label{rem:border_dispute}
        We emphasize that the region $\cR_{\delta}^\uparrow$ is an a sense the largest in which we can control the poles of $\cP_{\delta}^\uparrow(\varphi,\lambda) \langle f,E(\varphi,-\overline{\lambda}) \rangle_{G}$. Indeed, note that in coordinates (with respect to the decomposition of $\delta$ given in \eqref{eq:tau_gamma_n} and \eqref{eq:tau_gamma_n+1}) 
        \begin{equation*}
            \underline{\rho}_\delta+\underline{\rho}_\delta^{\uparrow}-\underline{z}_\delta=\left( (1/4,1/4,-1/4,-1/4,-1/4),(1/4,1/4,-1/4,-1/4,-1/4) \right).
        \end{equation*}
        In the special case $\delta=\tau$ we get for $t \in \rr$ with the coordinates of \eqref{eq:defi_tau_n} and \eqref{eq:defi_tau_n+1} instead
         \begin{equation*}
            \underline{\rho}_\tau+\underline{\rho}_\tau^{\uparrow}-t\underline{z}_\tau=\left((1/4,1/4,-1/4,-t/4),(1/4,1/4,-1/4,-t/4) \right),
        \end{equation*}
        so that we need $t \geq 1$ (i.e. $\nu \geq \underline{z}_\tau$) in the definition of $\cR_{\tau}^{\uparrow,+}$ given in \eqref{eq:our_region_tau}.
    \end{rem}

    \subsubsection{Computation of the residues}
    \label{subsubsec:residues_final}

    We keep $(J,R,\delta,J_1,J_2) \in \Pi_{H,\mathrm{c}}^\uparrow(\tau)$ and describe the residues obtained along the affine hyperplanes cut out by \eqref{eq:thepoles2}. We assume that this tuple comes from a graph $\Gamma \in \cG_{\mathrm{c}}(\tau)$.

    Let $1 \leq i_1 \leq m_{\mathrm{c},1}(\delta)$ and $1 \leq i_2 \leq m_{\mathrm{c},2}(\delta)$ such that $\delta_{\mathrm{c},1,i_1} \simeq \delta_{\mathrm{c},2,i_2}$. Let $\Lambda$ be the affine linear form $\Lambda(\lambda)=\lambda(1)_{\mathrm{c},i_1}+\lambda(2)_{\mathrm{c},i_2}-1/2$. If we keep track of the order in \eqref{eq:defi_tau_n} and \eqref{eq:defi_tau_n+1}, there exist uniquely determined indices $j_1$ and $j_2$ such that $\delta_{\mathrm{c},1,i_1}=\tau_{\mathrm{c},1,j_1}$ and $\delta_{\mathrm{c},2,i_2}=\tau_{\mathrm{c},2,j_2}$. We denote by $\Gamma_\Lambda$ the graph obtained by adding the edge $\{\tau_{\mathrm{c},1,j_1},\tau_{\mathrm{c},2,j_2}\}$ to $\Gamma$ and by $(J_\Lambda,R_\Lambda,\tau_\Lambda,J_{1,\Gamma},J_{2,\Gamma})$ the tuple in $\Pi_{H,\mathrm{c}}^\uparrow(\tau)$ associated to $\Gamma_\Lambda$. It gives rise to a regularized period $\cP_{\tau_\Lambda}^\uparrow$.

    As in our previous computations of residues from \S\ref{subsubsec:computation_residues}, we have a unique element $w$ acting by blocks on $M_R$ such that $w \delta=\tau_\Lambda$. Once again, we need to ask that it preserves the order to ensure that it is unique. Moreover, up to choosing appropriately our constants, we have 
    \begin{equation*}
        w\left(\Lambda^{-1}(\{0\}) \cap (\fa_{\delta,\cc}^*-\underline{\rho}_\delta-\underline{\rho}_{\delta}^\uparrow)\right)=\fa_{\tau_\Lambda,\cc}^*-\underline{\rho}_{\delta_\Lambda}-\underline{\rho}_{\delta_\Lambda}^\uparrow,
    \end{equation*}
    and 
    \begin{equation*}
        w \left(\Lambda^{-1}(\{0\}) \cap \cR_{\delta}^\uparrow \right) \subset \cR_{\tau_\Lambda}^\uparrow.
    \end{equation*}

    \begin{lem}
    \label{lem:computation_residues_2}
        For $\lambda \in \Lambda^{-1}(\{0\}) \cap (\fa_{\delta,\cc}^*-\underline{\rho}_\delta-\underline{\rho}_\delta^\uparrow)$ in general position we have
        \begin{equation*}
            \underset{\Lambda}{\Res}\left(\sum_{\varphi \in \cB_{R,\delta}(J)} \cP_{\delta}^\uparrow(\varphi,\mu) \langle f,E(\varphi,-\overline{\mu}) \rangle_G \right)(\lambda)=\sum_{\varphi \in \cB_{R_\Lambda,\tau_\Lambda}(J)} \cP_{\tau_\Lambda}^\uparrow(\varphi,w\lambda) \langle f,E(\varphi,-w\overline{\lambda}) \rangle_G.
        \end{equation*}
    \end{lem}

    \begin{proof}
        For simplicity, we set $R'=R_\Lambda$ and $\delta'=\tau_\Lambda$. Let $\varphi \in \cB_{R,\delta}(J)$. By the proof of Lemma~\ref{lem:poles_of_the_product2}, we know that the singularity comes from the period $\cP_{\delta}^\uparrow(\varphi,\mu)$. By the definition of $\cP_{\delta}^\uparrow$ from \eqref{eq:P_up}, we have 
        \begin{equation*}
        \cP_\delta^\uparrow(\varphi,\lambda):=\cP^{R^{\uparrow}}(M(w^\uparrow,\lambda)\varphi_{R_{\delta}^\uparrow},w^\uparrow \lambda),
        \end{equation*}
        where the parabolic subgroups $R^\uparrow$ and $R_\delta^\uparrow$ as well as the element $w_{\delta}^\uparrow$ are described in \S\ref{subsubsec:regularized_increasing}. In particular, the intertwining operator $M(w^\uparrow,\lambda)\varphi_{R_{\delta}^\uparrow}$ is regular for $\lambda$ in general position in the affine hyperplane $\Lambda^{-1}(0)$. We may therefore apply Proposition~\ref{prop:order_residues} to compute the residues. This proposition gives us a Rankin--Selberg parabolic subgroup $\widetilde{R}$ and an element $w'$ acting by blocks on $w^\uparrow.R_{\delta}^\uparrow$ such that for $ \lambda \in \Lambda^{-1}(\{0\}) \cap (\fa_{\delta,\cc}^*-\underline{\rho}_\delta-\underline{\rho}_\delta^\uparrow)$ in general position we have
        \begin{equation*}
             \underset{\Lambda}{\Res}\left(\cP^{R^{\uparrow}}(M(w^\uparrow,\mu)\varphi_{R_{\delta}^\uparrow},w^\uparrow \mu) \right)(\lambda)=\cP^{\widetilde{R}}\left(M(w'w^\uparrow,\lambda)\varphi_{R_{\delta}^\uparrow},w'w^\uparrow \lambda\right).
        \end{equation*}
        Let us denote by $R^{',\uparrow}$ and $w^{',\uparrow}$ the parabolic subgroup and element in the Weyl group built for $\delta'$ in \S\ref{subsubsec:regularized_increasing}. Then it follows from the description of $w'$ and $\widetilde{S}$ in Proposition~\ref{prop:order_residues} that there exists $w_\bfM \in W_\bfM(\delta')$ such that $w'w^\uparrow=w_\bfM w^{',\uparrow} w$ and $\widetilde{R}=w_\bfM. R^{',\uparrow}$. With the notation of \S\ref{subsubsec:functional_up}, $w_\bfM$ corresponds to a permutation in $\fS(m_-(\delta'))$ whose purpose is to put the components of $\delta'_-$ in the order prescribed by \eqref{eq:tau_gamma_n} and \eqref{eq:tau_gamma_n+1}. By the functional equation of Lemma~\ref{lem:functional_up}, we conclude that 
        \begin{equation*}
             \underset{\Lambda}{\Res}\left(\cP^{R^{\uparrow}}(M(w^\uparrow,\mu)\varphi_{R_{\delta}^\uparrow},w^\uparrow \mu) \right)(\lambda)=P_{\delta'}^\uparrow(M(w,\lambda)\varphi,w\lambda).
        \end{equation*}
        It remains to do a change of variables in the sum over $\cB_{R,\delta}(J)$ and to use the functional equation of Eisenstein series from Theorem~\ref{thm:analytic_Eisenstein} to conclude.
    \end{proof}

    \subsubsection{First change of contours}
    \label{subsubsec:first_change}

    We can now write the result of the first change of contours of this section. Recall that we have fixed $(I,Q,\tau,I_1,I_2) \in \Pi_H^\uparrow(\pi)$. As in \S\ref{subsubsec:contour_gln}, we take $c>0$ and $k'>0$ such that for all $T>0$ and $(J,R,\delta,J_1,J_2) \in \Pi_{H,\mathrm{c}}^\uparrow(\tau)$ we have
    \begin{equation}
    \label{eq:S_delta_T}
            \left\{ \lambda \in \fa_{\delta,\cc}^* \; | \; \max \Val{\Im(\lambda_i)} \leq T, \; \max \Val{\Re(\lambda_i)} \leq c(T+1)^{-k'} \right\}\subset \cS_{\delta,k,c_J}.
    \end{equation}
    For any $T>0$, we then set
    \begin{equation}
    \label{eq:eta_defi_spectral}
        \eta(T)=4c(2T+1)^{-k'}.
    \end{equation}
    We begin with the following lemma. Recall that our goal is compute the integral $I_{\tau}$ from \eqref{eq:last_I_tau}.

    \begin{lem}
    \label{lem:almost_finished}
        For every $\varepsilon>0$, there exists $T_\varepsilon>0$ such that for every $T>T_\varepsilon$ we have
        \begin{equation}
            \label{eq:almost_finished}
            \Val{I_\tau - \sum_{\Gamma \in \cG_\mathrm{c}(\tau)} \int_{\substack{i \fa_{\tau_{\Gamma}}^*-\underline{\rho}_{\Gamma}-\underline{\rho}_{\Gamma}^\uparrow+(1-\eta(T)) \underline{z}_{\Gamma} \\ \max \Val{\Im(\lambda_i)} \leq T}} \sum_{\varphi \in \cB_{R_{\Gamma},\tau_{\Gamma}}(J)} \cP_{\tau_{\Gamma}}^\uparrow(\varphi,\lambda) \langle f,E(\varphi,-\overline{\lambda}) \rangle_{G} d\lambda   }  \leq \varepsilon.
        \end{equation}
    \end{lem}

    \begin{rem}
        Note that in \eqref{eq:almost_finished} the bounds "$\max \Val{\Im(\lambda_i)} \leq T$" in the domain of integration are necessary. Indeed, if not for them we would leave the region $\cR_{\tau_\Gamma}^\uparrow$ as $\underline{\rho}_{\Gamma}+\underline{\rho}_{\Gamma}^\uparrow-(1-\eta(T)) \underline{z}_{\Gamma}$ does not belong to $\overline{\fa_{R_\Gamma}^{*,+}}$ (see Remark~\ref{rem:border_dispute}).
    \end{rem}

    \begin{proof}
        Let $\varepsilon>0$. The argument is very similar to the one used in Lemma~\ref{lem:residue_n}, so that we only sketch the main steps. For any $T$, let $\underline{\varepsilon}(T) \in \fa_{\tau}^*$ such that $\norm{\underline{\varepsilon}(T)}< (2T+1)^{-k'}/2$, $\underline{\varepsilon}(T)_{n}=0$ and all the coordinates of $\underline{\varepsilon}(T)_{n+1}$ are zero except those in the variables $\lambda(2)_{\mathrm{c},1}, \hdots, \lambda(2)_{\mathrm{c},m_{\mathrm{c},2}}$ which are all different. For any $T$ set
        \begin{equation}
        \label{eq:kappa_tau}
            \kappa_\tau(T)=-\underline{\rho}_\tau-\underline{\rho}_{\tau}^\uparrow+(1+2\eta(T))\underline{z}_{\tau,n}+(1-\eta(T))\underline{z}_{\tau,n+1}+\underline{\varepsilon}(T).
        \end{equation}
        By the estimates of Lemma~\ref{lem:poles_of_the_product2} and change of contour of integration (in the region $\cR_{\tau}^{\uparrow,+}$), we see that we may choose $T_0$ so that for any $T>T_0$ we have
        \begin{equation*}
            \Val{I_\tau - \int_{\substack{i \fa_{\tau}^*+\kappa_\tau(T) \\ \max \Val{\Im(\lambda_i)} \leq T}} \sum_{\varphi \in \cB_{Q,\tau}(J)} \cP_{\tau}^\uparrow(\varphi,\lambda) \langle f,E(\varphi,-\overline{\lambda}) \rangle_{G} d\lambda   }  \leq \varepsilon.
        \end{equation*}
        Note that we are now in $\cR_{\tau}^{\uparrow}$.

        We focus on the second integral. By Fubini's theorem, we can integrate last along $\lambda(1)_{\mathrm{c},1}$. We consider $\gamma$ the contour composed of the two imaginary segments $\Val{\Im(\lambda(1)_{\mathrm{c},1})} \leq T$ along $\Re(\lambda(1)_{\mathrm{c},1})=1/4+\eta(T)/2$ and $\Re(\lambda(1)_{\mathrm{c},1})=1/4-\eta(T)/4$, and of two curves in the region $2T \geq \Val{\Im(\lambda(1)_{\mathrm{c},1})} \geq T$ linking their ends. By Lemma~\ref{lem:poles_of_the_product2}, there are no singularities along this contour and we can bound the integral along the two additional curves. This holds because we have the upper bound \eqref{eq:rapid_decay_product2} and also a lower bound for the product \eqref{eq:thepoles2}, courtesy of our choices of $\eta(T)$ and $\underline{\varepsilon}(T)$. By the computation of the residues in Lemma~\ref{lem:computation_residues_2} and further changes of contours, we end up, for every $T>T_0$, with 
        \begin{equation*}
             \Val{I_\tau - \int_{\substack{i \fa_{\tau}^*+\kappa_\tau(T)-3\eta(T)\underline{z}_{1} \\ \max \Val{\Im(\lambda_i)} \leq T}} \sum_{\varphi \in \cB_{Q,\tau}(J)} \cP_{\tau}^\uparrow(\varphi,\lambda) \langle f,E(\varphi,-\overline{\lambda}) \rangle_{G} d\lambda   -\sum_{\Gamma} I_\Gamma(T)}  \leq C\varepsilon,
        \end{equation*}
        where $C$ is an absolute constant independent from $T$, $\underline{z}_1$ is the element in $\fa_\tau^*$ whose only non zero coordinate is $\lambda(1)_{\mathrm{c},1}=1/4$, the sum ranges over graphs $\Gamma \in \cG_{\mathrm{c}}(\tau)$ who have only one edge of the form $\{ \tau_{\mathrm{c},1,1},\tau_{\mathrm{c},2,j}\}$ and 
        \begin{equation*}
            I_\Gamma(T)=\int_{\substack{i \fa_{\tau_\Gamma}^*+\kappa_\Gamma(T) \\ \max \Val{\Im(\lambda_i)} \leq T}} \sum_{\varphi \in \cB_{R_\Gamma,\tau_\Gamma}(J)} \cP_{\tau_\Gamma}^\uparrow(\varphi,\lambda) \langle f,E(\varphi,-\overline{\lambda}) \rangle_{G} d\lambda,
        \end{equation*}
        $\kappa_\Gamma(T)$ being defined as in \eqref{eq:kappa_tau} for $\tau_{\Gamma}$. Note that the region of integration is included in $\cR_{\delta}^{\uparrow}$ thanks to \eqref{eq:S_delta_T}. We now conclude the proof of Lemma~\ref{lem:almost_finished} as in Lemma~\ref{lem:residue_n}, by doing an induction on the number of variables.
    \end{proof}

    \subsubsection{Second change of contours}
    \label{subsubsec:second_change}

    In this section, we fix $\Gamma \in \cG_{\mathrm{c}}(\tau)$. Let $(J,R,\delta,J_1,J_2)$ be the corresponding tuple in $\Pi_{H,\mathrm{c}}^\uparrow$. We have the tuple $(J_\emptyset,\delta_\emptyset,R_\emptyset) \in \Pi_H$ from \S\ref{subsubsec:empty_transformation}. As in \eqref{eq:underline_z_delta}, we set 
    \begin{equation*}
        \underline{z}_{\delta_\emptyset}:=((0,0,0,1/4,0,0),(0,0,0,1/4,0,0)) \in \fa_{R}^*,
    \end{equation*}
    where the $1/4$ only lives above $\boxtimes_i \delta_{\emptyset,\mathrm{c},1,i} \boxtimes_j \delta_{\emptyset,\mathrm{c},2,j}$.
   As noted in \eqref{eq:empty_relation}, we have 
    \begin{equation}
        \label{eq:effect_of_w} 
        w_\emptyset(\underline{\rho}_\delta+\underline{\rho}_{\delta}^\uparrow-\underline{z}_{\delta})=\underline{\rho}_{\delta_\emptyset}+\underline{\rho}_{\delta_\emptyset}^\uparrow-\underline{z}_{\delta_\emptyset}.
    \end{equation}
    Note that $\underline{\rho}_{\delta_\emptyset}^\uparrow-\underline{z}_{\delta_\emptyset} \in \fa_{\delta_\emptyset,\cc}^*$ in this case.

    We now consider the region
    \begin{equation}
    \label{eq:our_regiondown}
       \cR_{\delta_\emptyset} := \left\{ \lambda -\underline{\rho}_{\delta_\emptyset}-t(\underline{\rho}_{\delta_\emptyset}^\uparrow-\underline{z}_{\delta_\emptyset})\; \middle| \; \lambda \in \fa_{\delta_\emptyset,\cc}^* \cap \cS_{\delta_\emptyset,k,c_J}, 0 \leq t \leq 1
        \right\},
    \end{equation}
    where once again we ask that $k$ and $c_J$ are those prescribed by Theorem~\ref{thm:analytic_Eisenstein} and Proposition~\ref{prop:reg_P_pi_up}. By \eqref{eq:effect_of_w} we have $w_\emptyset \cR_{\delta}^\uparrow \subset \cR_{\delta_\emptyset}$.
    \begin{lem}
    \label{lem:poles_down}
          For $\varphi \in \cB_{R_\emptyset,\delta_\emptyset}(J)$, the possible poles of $\cP_{\delta_\emptyset}^\uparrow(\varphi,\lambda) \langle f,E(\varphi,-\overline{\lambda}) \rangle_{G}$ in the region $\cR_{\delta_\emptyset}$ of \eqref{eq:our_regiondown} are along the zeros of the polynomial
            \begin{equation}
            \prod_{\substack{\delta_{\emptyset,\mathrm{c},1,i} \simeq \delta^{\vee}_{\emptyset,\mathrm{c},2,j}}} \left( \lambda(1)_{i}+\lambda(2)_{j}+\frac{1}{2} \right)
            \label{eq:thepolesdown}
    \end{equation}
    Moreover, it is of rapid decay in the sense that if we denote by $L$ the polynomial in \eqref{eq:thepoles2}, for all $N>0$ there exists $C>0$ such that for all $\lambda$ in the region $\cR_{\delta_\emptyset}$ we have
    \begin{equation}
        \label{eq:rapid_decay_productdown} \sum_{\varphi \in \cB_{R_\emptyset,\delta_\emptyset}}\Val{L(\lambda)\cP_{\delta_\emptyset}^\uparrow(\varphi,\lambda) \langle f,E(\varphi,-\overline{\lambda}) \rangle_{G}} \leq \frac{C}{(1+\norm{\lambda}^2)^N}.
    \end{equation}
    In particular, $\varphi \mapsto \sum_{\varphi}L(\lambda)\cP_{\delta_\emptyset}^\uparrow(\varphi,\lambda) \langle f,E(\varphi,-\overline{\lambda}) \rangle_{G}$ is holomorphic in $\cR_{\delta_\emptyset}$.
    \end{lem}

    \begin{rem}
        \label{rem:L_non_zero}
        Note that the polynomial $L$ from \eqref{eq:thepolesdown} is non-zero in the region $i \fa_{\delta_\emptyset}^* -\underline{\rho}_{\delta_\emptyset}-t(\underline{\rho}_{\delta_\emptyset}^\uparrow-\underline{z}_{\delta_\emptyset})$ as soon as $t<1$. 
    \end{rem}

    \begin{proof}
        This is the same proof as Lemma~\ref{lem:poles_of_the_product2}, and relies on Corollary~\ref{cor:zeros_of_E} (poles of Eisenstein series) and on Proposition~\ref{prop:reg_P_pi_up}.
    \end{proof}
    
    We now set
    \begin{equation}
        \label{eq:I_empty}
        I_{\delta_\emptyset}:=\int_{i \fa_{\delta_\emptyset}^*-\underline{\rho}_{\delta_\emptyset}} \sum_{\varphi \in \cB_{R_\emptyset,\delta_\emptyset}(J)} \cP_{\delta_\emptyset}^\uparrow(\varphi,\lambda) \langle f,E(\varphi,-\overline{\lambda}) \rangle_{G} d\lambda.
    \end{equation}
    This integral is well-defined and absolutely convergent by Lemma~\ref{lem:poles_down}.

    \begin{lem}
    \label{lem:end_of_T}
        For all $\varepsilon>0$ there exists $T>0$ such that
        \begin{equation}
        \label{eq:end_of_T}
             \Val{I_{\delta_\emptyset}-\int_{\substack{i \fa_{\delta}^*-\underline{\rho}_{\delta}-\underline{\rho}_{\delta}^\uparrow+(1-\eta(T)) \underline{z}_{\delta} \\ \max \Val{\Im(\lambda_i)} \leq T}} \sum_{\varphi \in \cB_{R,\delta}(J)} \cP_{\delta}^\uparrow(\varphi,\lambda) \langle f,E(\varphi,-\overline{\lambda}) \rangle_{G} d\lambda  } \leq \varepsilon.
        \end{equation}
    \end{lem}

    \begin{proof}
        By the bounds and localization of the singularities from Lemma~\ref{lem:computation_residues_2}, we see that by change of contours, up to increasing $T$ we can assume that the integral on the right of \eqref{eq:end_of_T} differs in absolute value by a constant times $\varepsilon$ from 
        \begin{equation}
        \label{eq:everything_shifted}
            \int_{\substack{i \fa_{\delta}^*+(1-\eta(T))(\underline{z}_{\delta}-\underline{\rho}_{\delta}-\underline{\rho}_{\delta}^\uparrow) \\ \max \Val{\Im(\lambda_i)} \leq T}} \sum_{\varphi \in \cB_{R,\delta}(J)} \cP_{\delta}^\uparrow(\varphi,\lambda) \langle f,E(\varphi,-\overline{\lambda}) \rangle_{G} d\lambda
        \end{equation}
        Indeed, the key point is that this region of integration remains in $\cR_\delta^\uparrow$ by \eqref{eq:S_delta_T} and that we have a lower bound for the factor \eqref{eq:thepoles2} from Lemma~\ref{lem:poles_of_the_product2}, both information being available to us thanks to our choice of $\eta(T)$ in \eqref{eq:eta_defi_spectral}. Moreover, this estimate will remain true for any $T$ large enough.

        By Lemma~\ref{lem:up_is_empty}, for any $\varphi \in \cB_{R,\delta}(J)$ we have the relation 
        \begin{equation*}
             \cP_{\delta}^{\uparrow}(\varphi,\lambda)=\cP_{\delta_\emptyset}^\uparrow(M(w_\emptyset ,\lambda)\varphi,w_\emptyset\lambda),
        \end{equation*}
        where $w_\emptyset$ was defined in \S\ref{subsubsec:empty_transformation}. By changes of variables and thanks to the relation \eqref{eq:empty_relation}, we see that \eqref{eq:everything_shifted} is equal to
        \begin{equation*}
            \int_{\substack{i \fa_{\delta_\emptyset}^*-(1-\eta(T))(\underline{\rho}_{\delta_\emptyset}+\underline{\rho}_{\delta_\emptyset}^\uparrow-\underline{z}_{\delta_\emptyset}) \\ \max \Val{\Im(\lambda_i)} \leq T}} \sum_{\varphi \in \cB_{R_\emptyset,\delta_\emptyset}(J)} \cP_{\delta_\emptyset}^\uparrow(\varphi,\lambda) \langle f,E(\varphi,-\overline{\lambda}) \rangle_{G} d\lambda.
        \end{equation*}
        Because we are now in the region $\cR_{\delta_\emptyset}$, we can use Lemma~\ref{lem:poles_down} to conclude by changing yet again the contour that, up to increasing $T$, this integral differs in absolute value by a constant times $\varepsilon$ from
         \begin{equation*}
            \int_{\substack{i \fa_{\delta_\emptyset}^*-\underline{\rho}_{\delta_\emptyset}-(1-\eta(T))(\underline{\rho}_{\delta_\emptyset}^\uparrow-\underline{z}_{\delta_\emptyset}) \\ \max \Val{\Im(\lambda_i)} \leq T}} \sum_{\varphi \in \cB_{R_\emptyset,\delta_\emptyset}(J)} \cP_{\delta_\emptyset}^\uparrow(\varphi,\lambda) \langle f,E(\varphi,-\overline{\lambda}) \rangle_{G} d\lambda.
        \end{equation*}
        It remains to do one last change of contour to bring the region of integration from $i \fa_{\delta_\emptyset}^*-\underline{\rho}_{\delta_\emptyset}-(1-\eta(T))(\underline{\rho}_{\delta_\emptyset}^\uparrow-\underline{z}_{\delta_\emptyset})$ to $i \fa_{\delta_\emptyset}^*-\underline{\rho}_{\delta_\emptyset}$, which is possible by Lemma~\ref{lem:poles_down} because we do not cross any poles as noted in Remark~\ref{rem:L_non_zero}. This yields \eqref{eq:end_of_T} and concludes the proof of Lemma~\ref{lem:end_of_T}.
    \end{proof}

    \subsubsection{Going back to the non-increasing period}
    \label{subsubsec:non-increasing}

    We keep our graph $\Gamma \in \cG_{\mathrm{c}}(\tau)$ and $(J,R,\delta,J_1,J_2) \in \Pi_{H,\mathrm{c}}^\uparrow$ the corresponding tuple. Our goal is to express $I_{\delta_\emptyset}$ (defined in \eqref{eq:I_empty}) in terms of the regularized Rankin--Selberg period $\cP$, defined in Chapter~\ref{chap:RS_non_tempered}, rather than $\cP^\uparrow$. To do this, we use the $\delta \mapsto \delta^\downarrow \in \Pi_H$ construction described in \S\ref{subsubsec:connect_periods}.

    \begin{lem}
        \label{lem:empty_down_integral}
        We have 
        \begin{equation*}
            I_{\delta_\emptyset}= \int_{i \fa_{\delta^\downarrow}^*-\underline{\rho}_{\delta^\downarrow}} \sum_{\varphi \in \cB_{R^\downarrow,\delta^\downarrow}} \cP_{\delta^\downarrow}(\varphi,\lambda) \langle f,E(\varphi,-\overline{\lambda}) \rangle_G d \lambda.
        \end{equation*}
    \end{lem}

    \begin{proof}
        This follows from Lemma~\ref{lem:down_is_empty} and some changes of variables.
    \end{proof}

    \subsubsection{Counting the contributions}
    \label{subsubsec:counting_2}

    The last thing that we have to do is to sum over tuples $(J,R,\delta,J_1,J_2) \in \Pi_{H,\mathrm{c}}^\uparrow(\tau)$ rather than graphs $\cG_{\mathrm{c}}(\tau)$. This is the content of the next lemma.

      \begin{lem}
      \label{lem:counting_2}
       The fiber of the map
        \begin{equation*}
            \Gamma \in \cG_{\mathrm{c}}(\tau) \mapsto (J_{\Gamma},R_{\Gamma},\tau_{\Gamma},J_{1,\Gamma},J_{2,\Gamma}) \in \Pi_{H,\mathrm{c}}^\uparrow(\tau)
        \end{equation*}
        above a point in the image $(J,R,\delta,J_1,J_2)$ is of cardinal $|\Stab(\tau)||\Stab(\delta)|^{-1}$.
    \end{lem}

    \begin{proof}
        The proof is exactly the same as that of Lemma~\ref{lem:counting1}.
    \end{proof}

    \subsubsection{End of the proof of Proposition~\ref{prop:this_is_the_end}}
    \label{subsubsec:end_proof_2}
    The proposition now follows from Lemmas~\ref{lem:almost_finished},~\ref{lem:end_of_T} and \ref{lem:empty_down_integral} by taking $\varepsilon \to 0$, and by using the combinatorial result of Lemma~\ref{lem:counting_2}.

    \subsubsection{Conclusion}

    We write out explicitly the consequences of our work for the computation of $\int_{[H]}f(h)dh$. It follows from the combination of Proposition~\ref{prop:almost_final} and Proposition~\ref{prop:this_is_the_end}.

      \begin{prop}
    \label{prop:final}
        We have 
         \begin{align}
             &\int_{[H]}f(h)dh=\sum_{r=0}^n \sum_{(I_r,P,\pi) \in \Pi_H^\uparrow} \sum_{(I,Q,\tau,I_1,I_2) \in \Pi_H^\uparrow(\pi)} \sum_{(J,R,\delta,J_1,J_2) \in \Pi_{H,\mathrm{c}}^\uparrow(\tau)}  \nonumber \\
        \times&\frac{|\Stab(\pi)|}{|W(\pi)||\Stab(\delta)|} \int_{i \fa_{\delta^\downarrow}^*-\underline{\rho}_{\delta^\downarrow}} \sum_{\varphi \in \cB_{R^\downarrow,\delta^\downarrow}(J)} \cP_{\delta^\downarrow}(\varphi,\lambda) \langle f,E(\varphi,-\overline{\lambda}) \rangle_{G} d\lambda. \label{eq:final_eq}
        \end{align}
        where $\Pi_H^\uparrow(\pi)$ is the set of tuples obtained from the null graph built from $w \pi$ with $w$ being any element in $W(\pi)$ such that $\underline{d}(+,1) \geq \hdots \geq \underline{d}(+,m_+)$.
    \end{prop}

    \subsection{End of the proof of Theorem~\ref{thm:pseudo_spectral}}
    \label{subsec:end_proof}

    We can finally conclude the proof of Theorem~\ref{thm:pseudo_spectral}. It is now only a matter of changing the indices in our sums. For the reader's convenience, we restate the result.

    \begin{theorem}
    \label{thm:pseudo_spectral2}
    We have
    \begin{equation}
        \label{eq:final_spectral}
       \int_{[H]}f(h)dh=\sum_{(P,\pi) \in \Pi_H} \frac{1}{\Val{W(\pi)}}  \int_{\lambda \in i \fa_\pi^*} \sum_{\varphi \in \cB_{P,\pi}(J)} \cP_\pi(\varphi,\lambda-\underline{\rho}_\pi) \langle f,E(\varphi,\lambda+\underline{\rho}_\pi) \rangle_G d\lambda,
    \end{equation}
    where the integral on the RHS is absolutely convergent.
\end{theorem}

\begin{proof}
    To prove Theorem~\ref{thm:pseudo_spectral2}, we need to rewrite Proposition~\ref{prop:final}. Let $I_{\delta^\downarrow}$ be the last integral in \eqref{eq:final_eq}. By the functional equation of Eisenstein series (Theorem~\ref{thm:analytic_Eisenstein}) and of $\cP_\pi$ (Corollary~\ref{cor:independence_choice_couple}), for any $w \in W(\delta^\downarrow)$ we have
    \begin{equation}
        \label{eq:function_equation_I}
        I_{w \delta^\downarrow}=I_{\delta^\downarrow}.
    \end{equation}
    Let $0 \leq r \leq n$. Let $\Pi_{H,r}^\uparrow$ be the set of elements $(I_r,P,\pi)$ in $\Pi_H^\uparrow$. We denote by $\Pi_{H,r}^\uparrow / W$ the quotient of $\Pi_{H,r}^\uparrow$ by the relation $(I_r,P_1,\pi_1) \sim (I_r,P_2,\pi_2)$ if $\pi_1=w \pi_2$ for some $w \in W(\pi_2)$. Note that because of the shape of $I_r$ (see \eqref{eq:I_r_defi}), if $(I_r,P,\pi) \in \Pi_H^\uparrow$ and $w \in W(\pi)$, then $(I_r,w.P,w \pi) \in \Pi_H^\uparrow$. Moreover, the set $\Pi_H^\uparrow(\pi)$ is independent of the class of $\pi$.
    
    It follows from \eqref{eq:function_equation_I} that the RHS of \eqref{eq:final_eq} is equal to
    \begin{equation}
        \label{eq:graph_formula}
        \sum_{r=0}^n \sum_{(I_r,\overline{P},\overline{\pi}) \in \Pi_{H,r}^\uparrow / W} \sum_{(I,R,\tau,I_1,I_2) \in \Pi_H^\uparrow(\overline{\pi})} \sum_{(J,R,\delta,J_1,J_2) \in \Pi_{H,\mathrm{c}}^\uparrow(\tau)}\frac{1}{|\Stab(\delta)|} I_{\delta^\downarrow},
    \end{equation}
    where we write again $(I_r,\overline{P},\overline{\pi})$ for a representative of its class. 

    We also define an equivalence relation on $\Pi_H$ by declaring that $\pi_1 \sim \pi_2$ if $\pi_1=w \pi_2$ for some $w \in W(\pi_2)$. Let us denote by $\Pi_H / W$ the quotient. We claim that we have a bijection
    \begin{equation}
    \label{eq:bijec}
       (J,R,\delta,J_1,J_2) \in \left(\bigcup_{r=0}^n \bigcup_{(I_r,\overline{P},\overline{\pi}) \in \Pi_{H,r}^\uparrow / W} \bigcup_{(I,Q,\tau,I_1,I_2) \in \Pi_H^\uparrow(\overline{\pi})} \Pi_{H,\mathrm{c}}^\uparrow(\tau)\right) \mapsto \delta^\downarrow \in \Pi_H / W.
    \end{equation}
    We explain how to build its inverse. We start from $\overline{\Pi} \in \Pi_H / W$. We can choose a representative $\Pi$ so that we have the refinements of the decompositions of \eqref{eq:pi_n} and \eqref{eq:pi_n+1}  
    \begin{align*}
        \Pi_n&=\pi_+\boxtimes \left(\pi_{1,\cusp} \boxtimes \pi_{1,\mathrm{res}} \right) \boxtimes \pi_{2,\mathrm{res}}^{-,\vee} \boxtimes\left(\pi_{-,\cusp} \boxtimes \pi_{-,\mathrm{res}} \right)  , \\
        \Pi_{n+1}&=\pi_+^\vee \boxtimes \left(\pi_{2,\cusp} \boxtimes \pi_{2,\mathrm{res}} \right) \boxtimes \pi_{1,\mathrm{res}}^{-,\vee} \boxtimes \left(\pi_{-,\cusp}^\vee \boxtimes \pi_{-,\mathrm{res}}^\vee \right),
    \end{align*}
    where the representations with a "$\cusp$" index are cuspidal, while those with a "$\mathrm{res}$" index are residual. Moreover, we may assume that inside each $\pi_{\hdots}$, the blocks are ordered increasingly. For example, this means that if we write
    \begin{equation*}
        \pi_{+}=\boxtimes_{i=1}^{m_+} \Speh(\sigma_{+,i},\underline{d}(+,i)),
    \end{equation*}
    then we ask that $\underline{d}(+,1) \geq \hdots \geq \underline{d}(+,m_+)$. We now define $\pi_{n,\mathrm{res}}$ and $\pi_{n+1,\mathrm{res}}$ to be the unique reordering of the blocks of $\pi_{1,\mathrm{res}} \boxtimes \pi_{-,\mathrm{res}}$ and $\pi_{2,\mathrm{res}} \boxtimes \pi_{-,\mathrm{res}}^\vee$ respectively, such that, if we write
    \begin{equation}
        \label{eq:speh_decompo}
        \pi_{n,\mathrm{res}}=\boxtimes_{i=1}^{m_n} \Speh(\sigma_{n,i},\underline{d}(n,i)), \quad \pi_{n+1,\mathrm{res}}=\boxtimes_{i=1}^{m_{n+1}} \Speh(\sigma_{n+1,i},\underline{d}(n+1,i)),
    \end{equation}
    then $\underline{d}(n,1) \geq \hdots \geq \underline{d}(n,m_n)$ and $\underline{d}(n,1) \geq \hdots \geq \underline{d}(n,m_{n+1})$, and moreover such that the respective blocks of $\pi_{1,\mathrm{res}}$, $\pi_{-,\mathrm{res}}$, $\pi_{2,\mathrm{res}}$ and $\pi_{-,\mathrm{res}}^\vee$ remain in the same order.
    
    It follows from our construction that $\Pi=\delta^\downarrow$ for $(J,R,\delta,J_1,J_2) \in \Pi_H^\uparrow$ which is defined by
    \begin{align}
        \delta_n& = \pi_+ \boxtimes \pi_{2,\mathrm{res}}^{-,\vee} \boxtimes \pi_{n,\mathrm{res}} \boxtimes \pi_{1,\cusp} \boxtimes \pi_{-,\cusp},\label{eq:delta_n_final} \\
        \delta_{n+1}& = \pi_+^\vee \boxtimes \pi_{1,\mathrm{res}}^{-,\vee} \boxtimes \pi_{n+1,\mathrm{res}} \boxtimes \pi_{2,\cusp} \boxtimes \pi_{-,\cusp}^\vee, \label{eq:delta_n+1_final}.
    \end{align}
    Here $J$ is defined in the obvious way and the sets $J_1$ and $J_2$ correspond to the indices of the discrete automorphic representations $\pi_{-,\mathrm{res}}$ and $\pi_{-,\mathrm{res}}^\vee$ in $\pi_{n,\mathrm{res}}$ and $\pi_{n+1,\mathrm{res}}$ respectively .
    
    We then have $(J,R,\delta,J_1,J_2) \in \Pi_{H,\mathrm{c}}^\uparrow(\tau)$ where $(I,Q,\tau,I_1,I_2)$ is also defined by \eqref{eq:delta_n_final} and \eqref{eq:delta_n+1_final}, up to grouping $\pi_{1,\cusp} \boxtimes  \pi_{-,\cusp}$ and $\pi_{2,\cusp} \boxtimes  \pi_{-,\cusp}^\vee$ together. 
    
    With the notation of \eqref{eq:speh_decompo}, we set $\sigma_n=\boxtimes_i \sigma_{n,i}$ and $\sigma_{n+1}=\boxtimes_i \sigma_{n+1,i}$. Now $(I,Q,\tau,I_1,I_2) \in \Pi_H^\uparrow(\pi)$ with $\pi$ satisfying
      \begin{align*}
        \pi_n&=\left(\pi_+ \boxtimes \pi_{2,\mathrm{res}}^{-,\vee} \boxtimes \pi_{n,\mathrm{res}}^- \right) \boxtimes \left( \sigma_n \boxtimes \pi_{1,\cusp} \boxtimes \pi_{-,\cusp}\right),\\
        \pi_{n+1}&=\left(\pi_+^\vee \boxtimes \pi_{1,\mathrm{res}}^{-,\vee} \boxtimes \pi_{n+1,\mathrm{res}}^- \right) \boxtimes \left( \sigma_{n+1} \boxtimes \pi_{2,\cusp} \boxtimes \pi_{-,\cusp}^\vee\right).
    \end{align*}
    Then $(I_r,P,\pi) \in \Pi_H^\uparrow$ for some $0 \leq r \leq n$ and some standard parabolic subgroup $P$ of $G$, so that the map $\overline{\Pi} \mapsto (J,R,\delta,J_1,J_2)$ lands in the union of \eqref{eq:bijec}. By construction, it is the inverse of $(J,R,\delta,J_1,J_2) \mapsto \delta^\downarrow$, so that \eqref{eq:bijec} is indeed a bijection.

    Finally, we note that the unions in \eqref{eq:bijec} are disjoint. This can be proved by similar arguments as the ones we just used. We then see that \eqref{eq:final_spectral} follows from \eqref{eq:graph_formula} and the orbit-counting theorem.
\end{proof}

\printbibliography

\begin{flushleft}
Paul Boisseau \\
Max Planck Institute for Mathematics, \\
Vivatsgasse 7, \\
53111 Bonn, Germany
\medskip
	
email:\\
boisseau@mpim-bonn.mpg.de \\
\end{flushleft}

\end{document}